\documentclass[11pt,draft]{amsart}

\usepackage{a4,cite,amsfonts,amsmath,amsthm,amscd,amssymb,latexsym}

\newtheorem{theorem}{Theorem}[section]

\newtheorem{proposition}[theorem]{Proposition}

\newtheorem{lemma}[theorem]{Lemma}

\newtheorem{corollary}[theorem]{Corollary}

\theoremstyle{definition}

\newtheorem{example}[theorem]{Example}

\newtheorem{question}[theorem]{Question}
 
\newtheorem{remark}[theorem]{Remark}

\DeclareMathOperator{\con}{con}

\DeclareMathOperator{\ini}{ini}

\DeclareMathOperator{\mul}{mul}

\DeclareMathOperator{\occ}{occ}

\DeclareMathOperator{\simple}{sim}

\DeclareMathOperator{\var}{var}

\numberwithin{equation}{section}

\numberwithin{figure}{section}

\numberwithin{table}{section}

\makeatletter

\renewcommand*\subjclass[2][2010]{\def\@subjclass{#2}\@ifundefined{subjclassname@#1}{\ClassWarning{\@classname}{Unknown edition (#1) of Mathematics Subject Classification; using '2010'.}}{\@xp\let\@xp\subjclassname\csname subjclassname@#1\endcsname}}

\makeatother

\renewcommand{\subjclassname}{\textup{2010} Mathematics Subject Classification}

\newcommand{\contsection}[2]{\ref{#1}. #2\dotfill\pageref{#1}}

\newcommand{\contsectionstar}[2]{#2\dotfill\pageref{#1}}

\begin{document}

\title{Chain varieties of monoids}

\thanks{The work is partially supported by Russian Foundation for Basic Research (grant 17-01-00551) and by the Ministry of Education and Science of the Russian Federation (project 1.6018.2017/8.9).}

\author{S.\,V.\,Gusev}

\address{Ural Federal University, Institute of Natural Sciences and Mathematics, Lenina 51, 620000 Ekaterinburg, Russia}

\email{sergey.gusb@gmail.com,\,bvernikov@gmail.com}

\author{B.\,M.\,Vernikov}

\begin{abstract}
A variety of universal algebras is called a chain variety if its subvariety lattice is a chain. Non-group chain varieties of semigroups were completely classified by Sukhanov in 1982. Here we completely determine non-group chain varieties of monoids as algebras of type $(2,0)$.
\end{abstract}

\keywords{Monoid, variety, lattice of varieties, chain}

\subjclass{Primary 20M07, secondary 08B15}

\maketitle

\section*{Contents}

{\small
\begin{itemize}
\item[]\contsection{introduction}{Introduction and summary}
\item[]\contsection{preliminaries}{Preliminaries}
\item[]\contsection{decomposition}{$k$-decomposition of a word and related notions}
\item[]\contsection{necessity}{The proof of the ``only if part''}
\begin{itemize}
\item[]\contsection{necessity: red to over D_2}{Reduction to the case when $\mathbf D_2\subseteq\mathbf V$}
\item[]\contsection{necessity: red to over L}{Reduction to the case when $\mathbf L\subseteq\mathbf V$}
\item[]\contsection{necessity: over L}{The case when $\mathbf L\subseteq\mathbf V$}
\end{itemize}
\item[]\contsection{sufficiency: not K}{The proof of the ``if part'': all varieties except \textbf K}
\item[]\contsection{sufficiency: K}{The proof of the ``if part'': the variety \textbf K}
\begin{itemize}
\item[]\contsection{sufficiency: K - red to [E,K]}{Reduction to the interval $[\mathbf E,\mathbf K]$}
\item[]\contsection{sufficiency: K - aux}{Several auxiliary results}
\begin{itemize}
\item[]\contsection{sufficiency: K - aux-FHIJK}{Some properties of the varieties $\mathbf F_k$, $\mathbf H_k$, $\mathbf I_k$, $\mathbf J_k^m$, $\mathbf K$\\
\phantom{\ref{sufficiency: K - aux-FHIJK}.\ \,}and their identities}
\item[]\contsection{sufficiency: K - aux-decompositions}{$k$-decompositions of sides of the identities $\alpha_k$, $\beta_k$, $\gamma_k$\\
\phantom{\ref{sufficiency: K - aux-decompositions}\ \ }and $\delta_k^m$}
\item[]\contsection{sufficiency: K - aux-swapping}{Swapping letters within $k$-blocks}
\end{itemize}
\item[]\contsection{sufficiency: K - red to [F_k,F_{k+1}]}{Reduction to intervals of the form $[\mathbf F_k,\mathbf F_{k+1}]$}
\item[]\contsection{sufficiency: K - structure of [F_k,F_{k+1}]}{Structure of the interval $[\mathbf F_k,\mathbf F_{k+1}]$}
\begin{itemize}
\item[]\contsection{structure of [F_k,F_{k+1}] 1 step}{If $\mathbf F_k\subset\mathbf X\subseteq\mathbf F_{k+1}$ then $\mathbf H_k\subseteq\mathbf X$}
\item[]\contsection{structure of [F_k,F_{k+1}] 2 step}{If $\mathbf H_k\subset\mathbf X\subseteq\mathbf F_{k+1}$ then $\mathbf I_k\subseteq\mathbf X$}
\item[]\contsection{structure of [F_k,F_{k+1}] 3 step}{If $\mathbf I_k\subset\mathbf X\subseteq\mathbf F_{k+1}$ then $\mathbf J_k^1\subseteq\mathbf X$}
\item[]\contsection{structure of [F_k,F_{k+1}] 4 step}{If $\mathbf J_k^m\subset\mathbf X\subseteq\mathbf F_{k+1}$ with $1\le m<k$ then $\mathbf J_k^{m+1}\subseteq\mathbf X$}
\item[]\contsection{structure of [F_k,F_{k+1}] 5 step}{The interval $[\mathbf J_k^k, \mathbf F_{k+1}]$ consists of $\mathbf J_k^k$ and $\mathbf F_{k+1}$ only}
\item[]\contsection{structure of [F_k,F_{k+1}] 6 step}{All inclusions are strict}
\end{itemize}
\end{itemize}
\item[]\contsection{corollaries}{Corollaries}
\item[]\contsectionstar{acknowledg}{Acknowledgments}
\item[]\contsectionstar{bibl}{\refname}
\end{itemize}
}

\section{Introduction and summary}
\label{introduction}
 
There are many articles devoted to the examination of the lattice \textbf{SEM} of all semigroup varieties. An overview of this area is contained in the detailed survey~\cite{Shevrin-Vernikov-Volkov-09}; see also the recent work~\cite{Vernikov-15} devoted to elements of \textbf{SEM} satisfying some special properties. In sharp contrast, the lattice \textbf{MON} of all monoid varieties has received much less attention over the years; when referring to monoid varieties, we consider monoids as algebras with an associative binary operation and the nullary operation that fixes the identity element. As far as we know, there are only three papers containing substantial results on this subject. We have in mind the article~\cite{Head-68} where the lattice of commutative monoid varieties is completely described, the article~\cite{Wismath-86} which contains a complete description of the lattice of band monoid varieties, and the article~\cite{Pollak-81} where an example of a monoid variety without covers in the lattice \textbf{MON} is found.
 
Recently, the situation began to change gradually. The papers~\cite{Jackson-05,Jackson-Lee-17+,Lee-12a,Lee-12b,Lee-13,Lee-14a,Lee-14b} are devoted principally to an examination of identities of monoids but contain also some intermediate results about lattices of varieties. Moreover, the article~\cite{Jackson-Lee-17+} contains some results about the lattice \textbf{MON} that are of undoubted independent interest.
 
Thus nowadays, interest in the lattice \textbf{MON} has grown. Nevertheless, many questions in this area still remain open. For example, it is known that the lattice \textbf{MON} is not modular (see, e.g.,~\cite[Proposition~4.1]{Lee-12a} or Fig.~\ref{two lattices}b) below), but it was unknown up to the recent time whether this lattice satisfies some non-trivial identity. Only recently the first author gave a negative answer to this question~\cite{Gusev-18}. In contrast, the fact that the lattice \textbf{SEM} does not satisfy any non-trivial lattice identity is known since the beginning of 1970's~\cite{Burris-Nelson-71a,Burris-Nelson-71b}.
 
The problem of describing monoid varieties with modular or even distributive subvariety lattice seems to be quite difficult. As a first step on this direction, it seems natural to consider the extreme strengthening of the distributive law, namely the property of being a chain. Varieties whose subvariety lattice is a chain are called \emph{chain} varieties. Non-group chain varieties of semigroups were listed by Sukhanov in~\cite{Sukhanov-82} (see Fig.~\ref{all chain sem} in Section~\ref{corollaries} below), while locally finite chain group varieties were completely determined by Artamonov in~\cite{Artamonov-78}. Note that the problem of completely describing arbitrary chain varieties of groups seems to be extremely difficult. This is confirmed by the fact that there are uncountably many periodic non-locally finite varieties of groups with the 3-element subvariety lattice~\cite{Kozhevnikov-12}.
 
Some non-trivial examples of chain varieties of monoids appeared in~\cite{Jackson-05,Lee-12a,Lee-14a}. We introduce here one of these examples. To do this, we need some notation. We denote by $F$ the free semigroup over a countably infinite alphabet $A$. Elements of both $F$ and $A$ are denoted by small Latin letters. However, elements of $F$ unlike elements of $A$ are written in bold. As usual, elements of $F$ and $A$ are called \emph{words} and \emph{letters} respectively. The symbol $F^1$ stands for the semigroup $F$ with a new identity element adjoined. We treat this identity element as the empty word and denote it by $\lambda$. The following notion was introduced by Perkins~\cite{Perkins-69} and often appeared in the literature (see~\cite{Jackson-05,Jackson-Lee-17+,Jackson-Sapir-00,Lee-12a,Lee-14a}, for instance; in~\cite[Remark~2.4]{Jackson-Lee-17+} there is a number of other references). Let $W$ be a set of possibly empty words. We denote by $\overline W$ the set of all subwords of words from $W$ and by $I\bigl(\,\overline W\,\bigr)$ the set $F^1 \setminus \overline W$. It is clear that $I\bigl(\,\overline W\,\bigr)$ is an ideal of $F^1$. Then $S(W)$ denotes the Rees quotient monoid $F^1/I\bigl(\,\overline W\,\bigr)$. If $W=\{\bf w\}$ then we will write $S(\bf w)$ rather than $S\bigl(\{\bf w\}\bigr)$. It is verified in~\cite[Lemmas~4.4 and~5.10]{Jackson-05} that the variety generated by the monoid $S(xzxyty)$  is a chain variety. Besides that, this variety turns out to be non-finitely based~\cite[Lemma~5.5]{Jackson-05}.

However, chain monoid varieties were not studied systematically so far. In this paper we obtain a complete description of non-group chain varieties of monoids. In order to formulate the main result of the article, we need some new notation. Two sides of identities we connect by the symbol~$\approx$, while the symbol~$=$ denotes the equality relation on $F^1$. One can introduce notation for the following three identities:
\begin{align*}
\sigma_1:&\enskip xyzxty\approx yxzxty,\\
\sigma_2:&\enskip xtyzxy\approx xtyzyx,\\
\gamma_1:&\enskip y_1y_0x_1y_1x_0x_1 \approx y_1y_0y_1x_1x_0x_1.
\end{align*}
Note that the identities $\sigma_1$ and $\sigma_2$ are dual to each other. The identity $\gamma_1$ belongs to a countably infinite series of identities $\gamma_k$ that will be defined in Subsection~\ref{sufficiency: K - red to [E,K]}. For an identity system $\Sigma$, we denote by $\var\,\Sigma$ the variety of monoids given by $\Sigma$. Let us fix notation for the following varieties:
\begin{align*}
&{\bf C}_n=\var\{x^n\approx x^{n+1},\,xy\approx yx\}\text{ where }n\ge2,\\
&{\bf D}=\var\{x^2\approx x^3,\,x^2y\approx yx^2,\,\sigma_1,\,\sigma_2,\,\gamma_1\},\\
&{\bf K}=\var\{xyx\approx xyx^2,\,x^2y^2\approx y^2x^2,\,x^2y\approx x^2yx\},\\
&{\bf LRB}=\var\{xy\approx xyx\},\\
&{\bf N}=\var\{x^2y\approx yx^2,\,x^2yz\approx xyxzx,\,\sigma_2,\,\gamma_1\},\\
&{\bf RRB}=\var\{yx\approx xyx\}.
\end{align*}
To define one more variety, we need some additional notation. For an arbitrary natural number $n$, we denote by $S_n$ the full symmetric group on the set $\{1,2,\dots,n\}$. For arbitrary permutations $\pi,\tau\in S_n$, we put
\begin{align*}
\mathbf w_n(\pi,\tau)&=\biggl(\prod_{i=1}^n z_it_i\biggr) x \biggl(\prod_{i=1}^n z_{\pi(i)}z_{n+\tau(i)}\biggr) x \biggl(\prod_{i=n+1}^{2n} t_iz_i\biggr),\\[-3pt]
\mathbf w_n'(\pi,\tau)&=\biggl(\prod_{i=1}^n z_it_i\biggr) x^2 \biggl(\prod_{i=1}^{n} z_{\pi(i)}z_{n+\tau(i)}\biggr)\biggl(\prod_{i=n+1}^{2n} t_iz_i\biggr).
\end{align*}
Note that the words $\mathbf w_n(\pi,\tau)$ and $\mathbf w_n'(\pi,\tau)$ with the trivial permutations $\pi$ and $\tau$ appeared earlier in~\cite[proof of Proposition~5.5]{Jackson-05}. Put
$$
{\bf L}\!=\!\var\{x^2y\!\approx\!yx^2,xyxzx\!\approx\!x^2yz,\sigma_1,\sigma_2,{\bf w}_n(\pi,\tau)\!\approx\!{\bf w}'_n(\pi,\tau)\!\mid\!n\!\in\!\mathbb N,\pi,\tau\!\in\!S_n\}.
$$
If ${\bf X}$ is a monoid variety then we denote by $\overleftarrow{{\bf X}}$ the variety \emph{dual to} \textbf X, i.e. the variety consisting of monoids antiisomorphic to monoids from ${\bf X}$.
 
The main result of the paper is the following
 
\begin{theorem}
\label{main result}
A non-group monoid variety is a chain variety if and only if it is contained in one of the varieties ${\bf C}_n$ for some $n\ge2$, ${\bf D}$, ${\bf K}$, $\overleftarrow{{\bf K}}$, ${\bf L}$, ${\bf LRB}$, ${\bf N}$, $\overleftarrow{{\bf N}}$ and ${\bf RRB}$.
\end{theorem}

As we will see below, the variety \textbf L is generated by the monoid $S(xzxyty)$ (see Lemma~\ref{L = var S(xzxyty)}). In view of the results of~\cite{Jackson-05}  mentioned above, the variety \textbf L is non-finitely based. Our Theorem implies that \textbf L is the unique non-finitely based non-group chain variety of monoids (see Corollary~\ref{list} below). For comparison, we note that all non-group chain semigroup varieties and locally finite chain group varieties are finitely based. This follows from the above-mentioned results of~\cite{Artamonov-78,Sukhanov-82}. Note also that following the above-mentioned result of~\cite{Kozhevnikov-12}, there exist non-finitely based non-locally finite chain varieties of groups. But explicit examples of such varieties have not yet been specified anywhere.

The complete list of all non-group chain varieties of monoids will be given in Corollary~\ref{list} below. The unique non-finitely based non-group chain variety of monoids mentioned above is the variety \textbf L (see Corollary~\ref{L is limit} below).
 
A minimal non-chain variety is called a \emph{just non-chain} variety. It is noted in~\cite[Corollary~2]{Sukhanov-82} that, among non-group varieties in \textbf{SEM}, any chain variety is contained in some maximal chain variety and any non-chain variety contains some just non-chain subvariety. However, similar results do not hold for non-group varieties in \textbf{MON}. Specifically, the varieties ${\bf C}_3$, $\mathbf C_4$, \dots are not contained in any maximal chain variety (see Fig.~\ref{all chain mon} in Section~\ref{corollaries}), while it follows from Theorem~\ref{main result} that there is a non-chain variety of monoids that does not contain any just non-chain subvariety (see Corollary~\ref{does not contain just non-chain} below).
 
In~\cite{Sukhanov-82} non-group chain varieties of semigroups were described in two ways. The first one is a description in the identity language. Theorem~\ref{main result} is an analogue of this result in the case of monoids. The second way is by presenting the full list of non-group just non-chain varieties of semigroups; this gives a characterization of chain varieties because, in view of~\cite[Corollary~2]{Sukhanov-82}, a non-group variety of semigroups is a chain variety if and only if it does not contain any just non-chain subvariety. As we have mentioned in the preceding paragraph, an analogous claim is false for monoids. Therefore, the second way of describing chain varieties is not applicable in the case of monoids. Due to this reason, we do not consider just non-chain monoid varieties here.
 
The article consists of seven sections. Section~\ref{preliminaries} contains definitions, notation and auxiliary results. In Section~\ref{decomposition} we introduce a series of new notions and notation and prove a number of results of technical character concerning these notions. These notions and results play a valuable role in the proof of Theorem~\ref{main result}. Section~\ref{necessity} is devoted to the proof of the ``only if'' part of Theorem~\ref{main result}, while the ``if'' part is verified in Sections~\ref{sufficiency: not K} and~\ref{sufficiency: K}. Finally, in Section~\ref{corollaries} some corollaries of Theorem~\ref{main result} and its proof are established.
 
\section{Preliminaries}
\label{preliminaries}
 
A word is called a \emph{semigroup} one if it does not contain the symbol of nullary operation~1. An identity is called a \emph{semigroup} one if both its sides are semigroup words. Note that an identity of the form ${\bf w}\approx 1$ is equivalent to the pair of identities ${\bf w}x\approx x{\bf w} \approx x$ where the letter $x$ does not occur in the word $\bf w$. Further, any monoid satisfies the identities ${\bf u} \cdot 1 \approx 1\cdot\mathbf u\approx\mathbf u$ for any word $\bf u$. These observations allow us to assume that all identities that appear below are semigroup ones.
 
The \emph{content} of a word \textbf w, i.e., the set of all letters occurring in $\bf w$, is denoted by $\con({\bf w})$. We denote by \textbf{SL} the variety of all semilattice monoids. The following statement is well known in fact. But it never appeared anywhere in this form, as far as we know. For the sake of completeness, we give its proof here.
 
\begin{lemma}
\label{group variety}
For a monoid variety $\mathbf V$, the following are equivalent:
\begin{itemize}
\item[\textup{a)}] $\mathbf V$ is a group variety;
\item[\textup{b)}] $\mathbf V$ satisfies an identity ${\bf u}\approx {\bf v}$ with $\con({\bf u})\ne \con({\bf v})$;
\item[\textup{c)}] $\mathbf{SL\nsubseteq V}$.
\end{itemize}
\end{lemma}
 
\begin{proof}
The implication a)\,$\longrightarrow$\,c) is obvious.
 
\smallskip
 
The implication c)\,$\longrightarrow$\,b) immediately follows from the evident fact that the variety \textbf{SL} satisfies any identity $\mathbf{u\approx v}$ with $\con({\bf u})=\con({\bf v})$.
 
\smallskip
 
b)\,$\longrightarrow$\,a) By the hypothesis, there is a letter $x$ that occurs in precisely one of the words \textbf u and \textbf v. Let $y$ be a letter with $y\notin \con({\bf uv})$. Clearly, the identities ${\bf u}y\approx {\bf v}y$ and $y\mathbf u\approx y\mathbf v$ hold in ${\bf V}$. One can substitute~1 for all letters occurring in these identities except $x$ and $y$. Then we obtain ${\bf V}$ satisfies $x^ny\approx y$ and $yx^n\approx y$ for some $n$. Hence ${\bf V}$ is a group variety.
\end{proof}
 
A letter is called \emph{simple} [\emph{multiple}] \emph{in a word} $\bf w$ if it occurs in $\bf w$ once [at least twice]. The set of all simple [multiple] letters in a word \textbf w is denoted by $\simple(\mathbf w)$ [respectively $\mul(\mathbf w)$]. The following statement is well known and can be easily verified.
 
\begin{proposition}
\label{word problem C_2}
A non-trivial identity ${\bf u}\approx {\bf v}$ holds in the variety ${\bf C}_2$ if and only if the claim
\begin{equation}
\label{sim(u)=sim(v) & mul(u)=mul(v)}
\simple(\mathbf u)=\simple(\mathbf v)\text{ and }\mul(\mathbf u)=\mul(\mathbf v)
\end{equation}
is true.\qed
\end{proposition}
 
A word \textbf w is called an \emph{isoterm} for a class of semigroups if no semigroup in the class satisfies any non-trivial identity of the form $\mathbf w\approx\mathbf w'$. The following statement is known in fact and plays an important role below.
 
\begin{lemma}
\label{S(W) in V}
Let ${\bf V}$ be a monoid variety and $W$ a set of possibly empty words. Then $S(W)$ lies in ${\bf V}$ if and only if each word in $W$ is an isoterm for ${\bf V}$.
\end{lemma}
 
\begin{proof}
It is easy to verify that it suffices to consider the case when $W$ consists of one word (see the paragraph after Lemma~3.3 in~\cite{Jackson-05}). Then necessity is obvious, while sufficiency is proved in~\cite[Lemma~5.3]{Jackson-Sapir-00}.
\end{proof}
 
The variety generated by a monoid $M$ is denoted by $\var M$.
 
\begin{lemma}[\mdseries{\cite[Corollary~6.1.5]{Almeida-94}}]
\label{C_{n+1}=var S(x^n)}
$\mathbf C_{n+1}=\var S(x^n)$ for any natural $n$.\qed
\end{lemma}
 
\begin{lemma}
\label{C_{n+1} nsubseteq V}
Let ${\bf V}$ be a monoid variety and $n$ a natural number. If ${\bf C}_{n+1} \nsubseteq {\bf V}$ then ${\bf V}$ satisfies an identity $x^n\approx x^{n+m}$ for some $m$.
\end{lemma}
 
\begin{proof}
We can assume that $\mathbf V$ is not a group variety because the required conclusion is evident otherwise. Lemmas~\ref{S(W) in V} and~\ref{C_{n+1}=var S(x^n)} apply with the conclusion that the variety ${\bf V}$ satisfies a non-trivial identity of the form $x^n\approx {\bf w}$. Then $\con({\bf w})= \{x\}$ by Lemma~\ref{group variety}, whence ${\bf w}= x^k$ for some $k\ne n$. Clearly, the identity $x^n\approx x^k$ implies an identity $x^n\approx x^{n+m}$ for some $m$. Thus, the variety ${\bf V}$ satisfies the identity $x^n\approx x^{n+m}$.
\end{proof}
 
As in the case of semigroups, a variety of monoids is called \emph{completely regular} if it consists of \emph{completely regular monoids}~(i.e., unions of groups). It is well known that a variety is completely regular if and only if it satisfies an identity $x\approx x^{m+1}$ for some $m$. This observation, together with Lemma~\ref{C_{n+1} nsubseteq V} and the evident fact that the variety $\mathbf C_2$ is non-completely regular, implies the following
 
\begin{corollary}
\label{cr or over C_2}
A monoid variety ${\bf V}$ is completely regular if and only if ${\bf C}_2 \nsubseteq {\bf V}$.\qed
\end{corollary}
 
For any natural number $k$, we denote by $\mathbf D_k$ the subvariety of the variety \textbf D given within \textbf D by the identity $x^2y_1y_2\cdots y_k\approx xy_1xy_2x\cdots xy_kx$. The proof of Proposition~4.1 in~\cite{Lee-14a} implies the following
 
\begin{lemma}
\label{D_n=var S(w)}
$\mathbf D_1=\var S(xy)$ and ${\bf D}_{n+1} = \var S(xy_1xy_2x\cdots xy_nx)$ for any natural $n$.\qed
\end{lemma}
 
We denote by \textbf T the trivial variety of monoids. The subvariety lattice of a monoid variety \textbf X is denoted by $L(\mathbf X)$. Proposition~4.1 of~\cite{Lee-14a} and its proof readily imply also the following
 
\begin{lemma}
\label{L(D)}
The lattice $L({\bf D})$ is the chain $\mathbf{T\subset SL\subset C}_2\subset\mathbf D_1\subset\mathbf D_2\subset\cdots\subset\mathbf D_k\subset\cdots\subset\mathbf D$.\qed
\end{lemma}
 
The following statement immediately follows from~\cite[Proposition~4.7]{Wismath-86}.
 
\begin{lemma}
\label{L(BM)}
\begin{itemize}
\item[\textup{(i)}] Every variety of band monoids either contains the variety $\mathbf{LRB\vee RRB}$ or is contained in this variety.
\item[\textup{(ii)}] The lattice $L(\bf LRB \vee RRB)$ has the form shown in Fig.~\textup{\ref{two lattices}a)}.\qed
\end{itemize}
\end{lemma}
 
Put
$$
\mathbf E=\var\{x^2\approx x^3,\,x^2y\approx xyx,\,x^2y^2\approx y^2x^2\}.
$$
The following lemma is verified in~\cite[Proposition~4.1(i) and Lemma~3.3(iv)]{Lee-12a}.
 
\begin{lemma}
\label{L(LRB+C_2)}
\begin{itemize}
\item[\textup{(i)}] $\mathbf{LRB}\vee\mathbf C_2=\var\{x^2\approx x^3,\,x^2y\approx xyx\}$.
\item[\textup{(ii)}] The lattice $L({\bf LRB}\vee\mathbf C_2)$ has the form shown in Fig.~\textup{\ref{two lattices}b)}.\qed
\end{itemize}
\end{lemma}
 
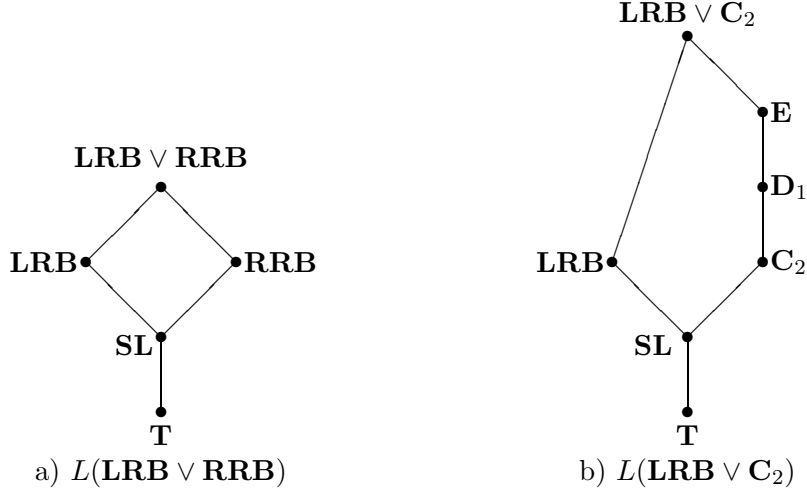
\begin{figure}[htb]
\unitlength=1mm
\linethickness{0.4pt}
\begin{center}
\begin{picture}(98,61)
\put(5,28){\circle*{1.33}}
\put(15,8){\circle*{1.33}}
\put(15,18){\circle*{1.33}}
\put(15,38){\circle*{1.33}}
\put(25,28){\circle*{1.33}}
\put(15,8){\line(0,1){10}}
\put(15,18){\line(1,1){10}}
\put(15,18){\line(-1,1){10}}
\put(5,28){\line(1,1){10}}
\put(25,28){\line(-1,1){10}}
\put(4,28){\makebox(0,0)[rc]{${\bf LRB}$}}
\put(15,42){\makebox(0,0)[cc]{${\bf LRB} \vee {\bf RRB}$}}
\put(26,28){\makebox(0,0)[lc]{${\bf RRB}$}}
\put(14,17){\makebox(0,0)[rc]{${\bf SL}$}}
\put(15,5){\makebox(0,0)[cc]{$\bf T$}}
\put(15,0){\makebox(0,0)[cc]{a) $L(\mathbf{LRB\vee RRB})$}}
\put(85,8){\circle*{1.33}}
\put(85,18){\circle*{1.33}}
\put(75,28){\circle*{1.33}}
\put(95,28){\circle*{1.33}}
\put(95,38){\circle*{1.33}}
\put(95,48){\circle*{1.33}}
\put(85,58){\circle*{1.33}}
\put(85,8){\line(0,1){10}}
\put(95,28){\line(0,1){20}}
\put(85,18){\line(1,1){10}}
\put(85,18){\line(-1,1){10}}
\put(75,28){\line(1,3){10}}
\put(95,48){\line(-1,1){10}}
\put(96,28){\makebox(0,0)[lc]{${\bf C}_2$}}
\put(96,38){\makebox(0,0)[lc]{${\bf D}_1$}}
\put(96,48){\makebox(0,0)[lc]{${\bf E}$}}
\put(85,61){\makebox(0,0)[cc]{${\bf LRB}\vee\mathbf C_2$}}
\put(74,28){\makebox(0,0)[rc]{${\bf LRB}$}}
\put(83,17){\makebox(0,0)[rc]{${\bf SL}$}}
\put(85,5){\makebox(0,0)[cc]{$\bf T$}}
\put(85,0){\makebox(0,0)[cc]{b) $L(\mathbf{LRB}\vee\mathbf C_2)$}}
\end{picture}
\end{center}
\caption{The lattices $L(\mathbf{LRB\vee RRB})$ and $L(\mathbf{LRB}\vee\mathbf C_2)$}
\label{two lattices}
\end{figure}
 
Let \textbf w be a word and $x$ a letter. We denote by $\occ_x({\bf w})$ the number of occurrences of $x$ in ${\bf w}$. If $x\in \con({\bf w})$ and $i\le \occ_x({\bf w})$ then $\ell_i(\mathbf w,x)$ denotes the length of the minimal prefix \textbf p of \textbf w with $\occ_x(\mathbf p)=i$.
 
\begin{example}
\label{example l_i(w,x)}
If $\mathbf w=xyx^2zy$ then, evidently, $\occ_x(\mathbf w)=3$, $\occ_y(\mathbf w)=2$ and $\occ_z(\mathbf w)=1$. Further, the shortest prefixes \textbf p of \textbf w with $\occ_x(\mathbf p)=1$, $\occ_x(\mathbf p)=2$ and $\occ_x(\mathbf p)=3$ are $x$, $xyx$ and $xyx^2$ respectively, whence $\ell_1(\mathbf w,x)=1$, $\ell_2(\mathbf w,x)=3$ and $\ell_3(\mathbf w,x)=4$. Analogously, $\ell_1(\mathbf w,y)=2$, $\ell_2(\mathbf w,y)=6$ and $\ell_1(\mathbf w,z)=5$.
\end{example}
 
Below we often deal with inequalities like $\ell_i(\mathbf w,x)<\ell_j(\mathbf w,y)$. Clearly, this inequality means simply that $i$th occurrence of $x$ in \textbf w precedes $j$th occurrence of $y$ in \textbf w.
 
If \textbf w is a word and $X$ is a set of letters then ${\bf w}_X$ denotes the word obtained from ${\bf w}$ by deleting all letters from $X$. If $X=\{x\}$ then we write ${\bf w}_x$ rather than ${\bf w}_{\{x\}}$.
 
\begin{lemma}
\label{u_mul=v_mul}
If a non-commutative variety of monoids $\mathbf V$ satisfies an identity $\mathbf{u\approx v}$ such that the claim~\eqref{sim(u)=sim(v) & mul(u)=mul(v)} holds then
\begin{equation}
\label{eq u_mul=v_mul}
{\bf u}_{\mul(\mathbf u)} = {\bf v}_{\mul(\mathbf u)}.
\end{equation}
\end{lemma}
 
\begin{proof}
According to the claim~\eqref{sim(u)=sim(v) & mul(u)=mul(v)}, $\simple(\mathbf u)=\simple(\mathbf v)$ and $\mul(\mathbf u)=\mul(\mathbf v)$. It is evident that the claim \eqref{eq u_mul=v_mul} holds whenever the set $\simple(\mathbf u)$ contains $<2$ letters. Suppose now that $\simple(\mathbf u)$ contains at least two different letters and the claim~\eqref{eq u_mul=v_mul} is false. Then there are letters $x,y\in\simple(\mathbf u)$ such that $\ell_1({\bf u},x)< \ell_1({\bf u},y)$ and $\ell_1({\bf v},x)>\ell_1({\bf v},y)$. One can substitute~1 for all letters occurring in the identity ${\bf u}\approx {\bf v}$ except $x$ and $y$. Then we obtain $xy\approx yx$ contradicting the fact that \textbf V is non-commutative.
\end{proof}
 
\begin{proposition}
\label{word problem D_1}
A non-trivial identity ${\bf u}\approx {\bf v}$ holds in the variety ${\bf D}_1$ if and only if the claims~\eqref{sim(u)=sim(v) & mul(u)=mul(v)} and \eqref{eq u_mul=v_mul} are true.
\end{proposition}
 
\begin{proof}
\emph{Necessity}. The inclusion ${\bf C}_2 \subseteq{\bf D}_1$ and Proposition~\ref{word problem C_2} imply that the identity ${\bf u}\approx {\bf v}$ satisfies the claim~\eqref{sim(u)=sim(v) & mul(u)=mul(v)}. Since the variety $\mathbf D_1$ is non-commutative, Lemma~\ref{u_mul=v_mul} implies that the claim~\eqref{eq u_mul=v_mul} holds too.
 
\smallskip
 
\emph{Sufficiency}. Suppose that the identity ${\bf u}\approx {\bf v}$ satisfies the claims~\eqref{sim(u)=sim(v) & mul(u)=mul(v)} and~\eqref{eq u_mul=v_mul}. Let $\simple({\bf u})=\{y_1,y_2,\dots,y_m\}$. We may assume without loss of generality that
$$
{\bf u}={\bf u}_0y_1{\bf u}_1y_2{\bf u}_2\cdots y_m{\bf u}_m
$$
where $\con({\bf u}_0{\bf u}_1\cdots {\bf u}_m)=\mul({\bf u})$. It follows from the claim~\eqref{sim(u)=sim(v) & mul(u)=mul(v)} that $\simple({\bf v})=\{y_1,y_2$, $\dots,y_m\}$. Moreover, ${\bf v}= {\bf v}_0y_1{\bf v}_1y_2{\bf v}_2\cdots y_m{\bf v}_m$ by the claim~\eqref{eq u_mul=v_mul}. We can apply the claim~\eqref{sim(u)=sim(v) & mul(u)=mul(v)} again and conclude that $\con({\bf u}_0{\bf u}_1\cdots {\bf u}_m)=\con({\bf v}_0{\bf v}_1\cdots {\bf v}_m)$. Now it is easy to see that the identity system $\{x^2\approx x^3,x^2y\approx xyx\approx yx^2\}$ implies the identities
$$
\mathbf u=\mathbf u_0y_1\mathbf u_1y_2\mathbf u_2\cdots y_m\mathbf u_m\approx\mathbf v_0y_1\mathbf v_1y_2\mathbf v_2\cdots y_m\mathbf v_m=\mathbf v,
$$
whence $\mathbf D_1$ satisfies $\mathbf u\approx\mathbf v$.
\end{proof}
 
\begin{lemma}
\label{non-cr and non-commut}
If a variety of monoids ${\bf V}$ is non-completely regular and non-commutative then ${\bf D}_1\subseteq {\bf V}$.
\end{lemma}
 
\begin{proof}
Suppose that $\mathbf D_1\nsubseteq\mathbf V$. Then there is an identity $\mathbf{u\approx v}$ that holds in \textbf V but is false in $\mathbf D_1$. Corollary~\ref{cr or over C_2} implies that $\mathbf C_2\subseteq\mathbf V$. Then $\mathbf{u\approx v}$ holds in $\mathbf C_2$, whence the claim~\eqref{sim(u)=sim(v) & mul(u)=mul(v)} holds by Proposition~\ref{word problem C_2}. Now Lemma~\ref{u_mul=v_mul} and the claim that the variety \textbf V is non-commutative imply that the equality~\eqref{eq u_mul=v_mul} is true. Now Proposition~\ref{word problem D_1} applies and we conclude that the identity $\mathbf{u\approx v}$ holds in $\mathbf D_1$, a contradiction.
\end{proof}
 
\begin{lemma}
\label{D_{n+1} nsubseteq X}
If ${\bf X}$ is a non-completely regular variety of monoids and ${\bf D}_{n+1} \nsubseteq {\bf X}$ for some $n$ then ${\bf X}$ satisfies an identity of the form
\begin{equation}
\label{eq D_{n+1} nsubseteq X}
xy_1xy_2x\cdots xy_nx\approx x^{k_1}y_1x^{k_2}y_2x^{k_2}\cdots x^{k_n}y_nx^{k_{n+1}}
\end{equation}
where $k_i > 1$ for some $i$.
\end{lemma}
 
\begin{proof}
If the variety $\bf X$ is commutative then it satisfies the identity
$$
xy_1xy_2x\cdots xy_nx\approx x^{n+1}y_1y_2\cdots y_n,
$$
and we are done. Suppose now that $\bf X$ is non-commutative. Then ${\bf X}$ satisfies a non-trivial identity of the form $xy_1xy_2x\cdots xy_nx\approx {\bf w}$ by Lemmas~\ref{S(W) in V} and~\ref{D_n=var S(w)}. Now Lemma~\ref{non-cr and non-commut} applies with the conclusion that $\mathbf D_1 \subseteq\mathbf X$. According to Proposition~\ref{word problem D_1},
$$
{\bf w}= x^{k_1}y_1x^{k_2}y_2x^{k_2}\cdots y_nx^{k_{n+1}}.
$$
If $k_i > 1$ for some $i$ then we are done. Suppose that $k_i \le 1$ for all $i$. There is $1\le i\le n+1$ with $k_i=0$ because the identity $xy_1xy_2x\cdots xy_nx\approx\mathbf w$ is trivial otherwise. Substitute $xy_i$ for $y_i$ in this identity for all $i$ such that $k_i=0$. If $k_{n+1}=0$ then we multiply the resulted identity by $x$ on the right. Thus, we obtain an identity of the form~\eqref{eq D_{n+1} nsubseteq X} where $k_i > 1$ for some $i$.
\end{proof}
 
\section{$k$-decomposition of a word and related notions}
\label{decomposition}
 
Here we introduce a series of notions and examine their properties. These notions and results play a key role in the most complicated part of the proof of Theorem~\ref{main result} in Section~\ref{sufficiency: K}.
 
For a word \textbf u and letters $x_1,x_2,\dots,x_k\in \con({\bf u})$, let ${\bf u}(x_1,x_2,\dots,x_k)$ denote the word obtained from \textbf u by retaining the letters $x_1,x_2,\dots,x_k$. Equivalently,
$$
{\bf u}(x_1,x_2,\dots,x_k)={\bf u}_{\con({\bf u})\setminus \{x_1,x_2,\dots,x_k\}}.
$$
Let $\bf w$ be a word and $\simple(\mathbf w)=\{t_1,t_2,\dots, t_m\}$. We can assume without loss of generality that $\mathbf w(t_1,t_2,\dots, t_m)=t_1t_2\cdots t_m$. Then
\begin{equation}
\label{blocks and dividers}
\mathbf w = t_0\mathbf w_0 t_1 \mathbf w_1 \cdots t_m \mathbf w_m
\end{equation}
where $\mathbf w_0,\mathbf w_1,\dots,\mathbf w_m$ are possibly empty words and $t_0=\lambda$. The words $\mathbf w_0$, $\mathbf w_1$, \dots, $\mathbf w_m$ are called 0-\emph{blocks} of a word $\bf w$, while $t_0,t_1,\dots,t_m$ are said to be 0-\emph{dividers} of $\mathbf w$. The representation of the word \textbf w as a product of alternating 0-dividers and 0-blocks, starting with the 0-divider $t_0$ and ending with the 0-block $\mathbf w_m$ is called a 0-\emph{decomposition} of the word \textbf w.
 
Let now $k$ be a natural number. We define a $k$-decomposition of $\bf w$ by induction on $k$. Let~\eqref{blocks and dividers} be a \mbox{$(k-1)$}-decomposition of the word $\bf w$ with \mbox{$(k-1)$}-blocks $\mathbf w_0,\mathbf w_1,\dots,\mathbf w_m$ and \mbox{$(k-1)$}-dividers $t_0,t_1,\dots,t_m$. For any $i=0,1,\dots,m$, let $s_{i1},s_{i2},\dots, s_{ir_i}$ be all simple in the word $\mathbf w_i$ letters that does not occur in the word \textbf w to the left of $\mathbf w_i$. We can assume that $\mathbf w_i(s_{i1},s_{i2},\dots,s_{ir_i})=s_{i1}s_{i2}\cdots s_{ir_i}$. Then
\begin{equation}
\label{blocks and dividers within block}
\mathbf w_i = \mathbf v_{i0} s_{i1} \mathbf v_{i1} s_{i2} \mathbf v_{i2} \cdots s_{ir_i} \mathbf v_{ir_i}
\end{equation}
for possibly empty words $\mathbf v_{i0},\mathbf v_{i1}, \dots,\mathbf v_{ir_i}$. Put $s_{i0}=t_i$. The words $\mathbf v_{i0},\mathbf v_{i1}, \dots$, $\mathbf v_{ir_i}$ are called $k$-\emph{blocks} of the word $\bf w$, while the letters $s_{i0},s_{i1},\dots, s_{ir_i}$ are said to be $k$-\emph{dividers} of $\mathbf w$. 
 
\begin{remark}
\label{only first occurrences}
Note that only first occurrence of a letter in a given word might be a $k$-divider of this word for some $k$. In view of this observation, below we use expressions like``a letter $x$ is (or is not) a $k$-divider of a word \textbf w'' meaning that first occurrence of $x$ in \textbf w has the specified property.
\end{remark}
 
For any $i=0,1,\dots,m$, we represent \mbox{$(k-1)$}-block $\mathbf w_i$ in the form~\eqref{blocks and dividers within block}. As a result, we obtain the representation of the word $\bf w$ as a product of alternating $k$-dividers and $k$-blocks, starting with the $k$-divider $s_{00}=t_0$ and ending with the $k$-block $\mathbf v_{mr_m}$. This representation is called a $k$-\emph{decomposition of the word} $\bf w$. 
 
\begin{remark}
\label{max decomposition}
Since the length of the word $\bf w$ is finite, there is a number $k$ such that the $k$-decomposition of $\bf w$ coincides with its $n$-decompositions for all $n>k$.
\end{remark}
 
For reader convenience, we illustrate the notions of $k$-blocks, $k$-dividers and $k$-decomposition of a word by the following
 
\begin{example}
\label{example decompositions}
Let $\mathbf w=xyxzytszxs$. A unique simple letter in \textbf w is $t$. Therefore, the 0-decomposition of \textbf w has the form
\begin{equation}
\label{example 0-decomposition}
\lambda\cdot\underline{xyxzy}\cdot t\cdot\underline{szxs}
\end{equation}
(here and below throughout this example we underline blocks to distinguish them from dividers). A unique simple letter of the most left 0-block $xyxzy$ is $z$; the 0-block $szxs$ contains two simple letters, namely $z$ and $x$ but both these letters occur in \textbf w to the left of this block. Therefore, the 1-decomposition of \textbf w has the form
$$
\lambda\cdot\underline{xyx}\cdot z\cdot\underline y\cdot t\cdot\underline{szxs}.
$$
Analogous arguments show that the 2-decomposition of \textbf w has the form
$$
\lambda\cdot\underline x\cdot y\cdot\underline x\cdot z\cdot\underline y\cdot t\cdot\underline{szxs}
$$
and if $k\ge 3$ then the $k$-decomposition of \textbf w has the form
$$
\lambda\cdot\underline\lambda\cdot x\cdot\underline\lambda\cdot y\cdot\underline x\cdot z\cdot\underline y\cdot t\cdot\underline{szxs}.
$$
\end{example}
 
For a given word \textbf w, a letter $x\in\con(\mathbf w)$, a natural number $i\le\occ_x(\mathbf w)$ and an integer $k\ge 0$, we denote by $h_i^k(\mathbf w,x)$ the right-most $k$-divider of \textbf w that precedes $i$th occurrence of $x$ in \textbf w. The (possibly empty) letter $h_i^k(\mathbf w,x)$ is called an \mbox{$(i,k)$}-\emph{restrictor} of the letter $x$ in the word \textbf w. This notion is illustrated by the following
 
\begin{example}
\label{example h_i^k(w,x)}
Let \textbf w be the same word as in Example~\ref{example decompositions}. The 0-decomposition of \textbf w has the form~\eqref{example 0-decomposition}. We see that the right-most 0-divider of \textbf w that precedes the first two occurrences of $x$, the two occurrences of $y$, and the first occurrences of $z$ and $t$ is $\lambda$, while the right-most 0-divider of \textbf w that precedes third occurrence of $x$, second occurrence of $z$ and both occurrences of $s$ is $t$. This means that $h_1^0(\mathbf w,x)=h_2^0(\mathbf w,x)=\lambda$, $h_3^0(\mathbf w,x)=t$, $h_1^0(\mathbf w,y)=h_2^0(\mathbf w,y)=\lambda$, $h_1^0(\mathbf w,z)=\lambda$, $h_2^0(\mathbf w,z)=t$, $h_1^0(\mathbf w,s)=h_2^0(\mathbf w,s)=t$ and $h_1^0(\mathbf w,t)=\lambda$. Analogously, based on Example~\ref{example decompositions}, it is easy to find all other restrictors of letters in the word \textbf w. Results are presented in Table~\ref{table h_i^k(w,x)}.
\end{example}
 
\begin{table}[tbh]
\caption{Restrictors of letters in the word $xyxzytszxs$}
\begin{center}
\begin{tabular}{|c|c|c|c||c|c|c|c|}
\hline
\rule{0pt}{9pt}$a$&$k$&$i$&$h_i^k(\mathbf w,a)$&$a$&$k$&$i$&$h_i^k(\mathbf w,a)$\\
\hline
&&1&$\lambda$&&0&1&$\lambda$\\
\cline{3-4}
\cline{7-8}
&0&2&$\lambda$&&&2&$t$\\
\cline{3-4}
\cline{6-8}
&&3&$t$&&1&1&$\lambda$\\
\cline{2-4}
\cline{7-8}
&&1&$\lambda$&$z$&&2&$t$\\
\cline{3-4}
\cline{6-8}
&1&2&$\lambda$&&2&1&$y$\\
\cline{3-4}
\cline{7-8}
$x$&&3&$t$&&&2&$t$\\
\cline{2-4}
\cline{6-8}
&&1&$\lambda$&&$\ge 3$&1&$y$\\
\cline{3-4}
\cline{7-8}
&2&2&$y$&&&2&$t$\\
\cline{3-8}
&&3&$t$&&0&1&$t$\\
\cline{2-4}
\cline{7-8}
&&1&$\lambda$&&&2&$t$\\
\cline{3-4}
\cline{6-8}
&$\ge 3$&2&$y$&&1&1&$t$\\
\cline{3-4}
\cline{7-8}
&&3&$t$&$s$&&2&$t$\\
\cline{1-4}
\cline{6-8}
&0&1&$\lambda$&&2&1&$t$\\
\cline{3-4}
\cline{7-8}
&&2&$\lambda$&&&2&$t$\\
\cline{2-4}
\cline{6-8}
&1&1&$\lambda$&&$\ge 3$&1&$t$\\
\cline{3-4}
\cline{7-8}
$y$&&2&$z$&&&2&$t$\\
\cline{2-8}
&2&1&$\lambda$&&0&1&$\lambda$\\
\cline{3-4}
\cline{6-8}
&&2&$z$&$t$&1&1&$z$\\
\cline{2-4}
\cline{6-8}
&$\ge 3$&1&$x$&&2&1&$z$\\
\cline{3-4}
\cline{6-8}
&&2&$z$&&$\ge 3$&1&$z$\\
\hline
\end{tabular}
\end{center}
\label{table h_i^k(w,x)}
\end{table}
 
\begin{lemma}
\label{simple observations}
Let $\mathbf w$ be a word, $t$ be a letter and $k,r$ be numbers with $r<k$.
\begin{itemize}
\item[\textup{(i)}] If $t$ is an $r$-divider of $\bf w$ then $t$ is a $k$-divider of $\bf w$ too.
\item[\textup{(ii)}] If $h_1^k({\bf w},x)=h_2^k({\bf w},x)$ then $h_1^r({\bf w},x)=h_2^r({\bf w},x)$ as well.
\item[\textup{(iii)}] If $t_0\mathbf w_0t_1\mathbf w_1\cdots t_m\mathbf w_m$ is the $k$-decomposition of $\bf w$ and $m>0$ then $t_m\in\simple(\mathbf w)$.
\end{itemize}
\end{lemma}
 
\begin{proof}
The claims~(i) and~(ii) are obvious. One can verify the claim~(iii). Suppose that $t_m\in \mul(\mathbf w)$. Then $t_m$ is not a 0-divider of $\bf w$. Let $p$ be the least natural number such that $t_m$ is a $p$-divider but not a \mbox{$(p-1)$}-divider of $\bf w$. Evidently, $p\le k$.
 
Suppose that $h_1^{p-1}(\mathbf w,t_m)=h_2^{p-1}(\mathbf w,t_m)$. This means that there are no \mbox{$(p-1)$}-dividers in \textbf w between the first and the second occurrences of $t_m$ in \textbf w. In other words, the first and the second occurrences of $t_m$ in \textbf w lie in the same \mbox{$(p-1)$}-block of \textbf w. Therefore, $t_m$ is not simple in this \mbox{$(p-1)$}-block. In particular, $t_m$ is not a $p$-divider of $\bf w$, contradicting the choice of $t_m$. Thus, $h_1^{p-1}(\mathbf w,t_m)\ne h_2^{p-1}(\mathbf w,t_m)$. Note that the arguments of this paragraph is very typical. Below we use arguments like these many times, without repeating them explicitly.
 
Note that $t_m\ne h_2^{p-1}(\mathbf w,t_m)$ because $t_m$ is not a \mbox{$(p-1)$}-divider of $\bf w$. Put $t_{m+1}= h_2^{p-1}(\mathbf w,t_m)$. Since $p-1<k$, the claim~(i) implies that $t_{m+1}$ is a $k$-divider of $\bf w$. The last $k$-divider of $\bf w$ is $t_m$. Therefore, first occurrence of $t_{m+1}$ in \textbf w precedes first occurrence of $t_m$ in \textbf w. Therefore, $h_1^{p-1}(\mathbf w,t_m)=t_{m+1}= h_2^{p-1}(\mathbf w,t_m)$, a contradiction.
\end{proof}
 
For a given word \textbf w and a letter $x\in\con(\mathbf w)$, we define some number that is called a \emph{depth} of $x$ in $\bf w$ and is denoted by $D({\bf w},x)$. If $x\in\simple(\mathbf w)$ then we put $D({\bf w},x)=0$. Suppose now that $x\in\mul(\mathbf w)$. If there is a natural $k$ such that the first and the second occurrences of $x$ in \textbf w lie in different \mbox{$(k-1)$}-blocks of \textbf w then the depth of $x$ in $\bf w$ equals the minimal number $k$ with this property. Finally, if, for any natural $k$, the first and the second occurrences of $x$ in \textbf w lie in the same $k$-block of \textbf w then we put $D({\bf w},x)=\infty$. In other words, $D({\bf w},x)=k$ if and only if $h_1^{k-1}(\mathbf w, x)\ne h_2^{k-1}(\mathbf w, x)$ and $k$ is the least number with this property, while $D({\bf w},x)=\infty$ if and only if $h_1^{k-1}(\mathbf w, x)=h_2^{k-1}(\mathbf w, x)$ for any $k$. The definition of the depth of a letter in a word is illustrated by the following
 
\begin{example}
\label{example D(w,x)}
As in Examples~\ref{example decompositions} and~\ref{example h_i^k(w,x)}, put $\mathbf w=xyxzytszxs$. Here we systematically use information about restrictors of letters in the word \textbf w indicated in Table~\ref{table h_i^k(w,x)}. In particular, in view of this table, $h_1^k(\mathbf w,x)=\lambda$ for all $k$, while $h_2^0(\mathbf w,x)=h_2^1(\mathbf w,x)=\lambda$ and $h_2^2(\mathbf w,x)=y$. Therefore, $D(\mathbf w,x)=3$. Further, $h_1^0(\mathbf w,y)=h_2^0(\mathbf w,y)=\lambda$, $h_1^1(\mathbf w,y)=\lambda$ and $h_2^1(\mathbf w,y)=z$. Hence $D(\mathbf w,y)=2$. The equalities $h_1^0(\mathbf w,z)=\lambda$ and $h_2^0(\mathbf w,z)=t$ imply that $D(\mathbf w,z)=1$. Further, $h_1^k({\bf w},s)=h_2^k({\bf w},s)= t$ for each $k\ge 0$, whence $D(\mathbf w,s)=\infty$. Finally, $D(\mathbf w,t)=0$ because $t\in\simple(\mathbf w)$.
\end{example}
 
The following criterion for a letter of a word to be a $k$-divider is often used in the proof of Theorem~\ref{main result}.
 
\begin{lemma}
\label{k-divider and depth}
A letter $t$ is a $k$-divider of a word $\bf w$ if and only if $D(\mathbf w,t)\le k$.
\end{lemma}
 
\begin{proof}
This statement is evident whenever $k=0$ because both the property of $t$ to be a 0-divider of \textbf w and the equality $D(\mathbf w,t)=0$ are equivalent to the claim that $t$ is simple in \textbf w. Further, if $k>0$ then a property of a letter $t$ to be a $k$-divider of \textbf w is equivalent to the claim that the first and the second occurrences of $t$ lie in different \mbox{$(k-1)$}-blocks of \textbf w. In turn, the last claim is equivalent to the non-equality $h_1^{k-1}(\mathbf w,t)\ne h_2^{k-1}(\mathbf w,t)$, i.e., to the required statement that $D(\mathbf w,t)\le k$.
\end{proof}
 
The words $\bf u$ and $\bf v$ are said to be $k$-\emph{equivalent} if these words have the same set of $k$-dividers and these $k$-dividers appear in \textbf u and in \textbf v in the same order.
 
\begin{lemma}
\label{k-equivalent}
Let $k$ be a non-negative integer. Words $\bf u$ and $\bf v$ are $k$-equivalent if and only if the claim~\eqref{sim(u)=sim(v) & mul(u)=mul(v)} is true and, for any $x\in\con(\mathbf{uv})$, $h_1^k({\bf u},x)= h_1^k({\bf v},x)$ whenever either $D({\bf u},x)\le k$ or $D({\bf v},x)\le k$.
\end{lemma}
 
\begin{proof}
\emph{Sufficiency}. Suppose that
\begin{equation}
\label{t_0u_0t_1u_1 ... t_mu_m}
t_0\mathbf u_0t_1 \mathbf u_1 \cdots t_m \mathbf u_m
\end{equation}
and $s_0\mathbf v_0s_1 \mathbf v_1 \cdots s_r \mathbf v_r$ are $k$-decompositions of the words $\bf u$ and $\bf v$, respectively. Evidently, $t_0=s_0=\lambda$. If $m=r=0$ then the required fact is evident. Let now $m>0$. In view of Lemma~\ref{k-divider and depth}, $D(\mathbf u, t_i)\le k$ for any $1\le i\le m$. By the hypothesis, this implies that $t_{i-1}=h_1^k(\mathbf u,t_i)=h_1^k(\mathbf v,t_i)$ for any $1\le i\le m$, whence $t_{i-1}$ is a $k$-divider of the word $\bf v$. According to Lemma~\ref{simple observations}(iii), $t_m\in\simple(\mathbf u)$. Then the claim~\eqref{sim(u)=sim(v) & mul(u)=mul(v)} implies that $t_m\in\simple(\mathbf v)$, whence $t_m$ is a 0-divider of $\bf v$. Now Lemma~\ref{simple observations}(i) applies with the conclusion that $t_m$ is a $k$-divider of $\bf v$. So, the letters $t_1,t_2,\dots, t_m$ are $k$-dividers of the word $\bf v$, whence $m\le r$. By symmetry, $r\le m$. We prove that $m=r$. Further, $t_1$ coincides with $s_p$ for some $p$. If $p\ne 1$ then $h_1^k(\mathbf v,t_1)\ne t_0$. This contradicts the fact that $h_1^k(\mathbf v,t_1)= h_1^k(\mathbf u,t_1)=t_0$. So, $p=1$ and therefore, $t_1=s_1$. By induction, we can verify that $t_j=s_j$ for any $j\le m$.
 
\smallskip
 
\emph{Necessity}. Suppose that~\eqref{t_0u_0t_1u_1 ... t_mu_m} is the $k$-decomposition of the word $\bf u$. Then the $k$-decomposition of $\bf v$ has the form
\begin{equation}
\label{t_0v_0t_1v_1 ... t_mv_m}
t_0\mathbf v_0t_1 \mathbf v_1 \cdots t_m \mathbf v_m.
\end{equation}
Let $x\in \con({\bf u})$ and $D({\bf u},x)\le k$. Lemma~\ref{k-divider and depth} implies that $x=t_i$ for some $1\le i\le m$. Therefore, $h_1^k(\mathbf v, x)=h_1^k(\mathbf u, x)=t_{i-1}$. Analogously, we verify that if $x\in \con({\bf v})$ and $D({\bf v},x)\le k$ then $h_1^k(\mathbf v,x)= h_1^k(\mathbf u,x)$.
\end{proof}
 
\begin{lemma}
\label{h_2^{k-1}}
Let $\bf w$ be a word, $x$ be a letter multiple in $\mathbf w$ with $D(\mathbf w,x)= k$ and $t$ be a \mbox{$(k-1)$}-divider of $\bf w$.
\begin{itemize}
\item[\textup{(i)}] If $t=h_2^{k-1}(\mathbf w,x)$ then $\ell_1(\mathbf w,x)<\ell_1(\mathbf w,t)$.
\item[\textup{(ii)}] If $\ell_1(\mathbf w,x)<\ell_1(\mathbf w,t)<\ell_2(\mathbf w,x)$ then $D(\mathbf w,t)= k-1$; if besides that $k>1$ then $\ell_2(\mathbf w,x)<\ell_2(\mathbf w,t)$.
\end{itemize}
\end{lemma}
 
\begin{proof}
(i) Suppose that $\ell_1(\mathbf w,t)<\ell_1(\mathbf w,x)$. Then the equality $t=h_2^{k-1}(\mathbf w,x)$ implies that $t=h_1^{k-1}(\mathbf w,x)$. Thus, $h_1^{k-1}(\mathbf w,x)=h_2^{k-1}(\mathbf w,x)$. This contradicts the assumption that $D(\mathbf w,x)= k$. So, $\ell_1(\mathbf w,x)\le\ell_1(\mathbf w,t)$. Since $t$ is a \mbox{$(k-1)$}-divider, Lemma~\ref{k-divider and depth} implies that $D(\mathbf w,t)\le k-1$. In particular, $D(\mathbf w,t)\ne D(\mathbf w,x)$, whence $t\ne x$. Therefore, $\ell_1(\mathbf w,x)<\ell_1(\mathbf w,t)$.
 
\smallskip
 
(ii) Suppose now that $\ell_1(\mathbf w,x)<\ell_1(\mathbf w,t)<\ell_2(\mathbf w,x)$. Put $r=D(\mathbf w,t)$. By Lemma~\ref{k-divider and depth}, $r\le k-1$. If $D(\mathbf w,t)=r< k-1$ then $t$ is an $r$-divider by Lemma~\ref{k-divider and depth}. Therefore, $t=h_2^r(\mathbf w,x)$. Further, $t\ne h_1^r(\mathbf w,x)$ because $\ell_1(\mathbf w,x)<\ell_1(\mathbf w,t)$. Thus, $h_1^r(\mathbf w,x)\ne h_2^r(\mathbf w,x)$. This means that $D(\mathbf w,x)\le r+1<k$, a contradiction. So, $D(\mathbf w,t)= k-1$.
 
Let now $k>1$. Then $t\in\mul(\mathbf w)$. Suppose that $\ell_2(\mathbf w,t)<\ell_2(\mathbf w,x)$. Put $s=h_2^{k-2}(\mathbf w,t)$. In view of the claim~(i), $\ell_1(\mathbf w,t)<\ell_1(\mathbf w,s)$. Arguments similar to those from the previous paragraph imply that $D(\mathbf w,s)= k-2$. According to Lemma~\ref{k-divider and depth}, $s$ is a \mbox{$(k-2)$}-divider of $\bf w$. The choice of $s$ guarantees that first occurrence of $s$ in \textbf w precedes second occurrence of $t$. On the other hand, second occurrence of $t$ precedes second occurrence of $x$. Thus, first occurrence of $s$ precedes second occurrence of $x$. At the same time, first occurrence of $x$ precedes first occurrence of $s$ because $\ell_1(\mathbf w,x)<\ell_1(\mathbf w,t)<\ell_1(\mathbf w,s)$. Therefore, first and second occurrences of $x$ in $\bf w$ lie in different \mbox{$(k-2)$}-blocks. Hence, $D(\mathbf w,x)\le k-1$, a contradiction.
\end{proof}
 
\begin{lemma}
\label{the same l-dividers}
Let $\bf u$ and $\bf v$ be words and $\ell$ be a natural number. Suppose that the claims~\eqref{sim(u)=sim(v) & mul(u)=mul(v)} and
\begin{equation}
\label{eq the same l-dividers}
h_i^{\ell-1}({\bf u},x)= h_i^{\ell-1}({\bf v},x)\text{ for }i=1,2\text{ and all }x\in \con({\bf u})
\end{equation}
are true. Then the words $\mathbf u$ and $\mathbf v$ have the same set of $\ell$-dividers.
\end{lemma}
 
\begin{proof}
Let $t$ be an arbitrary $\ell$-divider of $\bf u$. If $t\in\simple(\mathbf u)$ then $t\in\simple(\mathbf v)$ by the claim~\eqref{sim(u)=sim(v) & mul(u)=mul(v)}. Therefore, $t$ is a 0-divider of $\bf v$. According to Lemma~\ref{simple observations}(i), $t$ is an $\ell$-divider of $\bf v$. Suppose now that $t\in\mul(\mathbf u)$. The claim~\eqref{sim(u)=sim(v) & mul(u)=mul(v)} implies that $t\in\mul(\mathbf v)$. Since $t$ is an $\ell$-divider of $\bf u$, $h_1^{\ell-1}({\bf u},t)\ne h_2^{\ell-1}({\bf u},t)$. Then $h_1^{\ell-1}({\bf v},t)\ne h_2^{\ell-1}({\bf v},t)$ by the claim~\eqref{eq the same l-dividers}. This implies that $t$ is an $\ell$-divider of $\bf v$. Similarly we prove that if $s$ is an $\ell$-divider of $\bf v$ then $s$ is an $\ell$-divider of $\bf u$.
\end{proof}
 
\begin{lemma}
\label{h_i^k(u,x)=h_i^k(v,x) to h_i^s(u,x)=h_i^s(v,x)}
Let $\bf u$ and $\bf v$ be words and $k$ be a natural number. Suppose that the claims~\eqref{sim(u)=sim(v) & mul(u)=mul(v)} and~\eqref{eq the same l-dividers} with $\ell=k$ are true. Then the claim~\eqref{eq the same l-dividers} with $\ell=s$ is true for any $1\le s\le k$.
\end{lemma}
 
\begin{proof}
If $k=1$ then the assertion is valid by the hypothesis. Suppose now that $k>1$. Let~\eqref{t_0u_0t_1u_1 ... t_mu_m} be the \mbox{$(k-1)$}-decomposition of $\bf u$. In view of Lemma~\ref{k-equivalent}, the \mbox{$(k-1)$}-decomposition of $\bf v$ has the form~\eqref{t_0v_0t_1v_1 ... t_mv_m}. Let $s<k$ be the least number such that~\eqref{eq the same l-dividers} with $\ell=s$ is false. Then there exists a letter $x$ such that $h_i^{s-1}(\mathbf u,x)\ne h_i^{s-1}(\mathbf v,x)$ for some $i\in\{1,2\}$. By the definition of \mbox{$(i,s-1)$}-restrictors, $h_i^{s-1}(\mathbf u,x)$ and $h_i^{s-1}(\mathbf v,x)$ are some \mbox{$(s-1)$}-dividers of \textbf u and \textbf v respectively. Lemma~\ref{simple observations}(i) implies that \mbox{$(s-1)$}-dividers of \textbf u and \textbf v are \mbox{$(k-1)$}-dividers of these words. Therefore, $h_i^{s-1}(\mathbf u,x)=t_p$ and $h_i^{s-1}(\mathbf v,x)=t_q$ for some $p\ne q$. We may assume without loss of generality that $p<q$. By the hypothesis, $h_i^{k-1}(\mathbf u,x)= h_i^{k-1}(\mathbf v,x)$, whence this \mbox{$(i,k-1)$}-restrictor of $x$ coincide with $t_n$ for some $n$. Clearly, $n\ge q$ because $s<k$. Since $t_n$ precedes $i$th occurrence of $x$ in \textbf u, we have $\ell_1(\mathbf u, t_q)<\ell_i(\mathbf u,x)$. Since $t_p$ is an \mbox{$(i,s-1)$}-restrictor of $x$ in \textbf u, there are no \mbox{$(s-1)$}-dividers of \textbf u between first occurrence of $t_p$ and $i$th occurrence of $x$ in \textbf u. In particular, $t_q$ is not an \mbox{$(s-1)$}-divider of \textbf u. Further, Lemma~\ref{k-divider and depth} implies that $D(\mathbf u, t_q)>s-1$. In particular, $D(\mathbf u, t_q)>0$, whence $t_q\in\mul(\mathbf u)$. If $s=1$ then $t_q$ is a 0-divider of $\bf v$, whence $t_q$ is simple in \textbf v. This contradicts the claim~\eqref{sim(u)=sim(v) & mul(u)=mul(v)}. Thus, $s>1$. This means that $h_1^{s-2}(\mathbf u, t_q)= h_2^{s-2}(\mathbf u, t_q)$. Since the claim~\eqref{eq the same l-dividers} with $\ell=s-1$ is true, we obtain $h_1^{s-2}(\mathbf v, t_q)= h_2^{s-2}(\mathbf v, t_q)$. According to Lemma~\ref{simple observations}(ii), $h_1^{r-2}(\mathbf v, t_q)= h_2^{r-2}(\mathbf v, t_q)$ for all $r\le s$. Then $D(\mathbf v, t_q)>s-1$. Lemma~\ref{k-divider and depth} implies that $t_q$ is not an \mbox{$(s-1)$}-divider of $\bf v$, a contradiction with the equality $t_q=h_i^{s-1}(\mathbf v,x)$.
\end{proof}
 
\begin{lemma}
\label{D(u,x)=k iff D(v,x)=k}
Let $\bf u$ and $\bf v$ be words and $k$ be a natural number. Suppose that the claims~\eqref{sim(u)=sim(v) & mul(u)=mul(v)} and~\eqref{eq the same l-dividers} with $\ell=k$ are true. Then, for any letter $x\in\con(\mathbf u)$, $D(\mathbf u,x)=k$ if and only if $D(\mathbf v,x)=k$.
\end{lemma}
 
\begin{proof}
In view of Lemma~\ref{h_i^k(u,x)=h_i^k(v,x) to h_i^s(u,x)=h_i^s(v,x)}, the claim~\eqref{eq the same l-dividers} with $\ell=s$ is true for any $1\le s\le k$. Suppose that $D(\mathbf u,x)=k$. This implies that
$$
h_1^{s-1}({\bf v},x)=h_1^{s-1}({\bf u},x)=h_2^{s-1}({\bf u},x)=h_2^{s-1}({\bf v},x)
$$
whenever $1\le s<k$ but
$$
h_1^{k-1}({\bf v},x)=h_1^{k-1}({\bf u},x)\ne h_2^{k-1}({\bf u},x)=h_2^{k-1}({\bf v},x).
$$
This implies that $D(\mathbf v, x)=k$. By symmetry, if $D(\mathbf v, x)=k$ then $D(\mathbf u,x)=k$.
\end{proof}
 
\begin{lemma}
\label{if first then second}
Let $\bf w$ be a word, $r>1$ be a number and $y$ be a letter such that $D(\mathbf w, y)=r-2$. Then if $\ell_1(\mathbf w,z)<\ell_1({\bf w},y)$ for some letter $z$ with $D(\mathbf w,z)\ge r$ then $\ell_2(\mathbf w,z)<\ell_1({\bf w},y)$.
\end{lemma}
 
\begin{proof}
Let $z$ be a letter with $\ell_1(\mathbf w,z)<\ell_1({\bf w},y)$ and $D(\mathbf u,z)\ge r$. Lemma~\ref{k-divider and depth} implies that $y$ is an \mbox{$(r-2)$}-divider of $\bf w$. Then if $\ell_1({\bf u},y)<\ell_2(\mathbf u,z)$ then the \mbox{$(r-2)$}-divider $y$ is located between the first and the second occurrences of $z$ in \textbf u. This contradicts the equality $h_1^{r-2}(\mathbf u, z)=h_2^{r-2}(\mathbf u, z)$. The case $\ell_1({\bf u},y)=\ell_2({\bf u},z)$ also is impossible. Therefore, $\ell_2(\mathbf w,z)<\ell_1({\bf w},y)$.
\end{proof}
 
Below, in order to facilitate understanding of our considerations, we will sometimes write the number in brackets over a letter to indicate the number of the occurrences of this letter in the given word; for instance, we may write
$$
\mathbf w=\,\stackrel{(1)}{x_1}\,\stackrel{(1)}{x_2}\,\stackrel{(2)}{x_1}\,\stackrel{(1)}{x_3}\,\stackrel{(2)}{x_2}\,\stackrel{(3)}{x_1}.
$$
 
\begin{lemma}
\label{form of the identity}
Let ${\bf u}\approx {\bf v}$ be an identity and $s$ be a natural number. Suppose that the claims~\eqref{sim(u)=sim(v) & mul(u)=mul(v)} and~\eqref{eq the same l-dividers} with $\ell=s$ are true and there is a letter $x_s$ such that $D(\mathbf u, x_s)=s$. Then there exist letters $x_0,x_1,\dots, x_{s-1}$ such that $D(\mathbf u,x_r)=D(\mathbf v,x_r)=r$ for any $0\le r<s$ and the identity $\mathbf u \approx \mathbf v$ has the form
\begin{equation}
\label{form of u=v}
\begin{array}{rl}
&\mathbf u_{2s+1}\stackrel{(1)}{x_s}\mathbf u_{2s}\stackrel{(1)}{x_{s-1}}\mathbf u_{2s-1}\stackrel{(2)}{x_s}\mathbf u_{2s-2}\stackrel{(1)}{x_{s-2}}\mathbf u_{2s-3}\stackrel{(2)}{x_{s-1}}\mathbf u_{2s-4}\stackrel{(1)}{x_{s-3}}\\
&\cdot\,\mathbf u_{2s-5}\stackrel{(2)}{x_{s-2}}\cdots\mathbf u_4\stackrel{(1)}{x_1}\mathbf u_3\stackrel{(2)}{x_2}\mathbf u_2\stackrel{(1)}{x_0}\mathbf u_1\stackrel{(2)}{x_1}\mathbf u_0\\
\approx{}&\mathbf v_{2s+1}\stackrel{(1)}{x_s}\mathbf v_{2s}\stackrel{(1)}{x_{s-1}}\mathbf v_{2s-1}\stackrel{(2)}{x_s}\mathbf v_{2s-2}\stackrel{(1)}{x_{s-2}}\mathbf v_{2s-3}\stackrel{(2)}{x_{s-1}}\mathbf v_{2s-4}\stackrel{(1)}{x_{s-3}}\\
&\cdot\,\mathbf v_{2s-5}\stackrel{(2)}{x_{s-2}}\cdots\mathbf v_4\stackrel{(1)}{x_1}\mathbf v_3\stackrel{(2)}{x_2}\mathbf v_2\stackrel{(1)}{x_0}\mathbf v_1\stackrel{(2)}{x_1}\mathbf v_0
\end{array}
\end{equation}
for some possibly empty words $\mathbf u_0,\mathbf u_1,\dots,\mathbf u_{2s+1}$ and $\mathbf v_0,\mathbf v_1,\dots,\mathbf v_{2s+1}$.
\end{lemma}
 
\begin{proof}
In view of Lemma~\ref{h_i^k(u,x)=h_i^k(v,x) to h_i^s(u,x)=h_i^s(v,x)}, the claim~\eqref{eq the same l-dividers} with $\ell=r$ is true for any $1\le r\le s$. We use this fact below without references.
 
Put $x_{s-1}=h_2^{s-1}({\bf u},x_s)$. The claim~\eqref{eq the same l-dividers} with $\ell=s$ implies that $h_2^{s-1}({\bf v},x_s)=h_2^{s-1}({\bf u},x_s)=x_{s-1}$. According to Lemma~\ref{h_2^{k-1}}, $D(\mathbf u, x_{s-1})=s-1$ and $\ell_j(\mathbf u,x_s)<\ell_j(\mathbf u,x_{s-1})$ for any $j=1,2$. Recall that $D(\mathbf u,x_s)=s$. According to Lemma~\ref{D(u,x)=k iff D(v,x)=k}, $D(\mathbf v, x_s)=s$. Now we apply Lemma~\ref{h_2^{k-1}} again and obtain $D(\mathbf v, x_{s-1})=s-1$ and $\ell_j(\mathbf v,x_s)<\ell_j(\mathbf v,x_{s-1})$ for any $j=1,2$.
 
Further, put $x_{s-2}=h_2^{s-2}({\bf u},x_{s-1})$. According to Lemma~\ref{h_2^{k-1}}, $D(\mathbf u, x_{s-2})=s-2$ and $\ell_j(\mathbf u,x_{s-1})<\ell_j(\mathbf u,x_{s-2})$ for any $j=1,2$. The claim~\eqref{eq the same l-dividers} with $\ell=s-1$ implies that $h_2^{s-2}({\bf v},x_{s-1})=h_2^{s-2}({\bf u},x_{s-1})=x_{s-2}$. Now we apply Lemma~\ref{h_2^{k-1}} again and obtain $D(\mathbf v, x_{s-2})=s-2$ and $\ell_j(\mathbf v,x_{s-1})<\ell_j(\mathbf v,x_{s-2})$ for any $j=1,2$. Since $\ell_1(\mathbf u,x_s)<\ell_1(\mathbf u,x_{s-1})<\ell_1(\mathbf u,x_{s-2})$, we have $\ell_2(\mathbf u,x_s)<\ell_1(\mathbf u,x_{s-2})$ by Lemma~\ref{if first then second}. Analogously, $\ell_2(\mathbf v,x_s)<\ell_1(\mathbf v,x_{s-2})$.
 
Continuing these considerations, we define one by one the letters $x_r=h_2^r(\mathbf u,x_{r+1})$ for $r=s-3,s-4,\dots,1$ and prove that $D(\mathbf u,x_r)=D(\mathbf v, x_r)=r$, $\ell_j(\mathbf u,x_{r+1})<\ell_j(\mathbf u,x_r)$, $\ell_j(\mathbf v,x_{r+1})<\ell_j(\mathbf v,x_r)$ for any $j=1,2$, $\ell_2(\mathbf u,x_{r+2})<\ell_1(\mathbf u,x_r)$ and $\ell_2(\mathbf v,x_{r+2})<\ell_1(\mathbf v,x_r)$.
 
Finally, put $x_0=h_2^0({\bf u},x_1)$. According to Lemma~\ref{h_2^{k-1}}, $D(\mathbf u, x_0)=0$ and $\ell_1(\mathbf u,x_1)<\ell_1(\mathbf u,x_0)$. The claim~\eqref{eq the same l-dividers} with $\ell=1$ implies that $h_2^0({\bf v},x_1)=h_2^0({\bf u},x_1)=x_0$. Now we apply Lemma~\ref{h_2^{k-1}} again and obtain $D(\mathbf v, x_0)=0$ and $\ell_1(\mathbf v,x_1)<\ell_1(\mathbf v,x_0)$. Since $\ell_1(\mathbf u,x_2)<\ell_1(\mathbf u,x_1)<\ell_1(\mathbf u,x_0)$, we have $\ell_2(\mathbf u,x_2)<\ell_1(\mathbf u,x_0)$ by Lemma~\ref{if first then second}. Analogously, $\ell_2(\mathbf v,x_2)<\ell_1(\mathbf v,x_0)$.
 
In view of the above, we have the identity $\mathbf{u\approx v}$ has the form~\eqref{form of u=v} for some possibly empty words $\mathbf u_0,\mathbf u_1,\dots,\mathbf u_{2s+1}$ and $\mathbf v_0,\mathbf v_1,\dots,\mathbf v_{2s+1}$.
\end{proof}
 
\begin{lemma}
\label{does not contain dividers}
Let $\mathbf w=y_1y_2\cdots y_n$ where the letters $y_1,y_2,\dots,y_n$ are not necessarily pairwise different. Further, let $\mathbf u=\mathbf u'\xi(\mathbf w)\mathbf u''$ for some possibly empty words $\mathbf u'$ and $\mathbf u''$ and some endomorphism $\xi$ of $F^1$. Put $\xi(y_i)=\mathbf w_i$ for all $i=1,2,\dots,n$. If $D(\mathbf w, y_i)>0$ then the subword $\mathbf w_i$ of $\mathbf u$ does not contain any $r$-divider of $\mathbf u$ for all $r<D(\mathbf w, y_i)$. 
\end{lemma}
 
\begin{proof}
Let $1\le i\le n$ and $D(\mathbf w,y_i)>0$. Then $y_i\in\mul(\mathbf w)$, whence $\con(\mathbf w_i)\subseteq\mul\bigl(\xi(\mathbf w)\bigr)\subseteq\mul(\mathbf u)$. This implies that $\mathbf w_i$ does not contain any 0-divider of \textbf u. Let now $r>0$ be the least number such that there exists $i$ with $D(\mathbf w,y_i)>r$ but $\mathbf w_i$ contains some $r$-divider $t$ of \textbf u. The choice of $r$ and Lemma~\ref{k-divider and depth} imply that $D(\mathbf u,t)=r$. Clearly, $t\notin\con(\mathbf w_1\mathbf w_2\cdots\mathbf w_{i-1})$, whence $y_i$ differs from $y_1,y_2,\dots,y_{i-1}$. Since $y_i\in\mul(\mathbf w)$, there is some $j\ge i$ such that $\mathbf w_j$ contains second occurrence of $t$ in \textbf u. Put $x=h_2^{r-1}(\mathbf u,t)$. In view of Lemma~\ref{h_2^{k-1}}(i), $\ell_1(\mathbf u,t)<\ell_1(\mathbf u,x)$. Then there is $i\le\ell\le j$ such that $\mathbf w_\ell$ contains the \mbox{$(r-1)$}-divider $x$ of \textbf u. In view of the choice of $r$, $D(\mathbf w,y_\ell)\le r-1$. This implies that $y_i\ne y_\ell$, whence $\ell_1(\mathbf w,y_i)<\ell_1(\mathbf w,y_\ell)$. Further, since $y_i\in\mul(\mathbf w)$, there is $p\ge j$ such that $y_i= y_p$. We note that $\ell<p$ because $y_p=y_i\ne y_\ell$. So, we obtain $\ell_1(\mathbf w,y_i)<\ell_1(\mathbf w,y_\ell)<\ell_2(\mathbf w,y_i)$. Lemma~\ref{k-divider and depth} implies that $y_\ell$ is an \mbox{$(r-1)$}-divider of $\bf w$, whence $h_1^{r-1}(\mathbf w, y_i)\ne h_2^{r-1}(\mathbf w, y_i)$. We have a contradiction with the fact that $D(\mathbf w,y_i)>r$.
\end{proof}
 
\section{The proof of the ``only if'' part}
\label{necessity}
 
Throughout this section, ${\bf V}$ denotes a fixed non-group chain variety of monoids. We aim to verify that \textbf V is contained in one of the varieties listed in Theorem~\ref{main result}. The section is divided into three subsections.{\sloppy

}
 
\subsection{Reduction to the case when $\mathbf D_2\subseteq\mathbf V$}
\label{necessity: red to over D_2}
 
A variety of monoids is called \emph{aperiodic} if all its groups are singletons. Lemma~\ref{group variety} implies that ${\bf SL} \subseteq {\bf V}$. If ${\bf V}$ contains a non-trivial group then the variety generated by this group is incomparable with ${\bf SL}$. This contradicts the fact that ${\bf V}$ is a chain variety. Therefore, ${\bf V}$ is aperiodic, whence it satisfies the identity $x^n\approx x^{n+1}$ for some $n$. If ${\bf V}$ is commutative then $\mathbf{V\subseteq SL\subseteq C}_2$ whenever $n=1$ and ${\bf V} \subseteq {\bf C}_n$ otherwise.
 
Further, if ${\bf V}$ is a variety of band monoids then Lemma~\ref{L(BM)} and the observation that ${\bf V}$ cannot contain simultaneously the incomparable varieties ${\bf LRB}$ and ${\bf RRB}$ imply that ${\bf V}$ is contained in one of these two varieties.
 
Suppose now that ${\bf V}$ is non-commutative and is not a variety of band monoids. Then ${\bf V}$ is non-completely regular because every aperiodic completely regular variety consists of bands. Then Lemma~\ref{non-cr and non-commut} implies that ${\bf D}_1 \subseteq {\bf V}$. To continue our considerations, we need several assertions. 
 
\begin{lemma}
\label{D_1 subseteq X}
Let $\mathbf X$ be a monoid variety such that $\mathbf D_1\subseteq\mathbf X$. Then either $\mathbf X$ satisfies an identity of the form
\begin{equation}
\label{x^syx^t=yx^r}
x^syx^t\approx yx^r
\end{equation}
where $s\ge 1$, $t\ge 0$, $s+t\ge2$ and $r\ge 2$ or, for any identity $\mathbf{u\approx v}$ that holds in $\mathbf X$, the claim
\begin{equation}
\label{eq D_1 subseteq X}
h_1^0({\bf u},x)= h_1^0({\bf v},x)\text{ for all }x\in \con({\bf u})
\end{equation}
is true.
\end{lemma}
 
\begin{proof}
Let $\mathbf{u\approx v}$ be an identity that holds in \textbf X. The inclusion ${\bf D}_1 \subseteq{\bf X}$ and Proposition~\ref{word problem D_1} imply that the claims~\eqref{sim(u)=sim(v) & mul(u)=mul(v)} and~\eqref{eq u_mul=v_mul} are true. Hence if~\eqref{t_0u_0t_1u_1 ... t_mu_m} is the 0-decomposition of $\bf u$ then the 0-decomposition of $\bf v$ has the form~\eqref{t_0v_0t_1v_1 ... t_mv_m}. Suppose that the claim~\eqref{eq D_1 subseteq X} is false. Then there is a letter $x\in\mul(\mathbf u)$ such that $h_1^0({\bf u},x)\ne h_1^0({\bf v},x)$. The claim~\eqref{sim(u)=sim(v) & mul(u)=mul(v)} implies that $x\in\mul(\mathbf v)$. Further, we may assume without loss of generality that there are $i<j$ such that $t_i=h_1^0({\bf u},x)$ and $t_j=h_1^0({\bf v},x)$. Substituting $y$ for $t_j$ and~1 for all letters occurring in the identity ${\bf u}\approx {\bf v}$ except $x$ and $t_j$, we obtain \textbf X satisfies an identity of the form~\eqref{x^syx^t=yx^r} where $s\ge 1$, $t\ge 0$, $s+t\ge2$ and $r\ge 2$.
\end{proof}
 
When we make simultaneously several substitutions in some identity, say, substitute $\mathbf u_i$ for $x_i$ for $i=1,2,\dots,k$, then we will say for brevity that we perform the substitution
$$
(x_1,x_2,\dots,x_k)\mapsto(\mathbf u_1,\mathbf u_2,\dots,\mathbf u_k)
$$
in this identity.
 
\begin{proposition}
\label{word problem E}
A non-trivial identity ${\bf u}\approx {\bf v}$ holds in the variety ${\bf E}$ if and only if the claims~\eqref{sim(u)=sim(v) & mul(u)=mul(v)} and~\eqref{eq D_1 subseteq X} are true.
\end{proposition}
 
\begin{proof}
\emph{Necessity}. Suppose that ${\bf E}$ satisfies an identity ${\bf u}\approx {\bf v}$. The inclusion ${\bf D}_1 \subseteq{\bf E}$ and Proposition~\ref{word problem D_1} imply that this identity satisfies the claim~\eqref{sim(u)=sim(v) & mul(u)=mul(v)}. Suppose that the claim~\eqref{eq D_1 subseteq X} is false. Then Lemma~\ref{D_1 subseteq X} applies with the conclusion that \textbf E satisfies an identity of the form~\eqref{x^syx^t=yx^r} where $s\ge 1$, $t\ge 0$, $s+t\ge2$ and $r\ge 2$. Let us consider the semigroup
$$
P=\langle e,a\mid e^2=e,\,ae=a,\,ea=0\rangle=\{e,a,0\}.
$$
Note that \textbf E contains the monoid $P^1$, i.e., the semigroup $P$ with a new identity element adjoined. Making the substitution $(x,y)\mapsto(e,a)$ in the identity~\eqref{x^syx^t=yx^r} results in the contradiction $0=a$. Thus, $P^1$ and therefore, \textbf E violates~\eqref{x^syx^t=yx^r}, a contradiction.
 
\medskip
 
\emph{Sufficiency}. Suppose that the identity ${\bf u}\approx {\bf v}$ satisfies the claims~\eqref{sim(u)=sim(v) & mul(u)=mul(v)} and~\eqref{eq D_1 subseteq X}. Let~\eqref{t_0u_0t_1u_1 ... t_mu_m} be the 0-decomposition of $\bf u$. In view of Lemma~\ref{k-equivalent}, the 0-decomposition of $\bf v$ has the form~\eqref{t_0v_0t_1v_1 ... t_mv_m}. We are going to verify that ${\bf u}\approx {\bf v}$ holds in ${\bf E}$. Recall that the variety \textbf E is given by the identity system
\begin{equation}
\label{xx=xxx,xxy=xyx,xxyy=yyxx}
\{x^2\approx x^3,\,x^2y\approx xyx,\,x^2y^2\approx y^2x^2\}.
\end{equation}
Put $X=\con({\bf u}_0)=\{x_1,x_2,\dots,x_k\}$. Clearly, any block of an arbitrary word \textbf w does not contain letters simple in \textbf w. Therefore, we may assume without loss of generality that $\mathbf u_0=x_1^2x_2^2\cdots x_k^2$.
 
We will use induction on the parameter $m$ from~\eqref{t_0u_0t_1u_1 ... t_mu_m} and~\eqref{t_0v_0t_1v_1 ... t_mv_m}.
 
\smallskip
 
\emph{Induction base}. Let $m=0$. The claim~\eqref{sim(u)=sim(v) & mul(u)=mul(v)} implies that $\con(\mathbf u_0)=\con(\mathbf v_0)$. Since the variety \textbf E satisfies the identity
\begin{equation}
\label{xxyy=yyxx}
x^2y^2\approx y^2x^2,
\end{equation}
it also satisfies the identity $\mathbf v_0\approx x_1^2x_2^2\cdots x_k^2$. Therefore, the identities
$$
\mathbf u=t_0\mathbf u_0=t_0x_1^2x_2^2\cdots x_k^2\approx t_0\mathbf v_0=\mathbf v
$$
hold in \textbf E.
 
\smallskip
 
\emph{Induction step}. Let now $m>0$. The identity system~\eqref{xx=xxx,xxy=xyx,xxyy=yyxx} implies the identity
$$
{\bf u}\approx t_0x_1^2x_2^2\cdots x_k^2 t_1({\bf u}_1)_X\cdots t_m({\bf u}_m)_X.
$$
By the claim~\eqref{eq D_1 subseteq X}, $\con(\mathbf u_0)=\con(\mathbf v_0)$, whence the identity system~\eqref{xx=xxx,xxy=xyx,xxyy=yyxx} implies the identity
$$
{\bf v}\approx t_0x_1^2x_2^2\cdots x_k^2 t_1({\bf v}_1)_X\cdots t_m({\bf v}_m)_X.
$$
 
Put ${\bf u}'=({\bf u}_1)_X\cdots t_m({\bf u}_m)_X$ and ${\bf v}'=({\bf v}_1)_X\cdots t_m({\bf v}_m)_X$. It is easy to verify that the identity ${\bf u}' \approx {\bf v}'$ satisfies the claims~\eqref{sim(u)=sim(v) & mul(u)=mul(v)} and~\eqref{eq D_1 subseteq X}. By the induction assumption, the identity ${\bf u}' \approx {\bf v}'$ holds in ${\bf E}$, whence this variety satisfies
$$
\mathbf u\approx t_0x_1^2x_2^2\cdots x_k^2t_1\mathbf u'\approx t_0x_1^2x_2^2\cdots x_k^2t_1\mathbf v'\approx \mathbf v.
$$
Thus, ${\bf u}\approx {\bf v}$ holds in \textbf E.
\end{proof}
 
\begin{lemma}
\label{E nsubseteq X}
Let ${\bf X}$ be a non-completely regular variety of monoids. If ${\bf E} \nsubseteq {\bf X}$ and ${\bf X}$ satisfies the identity
\begin{equation}
\label{xx=xxx}
x^2\approx x^3
\end{equation}
then ${\bf X}$ satisfies also the identity
\begin{equation}
\label{yxx=xxyxx}
yx^2\approx x^2yx^2.
\end{equation}
\end{lemma}
 
\begin{proof}
If \textbf X is commutative then we apply the identity~\eqref{xx=xxx} and obtain \textbf X satisfies the identities $yx^2\approx yx^4\approx x^2yx^2$. Suppose now that \textbf X is non-commutative. Then Lemma~\ref{non-cr and non-commut} implies that $\mathbf D_1\subseteq\mathbf X$. Since ${\bf E} \nsubseteq {\bf X}$, there is an identity $\mathbf{u\approx v}$ that holds in \textbf X but fails in \textbf E. Then Proposition~\ref{word problem E} applies with the conclusion that either~\eqref{sim(u)=sim(v) & mul(u)=mul(v)} or~\eqref{eq D_1 subseteq X} is false. Proposition~\ref{word problem C_2} implies that~\eqref{sim(u)=sim(v) & mul(u)=mul(v)} is true because $\mathbf C_2 \subseteq \mathbf D_1 \subseteq \mathbf X$. Therefore, the claim~\eqref{eq D_1 subseteq X} is false. Now Lemma~\ref{D_1 subseteq X} applies and we conclude that \textbf X satisfies an identity of the form~\eqref{x^syx^t=yx^r} where $s\ge 1$, $t\ge 0$, $s+t\ge2$ and $r\ge 2$. Substitute $x^2$ for $x$ in this identity. Since \textbf X satisfies the identity~\eqref{xx=xxx}, we obtain the identity~\eqref{yxx=xxyxx} holds in \textbf X.
\end{proof}
 
Let us return to an examination of a chain variety \textbf V. Recall that we reduce considerations to the case when $\mathbf D_1\subseteq\mathbf V$. Hence ${\bf C}_3 \nsubseteq {\bf V}$ because the varieties ${\bf C}_3$ and ${\bf D}_1$ are incomparable. Then Lemma~\ref{C_{n+1} nsubseteq V} and the fact that ${\bf V}$ is aperiodic imply that the identity~\eqref{xx=xxx} holds in ${\bf V}$. Suppose now that ${\bf D}_2 \nsubseteq {\bf V}$. The variety ${\bf V}$ does not contain at least one of the incomparable varieties ${\bf E}$ and $\overleftarrow{{\bf E}}$. Assume without loss of generality that $\overleftarrow{{\bf E}} \nsubseteq {\bf V}$. The dual of Lemma~\ref{E nsubseteq X} implies then that ${\bf V}$ satisfies the identity
\begin{equation}
\label{xxy=xxyxx}
x^2y\approx x^2yx^2.
\end{equation}
Further, Lemma~\ref{D_{n+1} nsubseteq X} implies that the identity
\begin{equation}
\label{xyx=x^qyx^r}
xyx\approx x^qyx^r
\end{equation}
with $q>1$ or $r>1$ holds in ${\bf V}$.
 
If \textbf u and \textbf v are words and $\varepsilon$ is an identity then we will write $\mathbf u\stackrel{\varepsilon}\approx\mathbf v$ in the case when the identity $\mathbf u\approx\mathbf v$ follows from $\varepsilon$. If $q>1$ then ${\bf V}$ satisfies the identities
$$
xyx\stackrel{\eqref{xyx=x^qyx^r}}\approx x^qyx^r\stackrel{\eqref{xxy=xxyxx}}\approx x^qyx^{r+2}\stackrel{\eqref{xx=xxx}}\approx x^2yx^2\stackrel{\eqref{xxy=xxyxx}}\approx x^2y.
$$
Recall that \textbf V satisfies the identity~\eqref{xx=xxx} too. Then Lemma~\ref{L(LRB+C_2)}(i) applies and we conclude that ${\bf V} \subseteq {\bf LRB}\vee\mathbf C_2$. Since ${\bf V}$ is non-idempotent and chain, ${\bf V\subseteq E}$ by Lemma~\ref{L(LRB+C_2)}(ii). Therefore, $\mathbf{V\subseteq K}$.
 
Suppose now that $q\le1$. Then $r>1$. If $q=0$ then ${\bf V} \subseteq \mathbf{RRB}\vee\mathbf C_2$ by the dual of Lemma~\ref{L(LRB+C_2)}(i) because ${\bf V}$ satisfies the identity~\eqref{xx=xxx}. Since $\overleftarrow{{\bf E}} \nsubseteq {\bf V}$ and \textbf V is not a variety of band monoids, it follows from the dual of Lemma~\ref{L(LRB+C_2)}(ii) that ${\bf V}\subseteq{\bf D}_1\subseteq\mathbf D$.
 
Let now $q=1$. Then ${\bf V}$ satisfies the identity
\begin{equation}
\label{xyx=xyxx}
xyx\approx xyx^2
\end{equation}
because it satisfies~\eqref{xx=xxx}. Therefore, the identities $x^2yx\stackrel{\eqref{xyx=xyxx}}\approx x^2yx^2\stackrel{\eqref{xxy=xxyxx}}\approx x^2y$ hold in \textbf V. Thus, \textbf V satisfies
\begin{equation}
\label{xxy=xxyx}
x^2y\approx x^2yx.
\end{equation}
Recall that $\mathbf D_1\subseteq\mathbf V$. Therefore, $\mathbf{LRB\nsubseteq V}$. Hence there is an identity $\mathbf{u\approx v}$ that holds in \textbf V but fails in \textbf{LRB}. The \emph{initial part} of a word \textbf w, denoted by $\ini(\mathbf w)$, is the word obtained from \textbf w by retaining first occurrence of each letter. It is evident that an identity $\mathbf{a\approx b}$ holds in the variety $\mathbf{LRB}$ if and only if $\ini(\mathbf a)=\ini(\mathbf b)$. Hence $\mathbf{\ini(u)\ne\ini(v)}$. Proposition~\ref{word problem C_2} implies that $\con(\mathbf u)=\con(\mathbf v)$. Therefore, we can assume without any loss of generality that there are letters $x,y\in \con({\bf u})$ such that ${\bf u}(x,y)=x^sy{\bf w}_1$ and ${\bf v}(x,y)=y^tx{\bf w}_2$ where $s,t>0$ and $\con(\mathbf w_1)=\con(\mathbf w_2)=\{x,y\}$. Let us substitute~1 for all letters except $x$ and $y$ in ${\bf u}\approx {\bf v}$. We obtain that \textbf V satisfies the identity $x^sy\mathbf w_1\approx y^tx\mathbf w_2$. If $s=1$ then we substitute $x^2$ for $x$ in this identity and obtain an identity of the form $x^2y\mathbf w_1'\approx y^tx^2\mathbf w_2'$. Thus, we can assume that $s\ge 2$. Analogously, we can assume that $t\ge 2$. Moreover, the identity~\eqref{xx=xxx} allows us to assume that $s=t=2$. Now we can apply the identity~\eqref{xxy=xxyx} and deduce an identity of the form $x^2y^k\approx y^2x^m$ where $k,m>1$. Moreover, the identity~\eqref{xx=xxx} allows us to assume that $k=m=2$. We prove that the identity~\eqref{xxyy=yyxx} holds in ${\bf V}$. This means that ${\bf V} \subseteq {\bf K}$.
 
It remains to consider the case when ${\bf D}_2 \subseteq {\bf V}$. 
 
\subsection{Reduction to the case when $\mathbf{L\subseteq V}$}
\label{necessity: red to over L}
 
Here we need some notation and a series of auxiliary assertions. Let $n$ and $m$ be arbitrary non-negative integers such that $n+m>0$. For any $\theta\in S_{n+m}$, we put
\begin{align*}
\mathbf w_{n,m}(\theta)&=\biggl(\prod_{i=1}^n z_it_i\biggr) x \biggl(\prod_{i=1}^{n+m} z_{\theta(i)}\biggr) x \biggl(\prod_{i=n+1}^{n+m} t_iz_i\biggr),\\[-3pt]
\mathbf w_{n,m}'(\theta)&=\biggl(\prod_{i=1}^n z_it_i\biggr) x^2 \biggl(\prod_{i=1}^{n+m} z_{\theta(i)}\biggr)\biggl(\prod_{i=n+1}^{n+m} t_iz_i\biggr).
\end{align*}
Note that the words $\mathbf w_n(\pi,\tau)$ and $\mathbf w_n'(\pi,\tau)$ introduced in Section~\ref{introduction} are words of the form $\mathbf w_{n,n}(\theta)$ and $\mathbf w_{n,n}'(\theta)$ respectively for an appropriate permutation $\theta\in S_{2n}$.
 
\begin{lemma}
\label{w_{n,m}=w_{n,m}' in L}
The variety $\mathbf L$ satisfies the identities of the form
\begin{equation}
\label{w_{n,m}=w_{n,m}'}
{\bf w}_{n,m}(\theta)\approx {\bf w}'_{n,m}(\theta)
\end{equation}
for all $n$, $m$ and $\theta\in S_{n+m}$.
\end{lemma}
 
\begin{proof}
It suffices to verify that each identity of the form~\eqref{w_{n,m}=w_{n,m}'} follows from some identity of the form
\begin{equation}
\label{w_n=w_n'}
{\bf w}_n(\pi,\tau)\approx {\bf w}'_n(\pi,\tau).
\end{equation}
To do this, we fix an identity of the form~\eqref{w_{n,m}=w_{n,m}'}. It has the form
$$
\mathbf p_0x\mathbf q_0x\mathbf r_0\approx\mathbf p_0x^2\mathbf q_0\mathbf r_0
$$
where $\mathbf p_0=z_1t_1\cdots z_nt_n$, $\mathbf q_0=z_{\theta(1)}\cdots z_{\theta(n+m)}$ and $\mathbf r_0=t_{n+1}z_{n+1}\cdots t_{n+m}z_{n+m}$. The word $\textbf q_0$ may be uniquely represented in the form
$$
\mathbf q_0=\mathbf u_1\mathbf v_1\cdots\mathbf u_k\mathbf v_k
$$
where $\con(\mathbf u_1\cdots\mathbf u_k)=\{z_1,\dots,z_n\}$ and $\con(\mathbf v_1\cdots\mathbf v_k)=\{z_{n+1},\dots,z_{n+m}\}$ (we mean here that $\mathbf u_1=\lambda$ whenever $\theta(1)>n$, and $\mathbf v_k=\lambda$ whenever $\theta(n+m)\le n$). Each of the words $\mathbf u_1,\dots,\mathbf u_k$ (except $\mathbf u_1$ whenever $\mathbf u_1=\lambda$) has the form $z_{j_1}\cdots z_{j_s}$ where $j_1,\dots,j_s\le n$, while each of the words $\mathbf v_1,\dots,\mathbf v_k$ (except $\mathbf v_k$ whenever $\mathbf v_k=\lambda$) has the form $z_{j_1}\cdots z_{j_s}$ where $j_1,\dots,j_s>n$.
 
Suppose at first that $\mathbf u_1=\lambda$. Let $z$ and $t$ be letters that do not occur in the word $\mathbf p_0\mathbf q_0\mathbf r_0x$. Put $\mathbf p'=zt\mathbf p_0$, $\mathbf q'=z\mathbf q_0$ and $\mathbf r'=\mathbf r_0$. The identity $\mathbf p'x\mathbf q'x\mathbf r'\approx\mathbf p'x^2\mathbf q'\mathbf r'$ evidently implies the identity~\eqref{w_{n,m}=w_{n,m}'}. Up to the evident renaming of letters, the identity $\mathbf p'x\mathbf q'x\mathbf r'\approx\mathbf p'x^2\mathbf q'\mathbf r'$ has the form indicated in the previous paragraph with $\mathbf u_1\ne\lambda$. Thus, we can assume that $\mathbf u_1\ne\lambda$. Analogous arguments allow us to suppose that $\mathbf v_k\ne\lambda$.
 
Let now $\mathbf u_1=z_{j_1}\cdots z_{j_s}$ with $j_1,\dots,j_s\le n$. Let $z_{j_1}'$, $t_{j_1}'$, \dots, $z_{j_{s-1}}'$, $t_{j_{s-1}}'$ be letters that do not occur in the word $\mathbf p_0\mathbf q_0\mathbf r_0x$. Put $\mathbf p_1=\mathbf p_0$. Denote by $\mathbf q_1$ the word that is obtained from $\mathbf q_0$ by replacing of the word $\mathbf u_1$ with $z_{j_1}z_{j_1}'\cdots z_{j_{s-1}}z_{j_{s-1}}'z_{j_s}$. Finally, we put $\mathbf r_1=\mathbf r_0t_{j_1}'z_{j_1}'\cdots t_{j_{s-1}}'z_{j_{s-1}}'$. The identity $\mathbf p_0x\mathbf q_0x\mathbf r_0\approx \mathbf p_0x^2\mathbf q_0\mathbf r_0$ follows from $\mathbf p_1x\mathbf q_1x\mathbf r_1\approx\mathbf p_1x^2\mathbf q_1\mathbf r_1$ because the former identity may be obtained by substitution of~1 for the letters $z_{j_1}',t_{j_1}',\dots,z_{j_{s-1}}'$, $t_{j_{s-1}}'$ in the latter identity.
 
Further, let $\mathbf v_1=z_{j_1}\cdots z_{j_s}$ where $j_1,\dots,j_s>n$. Let $z_{j_1}',t_{j_1}',\dots,z_{j_{s-1}}',t_{j_{s-1}}'$ be letters that do not occur in $\mathbf p_1\mathbf q_1\mathbf r_1x$. Put $\mathbf p_2=z_{j_1}'t_{j_1}'\cdots z_{j_{s-1}}'t_{j_{s-1}}'\mathbf p_1$. Further, we denote by $\mathbf q_2$ the word that is obtained from $\mathbf q_1$ by replacing of the word $\mathbf v_1$ with $z_{j_1} z_{j_1}'\cdots z_{j_{s-1}}z_{j_{s-1}}'z_{j_s}$. Finally, we put $\mathbf r_2=\mathbf r_1$. The identity $\mathbf p_1x\mathbf q_1x\mathbf r_1\approx\mathbf p_1x^2\mathbf q_1\mathbf r_1$ follows from $\mathbf p_2x\mathbf q_2x\mathbf r_2\approx\mathbf p_2x^2\mathbf q_2\mathbf r_2$ because the former identity may be obtained by substitution of~1 for the letters $z_{j_1}'$, $t_{j_1}'$, \dots, $z_{j_{s-1}}'$, $t_{j_{s-1}}'$ in the latter identity.
 
We continue this process and apply the analogous modifications of our identity with the use of the words $\mathbf u_2,\mathbf v_2,\dots,\mathbf u_k,\mathbf v_k$. As a result, we obtain an identity of the form
\begin{equation}
\label{resulting identity}
\mathbf p_{2k}x\mathbf q_{2k}x\mathbf r_{2k}\approx\mathbf p_{2k}x^2\mathbf q_{2k}\mathbf r_{2k}
\end{equation}
that implies the identity of the form~\eqref{w_{n,m}=w_{n,m}'} fixed at the beginning of the proof. We can evidently rename the letters and assume that $\mathbf p_{2k}=z_1t_1\cdots z_pt_p$, $\mathbf q_{2k}=z_{\xi(1)}\cdots z_{\xi(p+q)}$ and $\mathbf r_{2k}=t_{p+1}z_{p+1}\cdots t_{p+q}z_{p+q}$ for some natural numbers $p,q$ and some permutation $\xi\in S_{p+q}$ with $\xi(i)\le p$ for all odd $i$ and $\xi(i)>p$ for all even $i$. It remains to verify that $p=q$. For $i=1,\dots,k$, we denote the length of the word $\mathbf u_i$ by $n_i$ and the length of the word $\mathbf v_i$ by $m_i$. Then $n_1+\cdots+n_k=n$ and $m_1+\cdots+m_k=m$. It is easy to see that
\begin{align*}
p={}&n+(m_1-1)+\cdots+(m_k-1)=n+m-k\\
={}&m+n-k=m+(n_1-1)+\cdots+(n_k-1)=q.
\end{align*}
Therefore, the identity~\eqref{resulting identity} has the form~\eqref{w_n=w_n'}.
\end{proof}
 
\begin{lemma}
\label{u=xxu_x holds in X}
Suppose that a monoid variety ${\bf X}$ satisfies the identities
\begin{align}
\label{xyxzx=xxyz}
&xyxzx\approx x^2yz,\\
\label{xxy=yxx}
&x^2y\approx yx^2
\end{align}
and~\eqref{w_{n,m}=w_{n,m}'} for all $n$, $m$ and $\theta\in S_{n+m}$. Let ${\bf u}$ be a word. If there is a letter $x\in\mul({\bf u})$ such that ${\bf u}(x,y)\ne xyx$ for any letter $y$ then ${\bf X}$ satisfies the identity
\begin{equation}
\label{u=xxu_x}
{\bf u}\approx x^2{\bf u}_x.
\end{equation}
\end{lemma}
 
\begin{proof}
Suppose at first that $\occ_x(\mathbf u)>2$. Then $\mathbf u = \mathbf u_1 x \mathbf u_2 x \mathbf u_3 \cdots \mathbf u_n x \mathbf u_{n+1}$ where $n>2$ and $\mathbf u_1,\mathbf u_2, \dots, \mathbf u_{n+1}$ are possibly empty words with $x\notin \con(\mathbf u_1\mathbf u_2\cdots \mathbf u_{n+1})$. Clearly, $\mathbf u_1\mathbf u_2\cdots\mathbf u_{n+1}=\mathbf u_x$. Then \textbf X satisfies the identities{\sloppy

}
$$
\mathbf u\!=\!\mathbf u_1x\mathbf u_2x\mathbf u_3\cdots\mathbf u_nx\mathbf u_{n+1}\!\stackrel{\eqref{xyxzx=xxyz}}\approx\!\mathbf u_1x^2\mathbf u_2\mathbf u_3\cdots\mathbf u_{n+1}\!\stackrel{\eqref{xxy=yxx}}\approx\!x^2\mathbf u_1\mathbf u_2\cdots\mathbf u_{n+1}\!=\!x^2\mathbf u_x,
$$
whence~\eqref{u=xxu_x} holds in \textbf X.
 
It remains to consider the case when $\occ_x({\bf u})=2$. Then ${\bf u}= {\bf u}_1x\mathbf u_2x{\bf u}_3$ and $x\notin \con(\mathbf u_1\mathbf u_2\mathbf u_3)$. If $\mathbf u_2=\lambda$ then $\mathbf u=\mathbf u_1x^2\mathbf u_3\stackrel{\eqref{xxy=yxx}}\approx x^2\mathbf u_1\mathbf u_3=x^2\mathbf u_x$ hold in \textbf X, and we are done. Let now $\mathbf u_2\ne\lambda$.
 
If $y\in \con(\mathbf u_2)$ and $y\in\simple(\mathbf u)$ then $\mathbf u(x,y)=xyx$, a contradiction. Thus, $y\in\mul(\mathbf u)$ for any $y\in \con(\mathbf u_2)$. Suppose that $\occ_y(\mathbf u)>2$ for some $y\in\con(\mathbf u_2)$. Then we can use the same arguments as in the first paragraph of the proof and conclude that the variety \textbf X satisfies the identity $\mathbf u\approx y^2\mathbf u_y$. This identity can be rewritten in the form $\mathbf u\approx\mathbf u_1'x\mathbf u_2'x\mathbf u_3'$ where $\mathbf u_1'=y^2\mathbf u_1$, $\mathbf u_2'=(\mathbf u_2)_y$ and $\mathbf u_3'=(\mathbf u_3)_y$. Thus, we can remove from $\mathbf u_2$ all letters $y$ with $\occ_y(\mathbf u)>2$. In other words, we can assume that either $\mathbf u_2=\lambda$ or $\occ_y(\mathbf u)=2$ for all $y\in \con(\mathbf u_2)$. The former case is already considered in the previous paragraph. Now we examine the latter case.
 
Recall that a word \textbf w is called \emph{linear} if $\occ_x(\mathbf w)\le 1$ for any letter $x$. Suppose that the word $\mathbf u_2$ is linear, say, $\mathbf u_2=y_1y_2\cdots y_k$ for some letters $y_1,y_2$, \dots, $y_k$. Then either $y_i\in \con(\mathbf u_1)\setminus \con(\mathbf u_3)$ or $y_i\in \con(\mathbf u_3)\setminus \con(\mathbf u_1)$ for any $1\le i\le k$. Renaming, if necessary, the letters $y_1,y_2,\dots,y_k$, we may assume that $y_1,y_2,\dots,y_n\in \con(\mathbf u_1)\setminus \con(\mathbf u_3)$ and $y_{n+1},\dots,y_{n+m}\in \con(\mathbf u_3)\setminus \con(\mathbf u_1)$ for some $n$ and $m$ with $n+m=k$. Then
$$
\mathbf u=\mathbf u_1xy_{\theta(1)}y_{\theta(2)}\cdots y_{\theta(n+m)}x\mathbf u_3
$$
for an appropriate permutation $\theta\in S_{n+m}$. We have also
$$
\mathbf u_1=\mathbf w_0y_1\mathbf w_1y_2\mathbf w_2\cdots y_n\mathbf w_n, \mathbf u_3=\mathbf w_{n+1}y_{n+1}\mathbf w_{n+2}y_{n+2}\cdots \mathbf w_{n+m}y_{n+m}\mathbf w_{n+m+1}
$$
for some possibly empty words $\mathbf w_0,\mathbf w_1,\dots,\mathbf w_{n+m+1}$. Then \textbf X satisfies the identities
\begin{align*}
\mathbf u&\stackrel{\phantom{\eqref{w_{n,m}=w_{n,m}'}}}=\mathbf w_0\biggl(\prod_{i=1}^ny_i\mathbf w_i\biggr)x\biggl(\prod_{i=1}^{n+m}y_{\theta(i)}\biggr)x\biggl(\prod_{i=n+1}^{n+m}\mathbf w_iy_i\biggr)\mathbf w_{n+m+1}
\end{align*}
\begin{align*}
&\stackrel{\eqref{w_{n,m}=w_{n,m}'}}\approx\mathbf w_0\biggl(\prod_{i=1}^ny_i\mathbf w_i\biggr)x^2\biggl(\prod_{i=1}^{n+m}y_{\theta(i)}\biggr)\biggl(\prod_{i=n+1}^{n+m}\mathbf w_iy_i\biggr)\mathbf w_{n+m+1}\\[-3pt]
&\stackrel{\eqref{xxy=yxx}}\approx x^2\mathbf w_0\biggl(\prod_{i=1}^ny_i\mathbf w_i\biggr)\biggl(\prod_{i=1}^{n+m}y_{\theta(i)}\biggr)\biggl(\prod_{i=n+1}^{n+m}\mathbf w_iy_i\biggr)\mathbf w_{n+m+1}\\[-3pt]
&\stackrel{\phantom{\eqref{w_{n,m}=w_{n,m}'}}}=x^2\mathbf u_x.
\end{align*}
We have \textbf X satisfies the identity~\eqref{u=xxu_x} again.
 
It remains to consider the case when the word $\mathbf u_2$ is not linear. Then there is a letter $y\in \con(\mathbf u_2)$ such that $\mathbf u_2=\mathbf v_1y\mathbf v_2y\mathbf v_3$ where $\mathbf v_1$, $\mathbf v_2$ and $\mathbf v_3$ are possibly empty words, $y\notin \con(\mathbf v_1\mathbf v_2\mathbf v_3)$ and the word $\mathbf v_2$ is either empty or linear. If $\mathbf v_2$ is linear then the same arguments as in the previous paragraph show that
$$
\mathbf u=\mathbf u_1x\mathbf v_1y\mathbf v_2y\mathbf v_3x\mathbf u_3\approx y^2\mathbf u_y=y^2\mathbf u_1x\mathbf v_1\mathbf v_2\mathbf v_3x\mathbf u_3=\mathbf u_1'x\mathbf u_2'x\mathbf u_3
$$
hold in \textbf X where $\mathbf u_1'=y^2\mathbf u_1$ and $\mathbf u_2'=\mathbf v_1\mathbf v_2\mathbf v_3$. If $\mathbf v_2=\lambda$ then
$$
\mathbf u=\mathbf u_1x\mathbf v_1y^2\mathbf v_3x\mathbf u_3\stackrel{\eqref{xxy=yxx}}\approx y^2\mathbf u_1x\mathbf v_1\mathbf v_3x\mathbf u_3=\mathbf u_1'x\mathbf u_2'x\mathbf u_3
$$
is valid in \textbf X where $\mathbf u_1'=y^2\mathbf u_1$ and $\mathbf u_2'=\mathbf v_1\mathbf v_3$. In both the cases $y\notin \con(\mathbf u_2')$. In other words, we can remove the letter $y$ from $\mathbf u_2$. Further, we will repeat these arguments as long as the word $\mathbf u_2$ will remain non-empty and non-linear. In other words, we may assume that the word $\mathbf u_2$ is either empty or linear. Both these cases have been already considered above. Thus, we prove that \textbf X satisfies the identity~\eqref{u=xxu_x} always.
\end{proof}

\begin{lemma}
\label{L = var S(xzxyty)}
${\bf L}=\var S(xzxyty)$.
\end{lemma}
 
\begin{proof}
Put ${\bf Z}=\var S(xzxyty)$. First, we are going to verify that $\mathbf{ Z \subseteq L}$. In view of Lemma~\ref{S(W) in V}, to achieve this aim it suffices to check that the word $xzxyty$ is an isoterm for $\bf L$. Put
$$
\Psi\!=\!\{x^2y\!\approx\!yx^2,\,xyxzx\!\approx\!x^2yz,\,\sigma_1,\,\sigma_2,\,{\bf w}_n(\pi,\tau)\!\approx\!{\bf w}'_n(\pi,\tau)\mid n\in\mathbb N,\ \pi,\tau\in S_n\}.
$$
We recall that $\mathbf L=\var\Psi$. We suppose that $\bf L$ satisfies a non-trivial idenity $xzxyty\approx \bf w$ for some word $\bf w$. Therefore, there exists a \emph{deduction} of the identity $xzxyty\approx \mathbf w$  from the identity system $\Psi$, i.e., a sequence of words
\begin{equation}
\label{sequence of words}
\mathbf v_0,\mathbf v_1,\ldots, \mathbf v_m
\end{equation}
such that $\mathbf v_0 = xzxyty$, $\mathbf v_m = \mathbf w$ and, for any $0\le i<m$, there exist words $\mathbf a_i$, $\mathbf b_i$, the identity $\mathbf s_i\approx \mathbf t_i\in\Psi$ and endomorphism $\xi_i$ of $F^1$ such that either $\mathbf v_i=\mathbf a_i\xi_i(\mathbf s_i)\mathbf b_i$ and $\mathbf v_{i+1}=\mathbf a_i\xi_i(\mathbf t_i)\mathbf b_i$ or $\mathbf v_i=\mathbf a_i\xi_i(\mathbf t_i)\mathbf b_i$ and $\mathbf v_{i+1}=\mathbf a_i\xi_i(\mathbf s_i)\mathbf b_i$. We can assume without loss of generality that the sequence~\eqref{sequence of words} is the shortest deduction of the identity $xzxyty\approx \mathbf w$  from the identity system $\Psi$. In particular, this means that $xzxyty\ne\mathbf v_1$. We note that if $\xi_0(x)=\lambda$ then $\xi_0(\mathbf s_0)=\xi_0(\mathbf t_0)$ for any $\mathbf s_0\approx\mathbf t_0\in\Psi$. The latest equality implies that $xzxyty=\mathbf v_1$, but this is impossible. Thus, we can assume that $\xi_0(x)\ne\lambda$.

Suppose that $xzxyty=\mathbf v_0=\mathbf a_0\xi_0(\mathbf s_0)\mathbf b_0$ and $\mathbf v_1=\mathbf a_0\xi_0(\mathbf t_0)\mathbf b_0$. The case when $\mathbf s_0=x^2y$ is impossible because the word $\xi_0(\mathbf s_0)$ contains the square of a non-empty word, while the word $xzxyty$ is square-free. The case when $\mathbf s_0=xyxzx$ is also impossible because there is a  letter that occurs in the word $\xi_0(\mathbf s_0)$ at least three times, while every letter from $\con(xzxyty)$ occurs in the word $xzxyty$ no more than twice. Finally, the case when $\mathbf s_0={\bf w}_n(\pi,\tau)$ for some $n\in\mathbb N,\ \pi,\tau\in S_n$ is impossible because there exists a letter $c\in\xi_0(x)$ such that $c$ is multiple in $\xi_0(\mathbf s_0)$ and every letter located between the first and the second occurrences of $c$ in $\xi_0(\mathbf s_0)$ is multiple, while for every $d\in\mul(xzxyty)$ there is a letter $e\in \simple(xzxyty)$ such that $e$ lies between the first and the second occurrences of $d$ in $xzxyty$. So, the identity $\mathbf s_0\approx \mathbf t_0$ is either $\sigma_1$ or $\sigma_2$. By symmetry, we can consider only the case when $\mathbf s_0\approx \mathbf t_0$ is equal to $\sigma_1$. Then $\mathbf s_0=xyzxty$ and $\mathbf t_0=yxzxty$. Since $\xi_0(x)\ne\lambda$, we have $\con(\xi_0(x))$ contains a letter $a$. Then $a\in\{x,y\}$ because $a\in\mul(\xi_0(\mathbf s_0))$. Suppose that $a=x$. Then $\xi_0(y)=\lambda$ because
$$
xzxyty=\mathbf a_0\xi_0(\mathbf s_0)\mathbf b_0=\mathbf a_0\xi_0(x)\xi_0(y)\xi_0(z)\xi_0(x)\xi_0(t)\xi_0(y)\mathbf b_0.
$$
Therefore, $\xi_0(\mathbf t_0)=\xi_0(x)\xi_0(z)\xi_0(x)\xi_0(t)=\xi_0(\mathbf s_0)$. Then
$$
\mathbf v_1=\mathbf a_0\xi_0(\mathbf t_0)\mathbf b_0=\mathbf a_0\xi_0(\mathbf s_0)\mathbf b_0=xzxyty,
$$
contradicting the choice of the sequence~\eqref{sequence of words}. The case when $a=y$  is considered similarly.

Suppose now that $xzxyty=\mathbf v_0=\mathbf a_0\xi_0(\mathbf t_0)\mathbf b_0$. The case when
$$
\mathbf t_0\in\{yx^2,\,x^2yz,\,{\bf w}'_n(\pi,\tau)\mid n\in\mathbb N,\ \pi,\tau\in S_n\}
$$
is impossible because the word $\xi_0(\mathbf t_0)$ contains the square of a non-empty word in this case, while the word $xzxyty$ is square-free. So, the identity $\mathbf s_0\approx \mathbf t_0$ is either $\sigma_1$ or $\sigma_2$. Arguments similar to those from the previous paragraph allow us to obtain a contradiction with the fact that the words $xzxyty$ and $\mathbf v_1$ are distinct.

Thus, we have verified that $xzxyty$ is an isoterm for $\bf L$ and therefore, $\bf Z \subseteq L$. It remains to verify the opposite inclusion. Suppose that ${\bf Z}$ satisfies an identity ${\bf u}\approx {\bf v}$. We need to prove that ${\bf u}\approx {\bf v}$ holds in ${\bf L}$. Lemma~\ref{w_{n,m}=w_{n,m}' in L} allows us to use Lemma~\ref{u=xxu_x holds in X} below. Let $x$ be a letter multiple in $\bf u$ and ${\bf u}(x,y)\ne xyx$ for any letter $y$. By Lemma~\ref{u=xxu_x holds in X}, the variety ${\bf L}$ satisfies the identity~\eqref{u=xxu_x}. Obviously, ${\bf C}_2 \subseteq {\bf Z}$, whence $x\in \mul({\bf v})$ by Proposition~\ref{word problem C_2}. Since the word $xzxyty$ is an isoterm for ${\bf Z}$, the word $xyx$ is an isoterm for ${\bf Z}$ too. Therefore, ${\bf v}(x,y)\ne xyx$ for any letter $y$. We apply Lemma~\ref{u=xxu_x holds in X} again and conclude that the identity $\mathbf v \approx x^2 \mathbf v_x$ holds in ${\bf L}$. Thus, if the identity ${\bf u}_x\approx {\bf v}_x$ holds in the variety ${\bf L}$ then this variety satisfies the identities $\mathbf u \approx x^2 \mathbf u_x \approx x^2 \mathbf v_x \approx \mathbf v$. So, we can remove from $\mathbf u \approx \mathbf v$ all multiple letters $x$ such that ${\bf u}(x,y)\ne xyx$ for any $y$. In other words, we may assume without loss of generality that for any letter $x\in\mul({\bf u})$ there is a letter $y$ such that ${\bf u}(x,y)= xyx = {\bf v}(x,y)$. In particular, this means that $\occ_x(\mathbf u),\occ_x(\mathbf v)\le 2$ for any letter $x$.
 
Lemma~\ref{group variety} and the evident inclusion $\mathbf C_2\subseteq\mathbf Z$ imply that $\con(\mathbf u)=\con(\mathbf v)$. It is evident that for any letters $a,b\notin\con(\mathbf u)$, the identities $\mathbf{u\approx v}$ and $a\mathbf ub\approx a\mathbf vb$ are equivalent in the class of monoids. Therefore, we can assume without loss of generality that the first and the last letters in each of the words \textbf u and \textbf v are simple in this word. Let $\simple(\mathbf u)=\simple(\mathbf v)=\{t_0,t_1,\dots,t_m\}$. We can assume without any loss that $\mathbf v(t_1,t_2,\dots,t_m)=t_1t_2\cdots t_m$. In view of Lemma~\ref{D_n=var S(w)}, $\mathbf D_1 \subseteq\mathbf Z$. Then Proposition~\ref{word problem D_1} implies that
$$
\mathbf u=t_0\mathbf a_1t_1\mathbf a_2t_2\cdots t_{m-1}\mathbf a_mt_m\text{ and }\mathbf v=t_0\mathbf b_1t_1\mathbf b_2t_2\cdots t_{m-1}\mathbf b_mt_m
$$
for some possibly empty words $\mathbf a_1,\mathbf a_2,\dots,\mathbf a_m$ and $\mathbf b_1,\mathbf b_2,\dots,\mathbf b_m$.
 
Let $0\le i\le m-1$. Then $\mathbf u=\mathbf w_1t_i\mathbf a_{i+1}t_{i+1}\mathbf w_2$ where 
$$
\mathbf w_1=
\begin{cases}
t_0\mathbf a_1t_1\cdots t_{i-1}\mathbf a_i&\text{if }0<i\le m-1,\\ 
\lambda&\text{if }i=0
\end{cases}
$$
and
$$
\mathbf w_2=
\begin{cases}
\mathbf a_{i+2}t_{i+2}\cdots\mathbf a_mt_m&\text{if }0\le i<m-1,\\
\lambda&\text{if }i=m-1.
\end{cases}
$$
We are going to check that
\begin{equation}
\label{decomposition for a_{i+1}}
\mathbf a_{i+1}={\bf u}_1{\bf u}'_1{\bf u}_2{\bf u}'_2\cdots {\bf u}_k{\bf u}'_k
\end{equation}
and therefore, ${\bf u}= {\bf w}_1t_i{\bf u}_1{\bf u}'_1{\bf u}_2{\bf u}'_2\cdots {\bf u}_k{\bf u}'_kt_{i+1}{\bf w}_2$ for some possibly empty words $\mathbf u_1,\mathbf u_k'$ and non-empty words $\mathbf u_1',\mathbf u_2,\mathbf u_2',\dots,\mathbf u_k$ such that $\con({\bf u}_j) \subseteq \con({\bf w}_1)$ and $\con({\bf u}'_j) \subseteq \con({\bf w}_2)$ for all $j=1,\dots,k$. If $\mathbf a_{i+1}=\lambda$ then the equality~\eqref{decomposition for a_{i+1}} holds with $k=1$ and $\mathbf u_1=\mathbf u_1'=\lambda$. Suppose now that $\mathbf a_{i+1}\ne\lambda$. Let $x\in\con(\mathbf a_{i+1})$. Then $x\in\mul(\mathbf u)$. There is a letter $y\in\simple(\mathbf u)$ with $\mathbf u(x,y)=xyx$. Suppose that $x\in\mul(\mathbf a_{i+1})$. Therefore, $xyx$ is a subword of $\mathbf a_{i+1}$. This means that $y$ is simple in $\mathbf a_{i+1}$. But this is not the case because $y\ne t_j$ for any $0\le j\le m$. Thus, $x$ is simple in $\mathbf a_{i+1}$, whence $x\in\con(\mathbf w_1\mathbf w_2)$. We prove that every letter from $\con(\mathbf a_{i+1})$ is simple in $\mathbf a_{i+1}$ and occurs either in $\mathbf w_1$ or in $\mathbf w_2$.

Let $\mathbf u_1$ be the maximal prefix of $\mathbf a_{i+1}$ such that $\con(\mathbf u_1)\subseteq\con(\mathbf w_1)$ (if the first letter of $\mathbf a_{i+1}$ does not occur in $\mathbf w_1$ then $\mathbf u_1=\lambda$). Then $\mathbf a_{i+1}=\mathbf u_1\mathbf b$ for some possibly empty word $\mathbf b$. If $\mathbf b=\lambda$ then~\eqref{decomposition for a_{i+1}} holds with $k=1$ and $\mathbf u_1'=\lambda$. Otherwise, let $\mathbf u_1'$ be the maximal prefix of $\mathbf b$ such that $\con(\mathbf u_1')\subseteq\con(\mathbf w_2)$. Then $\mathbf a_{i+1}=\mathbf u_1\mathbf u_1'\mathbf c$ for some possibly empty word \textbf c. If $\mathbf c=\lambda$ then~\eqref{decomposition for a_{i+1}} holds with $k=1$. Otherwise, let $\mathbf u_2$ be the maximal prefix of \textbf c such that $\con(\mathbf u_2)\subseteq\con(\mathbf w_1)$. Continuing this process, we obtain the equality~\eqref{decomposition for a_{i+1}}.
 
Put
$$
\mathbf w_1'=
\begin{cases}
t_0\mathbf b_1t_1\cdots t_{i-1}\mathbf b_i&\text{if }0<i\le m-1,\\ 
\lambda&\text{if }i=0
\end{cases}
$$
and
$$
\mathbf w_2'=
\begin{cases}
\mathbf b_{i+2}t_{i+2}\cdots\mathbf b_mt_m&\text{if }0\le i<m-1,\\
\lambda&\text{if }i=m-1.
\end{cases}
$$
The same arguments as above show that $\mathbf b_{i+1}={\bf v}_1{\bf v}'_1{\bf v}_2{\bf v}'_2\cdots {\bf v}_r{\bf v}'_r$ for some natural $r$, possibly empty words $\mathbf v_1,\mathbf v_r'$ and non-empty words $\mathbf v_1',\mathbf v_2,\mathbf v_2',\dots,\mathbf v_r$ such that $\con({\bf v}_j) \subseteq \con({\bf w}_1')$ and $\con({\bf v}'_j) \subseteq \con({\bf w}_2')$ for all $j=1,\dots,r$. Therefore,
$$
{\bf v}= {\bf w}_1't_i{\bf v}_1{\bf v}'_1{\bf v}_2{\bf v}'_2\cdots {\bf v}_r{\bf v}'_rt_{i+1}{\bf w}_2'.
$$
Further, we may assume without loss of generality that $k \ge r$. We are going to verify that $k=r$, $\con({\bf u}_j)=\con({\bf v}_j)$ and $\con({\bf u}'_j)=\con({\bf v}'_j)$ for all $j=1,\dots,r$.
 
Let $x\in \con(\mathbf u_1)$. As we have shown above, $\mathbf u(x,t_i)=xt_ix$. Therefore, $\mathbf v(x,t_i)=xt_ix$ too, whence $\occ_x(\mathbf w_1')=1$. Note that $\mathbf v(x,t_{i+1})\ne xt_{i+1}x$ because $\mathbf u(x,t_{i+1})= x^2t_{i+1}$. Therefore, $x\notin\con(\mathbf w_2')$, whence $x\in \con(\mathbf v_1\mathbf v_2\cdots \mathbf v_r)$. If $x\notin \con(\mathbf v_1)$ then $x\in \con(\mathbf v_p)$ for some $p>1$. Then there exists a letter $y\in\con(\mathbf v_{p-1}')$. Note that $\mathbf u(y,t_{i+1})=\mathbf v(y,t_{i+1})=yt_{i+1}y$. Therefore, $y\in\con(\mathbf w_2)$, whence $y\in\con(\mathbf u_j')$ for some $1\le j\le k$. Then $\mathbf u(x,y,t_i,t_{i+1})=xt_ixyt_{i+1}y$, while $\mathbf v(x,y,t_i,t_{i+1})=xt_iyxt_{i+1}y$. This contradicts the fact that the word $xt_ixyt_{i+1}y$ is an isoterm for $\mathbf Z$. Thus, $x\in\con(\mathbf v_1)$, whence $\con({\bf u}_1) \subseteq \con({\bf v}_1)$. Analogously, $\con({\bf v}_1) \subseteq \con({\bf u}_1)$. Therefore, $\con({\bf u}_1) = \con({\bf v}_1)$.
 
Let $x\in \con(\mathbf u_1')$. As we have shown above, $\mathbf u(x,t_{i+1})=xt_{i+1}x$. Therefore, $\mathbf v(x,t_{i+1})=xt_{i+1}x$ too, whence $\occ_x(\mathbf w_2')=1$. Note that $\mathbf v(x,t_i)\ne xt_ix$ because $\mathbf u(x,t_i)= t_ix^2$. Therefore, $x\notin\con(\mathbf w_1')$, whence $x\in \con(\mathbf v_1'\mathbf v_2'\cdots \mathbf v_r')$. If $x\notin \con(\mathbf v_1')$ then $x\in \con(\mathbf v_p')$ for some $p>1$. Then there exists a letter $y\in\con(\mathbf v_p)$. Note that $\mathbf u(y,t_i)=\mathbf v(y,t_i)=yt_iy$. Therefore, $y\in\con(\mathbf w_1)$, whence $y\in\con(\mathbf u_j)$ for some $1\le j\le k$. Note that $y\notin\con(\mathbf u_1)$. Indeed, if $y\in\con(\mathbf u_1)$ then $y\in\con(\mathbf v_1)$ because $\con(\mathbf u_1)=\con(\mathbf v_1)$. Hence $\occ_y(\mathbf v)\ge \occ_y(\mathbf v_1\mathbf v_p\mathbf w_2')\ge 3$, a contradiction. So, $y\in\con(\mathbf u_j)$ for some $2\le j\le k$. Then $\mathbf u(x,y,t_i,t_{i+1})=xt_ixyt_{i+1}y$, while $\mathbf v(x,y,t_i,t_{i+1})=xt_iyxt_{i+1}y$. This contradicts the fact that the word $xt_ixyt_{i+1}y$ is an isoterm for $\mathbf Z$. Thus, $x\in\con(\mathbf v_1')$, whence $\con({\bf u}_1') \subseteq \con({\bf v}_1')$. Analogously, $\con({\bf v}_1') \subseteq \con({\bf u}_1')$. We prove that $\con({\bf u}_1') = \con({\bf v}_1')$.
 
Repeating with evident modifications arguments from the previous two paragraphs, we can check that $\con({\bf u}_i) =\con({\bf v}_i)$ and $\con({\bf u}_i') =\con({\bf v}_i')$ for $i=2,\dots,r$.
 
If $k>r$ then there is a letter $x\in \con(\mathbf u_{r+1})$. As we have shown above, $\mathbf u(x,t_i)=xt_ix$. Therefore, $\mathbf v(x,t_i)=xt_ix$ too, whence $\occ_x(\mathbf w_1')=1$. Note also that $\mathbf v(x,t_{i+1})=\mathbf u(x,t_{i+1})= x^2t_{i+1}$. In particular, $\mathbf v(x,t_{i+1})\ne xt_{i+1}x$. Therefore, $x\notin\con(\mathbf w_2')$, whence $x\in \con(\mathbf v_1\mathbf v_2\cdots \mathbf v_r)$. Then $x\in \con(\mathbf u_1\mathbf u_2\cdots \mathbf u_r)$ because $\con({\bf u}_i) =\con({\bf v}_i)$ for $i=1,2,\dots,r$. Thus, $\occ_x(\mathbf u)\ge \occ_x(\mathbf w_1{\bf u}_1{\bf u}_2\cdots {\bf u}_{r+1})\ge 3$, a contradiction. Therefore, $k=r$.{\sloppy

}
 
We prove that $k=r$, $\con({\bf u}_i)=\con({\bf v}_i)$ and $\con({\bf u}'_i)=\con({\bf v}'_i)$ for all $i=1,\dots,k$. One can fix an index $s\in\{1,2,\dots,k\}$. Then $\mathbf u_s$ and $\mathbf v_s$ are linear words depending on the same letters. The same is true for the words $\mathbf u_s'$ and $\mathbf v_s'$. The identity $\sigma_1$ [respectively $\sigma_2$] allows us to swap the first [the second] occurrences of two multiple letters whenever these occurrences are each adjacent to other. Therefore, the identities $\sigma_1$ and $\sigma_2$ allow us to reorder letters within the words $\mathbf u_s$ and $\mathbf u_s'$ in an arbitrary way. Thus, if we replace $\mathbf u_s$ by $\mathbf v_s$ and $\mathbf u_s'$ by $\mathbf v_s'$ in \textbf u then the word we obtain should be equal to \textbf u in \textbf L. This is true for all $s=1,\dots,k$. Hence \textbf L satisfies the identities
\begin{align*}
\mathbf u={}&\mathbf w_1t_i\mathbf a_{i+1}t_{i+1}\mathbf w_2=\mathbf w_1t_i{\bf u}_1{\bf u}'_1{\bf u}_2{\bf u}'_2\cdots {\bf u}_k{\bf u}'_kt_{i+1}\mathbf w_2\\
\approx{}&\mathbf w_1t_i{\bf v}_1{\bf v}'_1{\bf v}_2{\bf v}'_2\cdots {\bf v}_k{\bf v}'_kt_{i+1}\mathbf w_2=\mathbf w_1t_i\mathbf b_{i+1}t_{i+1}\mathbf w_2.
\end{align*}
Thus, if we replace $\mathbf a_{i+1}$ by $\mathbf b_{i+1}$ in \textbf u then the word we obtain should be equal to \textbf u in \textbf L. This is true for all $i=0,\dots,m-1$. Therefore, \textbf L satisfies the identities
$$
\mathbf u=t_0\mathbf a_1t_1\mathbf a_2t_2\cdots t_{m-1}\mathbf a_mt_m\approx t_0\mathbf b_1t_1\mathbf b_2t_2\cdots t_{m-1}\mathbf b_mt_m=\mathbf v.
$$
The lemma is proved.
\end{proof}
 
Lemma~\ref{L = var S(xzxyty)} and~\cite[Lemma~5.10]{Jackson-05} imply that any proper subvariety of ${\bf L}$ is contained in $\var S(xyx)$. Lemmas~\ref{D_n=var S(w)} and~\ref{L(D)} imply now the following
 
\begin{corollary}
\label{L(L)}
The lattice $L({\bf L})$ is the chain $\mathbf{T\subset SL\subset C}_2\subset\mathbf D_1\subset\mathbf D_2\subset\mathbf L$.\qed
\end{corollary}

A non-finitely based variety all whose proper subvarieties are finitely based is called \emph{limit}. The variety $\var S(xzxyty)$ is limit by~\cite[Proposition~5.1]{Jackson-05}. Thus, Lemma~\ref{L = var S(xzxyty)} implies
 
\begin{corollary}
\label{L is limit}
The variety ${\bf L}$ is a limit variety. In particular, it does not have a finite basis of identities.\qed
\end{corollary}
 
According to the result of~\cite{Kozhevnikov-12} mentioned in Section~\ref{introduction}, there are uncountably many periodic group varieties whose subvariety lattice is the 3-element chain. Let $\mathcal G$ be the class of all such varieties. Since the class of finitely based group varieties is countably infinite, the class $\mathcal G$ contains non-finitely based varieties. Group varieties whose subvariety lattice is the 2-element chain are varieties of Abelian groups of a prime exponent. They are finitely based. Thus, all non-finitely based varieties from the class $\mathcal G$ are limit varieties. But explicit examples of limit chain group varieties have not been published anywhere so far.
 
We denote by \textbf M the subvariety of \textbf N given within \textbf N by the following identity:
$$
\alpha_1:\enskip x_1y_1x_0x_1y_1\approx y_1x_1x_0x_1y_1.
$$
Note that $\alpha_1$ belongs to a countably infinite series of identities $\alpha_k$ that will be defined in Subsection~\ref{sufficiency: K - red to [E,K]}.
 
\begin{lemma}
\label{X contains D_2 but L or M}
Let ${\bf X}$ be a monoid variety and ${\bf D}_2\subseteq\mathbf X$.
\begin{itemize}
\item[\textup{(i)}] If ${\bf L}\nsubseteq{\bf X}$ then ${\bf X}$ satisfies the identity $\gamma_1$.
\item[\textup{(ii)}] If ${\bf M}\nsubseteq{\bf X}$ then ${\bf X}$ satisfies the identity $\sigma_1$.
\end{itemize}
\end{lemma}
 
\begin{proof}
(i) According to Lemmas~\ref{S(W) in V} and~\ref{L = var S(xzxyty)}, the variety $\bf X$ satisfies a non-trivial identity of the form $xzxyty\approx {\bf w}$. Note that the word $xyx$ is an isoterm for ${\bf X}$ by Lemmas~\ref{S(W) in V} and~\ref{D_n=var S(w)}. Then Fact~4.1(i) of~\cite{Sapir-15} implies that ${\bf w} = xzyxty$. Therefore, the identity $\gamma_1$ holds in ${\bf X}$.
 
\smallskip
 
(ii) According to Lemmas~\ref{S(W) in V} and~\ref{D_n=var S(w)}, the word $xyx$ is an isoterm for ${\bf X}$. Further, the variety ${\bf M}$ is generated by the monoid $S(xyzxty)$ (this fact is dual to Proposition~1 in Erratum to~\cite{Jackson-05}). Then $S(xyzxty)\notin\mathbf X$, whence $\bf X$ satisfies a non-trivial identity of the form $xyzxty\approx {\bf w}$ by Lemma~\ref{S(W) in V}. Fact~4.1(ii) of~\cite{Sapir-15} implies that ${\bf w} = yxzxty$. Therefore, the identity $\sigma_1$ holds in ${\bf X}$.
\end{proof}
 
Return to an examination of a chain variety \textbf V. In Subsection~\ref{necessity: red to over D_2} we reduce considerations to the case when ${\bf D}_2 \subseteq {\bf V}$. Then ${\bf E} \nsubseteq {\bf V}$ because the varieties $\mathbf D_2$ and \textbf E are non-comparable. The variety ${\bf V}$ satisfies the identity~\eqref{yxx=xxyxx} by Lemma~\ref{E nsubseteq X}. Similarly, the fact that $\overleftarrow{{\bf E}} \nsubseteq {\bf V}$ implies that ${\bf V}$ satisfies the identity~\eqref{xxy=xxyxx} by the dual of Lemma~\ref{E nsubseteq X}. Hence the identity~\eqref{xxy=yxx} holds in ${\bf V}$. If ${\bf V}$ does not contain ${\bf L}$, ${\bf M}$ and $\overleftarrow{{\bf M}}$ then Lemma~\ref{X contains D_2 but L or M} and the dual of its claim~(ii) imply that ${\bf V}$ satisfies $\sigma_1$, $\sigma_2$ and $\gamma_1$, whence ${\bf V} \subseteq {\bf D}$.
 
It remains to consider the case when ${\bf V}$ contains one of the varieties ${\bf L}$, ${\bf M}$ or $\overleftarrow{{\bf M}}$. Then ${\bf V}$ does not contain the variety ${\bf D}_3$ because ${\bf L}$, ${\bf M}$ and $\overleftarrow{{\bf M}}$ are non-comparable with $\mathbf D_3$. Lemma~\ref{D_{n+1} nsubseteq X} and the fact that ${\bf V}$ satisfies the identity~\eqref{xxy=yxx} imply that the identity~\eqref{xyxzx=xxyz} holds in ${\bf V}$.
 
Let ${\bf M} \subseteq {\bf V}$. Then ${\bf V}$ does not contain ${\bf L}$ and $\overleftarrow{{\bf M}}$. Lemma~\ref{X contains D_2 but L or M}(i) and the dual of Lemma~\ref{X contains D_2 but L or M}(ii) imply that ${\bf V} \subseteq {\bf N}$. Dual arguments show that if $\overleftarrow{{\bf M}} \subseteq {\bf V}$ then ${\bf V}\subseteq\overleftarrow{{\bf N}}$.
 
We reduce our considerations to the case when ${\bf L\subseteq V}$. 
 
\subsection{The case when $\mathbf{L\subseteq V}$}
\label{necessity: over L}
 
Clearly, here ${\bf M},\overleftarrow{{\bf M}} \nsubseteq {\bf V}$. Lemma~\ref{X contains D_2 but L or M}(ii) and the dual of it imply that \textbf V satisfies the identities $\sigma_1$ and $\sigma_2$. As we have already seen above, \textbf V satisfies the identities~\eqref{xyxzx=xxyz} and~\eqref{xxy=yxx} as well. Thus, \textbf V is contained in the variety
$$
{\bf O} = \var\{x^2y\approx yx^2,\,xyxzx\approx x^2yz,\,\sigma_1,\,\sigma_2\}.
$$
To complete the proof of the necessity in Theorem~\ref{main result}, it suffices to verify that $\mathbf{V\subseteq L}$ in the case under consideration. To achieve this goal, it remains to check that \textbf V satisfies all identities of the form~\eqref{w_n=w_n'} where $n$ is a natural number and $\pi,\tau\in S_n$. To do this, we need several auxiliary claims. Let $n\in\mathbb N$, $0\le k\le\ell\le n$ and $\pi,\tau\in S_n$. Put 
\begin{align*}
\mathbf w_n^{k,\ell}(\pi,\tau)={}&\biggl(\prod_{i=1}^n z_it_i\biggr)\biggl(\prod_{i=1}^k z_{\pi(i)}z_{n+\tau(i)}\biggr)x\biggl(\prod_{i=k+1}^\ell z_{\pi(i)}z_{n+\tau(i)}\biggr)x\\
&\cdot\,\biggl(\prod_{i=\ell+1}^n z_{\pi(i)}z_{n+\tau(i)}\biggr)\biggl(\prod_{i=n+1}^{2n} t_iz_i\biggr).
\end{align*}
We note that $\mathbf w_n^{0,n}(\pi,\tau)=\mathbf w_n(\pi,\tau)$ and $\mathbf w_n^{0,0}(\pi,\tau)=\mathbf w_n'(\pi,\tau)$.
 
\begin{lemma}
\label{S(w_n(pi,tau)) notin X}
Let ${\bf X}$ be a monoid variety such that ${\bf L} \subseteq{\bf X} \subseteq{\bf O}$, $n$ be a natural number and $\pi,\tau\in S_n$. If $S\bigl(\mathbf w_n(\pi,\tau)\bigr) \notin {\bf X}$ then ${\bf X}$ satisfies a non-trivial identity of the form
\begin{equation}
\label{w_n=w_n^{k,l}}
\mathbf w_n(\pi,\tau)\approx \mathbf w_n^{k,\ell}(\pi,\tau)
\end{equation}
for some $0\le k\le \ell \le n$.
\end{lemma}
 
\begin{proof}
Suppose that $S\bigl(\mathbf w_n(\pi,\tau)\bigr) \notin {\bf X}$. Then Lemma~\ref{S(W) in V} applies and we conclude that the variety ${\bf X}$ satisfies a non-trivial identity of the form
\begin{equation}
\label{w_n(pi,tau)=w}
\mathbf w_n(\pi,\tau)=\biggl(\prod_{i=1}^n z_it_i\biggr) x \biggl(\prod_{i=1}^n z_{\pi(i)}z_{n+\tau(i)}\biggr) x \biggl(\prod_{i=n+1}^{2n} t_iz_i\biggr)\approx{\bf w}.
\end{equation}
Put $\mathbf a=z_{\pi(1)}$, $\mathbf b=t_{\pi(1)}z_{\pi(1)+1}t_{\pi(1)+1}\cdots z_nt_nx$, $\mathbf c=z_{n+\tau(1)}$ and
$$
\mathbf d=z_{\pi(2)}z_{n+\tau(2)}\cdots z_{\pi(n)}z_{n+\tau(n)}xt_{n+1}z_{n+1}\cdots t_{n+\tau(1)-1}z_{n+\tau(1)-1}t_{n+\tau(1)}.
$$
The word $\mathbf w_n(\pi,\tau)$ contains the subword $\mathbf{abacdc}$. Therefore, the submonoid of the monoid $S\bigl(\mathbf w_n(\pi,\tau)\bigr)$ generated by the elements \textbf a, \textbf b, \textbf c and \textbf d is isomorphic to $S(xzxyty)$. Now Lemmas~\ref{S(W) in V} and~\ref{L = var S(xzxyty)} apply with the conclusion that the word $xzxyty$ is an isoterm for ${\bf X}$. Now we are going to verify that
\begin{equation}
\label{occurrences in w and w_n(pi,tau)}
\begin{array}{l}
\ell_2({\bf w},z_i)<\ell_1({\bf w},z_{n+j})\text{ if and only if }\\
\ell_2\bigl(\mathbf w_n(\pi,\tau),z_i\bigr)<\ell_1\bigl(\mathbf w_n(\pi,\tau),z_{n+j}\bigr)
\end{array}
\end{equation}
for any $1\le i,j\le n$. Indeed, let $1\le i,j\le n$. The word $xyx$ is an isoterm for \textbf X. Since $\bigl[\textbf w_n(\pi,\tau)\bigr](z_i,t_i)=z_it_iz_i$, $\bigl[\mathbf w_n(\pi,\tau)\bigr](z_{n+j},t_{n+j})=z_{n+j}t_{n+j}z_{n+j}$ and~\eqref{w_n(pi,tau)=w} holds in \textbf X, we have
$$
\mathbf w(z_i,t_i)=z_it_iz_i\text{ and }\mathbf w(z_{n+j},t_{n+j})=z_{n+j}t_{n+j}z_{n+j}.
$$
The variety \textbf X is non-commutative, whence $\ell_1({\bf w},t_i)<\ell_1({\bf w},t_{n+j})$. Therefore,
\begin{align*}
&\mathbf w(z_i,t_{n+j})=\bigl[\mathbf w_n(\pi,\tau)\bigr](z_i,t_{n+j})=z_i^2t_{n+j},\\
&\mathbf w(z_{n+j},t_i)=\bigl[\mathbf w_n(\pi,\tau)\bigr](z_{n+j},t_i)=t_iz_{n+j}^2.
\end{align*}
Summarizing the above, we have 
\begin{align*}
&\mathbf w(z_i,t_i,t_{n+j})=\bigl[\mathbf w_n(\pi,\tau)\bigr](z_i,t_i,t_{n+j})=z_it_iz_it_{n+j},\\
&\mathbf w(z_{n+j},t_i,t_{n+j})=\bigl[\mathbf w_n(\pi,\tau)\bigr](z_{n+j},t_i,t_{n+j})=t_iz_{n+j}t_{n+j}z_{n+j}.
\end{align*}

Suppose that $\ell_2({\bf w},z_i)<\ell_1({\bf w},z_{n+j})$. Then observations given in the previous paragraph imply that $\mathbf w(z_i,z_{n+j},t_i,t_{n+j})=z_it_iz_iz_{n+j}t_{n+j}z_{n+j}$. Since the word $xzxyty$ is an isoterm for \textbf X, we have 
$$
\bigl[\mathbf w_n(\pi,\tau)\bigr](z_i,z_{n+j},t_i,t_{n+j})=z_it_iz_iz_{n+j}t_{n+j}z_{n+j}=\mathbf w(z_i,z_{n+j},t_i,t_{n+j}),
$$ 
whence $\ell_2(\mathbf w_n(\pi,\tau),z_i)<\ell_1(\mathbf w_n(\pi,\tau),z_{n+j})$.
 
Suppose now that $\ell_2\bigl(\mathbf w_n(\pi,\tau),z_i\bigr)<\ell_1\bigl(\mathbf w_n(\pi,\tau),z_{n+j}\bigr)$. Then
$$
\bigl[\mathbf w_n(\pi,\tau)\bigr](z_i,z_{n+j},t_i,t_{n+j})=z_it_iz_iz_{n+j}t_{n+j}z_{n+j}.
$$
Now we apply the fact that $xzxyty$ is an isoterm for \textbf X again and obtain
$$
\mathbf w(z_i,z_{n+j},t_i,t_{n+j})=\bigl[\mathbf w_n(\pi,\tau)\bigr](z_i,z_{n+j},t_i,t_{n+j})=z_it_iz_iz_{n+j}t_{n+j}z_{n+j},
$$ 
whence $\ell_2({\bf w},z_i)<\ell_1({\bf w},z_{n+j})$.
 
The claim~\eqref{occurrences in w and w_n(pi,tau)} is proved. Then 
$$
{\bf w}_x= \biggl(\prod_{i=1}^n z_it_i\biggr) \biggl(\prod_{i=1}^n z_{\pi(i)}z_{n+\tau(i)}\biggr) \biggl(\prod_{i=n+1}^{2n} t_iz_i\biggr).
$$
Being a subvariety of \textbf O, the variety \textbf X satisfies the identities $xyxzx\approx x^2yz\approx yzx^2$. Therefore, we can assume that $\occ_x({\bf w})=2$ for any $x\in\con(\mathbf w)$. So,
$$
{\bf w}= \biggl(\prod_{i=1}^n \mathbf p_{2i-1}z_i\mathbf p_{2i}t_i\biggr)\mathbf q_0 \biggl(\prod_{i=1}^n z_{\pi(i)}\mathbf q_{2i-1}z_{n+\tau(i)}\mathbf q_{2i}\biggr) \biggl(\prod_{i=n+1}^{2n} t_i\mathbf r_{2i-2n-1}z_i\mathbf r_{2i-2n}\biggr)
$$
where 
$$
\biggl(\prod_{i=1}^{2n} \mathbf p_i\biggr)\biggl(\prod_{i=0}^{2n} \mathbf q_i\biggr)\biggl(\prod_{i=1}^{2n} \mathbf r_i\biggr) = x^2.
$$
 
Suppose at first that $x\in \con({\bf p}_{2j-1}{\bf p}_{2j})$ for some $1\le j\le n$ and $j$ is the least number with this property. If ${\bf p}_{2j-1}{\bf p}_{2j}=x$ then 
$$
\biggl(\prod_{i=2j+1}^{2n} \mathbf p_i\biggr)\biggl(\prod_{i=0}^{2n} \mathbf q_i\biggr)\biggl(\prod_{i=1}^{2n} \mathbf r_i\biggr) = x.
$$
It can be easily verified directly that substituting~1 for all letters except $x$ and $t_j$ in the identity~\eqref{w_n(pi,tau)=w} we obtain here the identity $t_jx^2\approx xt_jx$. But this is impossible because $xzx$ is an isoterm for ${\bf X}$. Therefore, ${\bf p}_{2j-1}{\bf p}_{2j}=x^2$, i.e., either $\mathbf p_{2j-1}=\mathbf p_{2j}=x$ or $\mathbf p_{2j-1}=x^2$ or $\mathbf p_{2j}=x^2$. If $\mathbf p_{2j-1}=\mathbf p_{2j}=x$ then \textbf X satisfies the identities
\begin{align*}
\mathbf w_n(\pi,\tau)\approx\mathbf w=\hskip8pt&\biggl(\prod_{i=1}^{j-1} z_it_i\biggr)xz_jxt_j\biggl(\prod_{i=j+1}^n z_it_i\biggr) \biggl(\prod_{i=1}^n z_{\pi(i)}z_{n+\tau(i)}\biggr) \biggl(\prod_{i=n+1}^{2n} t_iz_i\biggr)\\
\stackrel{\sigma_1}\approx\hskip8pt&\biggl(\prod_{i=1}^{j-1} z_it_i\biggr)z_jx^2t_j\biggl(\prod_{i=j+1}^n z_it_i\biggr) \biggl(\prod_{i=1}^n z_{\pi(i)}z_{n+\tau(i)}\biggr) \biggl(\prod_{i=n+1}^{2n} t_iz_i\biggr)\\
\stackrel{\eqref{xxy=yxx}}\approx{}&\biggl(\prod_{i=1}^n z_it_i\biggr)x^2 \biggl(\prod_{i=1}^n z_{\pi(i)}z_{n+\tau(i)}\biggr) \biggl(\prod_{i=n+1}^{2n} t_iz_i\biggr)\\
=\hskip8pt&\hskip5pt\mathbf w_n'(\pi,\tau)=\mathbf w_n^{0,0}(\pi,\tau),
\end{align*}
and we are done. If $\mathbf p_{2j-1}=x^2$ or $\mathbf p_{2j}=x^2$ then we can apply the identity~\eqref{xxy=yxx} and obtain the required conclusion. So, we can assume that $\mathbf p_1\mathbf p_2\cdots \mathbf p_{2n}=\lambda$.
 
The case when $x\in \con({\bf r}_{2j-1}{\bf r}_{2j})$ for some $1\le j\le n$ can be considered quite analogously to the previous case with the use of the identity $\sigma_2$ rather than $\sigma_1$.
 
Finally, let $x\notin \con({\bf p}_1{\bf p}_2\cdots {\bf p}_{2n})$ and $x\notin \con({\bf r}_1{\bf r}_2\cdots {\bf r}_{2n})$. Then ${\bf q}_0{\bf q}_1\cdots{\bf q}_{2n} = x^2$. Note that either $x \notin \con({\bf q}_0)$ or $x \notin \con({\bf q}_{2n})$ because otherwise the identity~\eqref{w_n(pi,tau)=w} is trivial. Assume without loss of generality that $x \notin \con({\bf q}_0)$, whence $\mathbf q_1\mathbf q_2\cdots \mathbf q_{2n}=x^2$. Let $x\in \con(\mathbf q_k)$ and $k$ is the least number with this property. If $\mathbf q_k = x^2$ then we can apply the identity~\eqref{xxy=yxx} and obtain the required conclusion. Suppose now that $x\in\con(\mathbf q_{\ell+1})$ for some $k\le\ell\le 2n-1$.{\sloppy

}
 
Each occurrence of $x$ in \textbf w lies either in a subword like $z_{\pi(i)}xz_{n+\tau(i)}$ or in a subword like $z_{\pi(i)}z_{n+\tau(i)}xz_{\pi(i+1)}z_{n+\tau(i+1)}$. We need to verify that \textbf w is equal in \textbf X to some word which has the same structure as \textbf w but contains only occurrences of $x$ of the second type. If both occurrences of $x$ in \textbf w are of the second type then we are done. Suppose that both occurrences are of the first type. Then the variety $\bf X$ satisfies the identities
\begin{align*}
\mathbf w_n(\pi,\tau)\approx \mathbf w={}&\biggl(\prod_{i=1}^n z_it_i\biggr) \biggl(\prod_{i=1}^{k-1} z_{\pi(i)}z_{n+\tau(i)}\biggr) z_{\pi(k)}xz_{n+\tau(k)}\biggl(\prod_{i=k+1}^{\ell} z_{\pi(i)}z_{n+\tau(i)}\biggr)\\[-3pt]
&\cdot\,\underline{z_{\pi(\ell+1)}x}\,z_{n+\tau(\ell+1)} \biggl(\prod_{i=\ell+2}^n z_{\pi(i)}z_{n+\tau(i)}\biggr)\biggl(\prod_{i=n+1}^{2n} t_iz_i\biggr)\\[-3pt]
\stackrel{\sigma_2}\approx{}&\biggl(\prod_{i=1}^n z_it_i\biggr) \biggl(\prod_{i=1}^{k-1} z_{\pi(i)}z_{n+\tau(i)}\biggr) z_{\pi(k)}\,\underline{xz_{n+\tau(k)}}\biggl(\prod_{i=k+1}^{\ell} z_{\pi(i)}z_{n+\tau(i)}\biggr)\\[-3pt]
&\cdot\,x\biggl(\prod_{i=\ell+1}^n z_{\pi(i)}z_{n+\tau(i)}\biggr)\biggl(\prod_{i=n+1}^{2n} t_iz_i\biggr)\\
\stackrel{\sigma_1}\approx{}&\biggl(\prod_{i=1}^n z_it_i\biggr) \biggl(\prod_{i=1}^k z_{\pi(i)}z_{n+\tau(i)}\biggr)x\biggl(\prod_{i=k+1}^\ell z_{\pi(i)}z_{n+\tau(i)}\biggr)x\\[-3pt]
&\cdot\,\biggl(\prod_{i=\ell+1}^n z_{\pi(i)}z_{n+\tau(i)}\biggr)\biggl(\prod_{i=n+1}^{2n} t_iz_i\biggr)\\
\stackrel{\phantom{\sigma_1}}={}&\,\mathbf w_n^{k,\ell}(\pi,\tau)
\end{align*}
(for reader convenience, we underline here pairs of adjacent letters that are transposed by one of the identities $\sigma_1$ or $\sigma_2$). Finally, if two occurrences of $x$ in \textbf w are of different types, then we can use analogous but simpler arguments. If occurrence of the first type lies in $\mathbf q_k$ [in $\mathbf q_{\ell+1}$] then it suffices to apply the identity $\sigma_1$ [respectively $\sigma_2$] only. Thus, in all cases an identity of the form~\eqref{w_n=w_n^{k,l}} holds in $\bf X$.
\end{proof}

\begin{lemma}
\label{w_q=w_q' implies w_m^{k,l}=w_m^{k,k}}
Let $m$ be a natural number, $0\le k<\ell<m$, $q=\ell-k$ and $\pi,\tau\in S_m$. Then there are permutations $\rho,\sigma\in S_q$ such that the identity $\mathbf w_q(\rho,\sigma)\approx\mathbf w_q'(\rho,\sigma)$ implies the identity $\mathbf w_m^{k,\ell}(\pi,\tau)\approx\mathbf w_m^{k,k}(\pi,\tau)$.
\end{lemma}
 
\begin{proof}
For convenience, we put
$$
\{z_{\pi(k+1)},z_{\pi(k+2)},\dots,z_{\pi(\ell)}\}=\{z_{p_1},z_{p_2},\dots,z_{p_q}\}
$$
and
$$
\{z_{m+\tau(k+1)},z_{m+\tau(k+2)},\dots,z_{m+\tau(\ell)}\}=\{z_{r_1},z_{r_2},\dots,z_{r_q}\}
$$
 where $1\le p_1<p_2<\cdots <p_q\le m< r_1<r_2<\cdots <r_q\le 2m$. The word $\mathbf w_m^{k,\ell}(\pi,\tau)$ has the form
$$ 
\mathbf u_0z_{p_1}\mathbf u_1\cdots z_{p_q}\mathbf u_q x z_{\pi(k+1)}z_{m+\tau(k+1)}\cdots z_{\pi(\ell)}z_{m+\tau(\ell)} x \mathbf u_{q+1}z_{r_1}\cdots \mathbf u_{2q}z_{r_q}\mathbf u_{2q+1}
$$
where
\begin{align*}
&\mathbf u_0=\prod_{i=1}^{p_1-1} z_it_i,\\[-3pt]
&\mathbf u_s=t_{p_s}\biggl(\prod_{i=p_s+1}^{p_{s+1}-1} z_it_i\biggr)\text{ for all } 1\le s<q,\\[-3pt]
&\mathbf u_q=t_{p_q}\biggl(\prod_{i=p_q+1}^m z_it_i\biggr)\biggl(\prod_{i=1}^k z_{\pi(i)}z_{m+\tau(i)}\biggr),\\[-3pt]
&\mathbf u_{q+1}=\biggl(\prod_{i=\ell+1}^m z_{\pi(i)}z_{m+\tau(i)}\biggr)\biggl(\prod_{i=m+1}^{r_1-1} t_iz_i\biggr)t_{r_1},\\[-3pt]
&\mathbf u_{q+1+s}=\biggl(\prod_{i=r_{s-1}+1}^{r_s-1} t_iz_i\biggr)t_{r_s}\text{ for all }1\le s<q,\\[-3pt]
&\mathbf u_{2q+1}=\prod_{i=r_q+1}^{2m} t_iz_i.
\end{align*}
 
We are going to rename all letters except $x$ in the word $\mathbf w_m^{k,\ell}(\pi,\tau)$. First, we rename all letters from the set
$$
\con\bigl(\mathbf w_m^{k,\ell}(\pi,\tau)\bigr)\setminus\{x,z_{p_1},z_{p_2}.\dots,z_{p_q},z_{r_1},z_{r_2},\dots,z_{r_q}\}
$$
by some pairwise different letters that do not occur in $\mathbf w_m^{k,\ell}(\pi,\tau)$. Further, we perform the substitution
$$
(z_{p_1},z_{p_2},\dots,z_{p_q},z_{r_1},z_{r_2},\dots,z_{r_q})\mapsto(z_1,z_2,\dots,z_q,z_{q+1},z_{q+2},\dots,z_{2q}).
$$
As a result, we get the word
$$
\mathbf u'=\mathbf u_0'z_1\mathbf u_1'\cdots z_q\mathbf u_q'xz_{\rho(1)}z_{q+\sigma(1)}\cdots z_{\rho(q)}z_{q+\sigma(q)}x\mathbf u_{q+1}'z_{q+1}\cdots\mathbf u_{2q}'z_{2q}\mathbf u_{2q+1}'
$$
for appropriate permutations $\rho,\sigma\in S_q$ and some words $\mathbf u_0',\mathbf u_1',\dots,\mathbf u_{2q+1}'$. 
 
Now we can perform the substitution $(t_1,t_2,\dots,t_{2q})\mapsto (\mathbf u_1',\mathbf u_2',\dots,\mathbf u_{2q}')$ in the identity $\mathbf w_q(\rho,\sigma)\approx\mathbf w_q'(\rho,\sigma)$. We get the identity
\begin{align*}
&z_1\mathbf u_1'\cdots z_q\mathbf u_q'xz_{\rho(1)}z_{q+\sigma(1)}\cdots z_{\rho(q)}z_{q+\sigma(q)}x\mathbf u_{q+1}'z_{q+1}\cdots \mathbf u_{2q}'z_{2q}\\
\approx{}&z_1\mathbf u_1'\cdots z_q\mathbf u_q'x^2z_{\rho(1)}z_{q+\sigma(1)}\cdots z_{\rho(q)}z_{q+\sigma(q)}\mathbf u_{q+1}'z_{q+1}\cdots \mathbf u_{2q}'z_{2q}.
\end{align*}
We apply this identity to the word $\bf u'$ and obtain the identity
$$
\mathbf u' \approx\mathbf u_0'z_1\mathbf u_1'\cdots z_q\mathbf u_q'x^2z_{\rho(1)}z_{q+\sigma(1)}\cdots z_{\rho(q)}z_{q+\sigma(q)}\mathbf u_{q+1}'z_{q+1}\cdots \mathbf u_{2q}'z_{2q}\mathbf u_{2q+1}'.
$$ 
Now we implement in this identity the renaming of letters, the reverse of the one described above. Then we obtain the identity
\begin{align*}
\mathbf w_m^{k,\ell}(\pi,\tau)\approx{}&\mathbf u_0z_{p_1}\mathbf u_1\cdots z_{p_q}\mathbf u_qx^2z_{\pi(k+1)}z_{m+\tau(k+1)}\cdots z_{\pi(\ell)}z_{m+\tau(\ell)}x\\
&\cdot\,\mathbf u_{q+1}z_{r_1}\cdots \mathbf u_{2q}z_{r_q}\mathbf u_{2q+1}=\mathbf w_m^{k,k}(\pi,\tau).
\end{align*}
The lemma is proved.
\end{proof}
 
Now we are well prepared to complete the proof of necessity of Theorem~\ref{main result}. Recall that we reduce our considerations to the case when ${\bf L\subseteq V\subseteq O}$. We denote by $\mathcal K$ the class of all varieties of the form $\var S\bigl(\mathbf w_n(\pi,\tau)\bigr)$ where $n\in\mathbb N$ and $\pi,\tau\in S_n$. It is clear that $\mathbf{L\subseteq X}$ whenever $\mathbf X\in \mathcal K$. We use this fact below without references. Let ${\bf X}\in \mathcal K$. We are going to verify that then the variety ${\bf X}$ contains at least two incomparable subvarieties from the class $\mathcal K$.
 
For an arbitrary permutation $\xi\in S_n$, we define the following two permutations from $S_{n+2}$:
{\arraycolsep=6pt
\begin{align*}
&\xi_1=
\begin{pmatrix}
1&2&3&4&5&\dots&n+2\\
\xi(1)+2&1&2&\xi(2)+2&\xi(3)+2&\dots&\xi(n)+2
\end{pmatrix}
,\\
&\xi_2=
\begin{pmatrix}
1&2&3&4&5&\dots&n+2\\
\xi(1)+2&2&1&\xi(2)+2&\xi(3)+2&\dots&\xi(n)+2
\end{pmatrix}
.
\end{align*}
}%
We have ${\bf X}=\var S\bigl(\mathbf w_n(\pi,\tau)\bigr)$ for some $n$, $\pi$ and $\tau$. Let $T_1 = S_{n+2}(\pi_1,\tau_1)$ and $T_2 = S_{n+2}(\pi_2,\tau_1)$. If $T_1\notin {\bf X}$ then Lemma~\ref{S(w_n(pi,tau)) notin X} allows us to assume without loss of generality that $\bf X$ satisfies a non-trivial identity of the form ${\bf w}_{n+2}(\pi_1,\tau_1) \approx {\bf w}_{n+2}^{k,\ell}(\pi_1,\tau_1)$ for some $1\le k\le\ell \le n+2$. Then we substitute 
\begin{itemize}
\item[$\bullet$] 1 for $z_1$, $z_2$, $z_{n+3}$, $z_{n+4}$, $t_1$, $t_2$, $t_{n+3}$ and $t_{n+4}$,
\item[$\bullet$] $z_{i-2}$ for $z_i$ whenever $3\le i\le n+2$ and $z_{i-4}$ for $z_i$ whenever $n+5\le i\le 2n+4$,
\item[$\bullet$] $t_{i-2}$ for $t_i$ whenever $3\le i\le n+2$ and $t_{i-4}$ for $t_i$ whenever $n+5\le i\le 2n+4$
\end{itemize}
in ${\bf w}_{n+2}(\pi_1,\tau_1) \approx {\bf w}_{n+2}^{k,\ell}(\pi_1,\tau_1)$. Then we obtain $\bf X$ satisfies the identity ${\bf w}_n(\pi,\tau) \approx {\bf w}_n^{s,t}(\pi,\tau)$ where $s=1$ whenever $k\le3$ and $s=k-2$ whenever $k>3$, while $t=1$ whenever $\ell\le3$ and $t=\ell-2$ whenever $\ell>3$. Since $s\ge1$, this identity is non-trivial. We obtain a contradiction with the fact that ${\bf X}=\var S\bigl(\mathbf w_n(\pi,\tau)\bigr)$ and Lemma~\ref{S(W) in V}. Thus, we have proved that $T_1\in {\bf X}$. Analogously, $T_2\in {\bf X}$.{\sloppy

}
 
Suppose that $T_1 \in \var T_2$. Then Lemma~\ref{S(W) in V} applies and we conclude that the word ${\bf w}_{n+2}(\pi_1,\tau_1)$ is an isoterm for $\var T_2$. At the same time, it is easy to verify that $\var T_2$ satisfies ${\bf w}_{n+2}(\pi_1,\tau_1) \approx {\bf w}'_{n+2}(\pi_1,\tau_1)$. Therefore, $\var T_1 \nsubseteq \var T_2$. Analogously, $\var T_2 \nsubseteq \var T_1$. We see that the varieties $\var T_1$ and $\var T_2$ are incomparable. Besides that, it is evident that these two varieties lie in $\mathcal K$.

Thus, if ${\bf X}=\var S\bigl(\mathbf w_n(\pi,\tau)\bigr)$ for some $n$, $\pi$ and $\tau$ then the variety ${\bf X}$ is not chain. Therefore, $S\bigl(\mathbf w_n(\pi,\tau)\bigr) \notin {\bf V}$ for all $n$, $\pi$ and $\tau$. For an arbitrary $n$, we denote the trivial permutation from $S_n$ by $\varepsilon$. Then $S\bigl(\mathbf w_1(\varepsilon,\varepsilon)\bigr)\notin \mathbf V$. According to Lemma~\ref{S(w_n(pi,tau)) notin X}, \textbf V satisfies a non-trivial identity of the form $\mathbf w_1(\varepsilon,\varepsilon)\approx \mathbf w_1^{k,\ell}(\varepsilon,\varepsilon)$ where $0\le k\le\ell\le 1$. Since $\mathbf w_1^{0,0}(\varepsilon,\varepsilon)=\mathbf w_1'(\varepsilon,\varepsilon)$, $\mathbf w_1^{0,1}(\varepsilon,\varepsilon)=\mathbf w_1(\varepsilon,\varepsilon)$ and the identity $\mathbf w_1(\varepsilon,\varepsilon)\approx \mathbf w_1^{k,\ell}(\varepsilon,\varepsilon)$ is non-trivial, we have \textbf V satisfies one of the identities $\mathbf w_1(\varepsilon,\varepsilon)\approx \mathbf w_1'(\varepsilon,\varepsilon)$ or $\mathbf w_1(\varepsilon,\varepsilon)\approx \mathbf w_1^{1,1}(\varepsilon,\varepsilon)$. Clearly, the latter identity together with~\eqref{xxy=yxx} implies the former one. Thus, \textbf V satisfies the identity $\mathbf w_1(\varepsilon,\varepsilon)\approx \mathbf w_1'(\varepsilon,\varepsilon)$.
 
Thus, there is a number $n$ such that $\bf V$ satisfies the identities of the form~\eqref{w_n=w_n'} for all $\pi,\tau\in S_n$ (for instance, $n=1$). We are going to verify that an arbitrary $n$ possesses this property. Arguing by contradiction, we suppose that the above-mentioned claim is true for $1,2,\dots,n$ but is false for $n+1$. Let $\pi_1,\tau_1\in S_{n+1}$. Since $S\bigl(\mathbf w_{n+1}(\pi_1,\tau_1)\bigr)\notin \mathbf V$, Lemma~\ref{S(w_n(pi,tau)) notin X} implies that $\bf V$ satisfies an identity of the form $\mathbf w_{n+1}(\pi_1,\tau_1)\approx \mathbf w_{n+1}^{k,\ell}(\pi_1,\tau_1)$ for some $0\le k\le\ell<n+1$. Suppose that $k<\ell$. Then Lemma~\ref{w_q=w_q' implies w_m^{k,l}=w_m^{k,k}} with $m=n+1$, $\pi=\pi_1$ and $\tau=\tau_1$ applies and we conclude that there exist permutations $\rho,\sigma\in S_{\ell-k}$ such that the identity $\mathbf w_{\ell-k}(\rho,\sigma)\approx \mathbf w_{\ell-k}'(\rho,\sigma)$ implies the identity $\mathbf w_{n+1}^{k,\ell}(\pi_1,\tau_1)\approx \mathbf w_{n+1}^{k,k}(\pi_1,\tau_1)$. The first of these identities holds in $\bf V$ because $\ell-k\le n$. Thus, in any case \textbf V satisfies the identity $\mathbf w_{n+1}(\pi_1,\tau_1)\approx \mathbf w_{n+1}^{k,k}(\pi_1,\tau_1)$. Note that $\mathbf w_{n+1}^{k,k}(\pi_1,\tau_1)\stackrel{\eqref{xxy=yxx}}\approx \mathbf w_{n+1}^{0,0}(\pi_1,\tau_1)=\mathbf w_{n+1}'(\pi_1,\tau_1)$. Therefore, the identity $\mathbf w_{n+1}(\pi_1,\tau_1)\approx\mathbf w_{n+1}'(\pi_1,\tau_1)$ holds in \textbf V. This is true for any $\pi_1,\tau_1\in S_{n+1}$. This contradicts the choice of $n$. So, the variety $\bf V$ satisfies the identities of the form~\eqref{w_n=w_n'} for all $n$ and $\pi,\tau\in S_n$, whence $\bf V = L$.
 
We have thus completed the proof of the ``only if'' part of Theorem~\ref{main result}.
 
\section{The proof of the ``if'' part: all varieties except \textbf K}
\label{sufficiency: not K}
 
In this and the following sections we are going to prove that if \textbf X is a subvariety of one of the varieties listed in Theorem~\ref{main result} then \textbf X is a chain variety. Since the property of being a chain variety is inherited by subvarieties, we can assume that \textbf X coincides with one of the varieties listed in Theorem~\ref{main result}. By symmetry, we can exclude from considerations the varieties $\overleftarrow{{\bf K}}$ and $\overleftarrow{{\bf N}}$. Thus, it suffices to verify that ${\bf C}_n$, ${\bf D}$, \textbf K, ${\bf L}$, ${\bf LRB}$, ${\bf N}$ and ${\bf RRB}$ are chain varieties. Here we consider all these varieties except the variety \textbf K, the last variety will be examined in the following section.
 
Lemmas~\ref{L(D)} and~\ref{L(BM)}(ii) and Corollary~\ref{L(L)} immediately imply that the varieties \textbf D, \textbf L, ${\bf LRB}$ and ${\bf RRB}$ are chain varieties.
 
\begin{proposition}
\label{L(C_n)}
The lattice $L({\bf C}_n)$ is the chain $\mathbf{T\subset SL\subset C}_2\subset\mathbf C_3\subset\cdots\subset\mathbf C_n$.
\end{proposition}
 
\begin{proof}
Let $\mathbf V\subseteq\mathbf C_n$. Then \textbf V is commutative and aperiodic. If $\mathbf C_2\nsubseteq\mathbf V$ then \textbf V is completely regular by Corollary~\ref{cr or over C_2}. Then $\mathbf{V\subseteq SL}$, whence \textbf V coincides with either \textbf T or \textbf{SL}. It remains to verify that if ${\bf C}_2 \subseteq {\bf V} \subseteq {\bf C}_n$ then ${\bf V}={\bf C}_s$ for some $2\le s\le n$. We will use induction on $n$. If $n=2$ then the assertion is obvious. Let now $n>2$. Suppose that ${\bf V} \ne {\bf C}_n$. Then Lemma~\ref{C_{n+1} nsubseteq V} implies that ${\bf V}$ satisfies the identity $x^{n-1}\approx x^n$, whence ${\bf V} \subseteq {\bf C}_{n-1}$. By the induction assumption, ${\bf V}={\bf C}_s$ for some $2\le s\le n-1$.
\end{proof}
 
It remains to consider the variety \textbf N. 
 
\begin{proposition}
\label{L(N)}
The lattice $L({\bf N})$ is the chain $\mathbf{T\subset SL\subset C}_2\subset\mathbf D_1\subset\mathbf D_2\subset\mathbf M\subset\mathbf N$.
\end{proposition}
 
\begin{proof}
First of all, we are going to check that the variety \textbf N satisfies identities of the form~\eqref{w_{n,m}=w_{n,m}'} for all $n$, $m$ and $\theta\in S_{n+m}$. Indeed, $\mathbf w_{n,m}(\theta)=\mathbf px\mathbf qx\mathbf r$ where $\mathbf p=z_1t_1\cdots z_nt_n$, $\mathbf q=z_{\theta(1)}\cdots z_{\theta(n+m)}$ and $\mathbf r=t_{n+1}z_{n+1}\cdots t_{n+m}z_{n+m}$. Suppose at first that $\theta(n+m)\le n$. Then
$$
\mathbf w_{n,m}(\theta)=z_1t_1\cdots\stackrel{(1)}{z_{\theta(n+m)}}t_{\theta(n+m)}\cdots z_nt_n\stackrel{(1)}xz_{\theta(1)}\cdots\stackrel{(2)}{z_{\theta(n+m)}}\stackrel{(2)}x\mathbf r.
$$
We see that the second occurrences of the letters $z_{\theta(n+m)}$ and $x$ in $\mathbf w_{n,m}(\theta)$ are each adjacent to other. The identity $\sigma_2$ allows us to swap these occurrences. In other words,
$$
\mathbf w_{n,m}(\theta)\stackrel{\sigma_2}\approx\mathbf pxz_{\theta(1)}\cdots z_{\theta(n+m-1)}xz_{\theta(n+m)}\mathbf r.
$$
Suppose now that $\theta(n+m)>n$. Then
$$
\mathbf w_{n,m}(\theta)=\mathbf p\stackrel{(1)}xz_{\theta(1)}\cdots\stackrel{(1)}{z_{\theta(n+m)}}\stackrel{(2)}xt_{n+1}z_{n+1}\cdots t_{\theta(n+m)}\stackrel{(2)}{z_{\theta(n+m)}}\cdots t_{n+m}z_{n+m}.
$$
We see that first occurrence of $z_{\theta(n+m)}$ and second occurrence $x$ in $\mathbf w_{n,m}(\theta)$ are each adjacent to other. The identity $\gamma_1$ allows us to transpose these occurrences. In other words,
$$
\mathbf w_{n,m}(\theta)\stackrel{\gamma_1}\approx\mathbf pxz_{\theta(1)}\cdots z_{\theta(n+m-1)}xz_{\theta(n+m)}\mathbf r.
$$
We see that in any case the identity
$$
\mathbf w_{n,m}(\theta)\approx\mathbf pxz_{\theta(1)}\cdots z_{\theta(n+m-1)}xz_{\theta(n+m)}\mathbf r
$$
holds in \textbf N. Analogous arguments show that we can successively swap  second occurrence of $x$ with $z_{\theta(n+m-1)}$, $z_{\theta(n+m-2)}$, \dots, $z_{\theta(1)}$ and obtain \textbf N satisfies the identities
$$
\mathbf w_{n,m}(\theta)\approx\mathbf px^2z_{\theta(1)}\cdots z_{\theta(n+m)}\mathbf r=\mathbf px^2\mathbf{qr}=\mathbf w_{n,m}'(\theta).
$$
Therefore, we can apply Lemma~\ref{u=xxu_x holds in X} below.
 
Suppose that ${\bf V}\subseteq{\bf N}$. If ${\bf M}\nsubseteq{\bf V}$ then ${\bf V}\subseteq{\bf D}$ by Lemma~\ref{X contains D_2 but L or M}(ii). Therefore, in view of Lemma~\ref{L(D)}, it suffices to consider the case when ${\bf M} \subseteq {\bf V}$. We need to verify that \textbf V coincides with one of the varieties \textbf M or \textbf N. Let ${\bf u}\approx {\bf v}$ be an arbitrary identity that holds in the variety ${\bf V}$. Our aim is to verify that $\mathbf{u\approx v}$ either implies the identity $\alpha_1$ or holds in the variety \textbf N. Proposition~\ref{word problem C_2} implies that $\simple(\mathbf u)=\simple(\mathbf v)$. Let $\simple(\mathbf u)=\{t_0,t_1,\dots,t_m\}$. As in the proof of Lemma~\ref{L = var S(xzxyty)}, we can assume that
$$
\mathbf u=t_0\mathbf a_1t_1\mathbf a_2t_2\cdots t_{m-1}\mathbf a_mt_m\text{ and }\mathbf v=t_0\mathbf b_1t_1\mathbf b_2t_2\cdots t_{m-1}\mathbf b_mt_m
$$
for some possibly empty words $\mathbf a_1,\mathbf a_2,\dots,\mathbf a_m$ and $\mathbf b_1,\mathbf b_2,\dots,\mathbf b_m$.
 
Let $x$ be a letter multiple in $\bf u$  and ${\bf u}(x,y)\ne xyx$ for any letter $y$. By Lemma~\ref{u=xxu_x holds in X}, the variety ${\bf V}$ satisfies the identity~\eqref{u=xxu_x}. Obviously, ${\bf C}_2 \subseteq \mathbf M \subseteq {\bf V}$, whence $x\in \mul({\bf v})$ by Proposition~\ref{word problem C_2}. Since $\mathbf D_2 \subseteq \mathbf M \subseteq {\bf V}$, we apply Lemmas~\ref{S(W) in V} and~\ref{D_n=var S(w)} and conclude that the word $xyx$ is an isoterm for ${\bf V}$. Therefore, ${\bf v}(x,y)\ne xyx$ for any letter $y$. We apply Lemma~\ref{u=xxu_x holds in X} again and conclude that the identity $\mathbf v \approx x^2 \mathbf v_x$ holds in ${\bf V}$. Thus, the identity $\mathbf u \approx \mathbf v$ follows from the identities~\eqref{u=xxu_x}, $\mathbf v \approx x^2 \mathbf v_x$ and ${\bf u}_x\approx {\bf v}_x$. So, we can remove from $\mathbf u \approx \mathbf v$ all multiple letters $x$ such that ${\bf u}(x,y)\ne xyx$ for any $y$. In other words, we may assume without loss of generality that for any letter $x\in\mul({\bf u})$ there is a letter $y$ such that ${\bf u}(x,y)= xyx = {\bf v}(x,y)$. In particular, $\occ_x({\bf u}),\occ_x({\bf v})\le 2$ for any letter $x$.
 
Let $0\le i\le m-1$. Then $\mathbf u=\mathbf w_1t_i\mathbf a_{i+1}t_{i+1}\mathbf w_2$ where 
$$
\mathbf w_1=
\begin{cases}
t_0\mathbf a_1t_1\cdots t_{i-1}\mathbf a_i&\text{if }0<i\le m-1,\\
\lambda&\text{if }i=0
\end{cases}
$$
and
$$
\textbf w_2=
\begin{cases}
\mathbf a_{i+2}t_{i+2}\cdots\mathbf a_mt_m&\text{if }0\le i<m-1,\\
\lambda&\text{if }i=m-1.
\end{cases}
$$
Analogously, $\mathbf v=\mathbf w_1't_i\mathbf b_{i+1}t_{i+1}\mathbf w_2'$ where
$$
\mathbf w_1'=
\begin{cases}
t_0\mathbf b_1t_1\cdots t_{i-1}\mathbf b_i&\text{if }0<i\le m-1,\\
\lambda&\text{if }i=0
\end{cases}
$$
and
$$
\mathbf w_2'=
\begin{cases}
\mathbf b_{i+2}t_{i+2}\cdots\mathbf b_mt_m&\text{if }0\le i<m-1,\\
\lambda&\text{if }i=m-1.
\end{cases}
$$
Suppose that the word $\mathbf a_{i+1}$ contains the subword $\mathbf d=x_ix_j$ where $x_i\in\con(\mathbf w_1)$ and $x_j\in\con(\mathbf w_2)$. The occurrence of the letter $x_i$ in the word \textbf d is second occurrence of $x_i$ in \textbf u, while the occurrence of the letter $x_j$ in the word \textbf d is first occurrence of $x_j$ in \textbf u. The identity $\gamma_1$ allows us to swap these two occurrences. Therefore, the variety ${\bf N}$ satisfies the identity ${\bf u}\approx {\bf w}_1t_i{\bf p}_1{\bf q}_1t_{i+1}{\bf w}_2$ where $\con({\bf p}_1) \subseteq \con({\bf w}_2)$ and $\con({\bf q}_1) \subseteq \con({\bf w}_1)$. Analogously, we can prove that ${\bf N}$ satisfies ${\bf v} \approx {\bf w}_1't_i{\bf p}_2{\bf q}_2t_{i+1}{\bf w}'_2$ where $\con({\bf p}_2) \subseteq \con({\bf w}'_2)$ and $\con({\bf q}_2) \subseteq \con({\bf w}_1')$.
 
We are going to verify that $\con({\bf p}_1)=\con({\bf p}_2)$ and $\con({\bf q}_1)=\con({\bf q}_2)$. Let $x\in\con(\mathbf p_1)$. Then $\mathbf u(x,t_{i+1})=xt_{i+1}x$. Therefore, $\mathbf v(x,t_{i+1})=xt_{i+1}x$. This means that $x\in\con(\mathbf w_1'\mathbf p_2\mathbf q_2)$ and $x\in\con(\mathbf w_2')$. If $x\in\con(\mathbf q_2)$ then $x\in\con(\mathbf w_1')$ as well, whence $\occ_x(\mathbf v)\ge 3$. Therefore, $x\notin\con(\mathbf q_2)$. Note that $\mathbf u(x,t_i)=t_ix^2$. Therefore, $\mathbf v(x,t_i)\ne xt_ix$, whence $x\notin\con(\mathbf w_1')$. We see that $x\in\con(\mathbf p_2)$. We have just proved that $\con(\mathbf p_1)\subseteq\con(\mathbf p_2)$. By symmetry, $\con(\mathbf p_2)\subseteq\con(\mathbf p_1)$ , whence $\con(\mathbf p_1)=\con(\mathbf p_2)$. Analogous arguments imply that $\con(\mathbf q_1)=\con(\mathbf q_2)$.
 
Therefore, $\mathbf p_1=x_1x_2\cdots x_k$ and $\mathbf p_2=x_{\pi(1)}x_{\pi(2)}\cdots x_{\pi(k)}$ for some letters $x_1,x_2$, $\dots,x_k\in \con({\bf w}_2)\cap \con({\bf w}_2')$ and some permutation $\pi\in S_k$, whence ${\bf N}$ satisfies the identities
$$
{\bf u}\approx {\bf w}_1t_i x_1x_2\cdots x_k \mathbf q_1 t_{i+1}{\bf w}_2\text{ and }{\bf v}\approx {\bf w}_1't_i x_{\pi(1)}x_{\pi(2)}\cdots x_{\pi(k)} \mathbf q_2 t_{i+1}{\bf w}'_2.
$$
Then the identity
\begin{equation}
\label{x_1...x_k permute}
{\bf w}_1t_i x_1x_2\cdots x_k{\bf q}_1t_{i+1}{\bf w}_2\approx {\bf w}_1't_i x_{\pi(1)}x_{\pi(2)}\cdots x_{\pi(k)}t_{i+1}{\bf q}_2{\bf w}'_2
\end{equation}
holds in ${\bf V}$.
 
Suppose that the permutation $\pi$ is non-trivial. Then there are $j$ and $\ell$ such that $j<\ell$ but $\pi(j)>\pi(\ell)$. Substituting~1 for all letters occurring in the identity~\eqref{x_1...x_k permute} except $x_j,x_\ell$ and $t_{i+1}$, we obtain the identity $x_jx_\ell t_{i+1}\mathbf s\approx x_\ell x_jt_{i+1}\mathbf s'$ where $\mathbf s,\mathbf s'\in \{x_jx_\ell,x_\ell x_j\}$. Now we apply the identity $\sigma_2$ and get $x_jx_\ell t_{i+1}x_jx_\ell\approx x_\ell x_jt_{i+1}x_jx_\ell$. The last identity is nothing but $\alpha_1$ (up to renaming of letters). So, if the permutation $\pi$ is non-trivial then ${\bf V}$ satisfies $\alpha_1$. This means that $\mathbf{V\subseteq M}$, whence $\mathbf V=\mathbf M$. In other words, if $\mathbf p_1\ne\mathbf p_2$ then $\mathbf V=\mathbf M$.
 
Let now $\mathbf p_1=\mathbf p_2$. The words $\mathbf q_1$ and $\mathbf q_2$ are linear and $\con(\mathbf q_1)=\con(\mathbf q_2)\subseteq\con(\mathbf w_1)\cap\con(\mathbf w_1')$. Thus, if some letter $z$ occurs in $\con(\mathbf q_1)$ then this occurrence is the second occurrence of $z$ in the word \textbf u. Hence the identity $\sigma_2$ allows us to reorder letters in $\mathbf q_1$ in an arbitrary way. Therefore, we can replace $\mathbf q_1$ by $\mathbf q_2$ in \textbf u, and the word we obtain should be equal to \textbf u in \textbf N. Thus, \textbf N satisfies the identities
$$
\mathbf u={\bf w}_1t_i \mathbf a_{i+1}t_{i+1}{\bf w}_2\approx{\bf w}_1t_i\mathbf p_1\mathbf q_1 t_{i+1}{\bf w}_2\approx {\bf w}_1t_i\mathbf p_2\mathbf q_2 t_{i+1}{\bf w}_2\approx{\bf w}_1t_i\mathbf b_{i+1}t_{i+1}{\bf w}_2.
$$
This is true for all $i=0,1,\dots,m-1$. Therefore, \textbf N satisfies the identities
$$
\mathbf u=t_0\mathbf a_1t_1\mathbf a_2t_2\cdots t_{m-1}\mathbf a_mt_m\approx t_0\mathbf b_1t_1\mathbf b_2t_2\cdots t_{m-1}\mathbf b_mt_m=\mathbf v.
$$
The proposition is proved.
\end{proof}
 
\section{The proof of the ``if'' part: the variety \textbf K}
\label{sufficiency: K}
 
Here we are going to verify that $\mathbf K$ is a chain variety. This case is much more complex than all the ones discussed in the previous section, taken together, and its consideration will be many times longer. For reader convenience, we divide this section into four subsections.
 
\subsection{Reduction to the interval $[\mathbf E,\mathbf K]$}
\label{sufficiency: K - red to [E,K]}
 
We fix notation for the following identity system:
$$
\Phi=\{xyx\approx xyx^2,\,x^2y^2\approx y^2x^2,\,x^2y\approx x^2yx\}.
$$
Note that $\mathbf K=\var\Phi$. For any $s\in\mathbb N$ and $1\le q\le s$, we put
$$
{\bf b}_{s,q}=x_{s-1}x_sx_{s-2}x_{s-1}\cdots x_{q-1}x_q.
$$
For brevity, we will write $\mathbf b_s$ rather than $\mathbf b_{s,1}$. We put also ${\bf b}_0=\lambda$ for convenience. We introduce the following four countably infinite series of identities:
\begin{align*}
\alpha_k:&\enskip x_ky_kx_{k-1}x_ky_k{\bf b}_{k-1} \approx y_kx_kx_{k-1}x_ky_k {\bf b}_{k-1},\\
\beta_k:&\enskip xx_kx{\bf b}_k\approx x_kx^2{\bf b}_k,\\
\gamma_k:&\enskip y_1y_0x_ky_1{\bf b}_k \approx y_1y_0y_1x_k{\bf b}_k,\\
\delta_k^m:&\enskip y_{m+1}y_mx_ky_{m+1} {\bf b}_{k,m}y_m{\bf b}_{m-1}\approx y_{m+1}y_my_{m+1}x_k{\bf b}_{k,m}y_m{\bf b}_{m-1}
\end{align*}
where $k\in\mathbb N$ and $1\le m\le k$. Note that the identities $\alpha_1$ and $\gamma_1$ have already appeared above. We define the following four countably infinite series of varieties:
$$
{\bf F}_k=\var\{\Phi,\alpha_k\},\,{\bf H}_k=\var\{\Phi,\beta_k\},\,{\bf I}_k=\var\{\Phi,\gamma_k\},\,{\bf J}_k^m=\var\{\Phi,\delta_k^m\}.
$$
 
In this section we are going to verify the following
 
\begin{proposition}
\label{L(K)}
\begin{itemize}
\item[\textup{1)}] The lattice $L(\mathbf K)$ is the set-theoretical union of the lattice $L(\mathbf E)$ and the interval $[\mathbf E,\mathbf K]$.
\item[\textup{2)}] The lattice $L(\mathbf E)$ is the chain $\mathbf{T\subset SL\subset C}_2\subset\mathbf D_1\subset\mathbf E$.
\item[\textup{3)}] If $\mathbf X$ is a monoid variety such that $\mathbf{E\subset X\subset K}$ then $\mathbf X$ belongs to the interval $[\mathbf F_k,\mathbf F_{k+1}]$ for some $k$.
\item[\textup{4)}] The interval $[\mathbf F_k,\mathbf F_{k+1}]$ is the chain
\begin{equation}
\label{[F_k,F_{k+1}]}
\mathbf F_k\subset\mathbf H_k\subset\mathbf I_k\subset\mathbf J_k^1\subset\mathbf J_k^2\subset\cdots\subset\mathbf J_k^k\subset\mathbf F_{k+1}.
\end{equation}
\end{itemize}
\end{proposition}
 
This proposition immediately implies that the lattice $L(\mathbf K)$ is the following chain:
\begin{align*}
\mathbf{T\subset SL\subset C}_2\subset\mathbf D_1\subset\mathbf E&\subset\mathbf F_1\subset\mathbf H_1\subset\mathbf I_1\subset\mathbf J_1^1\\
&\subset\mathbf F_2\subset\mathbf H_2\subset\mathbf I_2\subset\mathbf J_2^1\subset\mathbf J_2^2\\
&\rule{6pt}{0pt}\vdots\\
&\subset\mathbf F_k\subset\mathbf H_k\subset\mathbf I_k\subset\mathbf J_k^1\subset\mathbf J_k^2\subset\cdots\subset\mathbf J_k^k\\
&\rule{6pt}{0pt}\vdots\\
&\subset\mathbf K.
\end{align*}
In the remainder of this subsection we verify the claim~1) of Proposition~\ref{L(K)}. The claim~2) follows from Lemma~\ref{L(LRB+C_2)}(ii). The claims~3) and~4) are proved in Subsections~\ref{sufficiency: K - red to [F_k,F_{k+1}]} and~\ref{sufficiency: K - structure of [F_k,F_{k+1}]} respectively. Subsection~\ref{sufficiency: K - aux} contains auxiliary assertions.
 
Let \textbf X be a monoid variety with ${\bf X\subseteq K}$. We need to verify that either $\mathbf{E\subseteq X}$ or $\mathbf{X\subseteq E}$. Substituting~1 for $y$ in the identity~\eqref{xyx=xyxx}, we obtain ${\bf X}$ satisfies the identity~\eqref{xx=xxx}. If ${\bf X}$ is commutative then ${\bf X}\subseteq{\bf C}_2\subseteq{\bf E}$, and we are done. Thus, we can assume that ${\bf X}$ is non-commutative. The variety ${\bf X}$ is aperiodic because it satisfies the identity~\eqref{xx=xxx}. Suppose that the variety ${\bf X}$ is completely regular. Every aperiodic completely regular variety is a variety of band monoids and every band satisfying the identity~\eqref{xxyy=yyxx} is commutative. Thus, if ${\bf X}$ is completely regular then it is commutative, a contradiction. Hence we can assume that ${\bf X}$ is non-completely regular. Then $\mathbf D_1\subseteq\mathbf X$ by Lemma~\ref{non-cr and non-commut}.
 
Suppose that ${\bf E\nsubseteq X}$. Then ${\bf X}$ satisfies the identity~\eqref{yxx=xxyxx} by Lemma~\ref{E nsubseteq X}. Further, \textbf X satisfies the identity~\eqref{xxy=xxyx} as well because $\mathbf{X\subseteq K}$. Hence $x^2y\stackrel{\eqref{xxy=xxyx}}\approx x^2yx^2\stackrel{\eqref{yxx=xxyxx}}\approx yx^2$. We see that the identity~\eqref{xxy=yxx} holds in \textbf X. Besides that,
$$
xyx\stackrel{\eqref{xyx=xyxx}}\approx xyx^2\stackrel{\eqref{yxx=xxyxx}}\approx x^3yx^2\stackrel{\eqref{xx=xxx}}\approx x^2yx^2\stackrel{\eqref{yxx=xxyxx}}\approx yx^2\stackrel{\eqref{xxy=yxx}}\approx x^2y,
$$
whence the identity
\begin{equation}
\label{xyx=xxy}
xyx\approx x^2y
\end{equation}
holds in \textbf X. So, $\mathbf X$ satisfies the identities $yx^2 \stackrel{\eqref{xxy=yxx}}\approx x^2y \stackrel{\eqref{xyx=xxy}}\approx xyx$. The identities~\eqref{xxy=yxx}~\eqref{xxyy=yyxx} and~\eqref{xyx=xxy} evidently imply $\sigma_1$, $\sigma_2$ and $\gamma_1$. Thus, $\mathbf{X\subseteq D}_1\subseteq\mathbf E$. We proved that if ${\bf E\nsubseteq X}$ then ${\bf X\subseteq E}$. Hence the claim~1) of Proposition~\ref{L(K)} is proved.
 
\subsection{Several auxiliary results}
\label{sufficiency: K - aux}
 
Here we prove several lemmas that will be used many times below. This subsection is divided into three subsubsections.

\subsubsection{Some properties of the varieties $\mathbf F_k$, $\mathbf H_k$, $\mathbf I_k$, $\mathbf J_k^m$, $\mathbf K$ and their identities}
\label{sufficiency: K - aux-FHIJK}
 
\begin{lemma}
\label{identities in K}
The variety ${\bf K}$ satisfies:
\begin{itemize}
\item[\textup{(i)}] the identity $\sigma_2$;
\item[\textup{(ii)}] the identity
\begin{equation}
\label{xyxzx=xyxz}
xyxzx\approx xyxz;
\end{equation}
\item[\textup{(iii)}] any identity $\mathbf{u\approx v}$ such that $\con(\mathbf u)=\con(\mathbf v)$ and $\occ_x(\mathbf u),\occ_x(\mathbf v)\ge 2$ for any letter $x\in \con(\mathbf u)$.
\end{itemize}
\end{lemma}
 
\begin{proof}
(i) We have $xzytxy\stackrel{\eqref{xyx=xyxx}}\approx xzytx^2y^2\stackrel{\eqref{xxyy=yyxx}}\approx xzyty^2x^2\stackrel{\eqref{xyx=xyxx}}\approx xzytyx$.
 
\medskip
 
(ii) Here we have $xyxzx\stackrel{\eqref{xyx=xyxx}}\approx xyx^2zx\stackrel{\eqref{xxy=xxyx}}\approx xyx^2z\stackrel{\eqref{xyx=xyxx}}\approx xyxz$.
 
\medskip
 
(iii) According to the claim~(ii), \textbf V satisfies the identity~\eqref{xyxzx=xyxz}. This allows us to assume that $\occ_x(\mathbf u)=\occ_x(\mathbf v)=2$ for any letter $x\in\con(\mathbf u)$. Let $\con(\mathbf u)=\con(\mathbf v)=\{x_1,x_2,\dots, x_k\}$. We are going to verify that ${\bf u}\approx x_1^2x_2^2\cdots x_k^2$ holds in ${\bf K}$. We will use induction on $k$.
 
\smallskip
 
\emph{Induction base}. Suppose that $k=1$. Here the identity $\mathbf{u\approx v}$ has the form $x_1^2\approx x_1^2$, whence it trivially holds in $\bf K$.
 
\smallskip
 
\emph{Induction step}. Let now $k>1$. We may assume without loss of generality that $\ell_1(\mathbf u, x_i)<\ell_1(\mathbf u, x_k)$ for any $1\le i<k$. Then 
$$
\mathbf u = \mathbf u'x_kx_{j_1}x_{j_2}\cdots x_{j_s}x_kx_{j_{s+1}}x_{j_{s+2}}\cdots x_{j_{s+t}}
$$
where $x_{j_r}\in \con(\mathbf u')$ for any $1\le r\le s+t$. Then the identities
\begin{align*}
\mathbf u\stackrel{\eqref{xyx=xyxx}}\approx{}& \mathbf u'x_kx_{j_1}^2x_{j_2}^2\cdots x_{j_s}^2x_k^2x_{j_{s+1}}^2x_{j_{s+2}}^2\cdots x_{j_{s+t}}^2\\
\stackrel{\;\eqref{xxyy=yyxx}\;}\approx{}&\mathbf u'x_k^3x_{j_1}^2x_{j_2}^2\cdots x_{j_{s+t}}^2\\
\stackrel{\;\eqref{xx=xxx}\;}\approx{}&\mathbf u'x_k^2x_{j_1}^2x_{j_2}^2\cdots x_{j_{s+t}}^2\\
\stackrel{\;\eqref{xxyy=yyxx}\;}\approx{}&\mathbf u'x_{j_1}^2x_{j_2}^2\cdots x_{j_{s+t}}^2x_k^2\\
\stackrel{\;\eqref{xyxzx=xyxz}\;}\approx{}&\mathbf u'x_{j_1}x_{j_2}\cdots x_{j_{s+t}}x_k^2\\
\hskip-12pt=\hskip10pt&\mathbf u_{x_k}x_k^2
\end{align*}
hold in $\bf K$. The word ${\bf u}_{x_k}$ contains exactly $k-1$ letters. By the induction assumption, the identity ${\bf u}_{x_k} \approx x_1^2x_2^2\cdots x_{k-1}^2$ holds in ${\bf K}$, whence this variety satisfies the identities $\mathbf u\approx\mathbf u_{x_k}x_k^2\approx x_1^2x_2^2\cdots x_k^2$. Similarly, ${\bf v} \approx x_1^2x_2^2\cdots x_k^2$ holds in ${\bf K}$, whence ${\bf K}$ satisfies $\bf u \approx v$.
\end{proof}
 
\begin{lemma}
\label{basis for J_k^k}
The identity system $\Phi$ together with the identity
\begin{equation}
\label{xx_kxb_k=x^2x_kb_k}
xx_kx{\bf b}_k \approx x^2x_k{\bf b}_k
\end{equation}
forms an identity basis of the variety ${\bf J}_k^k$.
\end{lemma}
 
\begin{proof}
First of all, we note that the identity~\eqref{xx_kxb_k=x^2x_kb_k} holds in the variety ${\bf J}_k^k$. To check this fact, it suffices to perform the substitution $(y_k,y_{k+1})\mapsto(1,x)$ in the identity $\delta_k^k$ and use the equality $\mathbf b_k=x_{k-1}x_k\mathbf b_{k-1}$. So, we need to verify that $\delta_k^k$ follows from $\Phi$ and~\eqref{xx_kxb_k=x^2x_kb_k}. In view of Lemma~\ref{identities in K}, we can use the identities $\sigma_2$ and~\eqref{xyxzx=xyxz}. Here is the required deduction (letters in the right column refer to comments after the deduction):
\begin{align*}
y_{k+1}y_kx_ky_{k+1}\mathbf b_{k,k}y_k\mathbf b_{k-1}={}&y_{k+1}y_kx_ky_{k+1}x_{k-1}x_ky_k\mathbf b_{k-1}&&\textup{(a)}\\
\approx{}&y_{k+1}y_kx_ky_{k+1}x_{k-1}y_kx_k\mathbf b_{k-1}&&\textup{(b)}\\
\approx{}&y_{k+1}^2y_kx_kx_{k-1}y_kx_k\mathbf b_{k-1}&&\textup{(c)}\\
\approx{}&y_{k+1}^2y_kx_kx_{k-1}x_ky_k\mathbf b_{k-1}&&\textup{(d)}\\
={}&y_{k+1}^2y_kx_kx_{k-1}x_ky_kx_{k-2}x_{k-1}\mathbf b_{k-2}&&\textup{(e)}\\
\approx{}&y_{k+1}^2y_kx_kx_{k-1}x_ky_kx_{k-2}x_kx_{k-1}x_k\mathbf b_{k-2}&&\textup{(f)}\\
\approx{}&y_{k+1}y_ky_{k+1}x_kx_{k-1}x_ky_kx_{k-2}x_kx_{k-1}x_k\mathbf b_{k-2}&&\textup{(g)}\\
\approx{}&y_{k+1}y_ky_{k+1}x_kx_{k-1}x_ky_kx_{k-2}x_{k-1}\mathbf b_{k-2}&&\textup{(h)}\\
={}&y_{k+1}y_ky_{k+1}x_k\mathbf b_{k,k}y_k\mathbf b_{k-1}.&&\textup{(i)}
\end{align*}
 
(a) Here we use the equality $\mathbf b_{k,k}=x_{k-1}x_k$.
 
(b) Here we modify the subword $y_kx_ky_{k+1}x_{k-1}x_ky_k$ by performing the substitution $(x,t,y,z)\mapsto(y_k,1,x_k,y_{k+1}x_{k-1})$ in $\sigma_2$.
 
(c) Here we perform the substitution $(x,x_k)\mapsto(y_{k+1},y_kx_k)$ in \eqref{xx_kxb_k=x^2x_kb_k} and use the equality $\mathbf b_k=x_{k-1}x_k\mathbf b_{k-1}$.
 
(d) Here we modify the subword $y_kx_kx_{k-1}y_kx_k$ by performing the substitution $(x,t,y,z)\mapsto(y_k,1,x_k,x_{k-1})$ in $\sigma_2$.
 
(e) Here we use the equality $\mathbf b_{k-1}=x_{k-2}x_{k-1}\mathbf b_{k-2}$.
 
(f) Here \eqref{xyxzx=xyxz} allows us to add two new occurrences of the letter $x_k$ after its second occurrence in the word $y_{k+1}^2y_k\stackrel{(1)}{x_k}x_{k-1}\stackrel{(2)}{x_k}y_kx_{k-2}x_{k-1}\mathbf b_{k-2}$.
 
(g) Here we perform the substitution $(x,x_k,x_{k-1})\mapsto(y_{k+1},y_k,x_kx_{k-1}x_k)$ in \eqref{xx_kxb_k=x^2x_kb_k} and use the equality $\mathbf b_k=x_{k-1}x_kx_{k-2}x_{k-1}\mathbf b_{k-2}$.
 
(h) Here \eqref{xyxzx=xyxz} allows us to delete the third and the fourth occurrences of the letter $x_k$ in the word $y_{k+1}y_ky_{k+1}\stackrel{(1)}{x_k}x_{k-1}\stackrel{(2)}{x_k}y_kx_{k-2}\stackrel{(3)}{x_k}x_{k-1}\stackrel{(4)}{x_k}\mathbf b_{k-2}$.
 
(i) Here we use the equalities $\mathbf b_{k,k}=x_{k-1}x_k$ and $\mathbf b_{k-1}=x_{k-2}x_{k-1}\mathbf b_{k-2}$.
\end{proof}
 
\begin{lemma}
\label{between F_k and F_{k+1}}
The inclusions
\begin{equation}
\label{non-strict inclusions}
\mathbf F_k\subseteq\mathbf H_k\subseteq\mathbf I_k\subseteq\mathbf J_k^1\subseteq\mathbf J_k^2\subseteq\cdots\subseteq\mathbf J_k^k\subseteq\mathbf F_{k+1}
\end{equation}
are valid.
\end{lemma}
 
\begin{proof}
Since all varieties that appear in the inclusions~\eqref{non-strict inclusions} are contained in \textbf K, we can apply Lemma~\ref{identities in K}. In particular, this allows us to use below the identities $\sigma_2$ and~\eqref{xyxzx=xyxz}.
 
\smallskip
 
$1^\circ$. \emph{The inclusion} $\mathbf F_k\subseteq\mathbf H_k$. We need to verify that $\beta_k$ follows from $\Phi$ and $\alpha_k$. Here is the required deduction:
\begin{align*}
xx_kx\mathbf b_k&=xx_kxx_{k-1}x_k\mathbf b_{k-1}&&\textup{because }\mathbf b_k=x_{k-1}x_k\mathbf b_{k-1}\\
&\approx xx_kxx_{k-1}x_kx^2\mathbf b_{k-1}&&\textup{by }\eqref{xyxzx=xyxz}\\
&\approx x_kx^2x_{k-1}x_kx^2\mathbf b_{k-1}&&\textup{we perform the substitution}\\
&&&(x_k,y_k)\mapsto(x_kx,x)\textup{ in }\alpha_k
\end{align*}
\begin{align*}
&\approx x_kx^2x_{k-1}x_k\mathbf b_{k-1}&&\textup{by }\eqref{xyxzx=xyxz}\\
&=x_kx^2\mathbf b_k&&\textup{because }\mathbf b_k=x_{k-1}x_k\mathbf b_{k-1}.
\end{align*}
 
\smallskip
 
$2^\circ$. \emph{The inclusion} $\mathbf H_k\subseteq\mathbf I_k$. Here we need to verify that $\gamma_k$ follows from $\Phi$ and $\beta_k$. Indeed,
\begin{align*}
y_1y_0x_ky_1\mathbf b_k&\approx y_1y_0x_ky_1^2\mathbf b_k&&\textup{by }\eqref{xyxzx=xyxz}\\
&\approx y_1y_0y_1x_ky_1\mathbf b_k&&\textup{we modify the subword }x_ky_1^2\mathbf b_k\\
&&&\textup{by substitution }y_1\textup{ for }x\textup{ in }\beta_k\\
&\approx y_1y_0y_1x_k\mathbf b_k&&\textup{by }\eqref{xyxzx=xyxz}.
\end{align*}
 
\smallskip
 
$3^\circ$. \emph{The inclusion} $\mathbf I_k\subseteq\mathbf J_k^1$. It suffices to verify that $\delta_k^1$ follows from $\gamma_k$. Since $\mathbf b_{k,1}=\mathbf b_k$ and $\mathbf b_0=\lambda$, the identity $\delta_k^1$ has the form
$$
y_2y_1x_ky_2\mathbf b_ky_1\approx y_2y_1y_2x_k\mathbf b_ky_1.
$$
To deduce this identity from $\gamma_k$, it suffices to modify the subword $y_2y_1x_ky_2\mathbf b_k$ by performing the substitution $(y_0,y_1)\mapsto(y_1,y_2)$ in $\gamma_k$.
 
\smallskip
 
$4^\circ$. \emph{The inclusion $\mathbf J_k^m\subseteq\mathbf J_k^{m+1}$ where $1\le m<k$}. It suffices to verify that $\delta_k^{m+1}$ follows from $\delta_k^m$. Indeed, we get $\delta_k^{m+1}$ if we multiply $\delta_k^m$ by $x_{-1}x_0$ on the left and then increase by~1 the index of each letter in the identity we obtain.
 
\smallskip
 
$5^\circ$. \emph{The inclusion} $\mathbf J_k^k\subseteq\mathbf F_{k+1}$. In view of Lemma~\ref{basis for J_k^k}, it suffices to verify that $\alpha_{k+1}$ follows from $\Phi$ and~\eqref{xx_kxb_k=x^2x_kb_k}. We have:
\begin{align*}
x_{k+1}y_{k+1}x_kx_{k+1}y_{k+1}\mathbf b_k&\approx (x_{k+1}y_{k+1})^2x_k\mathbf b_k&&\textup{(a)}\\
&\approx (y_{k+1}x_{k+1})^2x_k\mathbf b_k&&\textup{(b)}\\
&\approx y_{k+1}x_{k+1}x_ky_{k+1}x_{k+1}\mathbf b_k&&\textup{(c)}\\
&\approx y_{k+1}x_{k+1}x_kx_{k+1}y_{k+1}\mathbf b_k.&&\textup{(d)}
\end{align*}
 
(a) Here we substitute $x_ky_k$ for $x$ in~\eqref{xx_kxb_k=x^2x_kb_k}.
 
(b) Here we apply the identity $(xy)^2\approx(yx)^2$ that holds in \textbf K according to Lemma~\ref{identities in K}(iii).
 
(c) Here we substitute $y_kx_k$ for $x$ in~\eqref{xx_kxb_k=x^2x_kb_k}.
 
(d) Here we perform the substitution $(x,t,y,z)\mapsto(y_{k+1},1,x_{k+1},x_k)$ in the identity $\sigma_2$.
\end{proof}
 
Below we often use the inclusions~\eqref{non-strict inclusions} without references to Lemma~\ref{between F_k and F_{k+1}}. Note that in fact strict inclusions~\eqref{[F_k,F_{k+1}]} are valid. We will prove these inclusions in Subsubsection~\ref{structure of [F_k,F_{k+1}] 6 step}.

\subsubsection{$k$-decompositions of sides of the identities $\alpha_k$, $\beta_k$, $\gamma_k$ and $\delta_k^m$}
\label{sufficiency: K - aux-decompositions}
 
\begin{lemma}
\label{depth and index}
Let ${\bf u}$ be a left-hand or right-hand side of one of the identities $\alpha_k$, $\beta_k$, $\gamma_k$ or $\delta_k^m$. Then:
\begin{itemize}
\item[\textup{1)}] If  $x_i,y_j\in\con(\mathbf u)$ then $D(\mathbf u,x_i)=i$ and $D(\mathbf u,y_j)=j$. The depth of the letter $x$ in the left-hand \textup[right-hand\textup] side of the identity $\beta_k$ equals $k+1$ \textup[respectively $\infty$\textup].
\item[\textup{2)}] The $k$-decomposition of the word $\mathbf u$ has the form indicated in Table~\textup{\ref{k-decomposition of eight words}}.
\end{itemize}
\end{lemma}

{\tabcolsep=4pt
\small
\begin{table}[tbh]
\caption{$k$-decompositions of some words}
\begin{center}
\begin{tabular}{|c|l|l|}
\hline
&\multicolumn{2}{c|}{The $k$-decomposition of the}\\
\cline{2-3}
The identity&\multicolumn{1}{c|}{left-hand side}&\multicolumn{1}{c|}{right-hand side}\\
\hline
$\alpha_k$&$\lambda\cdot\underline{\lambda}\cdot x_k \cdot\underline{\lambda}\cdot y_k\cdot \underline{\lambda}\cdot x_{k-1}\cdot\underline{x_ky_k}$&$\lambda\cdot\underline{\lambda}\cdot y_k \cdot\underline{\lambda}\cdot x_k\cdot \underline{\lambda}\cdot x_{k-1}\cdot\underline{x_ky_k}$\\
&\rule[-6pt]{0pt}{0pt}$\cdot\,x_{k-2}\cdot\underline{x_{k-1}}\cdots x_1\cdot\underline{x_2}\cdot x_0\cdot\underline{x_1}$&$\cdot\,x_{k-2}\cdot\underline{x_{k-1}}\cdots x_1\cdot\underline{x_2}\cdot x_0\cdot\underline{x_1}$\\
\hline
$\beta_k$&$\lambda\cdot\underline{x}\cdot x_k \cdot\underline{x}\cdot x_{k-1}\cdot\underline{x_k}\cdot x_{k-2}$&$\lambda\cdot\underline{\lambda}\cdot x_k\cdot\underline{x^2}\cdot x_{k-1}\cdot\underline{x_k}\cdot x_{k-2}$\\
&\rule[-6pt]{0pt}{0pt}$\cdot\,\underline{x_{k-1}}\cdots x_1\cdot\underline{x_2}\cdot x_0\cdot\underline{x_1}$&$\cdot\,\underline{x_{k-1}}\cdots x_1\cdot\underline{x_2}\cdot x_0\cdot\underline{x_1}$\\
\hline
$\gamma_k$&$\lambda\cdot\underline{\lambda}\cdot y_1\cdot\underline{\lambda}\cdot y_0\cdot\underline{\lambda}\cdot x_k \cdot\underline{y_1}\cdot x_{k-1}$&$\lambda\cdot\underline{\lambda}\cdot y_1\cdot\underline{\lambda}\cdot y_0\cdot\underline{y_1}\cdot x_k \cdot\underline{\lambda}\cdot x_{k-1}$\\
&\rule[-6pt]{0pt}{0pt}$\cdot\,\underline{x_k}\cdot x_{k-2}\cdot\underline{x_{k-1}}\cdots x_1\cdot\underline{x_2}\cdot x_0\cdot\underline{x_1}$&$\cdot\,\underline{x_k}\cdot x_{k-2}\cdot\underline{x_{k-1}}\cdots x_1\cdot\underline{x_2}\cdot x_0\cdot\underline{x_1}$\\
\hline
$\delta_k^m$ with&$\lambda\cdot\underline{\lambda}\cdot y_{m+1}\cdot\underline{\lambda}\cdot y_m\cdot\underline{\lambda}\cdot x_k \cdot\underline{y_{m+1}}$&$\lambda\cdot\underline{\lambda}\cdot y_{m+1}\cdot\underline{\lambda}\cdot y_m\cdot\underline{y_{m+1}}\cdot x_k \cdot\underline{\lambda}$\\
$m<k$&$\cdot\,x_{k-1}\cdot \underline{x_k}\cdots x_{m-1}\cdot \underline{x_my_m}\cdot x_{m-2}$&$\cdot\,x_{k-1}\cdot \underline{x_k}\cdots x_{m-1}\cdot \underline{x_my_m}\cdot x_{m-2}$\\
&\rule[-6pt]{0pt}{0pt}$\cdot\,\underline{x_{m-1}}\cdots x_1\cdot\underline{x_2}\cdot x_0\cdot\underline{x_1}$&$\cdot\,\underline{x_{m-1}}\cdots x_1\cdot\underline{x_2}\cdot x_0\cdot\underline{x_1}$\\
\hline
&$\lambda\cdot\underline{y_{k+1}}\cdot y_k\cdot\underline{\lambda}\cdot x_k \cdot\underline{y_{k+1}}\cdot x_{k-1}$&$\lambda\cdot\underline{y_{k+1}}\cdot y_k\cdot\underline{y_{k+1}}\cdot x_k \cdot\underline{\lambda}\cdot x_{k-1}$\\
$\delta_k^k$&$\cdot\,\underline{x_ky_k}\cdot x_{k-2}\cdot\underline{x_{k-1}}\cdots x_1\cdot\underline{x_2}\cdot x_0$&$\cdot\,\underline{x_ky_k}\cdot x_{k-2}\cdot\underline{x_{k-1}}\cdots x_1\cdot\underline{x_2}\cdot x_0$\\
&\rule[-6pt]{0pt}{0pt}$\cdot\,\underline{x_1}$&$\cdot\,\underline{x_1}$\\
\hline
\end{tabular}
\end{center}
\label{k-decomposition of eight words}
\end{table}
}

 As in Example~\ref{example decompositions}, we underline $k$-blocks of words in Table~\ref{k-decomposition of eight words} to distinguish them from $k$-dividers.
 
\begin{proof}[Proof of Lemma~\textup{\ref{depth and index}}]
We allow ourselves to verify both the claims for the left-hand side of the identity $\alpha_k$ only. In all other cases the proof is given by quite similar considerations. We denote  the left-hand side of the identity $\alpha_k$ by $\mathbf u_k$. So,
$$
\mathbf u_k=x_ky_kx_{k-1}x_ky_kx_{k-2}x_{k-1}x_{k-3}x_{k-2}\cdots x_1x_2x_0x_1.
$$

\smallskip

1) The letter $x_0$ is simple in $\mathbf u_k$, whence $D(\mathbf u_k,x_0)=0$. All other letters from $\con(\mathbf u_k)$ occur in $\mathbf u_k$ exactly twice. In particular, they are multiple in $\mathbf u_k$, and therefore their depth in $\mathbf u_k$ grater than 0. First occurrence of $x_1$ in $\mathbf u_k$ is not preceded by any simple letter. Therefore, $h_1^0(\mathbf u_k,x_1)=\lambda$. Further, only $x_0$ is the simple in $\mathbf u_k$ letter that precedes second occurrence of $x_1$ in $\mathbf u_k$. Hence $h_2^0(\mathbf u_k,x_1)=x_0$. We see
that $h_1^0(\mathbf u_k,x_1)\ne h_2^0(\mathbf u_k,x_1)$, whence $D(\mathbf u_k,x_1)=1$.

Both the first and the second occurrences of $x_2$ in $\mathbf u_k$ are not preceded by any simple in $\mathbf u_k$ letter. This means that $h_1^0(\mathbf u_k,x_2)= h_2^0(\mathbf u_k,x_2)=\lambda$, whence $D(\mathbf u_k,x_2)>1$. Second occurrence of $x_2$ in $\mathbf u_k$ is preceded by exactly one occurrence of $x_1$ and there are no any letters between these occurrences of $x_1$ and $x_2$. Besides that, $h_1^0(\mathbf u_k,x_1)\ne h_2^0(\mathbf u_k,x_1)$. Therefore, $h_2^1(\mathbf u_k,x_2)=x_1$. On the other hand, $h_1^1(\mathbf u_k,x_2)\ne x_1$ because $x_1$ does not occur before first occurrence of $x_2$ in $\mathbf u_k$. Thus, $h_1^1(\mathbf u_k,x_2)\ne h_2^1(\mathbf u_k,x_2)$, whence $D(\mathbf u_k,x_2)=2$.

We introduce some new notation to facilitate further considerations. For a letter $a\in\mul(\mathbf u_k)$, we denote by $\mathbf u_k[a;1,2]$ the subword of $\mathbf u_k$ located between the first and the second occurrences of $a$ in $\mathbf u_k$. For instance, $\mathbf u_k[x_k;1,2]=y_kx_{k-1}$, $\mathbf u_k[y_k;1,2]=x_{k-1}x_k$, while $\mathbf u_k[x_1;1,2]=x_2x_0$. Let now $2<r<k$. Suppose that we prove the equality $D(\mathbf u_k,x_i)=i$ for all $i=0,1,\dots,r-1$. We are going to check that $D(\mathbf u_k,x_r)=r$. Suppose that $D(\mathbf u_k,x_r)=s<r$. This means that $h_1^{s-1}(\mathbf u_k,x_r)\ne h_2^{s-1}(\mathbf u_k,x_r)$. Therefore, there is a letter $z$ such that first occurrence of $z$ in $\mathbf u_k$ lies in $\mathbf u_k[x_r;1,2]$ and $h_1^{s-2}(\mathbf u_k,z)\ne h_2^{s-2}(\mathbf u_k,z)$. But  $\mathbf u_k[x_r;1,2]=x_{r+1}x_{r-1}$ whenever $r<k-1$ and $\mathbf u_k[x_{k-1};1,2]=x_ky_kx_{k-2}$. In any case, a unique letter whose first occurrence in $\mathbf u_k$ lies in $\mathbf u_k[x_r;1,2]$ is $x_{r-1}$. In view of our assumption, $D(\mathbf u_k,x_{r-1})=r-1$. Since $s-2<r-2$ the latest equality implies that $h_1^{s-2}(\mathbf u_k,x_{r-1})=h_2^{s-2}(\mathbf u_k,x_{r-1})$. Thus, there are no letters $z$ with the above-mentioned properties. Therefore, $D(\mathbf u_k,x_r)\ge r$. Suppose now that $D(\mathbf u_k,x_r)=t>r$. Then $h_1^{r-1}(\mathbf u_k,x_r)=h_2^{r-1}(\mathbf u_k,x_r)$. Therefore, there are no letters $z$ such that first occurrence of $z$ in $\mathbf u_k$ lies in $\mathbf u_k[x_r;1,2]$ and $D(\mathbf u_k,z)=r-1$. But our assumption implies that the letter $x_{r-1}$ have these properties. Thus, $D(\mathbf u_k,x_r)=r$.

The arguments quite analogous to ones from the previous paragraph permit establish that $D(\mathbf u_k,y_k)=k$. It is necessary to take into account the equality $D(\mathbf u_k,x_{k-1})=k-1$ proved above and the fact that a unique letter whose first occurrence in $\mathbf u_k$ lies in $\mathbf u[y_k;1,2]$ is $x_{k-1}$.

It remains to verify that $D(\mathbf u_k,x_k)=k$. We note that both the first and the second occurrences of $x_k$ in $\mathbf u_k$ are not preceded by any simple letter, whence $h_1^0(\mathbf u_k,x_k)=h_2^0(\mathbf u_k,x_k)=\lambda$. Suppose that $h_1^i(\mathbf u_k,x_k)\ne h_2^i(\mathbf u_k,x_k)$ for some $0<i<k-1$. Then there is a letter $z$ such that first occurrence of $z$ in $\mathbf u_k$ lies in $\mathbf u[x_k;1,2]$ and $h_1^{i-1}(\mathbf u_k,z)=h_2^{i-1}(\mathbf u_k,z)$. The latest equality means that $D(\mathbf u_k,z)\le i<k-1$. Further, $\mathbf u[x_k;1,2]=y_kx_{k-1}$ and occurrences of both $y_k$ and $x_{k-1}$ are the first occurrences of these letters in $\mathbf u_k$. As we have seen above, $D(\mathbf u_k,y_k),D(\mathbf u_k,x_{k-1})\ge k-1$. Thus, $h_1^i(\mathbf u_k,x_k)=h_2^i(\mathbf u_k,x_k)$ for all $0\le i<k-1$. Now we check that $h_1^k(\mathbf u_k,x_k)\ne h_2^k(\mathbf u_k,x_k)$. Indeed, we have seen above that $D(\mathbf u_k,x_{k-1})=k-1$ and $D(\mathbf u_k,y_k)=k$. Therefore, $h_1^{k-2}(\mathbf u_k,x_{k-1})\ne h_2^{k-2}(\mathbf u_k,x_{k-1})$ and $h_1^{k-2}(\mathbf u_k,y_k)=h_2^{k-2}(\mathbf u_k,y_k)$. This implies that $h_2^{k-1}(\mathbf u_k,x_k)=x_{k-1}$. On the other hand, first occurrence of $x_k$ in $\mathbf u_k$ is not preceded by any letter, whence $h_1^{k-1}(\mathbf u_k,x_k)=\lambda$. We see that $h_1^k(\mathbf u_k,x_k)\ne h_2^k(\mathbf u_k,x_k)$. In view of the above, this means that $D(\mathbf u_k,x_k)=k$.

\smallskip

2) By Lemma~\ref{k-divider and depth}, $k$-dividers of a word $\mathbf w$ are exactly the first occurrences of letters $x\in\con(\mathbf w)$ with $D(\mathbf w,x)\le k$ and the empty word at the beginning of the word $\mathbf w$. As we have proved above, $D(\mathbf u_k,x)\le k$ for any letter $x\in\con(\mathbf u_k)$. Thus, $k$-dividers of $\mathbf u_k$ are just the first occurrences of all letters from $\con(\mathbf u_k)$ and the empty word at the beginning of $\mathbf u_k$. All subwords of $\mathbf u_k$ between these $ k$-dividers and only they are $k$-blocks of $\mathbf u_k$. Thus, the $k$-decomposition of the word $\mathbf u_k$ has the form indicated in Table~\ref{k-decomposition of eight words}.
\end{proof}

Note that the claim 1) of Lemma~\ref{depth and index} explains the choice of indexes of letters in the identities $\alpha_k$, $\beta_k$, $\gamma_k$ and $\delta_k^m$.

\subsubsection{Swapping letters within $k$-blocks}
\label{sufficiency: K - aux-swapping}
 
In this subsubsection we verify only one statement. It can be called the ``core'' of the whole proof of Theorem~\ref{main result}. Its proof is very long and based on a quite hard technique. At the same time, it is the basis for the rest of the proof of Theorem~\ref{main result} and plays a key role there.
 
\begin{lemma}
\label{u'abu''=u'bau''}
Let $\mathbf V$ be a monoid variety such that $\mathbf{V\subseteq K}$, $\mathbf u$ be a word and $k$ be a natural number. Further, let $\mathbf u=\mathbf u'ab\mathbf u''$ where $\mathbf u'$ and $\mathbf u''$ are possibly empty words, while $ab$ is a subword of some \mbox{$(k-1)$}-block of $\mathbf u$. Suppose that one of the following holds:
\begin{itemize}
\item[\textup{(i)}] $\mathbf V$ satisfies $\delta_k^m$, $a\in\con({\bf u}')$ and $D({\bf u},a)>m$;
\item[\textup{(ii)}] $\mathbf V$ satisfies $\gamma_k$ and $a\in\con(\mathbf u')$;
\item[\textup{(iii)}] $\mathbf V$ satisfies $\beta_k$ and $D(\mathbf u,a)\ne D(\mathbf u,b)$;
\item[\textup{(iv)}] $\mathbf V$ satisfies $\alpha_k$.
\end{itemize}
Then $\mathbf V$ satisfies the identity $\mathbf u\approx\mathbf u'ba\mathbf u''$.
\end{lemma}
 
\begin{proof}
We will prove the assertions~(i)--(iv) simultaneously. Suppose that the variety \textbf V satisfies the hypothesis of one of these four claims. In particular, $\mathbf V$ satisfies $\delta_k^k$ in any case. Let~\eqref{t_0u_0t_1u_1 ... t_mu_m} be the \mbox{$(k-1)$}-decomposition of $\bf u$ and $ab$ is a subword of $\mathbf u_i$ for some $0\le i\le m$. Then $\mathbf u_i= \mathbf u_i'ab\mathbf u_i''$ for some possibly empty words $\mathbf u_i'$ and $\mathbf u_i''$. Clearly, $\mathbf u' = t_0\mathbf u_0t_1\mathbf u_1\cdots t_i\mathbf u_i'$ and $\mathbf u''= \mathbf u_i''t_{i+1}\mathbf u_{i+1}\cdots t_m\mathbf u_m$.
 
If $a,b\in\con(\mathbf u')$ then
$$
{\bf u}= {\bf u}'ab{\bf u}''\stackrel{\eqref{xyx=xyxx}}\approx{\bf u}'a^2b^2{\bf u}''\stackrel{\eqref{xxyy=yyxx}}\approx{\bf u}'b^2a^2{\bf u}''\stackrel{\eqref{xyx=xyxx}}\approx{\bf u}'ba{\bf u}'',
$$
and we are done. Thus, we can assume without loss of generality that
\begin{equation}
\label{b notin u'}
b\notin \con(\mathbf u').
\end{equation}
If $D(\mathbf u, b)\le k-1$ then $b$ is a \mbox{$(k-1)$}-divider of $\bf u$ by Lemma~\ref{k-divider and depth}. But this is not the case because first occurrence of $b$ in \textbf u lies in the \mbox{$(k-1)$}-block $\mathbf u_i$. Therefore, $D(\mathbf u, b)\ge k$. Further, if $a\in\mul(\mathbf u')$ then Lemma~\ref{identities in K}(ii) implies that the identities $\mathbf u'ab\mathbf u''\approx\mathbf u'b\mathbf u''\approx\mathbf u'ba\mathbf u''$ hold in \textbf V. Thus, we can assume that
\begin{equation}
\label{a is not multiple in u'}
\text{if }a\in\con(\mathbf u')\text{ then }a\in\simple(\mathbf u').
\end{equation}
 
Further considerations are divided into three cases depending on the depth of $b$ in \textbf u: $D(\mathbf u, b)=k$, $k<D(\mathbf u, b)<\infty$ and $D(\mathbf u, b)=\infty$. Each of these cases is divided into subcases corresponding to the claims~(i)--(iv). Thus, the proof of each of the assertions~(i)--(iv) will be completed after considering the corresponding subcase of Case~3.
 
\medskip
 
\emph{Case} 1: $D(\mathbf u, b)=k$. This case is the most difficult from the technical point of view and the longest. By examining two other cases, we will repeatedly refer to properties that will be verified here. Let $\mathbf{p\approx q}$ be one of the identities $\alpha_k$, $\beta_k$, $\gamma_k$ or $\delta_k^m$. In a sense, the identity $\mathbf{p\approx q}$ ``looks like'' to $\mathbf u'ab\mathbf u''\approx\mathbf u'ba\mathbf u''$. We have in mind that the words \textbf p and \textbf q start with the same prefix (which is empty for $\alpha_k$ and $\beta_k$) and end with the same suffix, and the subword between these prefix and suffix is the product of two letters in \textbf p and the product of the same two letters in the reverse order in \textbf q. This makes it possible in principle to apply the identity $\mathbf{p\approx q}$ to one of the sides of the identity $\mathbf u'ab\mathbf u''\approx\mathbf u'ba\mathbf u''$ in order to obtain the other side of it. To realize this possibility, we need, with the use of the identities that hold in \textbf K, to reduce, say, the right-hand side of the identity $\mathbf u'ab\mathbf u''\approx\mathbf u'ba\mathbf u''$ to a form to which the identity $\mathbf{p\approx q}$ can be applied. To do this, we first need to find ``inside of'' the word \textbf u the letters $x_0,x_1,\dots,x_k$ which would appear in the same order as the letters with the same names in one of the sides of the identity $\mathbf{p\approx q}$.
 
Put $x_k=b$. Let $X_{k-1}$ be the set of \mbox{$(k-1)$}-dividers $z$ of $\bf u$ such that
$$
\ell_1(\mathbf u, x_k)<\ell_1(\mathbf u, z)<\ell_2(\mathbf u, x_k).
$$
The fact that $D(\mathbf u,x_k)=k$ implies that $h_1^{k-1}(\mathbf u,x_k)\ne h_2^{k-1}(\mathbf u,x_k)$, whence $h_2^{k-1}(\mathbf u,x_k)\in X_{k-1}$. Therefore, the set $X_{k-1}$ is non-empty. Further, Lemma~\ref{h_2^{k-1}}(ii) implies that $D(\mathbf u, z)=k-1$ and $\ell_2({\bf u},x_k)< \ell_2({\bf u},z)$ for any $z\in X_{k-1}$. Now we consider the letter $x_{k-1}\in X_{k-1}$ such that $\ell_2({\bf u},z)\le \ell_2({\bf u},x_{k-1})$ for any $z\in X_{k-1}$.{\sloppy

}
 
Let $X_{k-2}$ be the set of \mbox{$(k-2)$}-dividers $z$ of $\bf u$ such that $\ell_1(\mathbf u, x_{k-1})<\ell_1(\mathbf u, z)<\ell_2(\mathbf u, x_{k-1})$. Then the fact that $D(\mathbf u,x_{k-1})=k-1$ implies that $h_1^{k-2}(\mathbf u,x_{k-1})\ne h_2^{k-2}(\mathbf u,x_{k-1})$, whence $h_2^{k-2}(\mathbf u,x_{k-1})\in X_{k-2}$. Therefore, the set $X_{k-2}$ is non-empty. Further, Lemma~\ref{h_2^{k-1}}(ii) implies that $D(\mathbf u, z)=k-2$ and $\ell_2({\bf u},x_{k-1})< \ell_2({\bf u},z)$ for any $z\in X_{k-2}$. Now we consider the letter $x_{k-2}\in X_{k-2}$ such that $\ell_2({\bf u},z)\le \ell_2({\bf u},x_{k-2})$ for any $z\in X_{k-2}$. Since $\ell_1({\bf u},x_k)< \ell_1({\bf u},x_{k-1})<\ell_1({\bf u},x_{k-2})$, Lemma~\ref{if first then second} implies that $\ell_2({\bf u},x_k)<\ell_1({\bf u},x_{k-2})$.
 
Further, for $s=k-3,k-4,\dots,1$ we denote one by one the set $X_s$ and the letter $x_s$ in the following way: $X_s$ is the set of all $s$-dividers $z$ of \textbf u such that $\ell_1(\mathbf u, x_{s+1})<\ell_1(\mathbf u, z)<\ell_2(\mathbf u, x_{s+1})$, and $x_s$ is a letter such that $x_s\in X_s$ and $\ell_2({\bf u},z)\le \ell_2({\bf u},x_s)$ for any $z\in X_s$. Arguments similar to those from the previous two paragraphs allow us to verify that the set $X_s$ is non-empty, $D(\mathbf u,x_s)=s$, $\ell_j(\mathbf u,x_{s+1})<\ell_j(\mathbf u,x_s)$ for any $j=1,2$ and $\ell_2(\mathbf u,x_{s+2})<\ell_1(\mathbf u,x_s)$.
 
Finally, put $x_0=h_2^0({\bf u},x_1)$. In view of Lemma~\ref{h_2^{k-1}}, $D(\mathbf u, x_0)=0$ and $\ell_1(\mathbf u,x_1)<\ell_1(\mathbf u,x_0)$. Since $\ell_1({\bf u},x_2)< \ell_1({\bf u},x_1)$, Lemma~\ref{if first then second} implies that $\ell_2(\mathbf u,x_2)<\ell_1(\mathbf u,x_0)$. Then
\begin{equation}
\label{u = long word 1}
{\bf u}={\bf u}'ab{\bf v}_{2k}x_{k-1}{\bf v}_{2k-1}b{\bf v}_{2k-2}x_{k-2}{\bf v}_{2k-3}x_{k-1}\cdots {\bf v}_4x_1{\bf v}_3x_2{\bf v}_2x_0{\bf v}_1x_1{\bf v}_0
\end{equation}
for some possibly empty words $\mathbf v_0,\mathbf v_1,\dots, \mathbf v_{2k}$. One can verify that if $2\le s\le k$ then
\begin{equation}
\label{ell_2(u,z)<ell_2(u,x_{s-1})}
\ell_2(\mathbf u,z)<\ell_2(\mathbf u,x_{s-1})\text{ for any }z\in \con({\bf v}_{2s}{\bf v}_{2s-1}).
\end{equation}
Put
$$
\mathbf w_s= \mathbf u'ab{\bf v}_{2k}x_{k-1}{\bf v}_{2k-1}b{\bf v}_{2k-2}x_{k-2}{\bf v}_{2k-3}x_{k-1}\cdots \mathbf v_{2s+2}x_s\mathbf v_{2s+1}x_{s+1}.
$$
The word $\mathbf w_s$ is the prefix of \textbf u that immediately precedes the word $\mathbf v_{2s}$, while the word $\mathbf v_{2s-1}$ precedes second occurrence of $x_{s-1}$ in \textbf u. This implies the required conclusion whenever $z\in\con(\mathbf w_s)$. Suppose now that $z\notin\con(\mathbf w_s)$. Then $\ell_1(\mathbf u,x_s)<\ell_1(\mathbf u,z)<\ell_2(\mathbf u,x_s)$. If $z$ is an \mbox{$(s-1)$}-divider of $\bf u$ then $z\in X_{s-1}$, whence $\ell_2(\mathbf u,z)< \ell_2(\mathbf u,x_{s-1})$ by the choice of the letter $x_{s-1}$. Otherwise $D(\mathbf u, z)>s-1$ by Lemma~\ref{k-divider and depth}. Then since $\ell_1(\mathbf u,z)<\ell_1(\mathbf u,x_{s-2})$, Lemma~\ref{if first then second} implies that $\ell_2(\mathbf u,z)<\ell_1(\mathbf u,x_{s-2})$, whence $\ell_2(\mathbf u,z)<\ell_2(\mathbf u,x_{s-1})$.
 
Further realization of the plan outlined at the beginning of Case~1 depends on the identity that plays the role of $\mathbf{p\approx q}$. Therefore, further considerations are divided into four subcases.
 
\smallskip
 
\emph{Subcase} 1.1: \textbf V satisfies the hypothesis of the claim~(i), i.e., $\delta_k^m$ holds in \textbf V, $a\in\con({\bf u}')$ and $D({\bf u},a)>m$. The claim~\eqref{a is not multiple in u'} allows us to assume that $a\in\simple(\mathbf u')$. Then $\mathbf u'={\bf w}a{\bf v}$ for some possibly empty words $\mathbf v$ and $\bf w$. This implies that
\begin{equation}
\label{u = long word 2}
\begin{array}{rl}
{\bf u}={}&{\bf w}a{\bf v}ab{\bf v}_{2k}x_{k-1}{\bf v}_{2k-1}b{\bf v}_{2k-2}x_{k-2}{\bf v}_{2k-3}x_{k-1}\cdots{\bf v}_4x_1{\bf v}_3x_2{\bf v}_2x_0\\
&\cdot\,{\bf v}_1x_1{\bf v}_0.
\end{array}
\end{equation}
Put $D(\mathbf u,a)=r$. Further considerations are divided into two parts corresponding to the cases when $r\le k+1$ and $r>k+1$.
 
\smallskip
 
A) $r\le k+1$. Here we need to define two more letters, namely $y_{r-1}$ and $y_{r-2}$ and clarify the location of these letters within \textbf u. Let $Y_{r-1}$ be the set of \mbox{$(r-1)$}-dividers $z$ of $\bf u$ such that $\ell_1(\mathbf u, a)<\ell_1(\mathbf u, z)<\ell_2(\mathbf u, a)$. The fact that $D(\mathbf u,a)=r$ implies that $h_1^{r-1}(\mathbf u,a)\ne h_2^{r-1}(\mathbf u,a)$, whence $h_2^{r-1}(\mathbf u,a)\in Y_{r-1}$. Therefore, the set $Y_{r-1}$ is non-empty. Lemma~\ref{h_2^{k-1}}(ii) implies that $D(\mathbf u, z)=r-1$ and $\ell_2({\bf u},a)< \ell_2({\bf u},z)$ for any $z\in Y_{r-1}$. Then $\ell_1({\bf u},b)< \ell_2({\bf u},z)$ for any $z\in Y_{r-1}$. Now we consider the letter $y_{r-1}\in Y_{r-1}$ such that $\ell_2({\bf u},z)\le \ell_2({\bf u},y_{r-1})$ for any $z\in Y_{r-1}$.
 
Now we check some additional properties of the letter $x_r$, which are fulfilled under certain restrictions to $r$. Suppose that $r<k+1$. Then the letter $x_r$ is defined. Our aim is to prove that
\begin{equation}
\label{ell_2(x_r)<ell_2(y_{r-1}}
\ell_2(\mathbf u, x_r)<\ell_2(\mathbf u, y_{r-1}).
\end{equation}
Put $y_{r-2}=h_2^{r-2}(\mathbf u, y_{r-1})$. Since $D(\mathbf u, y_{r-1})=r-1$, Lemma~\ref{h_2^{k-1}} implies that $D(\mathbf u, y_{r-2})=r-2$ and $\ell_1({\bf u},y_{r-1})< \ell_1({\bf u},y_{r-2})$. Recall that $\ell_1(\mathbf u,a)<\ell_1({\bf u},y_{r-1})$, whence $\ell_1(\mathbf u,a)<\ell_1({\bf u},y_{r-2})$. Since $D(\mathbf u,a)=r$, we can apply Lemma~\ref{if first then second} and conclude that $\ell_2({\bf u},a)< \ell_1({\bf u},y_{r-2})$. Second occurrence of $a$ in \textbf u immediately precedes first occurrence of $b=x_k$, whence $\ell_1({\bf u},x_k)< \ell_1({\bf u},y_{r-2})$. Then Lemma~\ref{if first then second} implies that $\ell_2({\bf u},x_k)< \ell_1({\bf u},y_{r-2})$. This implies that $\ell_1(\mathbf u,x_{k-1})<\ell_2({\bf u},x_k)< \ell_1({\bf u},y_{r-2})$. If $k-1\ge r$ then Lemma~\ref{if first then second} applies with the conclusion that $\ell_2({\bf u},x_{k-1})< \ell_1({\bf u},y_{r-2})$. Continuing this process, we eventually obtain $\ell_2(\mathbf u, x_r)<\ell_1(\mathbf u, y_{r-2})$. The choice of $y_{r-2}$ implies that first occurrence of $y_{r-2}$ in \textbf u precedes second occurrence of $y_{r-1}$. Therefore, $\ell_2(\mathbf u, x_r)<\ell_2(\mathbf u, y_{r-1})$. So, we have proved that if $r<k+1$ then the claim~\eqref{ell_2(x_r)<ell_2(y_{r-1}} is true.
 
Let now $r>2$. Note that
$$
\ell_1(\mathbf u,y_{r-1})<\ell_2(\mathbf u,a)<\ell_1(\mathbf u,b)=\ell_1(\mathbf u,x_k)<\ell_1(\mathbf u,x_{k-1})<\cdots<\ell_1(\mathbf u,x_{r-3}).
$$
If $\ell_1({\bf u},x_{r-3})<\ell_2({\bf u},y_{r-1})$ then the letter $x_{r-3}$ lies between the first and the second occurrences of $y_{r-1}$ in \textbf u. Since $x_{r-3}$ is an \mbox{$(r-3)$}-divider of \textbf u, we obtain a contradiction with the equality $D(\mathbf u, y_{r-1})=r-1$. Thus,
\begin{equation}
\label{ell_2(y_{r-1})<ell_1(x_{r-3})}
\ell_2(\mathbf u, y_{r-1})<\ell_1(\mathbf u, x_{r-3})
\end{equation}
whenever $r>2$.
 
One can return to arbitrary $r\le k+1$. This restriction on $r$ guarantees that the letters $x_{r-2}$ and $x_{r-1}$ are defined. There are three possibilities for second occurrence of the letter $y_{r-1}$ in \textbf u:
\begin{align}
\label{2nd occurrence y_{r-1} in u 1}&\ell_1(\mathbf u, x_{r-2})<\ell_2(\mathbf u, y_{r-1})<\ell_2(\mathbf u, x_{r-1});\\
\label{2nd occurrence y_{r-1} in u 2}&\ell_2(\mathbf u,y_{r-1})<\ell_1(\mathbf u, x_{r-2});\\
\label{2nd occurrence y_{r-1} in u 3}&\ell_2(\mathbf u, x_{r-1})<\ell_2(\mathbf u,y_{r-1}).
\end{align}
 
The equality~\eqref{u = long word 2} may be rewritten in the form
\begin{equation}
\label{u = long word 3}
\begin{array}{rl}
{\bf u}={}&{\bf w}a{\bf v}ab{\bf v}_{2k}x_{k-1}{\bf v}_{2k-1}b{\bf v}_{2k-2}x_{k-2}{\bf v}_{2k-3}x_{k-1}\cdots{\bf v}_{2r}\stackrel{(1)}{x_{r-1}}\\
&\cdot\,{\bf v}_{2r-1}\stackrel{(2)}{x_r}{\bf v}_{2r-2}\stackrel{(1)}{x_{r-2}}{\bf v}_{2r-3}\stackrel{(2)}{x_{r-1}}{\bf v}_{2r-4}\stackrel{(1)}{x_{r-3}}{\bf v}_{2r-5}\stackrel{(2)}{x_{r-2}}\cdots\\
&\cdot\,{\bf v}_4x_1{\bf v}_3x_2{\bf v}_2x_0{\bf v}_1x_1{\bf v}_0.
\end{array}
\end{equation}
Suppose that the claim~\eqref{2nd occurrence y_{r-1} in u 1} holds. Then second occurrence of $y_{r-1}$ in \textbf u belongs to the word ${\bf v}_{2r-3}$, whence ${\bf v}_{2r-3}={\bf v}_{2r-3}'y_{r-1}{\bf v}_{2r-3}''$ for possibly empty words ${\bf v}_{2r-3}'$ and ${\bf v}_{2r-3}''$. Further, since $\ell_1(\mathbf u, a)<\ell_1(\mathbf u, y_{r-1})<\ell_2(\mathbf u, a)$, first occurrence of $y_{r-1}$ belongs to \textbf v. Therefore, $\mathbf v = \mathbf v_{2k+2} y_{r-1}\mathbf v_{2k+1}$ for possibly empty words ${\bf v}_{2k+2}$ and ${\bf v}_{2k+1}$.
 
Combining all we above, we can clarify the presentation~\eqref{u = long word 2} of the word $\mathbf u$ and write this word in the form
\begin{align*}
{\bf u}={}&{\bf w}a{\bf v}_{2k+2}y_{r-1}{\bf v}_{2k+1}ab{\bf v}_{2k}x_{k-1}{\bf v}_{2k-1}b{\bf v}_{2k-2}x_{k-2}{\bf v}_{2k-3}x_{k-1}\cdots\\
&\cdot\,{\bf v}_{2r}x_{r-1}{\bf v}_{2r-1}x_r{\bf v}_{2r-2}x_{r-2}{\bf v}_{2r-3}'y_{r-1}{\bf v}_{2r-3}''x_{r-1}{\bf v}_{2r-4}x_{r-3}{\bf v}_{2r-5}x_{r-2}\cdots\\
&\cdot\,{\bf v}_4x_1{\bf v}_3x_2{\bf v}_2x_0{\bf v}_1x_1{\bf v}_0.
\end{align*}
Note that $\mathbf u'={\bf w}a{\bf v}_{2k+2}y_{r-1}{\bf v}_{2k+1}$ and
\begin{align*}
\mathbf u''={}&{\bf v}_{2k}x_{k-1}{\bf v}_{2k-1}b{\bf v}_{2k-2}x_{k-2}{\bf v}_{2k-3}x_{k-1}\cdots{\bf v}_{2r}x_{r-1}{\bf v}_{2r-1}x_r{\bf v}_{2r-2}x_{r-2}\\
&\cdot\,{\bf v}_{2r-3}'y_{r-1}{\bf v}_{2r-3}''x_{r-1}{\bf v}_{2r-4}x_{r-3}{\bf v}_{2r-5}x_{r-2}\cdots{\bf v}_4x_1{\bf v}_3x_2{\bf v}_2x_0{\bf v}_1x_1{\bf v}_0.
\end{align*}
 
Similarly to the proof of~\eqref{ell_2(u,z)<ell_2(u,x_{s-1})}, we can verify that if $z\in\con(\mathbf v_{2k+2}\mathbf v_{2k+1})$ then $\ell_2(\mathbf u,z)\le \ell_2(\mathbf u,y_{r-1})$.
 
Now we are ready to begin the process of modifying the word \textbf u to get the word $\mathbf u'ba\mathbf u''$. But first, we will outline the general scheme of further considerations, since arguments of that type will be repeated many times below. We are based on the fact that the identity~\eqref{xyxzx=xyxz} is satisfied by the variety \textbf K. This allows us to add any letter that is multiple in a given word to any place after second occurrence of this letter in the word. Using this, we will add different missing letters or even words in different places in \textbf u (or in a word which equals \textbf u in \textbf V) in order to make it possible to apply that word to the identity that is fulfilled in \textbf V at the moment (now such identity is $\delta_k^m$). After that, we will apply this identity, and then ``reverse the process'', i.e., based on~\eqref{xyxzx=xyxz}, remove unnecessary letters or even words from the resulting word to obtain the word $\mathbf u'ba\mathbf u''$. 
 
Let us proceed with the implementation of this plan. First, we apply the identity~\eqref{xyxzx=xyxz} to the word $\bf u$ and insert the letter $y_{r-1}$ after second occurrence of $x_{r-1}$ in $\bf u$. We obtain the identity
\begin{equation}
\label{u = long word 4}
\begin{array}{rl}
{\bf u}\approx&{\bf w}a{\bf v}_{2k+2}y_{r-1}{\bf v}_{2k+1}ab{\bf v}_{2k}x_{k-1}{\bf v}_{2k-1}b{\bf v}_{2k-2}x_{k-2}{\bf v}_{2k-3}x_{k-1}\cdots\\
&\cdot\,{\bf v}_{2r}\stackrel{(1)}{x_{r-1}}{\bf v}_{2r-1}x_r{\bf v}_{2r-2}x_{r-2}{\bf v}_{2r-3}\stackrel{(2)}{x_{r-1}}y_{r-1}{\bf v}_{2r-4}x_{r-3}\cdots\\
&\cdot\,{\bf v}_{2r-5}x_{r-2}{\bf v}_4x_1{\bf v}_3x_2{\bf v}_2x_0{\bf v}_1x_1{\bf v}_0.
\end{array}
\end{equation}
Further, we apply the identity~\eqref{xyxzx=xyxz} sufficiently many times to the right-hand side of the identity~\eqref{u = long word 4} and replace there third occurrence of $y_{r-1}$ with $\mathbf v_{2k+2}y_{r-1}\mathbf v_{2k+1}$ and second occurrence of $x_{s-1}$ with $\mathbf v_{2s}x_{s-1}\mathbf v_{2s-1}$ for any $2\le s\le k$. We have \textbf V satisfies the identity
\begin{equation}
\label{u = long word 5}
\mathbf u\approx {\bf w}a{\bf v}_{2k+2}y_{r-1}{\bf v}_{2k+1} ab\mathbf p\mathbf v_0
\end{equation}
where
\begin{align*}
\mathbf p={}&{\bf v}_{2k}x_{k-1}{\bf v}_{2k-1}b{\bf v}_{2k-2}x_{k-2}{\bf v}_{2k-3}{\bf v}_{2k}x_{k-1}{\bf v}_{2k-1}{\bf v}_{2k-4}\cdots{\bf v}_{2r-2}x_{r-2}{\bf v}_{2r-3}\\
&\cdot\,\mathbf v_{2r} x_{r-1}\mathbf v_{2r-1}\mathbf v_{2k+2}y_{r-1}\mathbf v_{2k+1}{\bf v}_{2r-4}x_{r-3}{\bf v}_{2r-5}\mathbf v_{2r-2}x_{r-2}\mathbf v_{2r-3}{\bf v}_{2r-6}\cdots\\
&\cdot\,{\bf v}_4x_1{\bf v}_3\mathbf v_6 x_2\mathbf v_5{\bf v}_2x_0{\bf v}_1\mathbf v_4x_1\mathbf v_3.
\end{align*}
By the hypothesis, $r=D(\mathbf u,a)>m$. Then by Lemma~\ref{between F_k and F_{k+1}}, $\delta_k^{r-1}$ holds in \textbf V. Now we perform the substitution
$$
(x_0,\dots,x_{k-1},x_k,y_{r-1},y_r)\mapsto(\mathbf v_2x_0\mathbf v_1,\dots,\mathbf v_{2k}x_{k-1}\mathbf v_{2k-1},b, \mathbf v_{2k+2}y_{r-1}\mathbf v_{2k+1},a)
$$
in this identity. Then we obtain the identity
$$
a\mathbf v_{2k+2}y_{r-1}\mathbf v_{2k+1}ab\mathbf p\approx a\mathbf v_{2k+2}y_{r-1}\mathbf v_{2k+1}ba\mathbf p.
$$
This identity together with~\eqref{u = long word 5} implies that \textbf V satisfies the identity
$$
{\bf u}\approx \mathbf wa\mathbf v_{2k+2}y_{r-1}\mathbf v_{2k+1}ba\mathbf p\mathbf v_0.
$$
Now we apply the identity~\eqref{xyxzx=xyxz} to the right-hand side of the last identity ``in the opposite direction'' and replace the subword $\mathbf v_{2k+2}y_{r-1}\mathbf v_{2k+1}$ with $y_{r-1}$ and the subword $\mathbf v_{2s}x_{s-1}\mathbf v_{2s-1}$ with $x_{s-1}$ for any $2\le s\le k$. As a result, we obtain the identity
\begin{align*}
{\bf u}\approx{}&{\bf w}a{\bf v}_{2k+2}y_{r-1}{\bf v}_{2k+1}ba{\bf v}_{2k}x_{k-1}{\bf v}_{2k-1}b{\bf v}_{2k-2}x_{k-2}{\bf v}_{2k-3}x_{k-1}\cdots{\bf v}_{2r-2}x_{r-2}\\
&\cdot\,{\bf v}_{2r-3}x_{r-1}y_{r-1}{\bf v}_{2r-4}x_{r-3}{\bf v}_{2r-5}x_{r-2}{\bf v}_{2r-6}\cdots{\bf v}_4x_1{\bf v}_3x_2{\bf v}_2x_0{\bf v}_1x_1{\bf v}_0.
\end{align*}
Finally, we apply the identity~\eqref{xyxzx=xyxz} to the right-hand side of the last identity and delete third occurrence $y_{r-1}$. We obtain the identity
\begin{align*}
{\bf u}\approx{}&{\bf w}a{\bf v}_{2k+2}y_{r-1}{\bf v}_{2k+1}ba{\bf v}_{2k}x_{k-1}{\bf v}_{2k-1}b{\bf v}_{2k-2}x_{k-2}{\bf v}_{2k-3}x_{k-1}\cdots{\bf v}_{2r-2}x_{r-2}\\
&\cdot\,{\bf v}_{2r-3}'y_{r-1}{\bf v}_{2r-3}''x_{r-1}{\bf v}_{2r-4}x_{r-3}{\bf v}_{2r-5}x_{r-2}{\bf v}_{2r-6}\cdots{\bf v}_4x_1{\bf v}_3x_2{\bf v}_2\\
&\cdot\,x_0{\bf v}_1x_1{\bf v}_0\\
={}&{\bf u}'ba{\bf u}''.
\end{align*}
 
It remains to consider the case when either~\eqref{2nd occurrence y_{r-1} in u 2} or~\eqref{2nd occurrence y_{r-1} in u 3} holds. We are going to verify that in both the cases the identity~\eqref{u = long word 4} holds. It suffices to our aim because then we can complete considerations in the same arguments as above. If the claim~\eqref{2nd occurrence y_{r-1} in u 2} holds then~\eqref{ell_2(x_r)<ell_2(y_{r-1}} and~\eqref{u = long word 3} imply that the word \textbf u has the form
\begin{align*}
{\bf u}\!={}\!&{\bf w}a{\bf v}_{2k+2}y_{r-1}{\bf v}_{2k+1}ab{\bf v}_{2k}x_{k-1}{\bf v}_{2k-1}b{\bf v}_{2k-2}x_{k-2}{\bf v}_{2k-3}x_{k-1}\cdots{\bf v}_{2r}\stackrel{(1)}{x_{r-1}}{\bf v}_{2r-1}\\[-3pt]
&\cdot\,x_r{\bf v}_{2r-2}'y_{r-1}{\bf v}_{2r-2}''x_{r-2}{\bf v}_{2r-3}\stackrel{(2)}{x_{r-1}}{\bf v}_{2r-4}x_{r-3}{\bf v}_{2r-5}x_{r-2}\cdots{\bf v}_4x_1{\bf v}_3x_2\\
&\cdot\,{\bf v}_2x_0{\bf v}_1x_1{\bf v}_0
\end{align*}
for some possibly empty words ${\bf v}_{2r-2}',{\bf v}_{2r-2}''$ such that ${\bf v}_{2r-2}={\bf v}_{2r-2}'y_{r-1}{\bf v}_{2r-2}''$. Here we add one more occurrence of the letter $y_{r-1}$ immediately after second occurrence of $x_{r-1}$. As a result, we obtain the identity~\eqref{u = long word 4}. Finally, if~\eqref{2nd occurrence y_{r-1} in u 3} is the case then we use~\eqref{ell_2(y_{r-1})<ell_1(x_{r-3})}. The word \textbf u here has the form
\begin{align*}
{\bf u}={}&{\bf w}a{\bf v}_{2k+2}\stackrel{(1)}{y_{r-1}}{\bf v}_{2k+1}ab{\bf v}_{2k}x_{k-1}{\bf v}_{2k-1}b{\bf v}_{2k-2}x_{k-2}{\bf v}_{2k-3}x_{k-1}\cdots{\bf v}_{2r}x_{r-1}{\bf v}_{2r-1}\\[-3pt]
&\cdot\,x_r{\bf v}_{2r-2}x_{r-2}{\bf v}_{2r-3}x_{r-1}{\bf v}_{2r-4}'\stackrel{(2)}{y_{r-1}}{\bf v}_{2r-4}''\stackrel{(1)}{x_{r-3}}{\bf v}_{2r-5}x_{r-2}\cdots{\bf v}_4x_1{\bf v}_3x_2\\
&\cdot\,{\bf v}_2x_0{\bf v}_1x_1{\bf v}_0
\end{align*}
for some possibly empty words ${\bf v}_{2r-4}',{\bf v}_{2r-4}''$ such that ${\bf v}_{2r-4}={\bf v}_{2r-4}'y_{r-1}{\bf v}_{2r-4}''$. Then we can add third occurrence of the letter $x_{r-1}$ immediately before second occurrence of $y_{r-1}$ and obtain the identity
\begin{align*}
{\bf u}\approx{}&{\bf w}a{\bf v}_{2k+2}\stackrel{(1)}{y_{r-1}}{\bf v}_{2k+1}ab{\bf v}_{2k}x_{k-1}{\bf v}_{2k-1}b{\bf v}_{2k-2}x_{k-2}{\bf v}_{2k-3}x_{k-1}\cdots{\bf v}_{2r}\stackrel{(1)}{x_{r-1}}\\[-3pt]
&\cdot\,{\bf v}_{2r-1}x_r{\bf v}_{2r-2}x_{r-2}{\bf v}_{2r-3}\stackrel{(2)}{x_{r-1}}{\bf v}_{2r-4}'\stackrel{(3)}{x_{r-1}}\stackrel{(2)}{y_{r-1}}{\bf v}_{2r-4}''\stackrel{(1)}{x_{r-3}}{\bf v}_{2r-5}x_{r-2}\cdots\\
&\cdot\,{\bf v}_4x_1{\bf v}_3x_2{\bf v}_2x_0{\bf v}_1x_1{\bf v}_0.
\end{align*}
The last identity is nothing but~\eqref{u = long word 4} (up to renaming of $\mathbf v_{2r-3}x_{r-1}\mathbf v_{2r-4}'$ to $\mathbf v_{2r-3}$, and $\mathbf v_{2r-4}''$ to $\mathbf v_{2r-4}$).
 
\smallskip
 
B) $r>k+1$. Recall that the equality~\eqref{u = long word 2} is true. Suppose that the word ${\bf v}$ is non-empty. Let $y\in\con(\mathbf v)$. Suppose that $\ell_1({\bf u},x_{k-1})<\ell_2({\bf u},y)$. This implies that $h_1^{k-1}(\mathbf u, y)\ne h_2^{k-1}(\mathbf u, y)$ because $x_{k-1}$ is a \mbox{$(k-1)$}-divider of \textbf u. Then $y$ is a $k$-divider of $\bf u$. Since \textbf v (and, in particular, $y$) is located between the first and the second occurrences of $a$ in \textbf u, this contradicts the fact that $D(\mathbf u,a)=r>k+1$. So, $\ell_2({\bf u},y)\le \ell_1({\bf u},x_{k-1})$ for any $y\in\con(\mathbf v)$. Then we apply the identity~\eqref{xyxzx=xyxz} sufficiently many times to the right-hand side of the identity~\eqref{u = long word 2}, namely, we insert the word $\bf v$ after second occurrence of $b$ there. Clearly, we can formally insert the word $\bf v$ after second occurrence of $b$ whenever $\mathbf v=\lambda$ too. Further, in view of~\eqref{ell_2(u,z)<ell_2(u,x_{s-1})}, we can replace second occurrence of $x_{s-1}$ in the right-hand side of~\eqref{u = long word 2} with the word $\mathbf v_{2s}x_{s-1}\mathbf v_{2s-1}$ for any $2\le s\le k$. We have \textbf V satisfies the identity
\begin{equation}
\label{u = long word 6}
{\bf u}\approx{\bf w}a{\bf v}ab{\bf p}{\bf v}_0
\end{equation}
where
\begin{align*}
\mathbf p={}&{\bf v}_{2k}x_{k-1}{\bf v}_{2k-1}b{\bf v}{\bf v}_{2k-2}x_{k-2}{\bf v}_{2k-3}{\bf v}_{2k}x_{k-1}{\bf v}_{2k-1}\cdots{\bf v}_4x_1{\bf v}_3{\bf v}_6x_2{\bf v}_5{\bf v}_2x_0{\bf v}_1\\
&\cdot\,{\bf v}_4x_1{\bf v}_3.
\end{align*}
In view of Lemma~\ref{between F_k and F_{k+1}}, \textbf V satisfies the identity $\delta_k^k$. Now we perform the substitution
$$
(x_0,\dots,x_{k-1},x_k,y_k,y_{k+1})\mapsto(\mathbf v_2x_0\mathbf v_1,\dots,\mathbf v_{2k}x_{k-1}\mathbf v_{2k-1},b, \mathbf v,a)
$$
in this identity. Then we obtain the identity
$$
a\mathbf vab\mathbf p\approx a\mathbf vba\mathbf p.
$$
This identity together with~\eqref{u = long word 6} implies that \textbf V satisfies the identity
$$
{\bf u}\approx \mathbf wa\mathbf vba\mathbf p\mathbf v_0.
$$
Now we apply the identity~\eqref{xyxzx=xyxz} to the right-hand side of the last identity ``in the opposite direction'', namely delete the word $\bf v$ after second occurrence of $b$ and replace the subword $\mathbf v_{2s}x_{s-1}\mathbf v_{2s-1}$ with $x_{s-1}$ for any $2\le s\le k$. As a result, we obtain the identity
\begin{align*}
{\bf u}\approx{}&{\bf w}a{\bf v}ba{\bf v}_{2k}x_{k-1}{\bf v}_{2k-1}b{\bf v}_{2k-2}x_{k-2}{\bf v}_{2k-3}x_{k-1}\cdots{\bf v}_4x_1{\bf v}_3x_2{\bf v}_2x_0{\bf v}_1x_1{\bf v}_0\\
={}&{\bf u}'ba{\bf u}''.
\end{align*}
 
\smallskip
 
\emph{Subcase} 1.2: \textbf V satisfies the hypothesis of the claim~(ii), i.e., $\gamma_k$ holds in \textbf V and $a\in\con(\mathbf u')$. Recall that the equality~\eqref{u = long word 1} is true. The claim~\eqref{a is not multiple in u'} allows us to assume that $a\in\simple(\mathbf u')$. Then, as well as in Subcase~1.1, the word \textbf u has the form~\eqref{u = long word 2}. Note that $\mathbf u'={\bf w}a{\bf v}$ and
$$
{\bf u}''={\bf v}_{2k}x_{k-1}{\bf v}_{2k-1}b{\bf v}_{2k-2}x_{k-2}{\bf v}_{2k-3}x_{k-1}\cdots{\bf v}_4x_1{\bf v}_3x_2{\bf v}_2x_0{\bf v}_1x_1{\bf v}_0.
$$
 
Recall that the claim~\eqref{ell_2(u,z)<ell_2(u,x_{s-1})} is true for any $2\le s\le k$. Now we can apply the identity~\eqref{xyxzx=xyxz} sufficiently many times to the right-hand side of the identity~\eqref{u = long word 2} and replace second occurrence of $x_{s-1}$ with the word $\mathbf v_{2s}x_{s-1}\mathbf v_{2s-1}$ for any $2\le s\le k$. We have \textbf V satisfies the identity
\begin{align*}
{\bf u}\approx{}&{\bf w}a{\bf v}ab{\bf v}_{2k}x_{k-1}{\bf v}_{2k-1}b{\bf v}_{2k-2}x_{k-2}{\bf v}_{2k-3}{\bf v}_{2k}x_{k-1}{\bf v}_{2k-1}\cdots\\
&\cdot\,{\bf v}_4x_1{\bf v}_3{\bf v}_6x_2{\bf v}_5{\bf v}_2x_0{\bf v}_1{\bf v}_4x_1{\bf v}_3{\bf v}_0.
\end{align*}
Put $\mathbf p_1=a{\bf v}$ and
\begin{align*}
{\bf p}_2={}&{\bf v}_{2k}x_{k-1}{\bf v}_{2k-1}b{\bf v}_{2k-2}x_{k-2}{\bf v}_{2k-3}{\bf v}_{2k}x_{k-1}{\bf v}_{2k-1}\cdots{\bf v}_4x_1{\bf v}_3{\bf v}_6x_2{\bf v}_5{\bf v}_2x_0{\bf v}_1\\
&\cdot\,{\bf v}_4x_1{\bf v}_3.
\end{align*}
Then the last identity has the form
\begin{equation}
\label{u = long word 7}
{\bf u}\approx{\bf w}\mathbf p_1ab\mathbf p_2{\bf v}_0.
\end{equation}
By the hypothesis, \textbf V satisfies the identity $\gamma_k$. Now we perform the substitution
$$
(x_0,x_1,\dots, x_{k-1},x_k,y_0,y_1)\mapsto(\mathbf v_2x_0\mathbf v_1,\mathbf v_4x_1\mathbf v_3,\dots,\mathbf v_{2k}x_{k-1}\mathbf v_{2k-1},b,{\bf v},a)
$$
in this identity. Then we obtain the identity $\mathbf p_1ba\mathbf p_2\approx \mathbf p_1ab\mathbf p_2$. This identity together with~\eqref{u = long word 7} implies that \textbf V satisfies the identity ${\bf u}\approx {\bf w}\mathbf p_1ba\mathbf p_2{\bf v}_0$, i.e., the identity
\begin{align*}
{\bf u}\approx{}&{\bf w}a{\bf v}ba{\bf v}_{2k}x_{k-1}{\bf v}_{2k-1}b{\bf v}_{2k-2}x_{k-2}{\bf v}_{2k-3}{\bf v}_{2k}x_{k-1}{\bf v}_{2k-1}\cdots\\
&\cdot\,{\bf v}_4x_1{\bf v}_3{\bf v}_6x_2{\bf v}_5{\bf v}_2x_0{\bf v}_1{\bf v}_4x_1{\bf v}_3{\bf v}_0.
\end{align*}
Now we apply the identity~\eqref{xyxzx=xyxz} to the right-hand side of the last identity ``in the opposite direction'' and replace the subword $\mathbf v_{2s}x_{s-1}\mathbf v_{2s-1}$ with $x_{s-1}$ for any $2\le s\le k$. As a result, we obtain the identity
\begin{align*}
{\bf u}\approx{}&{\bf w}a{\bf v}ba{\bf v}_{2k}x_{k-1}{\bf v}_{2k-1}b{\bf v}_{2k-2}x_{k-2}{\bf v}_{2k-3}x_{k-1}\cdots{\bf v}_4x_1{\bf v}_3x_2{\bf v}_2x_0{\bf v}_1x_1{\bf v}_0\\
={}&{\bf u}'ba{\bf u}''.
\end{align*}
 
\smallskip
 
\emph{Subcase} 1.3: \textbf V satisfies the hypothesis of the claim~(iii), i.e., $\beta_k$ holds in \textbf V and $D(\mathbf u,a)\ne D(\mathbf u,b)$. Subcase~1.2 allows us to assume that $a\notin\con(\mathbf u')$. This fact and~\eqref{b notin u'} immediately imply that $\ell_1(\mathbf u,a)<\ell_1(\mathbf u,b)$. If $D(\mathbf u, a)\le k-1$ then $a$ is a \mbox{$(k-1)$}-divider of $\bf u$ by Lemma~\ref{k-divider and depth}. But this is not the case. Therefore, $D(\mathbf u, a)\ge k$. Since $D(\mathbf u,b)\ne D(\mathbf u,a)$ and $D(\mathbf u, b)=k$, we obtain $D(\mathbf u,a)>k$.
 
Note that $\ell_2(\mathbf u,a)<\ell_1({\bf u},x_{k-1})$ because $h_1^{k-1}(\mathbf u,a)=h_2^{k-1}(\mathbf u,a)$ and $x_{k-1}$ is a \mbox{$(k-1)$}-divider. Recall that the equality~\eqref{u = long word 1} is true. Then $\mathbf v_{2k}=\mathbf v_{2k}'a\mathbf v_{2k}''$ for some possibly empty words $\mathbf v_{2k}',\mathbf v_{2k}''$. Thus,
$$
{\bf u}={\bf u}'ab{\bf v}_{2k}'a{\bf v}_{2k}''x_{k-1}{\bf v}_{2k-1}b{\bf v}_{2k-2}x_{k-2}{\bf v}_{2k-3}x_{k-1}\cdots{\bf v}_4x_1{\bf v}_3x_2{\bf v}_2x_0{\bf v}_1x_1{\bf v}_0.
$$
Now we are going to verify that the identity
\begin{equation}
\label{u = long word 8}
{\bf u}\approx{\bf u}'aba{\bf v}_{2k}x_{k-1}{\bf v}_{2k-1}b{\bf v}_{2k-2}x_{k-2}{\bf v}_{2k-3}x_{k-1}\cdots{\bf v}_4x_1{\bf v}_3x_2{\bf v}_2x_0{\bf v}_1x_1{\bf v}_0
\end{equation}
holds in $\mathbf V$. This statement is evident whenever ${\bf v}_{2k}'=\lambda$. Suppose now that ${\bf v}_{2k}'=\mathbf v^\ast d$ for some possibly empty word $\mathbf v^\ast$ and some letter $d$. Then \textbf u may be rewritten in the form
$$
{\bf u}={\bf u}'\stackrel{(1)}ab{\bf v}^\ast d\stackrel{(2)}a{\bf v}_{2k}''x_{k-1}{\bf v}_{2k-1}b{\bf v}_{2k-2}x_{k-2}{\bf v}_{2k-3}x_{k-1}\cdots{\bf v}_4x_1{\bf v}_3x_2{\bf v}_2x_0{\bf v}_1x_1{\bf v}_0.
$$
Note that the subword $da$ located between $\mathbf v^\ast$ and ${\bf v}_{2k}''$ lies in some \mbox{$(k-1)$}-block of \textbf u. Indeed, the occurrence of $d$ in this subword is not a \mbox{$(k-1)$}-divider of \textbf u because otherwise the first and the second occurrences of $a$ in \textbf u lies in different \mbox{$(k-1)$}-blocks, contradicting the inequality $D(\mathbf u,a)>k$, while the occurrence of $a$ in this subword is not a \mbox{$(k-1)$}-divider of \textbf u because this is not first occurrence of $a$ in \textbf u.
 
According to Lemma~\ref{between F_k and F_{k+1}}, the variety $\bf V$ satisfies the identity $\gamma_k$. In view of the statement that was proved in Subcase 1.2, $\bf V$ satisfies the identity
$$
{\bf u}\approx{\bf u}'ab{\bf v}^\ast ad{\bf v}_{2k}''x_{k-1}{\bf v}_{2k-1}b{\bf v}_{2k-2}x_{k-2}{\bf v}_{2k-3}x_{k-1}\cdots{\bf v}_4x_1{\bf v}_3x_2{\bf v}_2x_0{\bf v}_1x_1{\bf v}_0.
$$
Acting in this way, we can successively swap the letter $a$ with all letters of the word $\mathbf v_{2k}'$ and obtain
$$
{\bf u}\approx{\bf u}'aba{\bf v}_{2k}'{\bf v}_{2k}''x_{k-1}{\bf v}_{2k-1}b{\bf v}_{2k-2}x_{k-2}{\bf v}_{2k-3}x_{k-1}\cdots{\bf v}_4x_1{\bf v}_3x_2{\bf v}_2x_0{\bf v}_1x_1{\bf v}_0
$$
holds in $\bf V$. Now we apply the identity~\eqref{xyxzx=xyxz} to the right-hand side of the last identity and insert the letter $a$ after the word ${\bf v}_{2k}'$. We obtain the identity~\eqref{u = long word 8}.
 
Recall that~\eqref{ell_2(u,z)<ell_2(u,x_{s-1})} is true for any $2\le s\le k$. Now we can apply the identity~\eqref{xyxzx=xyxz} sufficiently many times to the right-hand side of the identity~\eqref{u = long word 8} and replace second occurrence of $x_{s-1}$ with the word $\mathbf v_{2s}x_{s-1}\mathbf v_{2s-1}$ for any $2\le s\le k$. We have \textbf V satisfies the identity
\begin{equation}
\label{u = long word 9}
{\bf u}\approx{\bf u}'aba\mathbf p{\bf v}_0
\end{equation}
where
\begin{align*}
{\bf p}={}&{\bf v}_{2k}x_{k-1}{\bf v}_{2k-1}b{\bf v}_{2k-2}x_{k-2}{\bf v}_{2k-3}{\bf v}_{2k}x_{k-1}{\bf v}_{2k-1}\cdots{\bf v}_4x_1{\bf v}_3{\bf v}_6x_2{\bf v}_5{\bf v}_2x_0{\bf v}_1\\
&\cdot\,{\bf v}_4x_1{\bf v}_3.
\end{align*}
 
Now we perform the substitution
$$
(x_0,x_1,\dots,x_{k-1},x_k,x)\mapsto({\bf v}_2x_0{\bf v}_1,{\bf v}_4x_1{\bf v}_3,\dots,{\bf v}_{2k}x_{k-1}{\bf v}_{2k-1},b,a)
$$
in the identity $\beta_k$. Then we obtain the identity $aba\mathbf p\approx ba^2\mathbf p$. One can apply this identity to the identity~\eqref{u = long word 9}. We get that the identity
\begin{align*}
{\bf u}\approx{\bf u}'ba^2\mathbf p{\bf v}_0={}&{\bf u}'ba^2{\bf v}_{2k}x_{k-1}{\bf v}_{2k-1}b{\bf v}_{2k-2}x_{k-2}{\bf v}_{2k-3}{\bf v}_{2k}x_{k-1}{\bf v}_{2k-1}\cdots\\
&\cdot\,{\bf v}_4x_1{\bf v}_3{\bf v}_6x_2{\bf v}_5{\bf v}_2x_0{\bf v}_1{\bf v}_4x_1{\bf v}_3{\bf v}_0
\end{align*}
holds in $\bf V$. Now we apply the identity~\eqref{xyxzx=xyxz} to the right-hand side of the last identity ``in the opposite direction'' and replace the subword $\mathbf v_{2s}x_{s-1}\mathbf v_{2s-1}$ with $x_{s-1}$ for any $2\le s\le k$. As a result, we obtain the identity
$$
{\bf u}\approx{\bf u}'ba^2{\bf v}_{2k}x_{k-1}{\bf v}_{2k-1}b{\bf v}_{2k-2}x_{k-2}{\bf v}_{2k-3}x_{k-1}\cdots{\bf v}_4x_1{\bf v}_3x_2{\bf v}_2x_0{\bf v}_1x_1{\bf v}_0.
$$
Repeating arguments used above in the deduction of the identity~\eqref{u = long word 8}, we obtain \textbf V satisfies the identity
\begin{align*}
{\bf u}\approx{}&{\bf u}'ba{\bf v}_{2k}x_{k-1}{\bf v}_{2k-1}b{\bf v}_{2k-2}x_{k-2}{\bf v}_{2k-3}x_{k-1}\cdots{\bf v}_4x_1{\bf v}_3x_2{\bf v}_2x_0{\bf v}_1x_1{\bf v}_0\!=\!{\bf u}'ba{\bf u}''\\
={}&\mathbf u'ba\mathbf u''.
\end{align*}
 
\smallskip
 
\emph{Subcase} 1.4: \textbf V satisfies the hypothesis of the claim~(iv), i.e., $\alpha_k$ holds in \textbf V. By Subcases~1.2 and~1.3 and the claim~\eqref{b notin u'}, we can assume that $a,b\notin\con(\mathbf u')$ and $D(\mathbf u,b)=D(\mathbf u,a)$. Recall that the equality~\eqref{u = long word 1} is true.
 
Note that $\ell_2(\mathbf u,a)<\ell_1({\bf u},x_{k-2})$ because $h_1^{k-2}(\mathbf u,a)=h_2^{k-2}(\mathbf u,a)$ and $x_{k-2}$ is a \mbox{$(k-2)$}-divider. Therefore, there are possibly empty words $\mathbf v'$ and $\mathbf v''$ such that one of the following equalities holds:
$$
{\bf v}_{2k}={\bf v}'a{\bf v}'',\ {\bf v}_{2k-1}={\bf v}'a{\bf v}'' \text{ or } {\bf v}_{2k-2}={\bf v}'a{\bf v}''.
$$
Then one of the following equalities holds:
\begin{align*}
{\bf u}={}&{\bf u}'ab{\bf v}'a{\bf v}''x_{k-1}{\bf v}_{2k-1}b{\bf v}_{2k-2}x_{k-2}{\bf v}_{2k-3}x_{k-1}\cdots{\bf v}_4x_1{\bf v}_3x_2{\bf v}_2x_0{\bf v}_1x_1{\bf v}_0,\\
{\bf u}={}&{\bf u}'ab{\bf v}_{2k}x_{k-1}{\bf v}'a{\bf v}''b{\bf v}_{2k-2}x_{k-2}{\bf v}_{2k-3}x_{k-1}\cdots{\bf v}_4x_1{\bf v}_3x_2{\bf v}_2x_0{\bf v}_1x_1{\bf v}_0,\\
{\bf u}={}&{\bf u}'ab{\bf v}_{2k}x_{k-1}{\bf v}_{2k-1}b{\bf v}'a{\bf v}''x_{k-2}{\bf v}_{2k-3}x_{k-1}\cdots{\bf v}_4x_1{\bf v}_3x_2{\bf v}_2x_0{\bf v}_1x_1{\bf v}_0.
\end{align*}
We consider only the first case. Two other cases can be considered similarly. Since the variety $\bf V$ satisfies the identity~\eqref{xyxzx=xyxz}, we obtain the identity
\begin{equation}
\label{u = long word 10}
{\bf u}\approx{\bf u}'ab{\bf v}_{2k}x_{k-1}{\bf v}_{2k-1}ab{\bf v}_{2k-2}x_{k-2}{\bf v}_{2k-3}x_{k-1}\cdots{\bf v}_4x_1{\bf v}_3x_2{\bf v}_2x_0{\bf v}_1x_1{\bf v}_0
\end{equation}
holds in this variety.
 
Recall that the claim~\eqref{ell_2(u,z)<ell_2(u,x_{s-1})} is true for any $2\le s\le k$. Now we can apply the identity~\eqref{xyxzx=xyxz} sufficiently many times to the right-hand side of the identity~\eqref{u = long word 10} and replace second occurrence of $x_{s-1}$ with the word $\mathbf v_{2s}x_{s-1}\mathbf v_{2s-1}$ for any $2\le s\le k$. We have \textbf V satisfies the identity
\begin{equation}
\label{u = long word 11}
{\bf u}\approx{\bf u'}ab\mathbf p{\bf v}_0
\end{equation}
where
\begin{align*}
\mathbf p={}&{\bf v}_{2k}x_{k-1}{\bf v}_{2k-1}ab{\bf v}_{2k-2}x_{k-2}{\bf v}_{2k-3}{\bf v}_{2k}x_{k-1}{\bf v}_{2k-1}\cdots{\bf v}_4x_1{\bf v}_3{\bf v}_6x_2{\bf v}_5{\bf v}_2x_0{\bf v}_1\\
&\cdot\,{\bf v}_4x_1{\bf v}_3.
\end{align*}
Now we perform the substitution
$$
(x_0,x_1,\dots,x_{k-1},x_k,y_k)\mapsto(\mathbf v_2x_0\mathbf v_1,\mathbf v_4x_1\mathbf v_3,\dots,\mathbf v_{2k}x_{k-1}\mathbf v_{2k-1},a,b)
$$
in the identity $\alpha_k$. Then we obtain the identity $ab\mathbf p\approx ba\mathbf p$. Let us apply this identity to the identity~\eqref{u = long word 11}. We get that the identity ${\bf u}\approx {\bf u'}ba\mathbf p{\bf v}_0$ holds in $\bf V$. Now we apply the identity~\eqref{xyxzx=xyxz} to the right-hand side of the last identity ``in the opposite direction'' and replace the subword $\mathbf v_{2s}x_{s-1}\mathbf v_{2s-1}$ with $x_{s-1}$ for any $2\le s\le k$. As a result, we obtain the identity
$$
{\bf u}\approx{\bf u}'ba{\bf v}_{2k}x_{k-1}{\bf v}_{2k-1}ab{\bf v}_{2k-2}x_{k-2}{\bf v}_{2k-3}x_{k-1}\cdots{\bf v}_4x_1{\bf v}_3x_2{\bf v}_2x_0{\bf v}_1x_1{\bf v}_0.
$$
 
Now we apply the identity~\eqref{xyxzx=xyxz} again and delete the occurrence of $a$ located between $\mathbf v_{2k-1}$ and second occurrence of $b$ in the right-hand side of the last identity. We obtain \textbf V satisfies the identity
\begin{align*}
{\bf u}\approx{}&{\bf u}'ba{\bf v}'a{\bf v}''x_{k-1}{\bf v}_{2k-1}b{\bf v}_{2k-2}x_{k-2}{\bf v}_{2k-3}x_{k-1}\cdots{\bf v}_4x_1{\bf v}_3x_2{\bf v}_2x_0{\bf v}_1x_1{\bf v}_0\\
={}&{\bf u}'ba{\bf u}''.
\end{align*}
 
\medskip
 
\emph{Case} 2: $k<D(\mathbf u, b)<\infty$. As we will see below, this case reduces to the previous one by relatively simple arguments. Put $D(\mathbf u,b)=r$. Further considerations are divided into three subcases.
 
\smallskip
 
\emph{Subcase} 2.1: \textbf V satisfies the hypothesis of one of the claims~(i) and~(ii). Here $a\in\con(\mathbf u')$. Hence the occurrence of the letter $a$ in the subword $ab$ of the word \textbf u mentioned in the formulation of the lemma is not first occurrence of $a$ in \textbf u. Therefore, this occurrence of $a$ in \textbf u is not an \mbox{$(r-1)$}-divider of \textbf u. Lemma~\ref{k-divider and depth} together with the fact that $D(\mathbf u, b)=r$ implies that the occurrence of the letter $b$ in the same subword $ab$ of the word \textbf u also is not an \mbox{$(r-1)$}-divider of \textbf u. Therefore, the above-mentioned subword $ab$ of the word \textbf u lies in some \mbox{$(r-1)$}-block of \textbf u.
 
Let
\begin{equation}
\label{s_0w_0s_1w_1 ... s_nw_n}
s_0\mathbf w_0 s_1 \mathbf w_1 \cdots s_n \mathbf w_n
\end{equation}
be the \mbox{$(r-1)$}-decomposition of $\bf u$. Then there exists a number $0\le j\le n$ such that $\mathbf w_j = \mathbf w_j'ab\mathbf w_j''$, whence
$$
\mathbf u' = s_0\mathbf w_0s_1\mathbf w_1\cdots s_j\mathbf w_j'\text{ and }\mathbf u''= \mathbf w_j''s_{j+1}\mathbf w_{j+1}\cdots s_n\mathbf w_n.
$$
Since $\mathbf J_k^m\subseteq\mathbf J_r^m$ and $\mathbf I_k\subseteq\mathbf I_r$ by Lemma~\ref{between F_k and F_{k+1}}, we apply the statements that proved in Subcases~1.1 and~1.2 and obtain the required conclusion that the identity $\mathbf u\approx \mathbf u'ba\mathbf u''$ holds in $\bf V$.
 
\smallskip
 
\emph{Subcase} 2.2: \textbf V satisfies the hypothesis of the claim~(iii), i.e., $\beta_k$ holds in \textbf V and $D(\mathbf u,a)\ne D(\mathbf u,b)$. Subcase~2.1 allows us to assume that $a\notin\con(\mathbf u')$.
 
Suppose that $D(\mathbf u,a)=s<r$. If $s\le k-1$ then $a$ is a \mbox{$(k-1)$}-divider of $\bf u$ by Lemma~\ref{k-divider and depth}. But this is not the case because first occurrence of $a$ in \textbf u lies in the \mbox{$(k-1)$}-block $\mathbf u_i$. Therefore, $s\ge k$. Let~\eqref{s_0w_0s_1w_1 ... s_nw_n} be the \mbox{$(s-1)$}-decomposition of $\bf u$. Then there exists a number $0\le j\le n$ such that $\mathbf w_j = \mathbf w_j'ab\mathbf w_j''$, $\mathbf u' = s_0\mathbf w_0s_1\mathbf w_1\cdots s_j\mathbf w_j'$ and $\mathbf u''= \mathbf w_j''s_{j+1}\mathbf w_{j+1}\cdots s_n\mathbf w_n$. Put $\mathbf u^\ast=\mathbf u'ba\mathbf u''$. Since $a,b\notin\{s_1,s_2,\dots,s_n\}$, the \mbox{$(s-1)$}-decomposition of $\bf u^\ast$ has the form
$$
s_0\mathbf w_0 s_1 \mathbf w_1 \cdots s_j \mathbf w_j^\ast \cdots s_n \mathbf w_n
$$
where $\mathbf w_j^\ast = \mathbf w_j'ba\mathbf w_j''$. Then the claims~\eqref{sim(u)=sim(v) & mul(u)=mul(v)} and~\eqref{eq the same l-dividers} with $\mathbf v = \mathbf u^\ast$ and $\ell=s$ are true. Now Lemma~\ref{D(u,x)=k iff D(v,x)=k} applies with the conclusion that $D(\mathbf u^\ast, a)=s$. Since $\bf V$ satisfies the identity $\beta_s$ by Lemma~\ref{between F_k and F_{k+1}}, we apply the statement that proved in Subcase~1.3 and obtain the identity $\mathbf u^\ast= \mathbf u'ba\mathbf u''\approx \mathbf u'ab\mathbf u''=\mathbf u$ holds in $\bf V$.
 
Suppose now that $D(\mathbf u,a)>r$. Let now~\eqref{s_0w_0s_1w_1 ... s_nw_n} be the \mbox{$(r-1)$}-decomposition of $\bf u$. Then there exists a number $0\le j\le n$ such that $\mathbf w_j = \mathbf w_j'ab\mathbf w_j''$, whence $\mathbf u' = s_0\mathbf w_0s_1\mathbf w_1\cdots s_j\mathbf w_j'$ and $\mathbf u''= \mathbf w_j''s_{j+1}\mathbf w_{j+1}\cdots s_n\mathbf w_n$. Since $\mathbf H_k\subseteq\mathbf H_r$ by Lemma~\ref{between F_k and F_{k+1}}, we apply the statement that proved in Subcase~1.3 and obtain the identity $\mathbf u\approx \mathbf u'ba\mathbf u''$ holds in $\bf V$.
 
\smallskip
 
\emph{Subcase} 2.3: \textbf V satisfies the hypothesis of the claim~(iv), i.e., $\alpha_k$ holds in \textbf V. Subcase~2.2 allows us to assume that $D(\mathbf u,a)=D(\mathbf u,b)$. Put $D(\mathbf u,a)=r$. Then the subword $ab$ of the word \textbf u above-mentioned in the formulation of the lemma lies in some \mbox{$(r-1)$}-block of \textbf u. Let~\eqref{s_0w_0s_1w_1 ... s_nw_n} be the \mbox{$(r-1)$}-decomposition of $\bf u$. Then there exists a number $0\le j\le n$ such that $\mathbf w_j = \mathbf w_j'ab\mathbf w_j''$, whence $\mathbf u' = s_0\mathbf w_0s_1\mathbf w_1\cdots s_j\mathbf w_j'$ and $\mathbf u''= \mathbf w_j''s_{j+1}\mathbf w_{j+1}\cdots s_n\mathbf w_n$. Since $\mathbf F_k\subseteq\mathbf F_r$ by Lemma~\ref{between F_k and F_{k+1}}, we apply the statement that proved in Subcase~1.4 and obtain the identity $\mathbf u\approx \mathbf u'ba\mathbf u''$ holds in $\bf V$.
 
\medskip
 
\emph{Case} 3: $D(\mathbf u, b)=\infty$. This case, as well as the previous one, is divided into three subcases.
 
\smallskip
 
\emph{Subcase} 3.1: \textbf V satisfies the hypothesis of one of the claims~(i) and~(ii). Let $s$ be a non-negative integer. Repeating literally arguments from Subcase~2.1, we obtain the subword $ab$ of the word \textbf u above-mentioned in the formulation of the lemma lies in some $s$-block of \textbf u. By Remark~\ref{max decomposition}, there is a number $r\ge k$ such that~\eqref{s_0w_0s_1w_1 ... s_nw_n} is the $\ell$-decomposition of $\bf u$ for any $\ell\ge r$. Then $ab$ is a subword of $\mathbf w_j$ for some $0\le j\le n$. We have $\mathbf w_j = \mathbf w_j'ab\mathbf w_j''$ for some possibly empty words $\mathbf w_j'$ and $\mathbf w_j''$. Then $\mathbf u' = s_0\mathbf w_0s_1\mathbf w_1\cdots s_j\mathbf w_j'$ and $\mathbf u''= \mathbf w_j''s_{j+1}\mathbf w_{j+1}\cdots s_n\mathbf w_n$. One can prove that
\begin{equation}
\label{occ_z(w_j)>1}
\occ_z(\mathbf w_j)\ge2
\end{equation}
for any letter $z\in \con(\mathbf w_j)$. Suppose at first that $s_j=h_1^r(\mathbf u, z)$ and $\occ_z(\mathbf w_j)=1$. If $\occ_z(\mathbf u)=1$ then $z$ is a 0-divider of $\bf u$. Lemma~\ref{simple observations}(i) implies that then $z\in\{s_1,s_2,\dots,s_n\}$, a contradiction. Therefore, $\occ_z(\mathbf u)\ge2$. Since $\occ_z(\mathbf w_j)=1$, we have $s_j\ne h_2^r(\mathbf u,z)$. This means that $D(\mathbf u, z)\le r+1$. According to Lemma~\ref{k-divider and depth}, $z$ is an \mbox{$(r+1)$}-divider of $\bf u$. We obtain a contradiction with the fact that~\eqref{s_0w_0s_1w_1 ... s_nw_n} is the \mbox{$(r+1)$}-decomposition of $\bf u$. So, the claim~\eqref{occ_z(w_j)>1} is true whenever $s_j=h_1^r(\mathbf u, z)$. Suppose now that $s_j\ne h_1^r(\mathbf u, z)$. Then the \mbox{$(1,r)$}-restrictor of $z$ in \textbf u is $s_p$ for some $p<j$. This means that $z\in\con(s_0\mathbf w_0s_1\mathbf w_1\cdots s_{j-1}\mathbf w_{j-1})$. Then
$$
\mathbf u=\mathbf fz\mathbf g\mathbf w_{j1}z\mathbf w_{j2}s_{j+1}\mathbf w_{j+1}\cdots s_n\mathbf w_n
$$
for some possibly empty words $\mathbf f,\mathbf g,\mathbf w_{j1}$ and $\mathbf w_{j2}$ with
$$
\mathbf fz\mathbf g=s_0\mathbf w_0s_1\mathbf w_1\cdots s_{j-1}\mathbf w_{j-1}\mathbf s_j
$$
and $\mathbf w_j=\mathbf w_{j1}z\mathbf w_{j2}$. Then the identity~\eqref{xyx=xyxx} applies with the conclusion that \textbf V satisfies the identity
$$
\mathbf u\approx\mathbf fz\mathbf g\mathbf w_{j1}z^2\mathbf w_{j2}s_{j+1}\mathbf w_{j+1}\cdots s_n\mathbf w_n.
$$
Therefore, we can assume that~\eqref{occ_z(w_j)>1} is true again. Thus, this claim holds for any letter $z\in\con(\mathbf w_j)$. Then Lemma~\ref{identities in K}(iii) implies that the variety $\bf V$ satisfies ${\bf w}_j\approx{\bf w}_j'ba{\bf w}_j''$, whence
\begin{align*}
\mathbf u={}&s_0\mathbf w_0s_1\mathbf w_1\cdots s_j\mathbf w_js_{j+1}\mathbf w_{j+1}\cdots s_n\mathbf w_n\\
\approx{}&s_0\mathbf w_0s_1\mathbf w_1\cdots s_j\mathbf w_j'ba\mathbf w_j''s_{j+1}\mathbf w_{j+1}\cdots s_n\mathbf w_n\\
={}&\mathbf u'ba\mathbf u''
\end{align*}
hold in this variety.
 
We have completed the proof of the statements~(i) and~(ii).
 
\smallskip
 
\emph{Subcase} 3.2: \textbf V satisfies the hypothesis of the claim~(iii), i.e., $\beta_k$ holds in \textbf V and $D(\mathbf u,a)\ne D(\mathbf u,b)$. Then $D(\mathbf u, a)<\infty$. Put $D(\mathbf u,a)=r$. Repeating arguments from Subcase~1.3, we have $a\notin\con(\mathbf u')$ and $r\ge k$. Let~\eqref{s_0w_0s_1w_1 ... s_nw_n} be the \mbox{$(r-1)$}-decomposition of $\bf u$. Then there exists a number $0\le j\le n$ such that $\mathbf w_j = \mathbf w_j'ab\mathbf w_j''$, $\mathbf u' = s_0\mathbf w_0s_1\mathbf w_1\cdots s_j\mathbf w_j'$ and $\mathbf u''= \mathbf w_j''s_{j+1}\mathbf w_{j+1}\cdots s_n\mathbf w_n$. Put $\mathbf u^\ast=\mathbf u'ba\mathbf u''$. Since $a,b\notin\{s_1,s_2,\dots,s_n\}$, the \mbox{$(r-1)$}-decomposition of $\bf u^\ast$ has the form $s_0\mathbf w_0 s_1 \mathbf w_1 \cdots s_j \mathbf w_j^\ast \cdots s_n \mathbf w_n$ where $\mathbf w_j^\ast = \mathbf w_j'ba\mathbf w_j''$. Then the claims~\eqref{sim(u)=sim(v) & mul(u)=mul(v)} and~\eqref{eq the same l-dividers} with $\mathbf v = \mathbf u^\ast$ and $\ell=r$ are true. Now Lemma~\ref{D(u,x)=k iff D(v,x)=k} applies with the conclusion that $D(\mathbf u^\ast, a)=r$, whence $a$ is an $r$-divider of $\mathbf u^\ast$ by Lemma~\ref{k-divider and depth}. Then $h_1^r(\mathbf u^\ast, b)\ne h_2^r(\mathbf u^\ast, b)$. This implies that $D(\mathbf u^\ast, b)>r$. Since $\bf V$ satisfies the identity $\beta_r$ by Lemma~\ref{between F_k and F_{k+1}}, we apply the statement that proved in Subcase~1.3 and obtain the identities $\mathbf u^\ast= \mathbf u'ba\mathbf u''\approx \mathbf u'ab\mathbf u''=\mathbf u$ hold in $\bf V$.
 
We have completed the proof of the statement~(iii).
 
\smallskip
 
\emph{Subcase} 3.3: \textbf V satisfies the hypothesis of the claim~(iv), i.e., $\alpha_k$ holds in \textbf V. Subcase~3.2 allows us to assume that $D(\mathbf u,a)=D(\mathbf u,b)=\infty$. This fact together with Lemma~\ref{k-divider and depth} implies that the subword $ab$ of the word \textbf u above-mentioned in the formulation of the lemma lies in some $s$-block of \textbf u for any $s$. Now we repeat literally arguments used in Subcase~3.1 and prove that the identity $\mathbf u\approx\mathbf u'ba\mathbf u''$ holds in \textbf V.
 
We have completed the proof of the statement~(iv) and the lemma as a whole.
\end{proof}
 
\subsection{Reduction to intervals of the form $[\mathbf F_k,\mathbf F_{k+1}]$}
\label{sufficiency: K - red to [F_k,F_{k+1}]}
 
Here we prove the claim~3) of Proposition~\ref{L(K)}. We need several auxiliary results.
 
\begin{lemma}
\label{V subseteq H_s or J_s^s}
Let ${\bf V}$ be a monoid variety such that ${\bf V} \subseteq {\bf K}$ and $\mathbf V$ satisfies an identity ${\bf u}\approx {\bf v}$, and $s$ be a natural number. Suppose that the claims~\eqref{sim(u)=sim(v) & mul(u)=mul(v)} and~\eqref{eq the same l-dividers} with $\ell=s$ are true and there are letters $x$ and $x_s$ such that $D(\mathbf u, x_s)=s$, $\ell_i(\mathbf u,x)<\ell_1(\mathbf u,x_s)$ and $\ell_1(\mathbf v,x_s)<\ell_i(\mathbf v,x)$ for some $i\in\{1,2\}$.
\begin{itemize}
\item[\textup{(i)}] If $i=1$ then ${\bf V} \subseteq {\bf H}_s$.
\item[\textup{(ii)}] If $i=2$ then ${\bf V} \subseteq {\bf J}_s^s$.
\end{itemize}
\end{lemma}
 
\begin{proof}
Lemma~\ref{identities in K}(iii) allows us to assume that $\occ_y(\mathbf u),\occ_y(\mathbf v)\le 2$ for any letter $y$. Now Lemma~\ref{form of the identity} applies with the conclusion that there are letters $x_0,x_1,\dots, x_{s-1}$ such that $D(\mathbf u,x_r)=D(\mathbf v,x_r)=r$ for any $0\le r<s$ and the identity $\mathbf u \approx \mathbf v$ has the form~\eqref{form of u=v} for some possibly empty words $\mathbf u_0, \mathbf u_1,\dots,\mathbf u_{2s+1}$ and $\mathbf v_0, \mathbf v_1,\dots,\mathbf v_{2s+1}$.
 
Suppose that $i=1$. Then $\ell_1(\mathbf u,x)<\ell_1(\mathbf u,x_s)$ and $\ell_1(\mathbf v,x_s)<\ell_1(\mathbf v,x)$. Suppose that $\ell_1(\mathbf u,x_s)<\ell_2(\mathbf u,x)$. In view of the above, we have
\begin{itemize}
\item[$\bullet$] first occurrence of $x$ in $\bf u$ lies in $\mathbf u_{2s+1}$,
\item[$\bullet$] second occurrence of $x$ in $\bf u$ lies in $\mathbf u_{2s}\mathbf u_{2s-1}\cdots \mathbf u_0$,
\item[$\bullet$] the first and the second occurrences of $x$ in $\bf v$ lie in $\mathbf v_{2s}\mathbf v_{2s-1}\cdots \mathbf v_0$.
\end{itemize}
Now we substitute $x_sx^2$ for $x_s$ in the identity ${\bf u}\approx {\bf v}$ and obtain the identity
\begin{equation}
\label{alternating identity}
\begin{array}{rl}
&\mathbf u_{2s+1} x_sx^2 \mathbf u_{2s} x_{s-1} \mathbf u_{2s-1} x_sx^2 \mathbf u_{2s-2} x_{s-2} \mathbf u_{2s-3} x_{s-1} \cdots\\
&\cdot\,\mathbf u_4 x_1 \mathbf u_3 x_2 \mathbf u_2 x_0 \mathbf u_1x_1 \mathbf u_0\\
\approx{}&\mathbf v_{2s+1} x_sx^2 \mathbf v_{2s} x_{s-1} \mathbf v_{2s-1} x_sx^2 \mathbf v_{2s-2} x_{s-2} \mathbf v_{2s-3} x_{s-1} \cdots\\
&\cdot\,\mathbf v_4 x_1 \mathbf v_3x_2 \mathbf v_2 x_0 \mathbf v_1 x_1 \mathbf v_0.
\end{array}
\end{equation}
Further, we apply the identity~\eqref{xyxzx=xyxz} and delete the third and subsequent occurrences of $x$ in both the sides of the identity~\eqref{alternating identity}. As a result, we obtain the identity
\begin{align*}
&\mathbf u_{2s+1} x_sx (\mathbf u_{2s} x_{s-1} \mathbf u_{2s-1} x_s \mathbf u_{2(s-1)} x_{s-2} \mathbf u_{2s-1} x_{s-1} \cdots \mathbf u_4 x_1 \mathbf u_3 x_2 \mathbf u_2 x_0 \mathbf u_1 x_1 \mathbf u_0)_x\\
\approx{}&\mathbf v_{2s+1} x_sx^2 (\mathbf v_{2s} x_{s-1} \mathbf v_{2s-1} x_s \mathbf v_{2(s-1)} x_{s-2} \mathbf v_{2s-1} x_{s-1} \cdots\mathbf v_4 x_1 \mathbf v_3 x_2 \mathbf v_2 x_0 \mathbf v_1 x_1 \mathbf v_0)_x.
\end{align*}
Now substitute~1 for all letters occurring in the last identity except $x,x_0,x_1,\dots$, $x_s$. We get the identity
$$
xx_sxx_{s-1}x_sx_{s-2}x_{s-1}\cdots x_1x_2x_0x_1\approx x_sx^2x_{s-1}x_sx_{s-2}x_{s-1}\cdots x_1x_2x_0x_1,
$$
i.e., the identity $\beta_s$. Therefore, ${\bf V} \subseteq {\bf H}_s$. 

Suppose now that $\ell_2(\mathbf u,x)<\ell_1(\mathbf u,x_s)$. In view of the above, we have
\begin{itemize}
\item[$\bullet$] the first and the second occurrences of $x$ in $\bf u$ lie in $\mathbf u_{2s+1}$,
\item[$\bullet$] the first and the second occurrences of $x$ in $\bf v$ lie in $\mathbf v_{2s}\mathbf v_{2s-1}\cdots \mathbf v_0$.
\end{itemize}
Now we substitute $x_sx^2$ for $x_s$ in the identity ${\bf u}\approx {\bf v}$ and obtain the identity~\eqref{alternating identity}. The identity~\eqref{xyxzx=xyxz} allows us to delete the third and subsequent occurrences of $x$ in both the sides of the identity~\eqref{alternating identity}. As a result, we obtain the identity
\begin{align*}
&\mathbf u_{2s+1} x_s \mathbf u_{2s} x_{s-1} \mathbf u_{2s-1} x_s \mathbf u_{2s-2} x_{s-2} \mathbf u_{2s-3} x_{s-1} \cdots \mathbf u_4 x_1 \mathbf u_3 x_2 \mathbf u_2 x_0 \mathbf u_1 x_1 \mathbf u_0\\
\approx{}&\mathbf v_{2s+1} x_sx^2 (\mathbf v_{2s} x_{s-1} \mathbf v_{2s-1} x_s \mathbf v_{2(s-1)} x_{s-2} \mathbf v_{2s-1} x_{s-1} \cdots\mathbf v_4 x_1 \mathbf v_3 x_2 \mathbf v_2 x_0 \mathbf v_1 x_1 \mathbf v_0)_x.
\end{align*}
Now substitute~1 for all letters occurring in the last identity except $x,x_0,x_1,\dots$, $x_s$. We get the identity
\begin{equation}
\label{last minute identity}
x^2x_sx_{s-1}x_sx_{s-2}x_{s-1}\cdots x_1x_2x_0x_1\approx x_sx^2x_{s-1}x_sx_{s-2}x_{s-1}\cdots x_1x_2x_0x_1.
\end{equation}
Then ${\bf V}$ satisfies the identities
\begin{align*}
x_sx^2x_{s-1}x_sx_{s-2}x_{s-1}\cdots x_1x_2x_0x_1\stackrel{\eqref{last minute identity}}\approx{}& x^2x_sx_{s-1}x_sx_{s-2}x_{s-1}\cdots x_1x_2x_0x_1\\
\stackrel{\eqref{xx=xxx}}\approx{}&x^3x_sx_{s-1}x_sx_{s-2}x_{s-1}\cdots x_1x_2x_0x_1\\
\stackrel{\eqref{last minute identity}}\approx{}&xx_sx^2x_{s-1}x_sx_{s-2}x_{s-1}\cdots x_1x_2x_0x_1\\
\stackrel{\eqref{xyxzx=xyxz}}\approx{}&xx_sxx_{s-1}x_sx_{s-2}x_{s-1}\cdots x_1x_2x_0x_1,
\end{align*}
whence the identity $\beta_s$ holds in $\mathbf V$. Therefore, ${\bf V} \subseteq {\bf H}_s$. The claim~(i) is proved.
 
\smallskip
 
Suppose now that $i=2$. Then $\ell_1(\mathbf u,x)<\ell_2(\mathbf u,x)<\ell_1(\mathbf u,x_s)$. If $\ell_1(\mathbf v,x_s)< \ell_1(\mathbf v,x)$ then we return to the already proved claim~(i). Then $\mathbf{V\subseteq H}_s\subseteq\mathbf J_s^s$ by Lemma~\ref{between F_k and F_{k+1}}. So, we can assume that $\ell_1(\mathbf v,x)< \ell_1(\mathbf v,x_s)$. In view of the above, we have
\begin{itemize}
\item[$\bullet$] the first and the second occurrences of $x$ in $\bf u$ lie in $\mathbf u_{2s+1}$,
\item[$\bullet$] first occurrence of $x$ in $\bf v$ lies in $\mathbf v_{2s+1}$,
\item[$\bullet$] second occurrence of $x$ in $\bf v$ lies in $\mathbf v_{2s}\mathbf v_{2s-1}\cdots \mathbf v_0$.
\end{itemize}
Now we substitute $x_sx^2$ for $x_s$ in the identity ${\bf u}\approx {\bf v}$ and obtain the identity~\eqref{alternating identity}. The identity~\eqref{xyxzx=xyxz} allows us to delete the third and subsequent occurrences of $x$ in both the sides of the identity~\eqref{alternating identity}. As a result, we obtain the identity
\begin{align*}
&\mathbf u_{2s+1} x_s \mathbf u_{2s} x_{s-1} \mathbf u_{2s-1} x_s \mathbf u_{2s-2} x_{s-2} \mathbf u_{2s-3} x_{s-1} \cdots \mathbf u_4 x_1 \mathbf u_3 x_2 \mathbf u_2 x_0 \mathbf u_1 x_1 \mathbf u_0\\
\approx{}&\mathbf v_{2s+1} x_sx (\mathbf v_{2s} x_{s-1} \mathbf v_{2s-1} x_s \mathbf v_{2(s-1)} x_{s-2} \mathbf v_{2s-1} x_{s-1} \cdots\mathbf v_4 x_1 \mathbf v_3 x_2 \mathbf v_2 x_0 \mathbf v_1 x_1 \mathbf v_0)_x.
\end{align*}
Now substitute~1 for all letters occurring in the last identity except $x,x_0,x_1,\dots$, $x_s$. We get the identity
$$
x^2x_sx_{s-1}x_sx_{s-2}x_{s-1}\cdots x_1x_2x_0x_1\approx xx_sxx_{s-1}x_sx_{s-2}x_{s-1}\cdots x_1x_2x_0x_1,
$$
i.e., the identity~\eqref{xx_kxb_k=x^2x_kb_k} with $k=s$. Lemma~\ref{basis for J_k^k} implies now that ${\bf V} \subseteq {\bf J}_s^s$. The claim~(ii) is proved.
\end{proof}
 
\begin{lemma}
\label{V subseteq E or J_(ell-1)^(ell-1)}
Let ${\bf V}$ be a monoid variety such that ${\bf V} \subseteq {\bf K}$ and $\mathbf V$ satisfies an identity ${\bf u}\approx {\bf v}$. If the claim~\eqref{sim(u)=sim(v) & mul(u)=mul(v)} is true, while the claim~\eqref{eq the same l-dividers} is false for some $\ell>1$ then ${\bf V} \subseteq {\bf J}_{\ell-1}^{\ell-1}$.
\end{lemma}
 
\begin{proof}
Suppose that the claim~\eqref{sim(u)=sim(v) & mul(u)=mul(v)} is true, while the claim~\eqref{eq the same l-dividers} is false for some $\ell=k>1$ and $k$ is the least number with this property. Then there exists a letter $x$ such that $h_i^{k-1}({\bf u},x)\ne h_i^{k-1}({\bf v},x)$ where either $i=1$ or $i=2$. Let~\eqref{t_0u_0t_1u_1 ... t_mu_m} be the \mbox{$(k-1)$}-decomposition of $\bf u$. In particular, the set of \mbox{$(k-1)$}-dividers of \textbf u is $\{t_0,t_1,\dots,t_m\}$. Since the claim~\eqref{eq the same l-dividers} with $\ell=k-1$ is true, Lemma~\ref{the same l-dividers} applies with the conclusion that $\bf v$ has the same \mbox{$(k-1)$}-dividers as \textbf u (but the order of the first occurrences of these letters in the words \textbf u and \textbf v may be different). Put $t_p=h_i^{k-1}({\bf u},x)$ and $t_q=h_i^{k-1}({\bf v},x)$. Clearly, $p\ne q$. 
 
Suppose at first that $\ell_i(\mathbf u, x)<\ell_1(\mathbf u, t_q)$. The choice of $t_p$ and $t_q$ guarantees that $\ell_1(\mathbf u,t_p)<\ell_i(\mathbf u,x)$ and $\ell_1(\mathbf v, t_q)<\ell_i(\mathbf v, x)$. Therefore, $\ell_1(\mathbf u,t_p)<\ell_1(\mathbf u,t_q)$, whence $p<q$ in the case we consider. If $t_q$ is simple in \textbf u then the claim~\eqref{sim(u)=sim(v) & mul(u)=mul(v)} implies that $t_q$ is simple in \textbf v too. Therefore, the letter $t_q$ is a 0-divider of the words $\bf u$ and \textbf v. Since $t_q=h_i^{k-1}({\bf v},x)$, we have $t_q=h_i^0({\bf v},x)$. The claim~\eqref{eq the same l-dividers} with $\ell=1$ implies that $t_q=h_i^0({\bf u},x)$. But this contradicts the fact that $p<q$. So, $t_q$ is multiple in \textbf u, whence $t_q$ is multiple in \textbf v as well by the claim~\eqref{sim(u)=sim(v) & mul(u)=mul(v)}. Therefore, $D(\mathbf v, t_q)>0$. Besides that, $D(\mathbf v, t_q)\le k-1$ by Lemma~\ref{k-divider and depth} because $t_q$ is a \mbox{$(k-1)$}-divider of \textbf v. Put $r=D(\mathbf v,t_q)$. If $i=1$ then Lemma~\ref{V subseteq H_s or J_s^s}(i) with $s=r$ and $x_s=t_q$ applies with the conclusion that $\mathbf V\subseteq \mathbf H_r\subseteq \mathbf J_{k-1}^{k-1}$. If $i=2$ then $\mathbf V\subseteq \mathbf J_r^r\subseteq \mathbf J_{k-1}^{k-1}$ by Lemma~\ref{V subseteq H_s or J_s^s}(ii) with $s=r$ and $x_s=t_q$.
 
If $\ell_i(\mathbf v, x)<\ell_1(\mathbf v, t_p)$ then we can obtain the required conclusion using arguments similar to ones from the previous paragraph.
 
Finally, suppose that $\ell_1(\mathbf u, t_q)<\ell_i(\mathbf u, x)$ and $\ell_1(\mathbf v, t_p)<\ell_i(\mathbf v, x)$. The first of these inequalities implies that first occurrence of $t_q$ in \textbf u precedes $i$th occurrence of $x$ in \textbf u. But $t_p$ is the right-most \mbox{$(k-1)$}-divider of the word \textbf u precedes $i$th occurrence of $x$. Therefore, $\ell_1(\mathbf u, t_q)<\ell_1(\mathbf u, t_p)$. Analogously, it follows from $\ell_1(\mathbf v, t_p)<\ell_i(\mathbf v, x)$ and $t_q=h_i^{k-1}({\bf v},x)$ that $\ell_1(\mathbf v, t_p)<\ell_1(\mathbf v, t_q)$. Suppose that $t_p$ is simple in \textbf u. Then the claim~\eqref{sim(u)=sim(v) & mul(u)=mul(v)} implies that $t_p$ is simple in \textbf v too. Then the letter $t_p$ is a 0-divider of the words $\bf u$ and \textbf v. Since $t_p=h_i^{k-1}({\bf u},x)$, we have $t_p=h_i^0({\bf u},x)$. The claim~\eqref{eq the same l-dividers} with $\ell=1$ implies that $t_p=h_i^0({\bf v},x)$. Note that $\ell_1(\mathbf v, t_p)<\ell_1(\mathbf v, t_q)<\ell_i(\mathbf v, x)$. Being the right-most simple in \textbf v letter that is located to the left of $x$, the letter $t_p$ turns out to be also the right-most simple in \textbf v letter that is located to the left of $t_q$. In other words, $t_p=h_1^0({\bf v},t_q)$. The claim~\eqref{eq the same l-dividers} with $\ell=1$ implies that $t_p=h_1^0({\bf u},t_q)$. But this contradicts the fact that $\ell_1(\mathbf u, t_q)<\ell_1(\mathbf u, t_p)$. So, $t_p$ is multiple in \textbf u. Therefore, $D(\mathbf u, t_p)>0$. Besides that, $D(\mathbf u, t_p)\le k-1$ by Lemma~\ref{k-divider and depth} because $t_p$ is a \mbox{$(k-1)$}-divider of \textbf u. Put $r=D(\mathbf u,t_p)$. Then the hypothesis of Lemma~\ref{V subseteq H_s or J_s^s} with $i=1$, $s=r$, $x=t_q$ and $x_s=t_p$ holds. Therefore, Lemma~\ref{V subseteq H_s or J_s^s}(i) implies that $\mathbf V\subseteq \mathbf H_r\subseteq \mathbf J_{k-1}^{k-1}$.
\end{proof}
 
The following statement opens a series of one-type assertions, which also includes Propositions~\ref{word problem H_k},~\ref{word problem I_k} and~\ref{word problem J_k^r}. These results provide solutions of the word problem in the varieties $\mathbf F_k$, $\mathbf H_k$, $\mathbf I_k$, $\mathbf J_k^m$ and \textbf K. All of them are proved in the same scheme. For the ``only if'' part, this scheme almost does not changed from proposition to proposition. As to the ``if'' part, the mentioned scheme is generally outlined in the proof of Proposition~\ref{word problem F_k,K}(i) but technically its implementation will be more and more complicated each time.
 
\begin{proposition}
\label{word problem F_k,K}
A non-trivial identity ${\bf u}\approx {\bf v}$ holds:
\begin{itemize}
\item[\textup{(i)}]in the variety ${\bf F}_k$ if and only if the claims~\eqref{sim(u)=sim(v) & mul(u)=mul(v)} and~\eqref{eq the same l-dividers} with $\ell=k$ are true;
\item[\textup{(ii)}]in the variety ${\bf K}$ if and only if the claims~\eqref{sim(u)=sim(v) & mul(u)=mul(v)} and~\eqref{eq the same l-dividers} for all $\ell$ are true.
\end{itemize}
\end{proposition}
 
\begin{proof}
(i) \emph{Necessity}. Suppose that a non-trivial identity $\mathbf u\approx \mathbf v$ holds in $\mathbf F_k$. Proposition~\ref{word problem C_2} and the inclusion ${\bf C}_2\subseteq{\bf F}_k$ imply that the claim~\eqref{sim(u)=sim(v) & mul(u)=mul(v)} is true. Since $\mathbf F_k$ satisfies $\mathbf u\approx \mathbf v$, there is a sequence of words $\mathbf u=\mathbf w_0,\mathbf w_1,\dots,\mathbf w_n=\mathbf v$ such that, for any $i=0,1,\dots,n-1$, there are words $\mathbf p_i,\mathbf q_i\in F^1$, an endomorphism $\xi_i$ of $F^1$ and an identity $\mathbf a_i\approx \mathbf b_i$ from the system $\{\Phi,\alpha_k\}$ such that either $\mathbf w_i=\mathbf p_i\xi_i(\mathbf a_i)\mathbf q_i$ and $\mathbf w_{i+1}=\mathbf p_i\xi_i(\mathbf b_i)\mathbf q_i$ or $\mathbf w_i=\mathbf p_i\xi_i(\mathbf b_i)\mathbf q_i$ and $\mathbf w_{i+1}=\mathbf p_i\xi_i(\mathbf a_i)\mathbf q_i$. By induction we can assume without loss of generality that $\mathbf u=\mathbf p\xi(\mathbf a)\mathbf q$ and $\mathbf v=\mathbf p\xi(\mathbf b)\mathbf q$ for some possibly empty words $\mathbf p$ and $\mathbf q$, an endomorphism $\xi$ of $F^1$ and an identity $\mathbf a\approx \mathbf b\in\{\Phi,\alpha_k\}$. 
 
If $\mathbf a\approx\mathbf b\in\{xyx\approx xyx^2,x^2y\approx x^2yx\}$ then the required assertion is obvious because the first and second occurrences of the letters of $\bf u$ do not take part in replacing $\xi(\mathbf a)$ to $\xi(\mathbf b)$. Suppose now that $\mathbf a\approx \mathbf b$ coincides with the identity~\eqref{xxyy=yyxx}. Then, since $D(\mathbf a,x)=D(\mathbf a,y)=\infty$, Lemma~\ref{does not contain dividers} implies that the subword $\xi(\mathbf a)$ of $\bf u$ located between $\mathbf p$ and $\mathbf q$ is contained in some $s$-block for all $s$. In particular, this subword is contained is some \mbox{$(k-1)$}-block. This implies that the claim~\eqref{eq the same l-dividers} with $\ell=k$ is true.
 
Finally, suppose that $\mathbf a\approx \mathbf b$ coincides with $\alpha_k$. Then 
\begin{align*}
\xi(\mathbf a)={}&\mathbf a_k\mathbf b_k\mathbf a_{k-1}\mathbf a_k\mathbf b_k\mathbf a_{k-2}\mathbf a_{k-1}\cdots \mathbf a_1\mathbf a_2\mathbf a_0\mathbf a_1,\\
\xi(\mathbf b)={}&\mathbf b_k\mathbf a_k\mathbf a_{k-1}\mathbf a_k\mathbf b_k\mathbf a_{k-2}\mathbf a_{k-1}\cdots \mathbf a_1\mathbf a_2\mathbf a_0\mathbf a_1
\end{align*}
for some words $\mathbf a_0,\mathbf a_1,\dots,\mathbf a_k$ and $\mathbf b_k$, whence
\begin{align*}
\mathbf u={}&\mathbf p\mathbf a_k\mathbf b_k\mathbf a_{k-1}\mathbf a_k\mathbf b_k\mathbf a_{k-2}\mathbf a_{k-1}\cdots \mathbf a_1\mathbf a_2\mathbf a_0\mathbf a_1\mathbf q,\\
\mathbf v={}&\mathbf p\mathbf b_k\mathbf a_k\mathbf a_{k-1}\mathbf a_k\mathbf b_k\mathbf a_{k-2}\mathbf a_{k-1}\cdots \mathbf a_1\mathbf a_2\mathbf a_0\mathbf a_1\mathbf q.
\end{align*}
By Lemma~\ref{depth and index}, $D(\mathbf a,x_k)=D(\mathbf a,y_k)=k$. Then Lemma~\ref{does not contain dividers} implies that the subword $\mathbf a_k\mathbf b_k$ of $\bf u$ located between $\mathbf p$ and $\mathbf a_{k-1}$ is contained in some \mbox{$(k-1)$}-block. This implies that the claim~\eqref{eq the same l-dividers} with $\ell=k$ is true.
 
\smallskip
 
\emph{Sufficiency}. Let us describe the outline of our further arguments. We note that sufficiency in Propositions~\ref{word problem H_k},~\ref{word problem I_k} and~\ref{word problem J_k^r} will be proved below in the same scheme. Let $\mathbf{u\approx v}$ be an identity which satisfies the hypothesis of the proposition. We start with a consideration of the \mbox{$(k-1)$}-decomposition of the word \textbf u. Basing on Lemma~\ref{u'abu''=u'bau''} and using identities which hold in the variety $\textbf F_k$, we show that any \mbox{$(k-1)$}-block of \textbf u can be replaced by a word of some ``canonical form''. We replace all \mbox{$(k-1)$}-blocks of \textbf u with getting some word $\mathbf u^\sharp$. After that we consider the word \textbf v. It turns out that, up to identities in $\mathbf F_k$, this word has exactly the same \mbox{$(k-1)$}-blocks and \mbox{$(k-1)$}-dividers as the word \textbf u. This allows us to change \mbox{$(k-1)$}-blocks of \textbf v in the same way as \mbox{$(k-1)$}-blocks of \textbf u with getting the word $\mathbf u^\sharp$ again. This evidently implies that the identity $\mathbf{u\approx v}$ holds in $\mathbf F_k$.
 
Now we proceed to implement the above plan. Suppose that the identity ${\bf u}\approx {\bf v}$ satisfies the claims~\eqref{sim(u)=sim(v) & mul(u)=mul(v)} and~\eqref{eq the same l-dividers} with $\ell=k$. Let~\eqref{t_0u_0t_1u_1 ... t_mu_m} be the \mbox{$(k-1)$}-decomposition of \textbf u. Let us fix an index $i\in\{0,1,\dots,m\}$. Lemma~\ref{identities in K}(ii) allows us to suppose that every letter from $\con(\mathbf u_i)$ occurs in $\mathbf u_i$ at most twice. Put $\mul(\mathbf u_i)=\{x_1,x_2.\dots,x_p\}$, $\simple(\mathbf u_i)=\{y_1,y_2,\dots, y_q\}$ and
$$
\overline{\mathbf u_i}\,=x_1^2x_2^2\cdots x_p^2y_1y_2\cdots y_q.
$$
Note that $\overline{\mathbf u_i}$ is nothing but the ``canonical form'' of the \mbox{$(k-1)$}-block $\mathbf u_i$ mentioned above. Indeed, $\mathbf u=\mathbf w_1\mathbf u_i\mathbf w_2$ for some possibly empty words $\mathbf w_1$ and $\mathbf w_2$. Lemmas~\ref{identities in K}(ii) and~\ref{u'abu''=u'bau''}(iv) imply now that the variety $\mathbf F_k$ satisfies the identity
$$
\mathbf u=\mathbf w_1\mathbf u_i\mathbf w_2\approx\mathbf w_1\,\overline{\mathbf u_i}\,\mathbf w_2.
$$
In particular, $\mathbf F_k$ satisfies the identities
$$
\mathbf u=t_0\mathbf u_0 t_1\mathbf u_1\cdots t_{m-1}\mathbf u_{m-1}t_m\mathbf u_m\approx t_0\mathbf u_0 t_1\mathbf u_1\cdots t_{m-1}\mathbf u_{m-1}t_m\,\overline{\mathbf u_m}\,.
$$
 
Put $\mathbf u'=t_0\mathbf u_0 t_1\mathbf u_1\cdots t_{m-1}\mathbf u_{m-1}t_m\,\overline{\mathbf u_m}$\,. Note that the claims~\eqref{sim(u)=sim(v) & mul(u)=mul(v)} and~\eqref{eq the same l-dividers} with $\mathbf v=\mathbf u'$ and $\ell=k$ are true. Then Lemma~\ref{k-equivalent} implies that the words $\bf u$ and $\mathbf u'$ are \mbox{$(k-1)$}-equivalent, i.e., the letters $t_0,t_1,\dots,t_m$ are \mbox{$(k-1)$}-dividers of the word $\mathbf u'$, while $\mathbf u_0,\mathbf u_1,\dots,\mathbf u_{m-1},\overline{\mathbf u_m}$ are \mbox{$(k-1)$}-blocks of this word. After that, we can repeat literally arguments given above by changing the word \textbf u to the word $\mathbf u'$ and obtain the identities
$$
\mathbf u'=t_0\mathbf u_0 t_1\mathbf u_1\cdots t_{m-1}\mathbf u_{m-1}t_m\,\overline{\mathbf u_m}\,\approx t_0\mathbf u_0 t_1\mathbf u_1\cdots t_{m-1}\,\overline{\mathbf u_{m-1}}\,t_m\,\overline{\mathbf u_m}
$$
hold in $\mathbf F_k$. Continuing this process, we get that $\mathbf F_k$ satisfies the identities
\begin{equation}
\label{u = canonical form}
\begin{array}{rl}
\mathbf u={}&t_0\mathbf u_0 t_1\mathbf u_1\cdots t_{m-1}\mathbf u_{m-1}t_m\mathbf u_m\approx t_0\mathbf u_0 t_1\mathbf u_1\cdots t_{m-1}\mathbf u_{m-1}t_m\,\overline{\mathbf u_m}\\
\approx{}&t_0\mathbf u_0 t_1\mathbf u_1\cdots t_{m-1}\,\overline{\mathbf u_{m-1}}\,t_m\,\overline{\mathbf u_m}\approx\cdots\approx t_0\,\overline{\mathbf u_0}\,t_1\,\overline{\mathbf u_1}\,\cdots t_m\,\overline{\mathbf u_m}\,.
\end{array}
\end{equation}
Put $\mathbf u^\sharp=t_0\,\overline{\mathbf u_0}\,t_1\,\overline{\mathbf u_1}\,\cdots t_m\,\overline{\mathbf u_m}$\,.
 
One can return to the word \textbf v. By Lemma~\ref{k-equivalent}, the \mbox{$(k-1)$}-decomposition of $\bf v$ has the form~\eqref{t_0v_0t_1v_1 ... t_mv_m}. The claim~\eqref{eq the same l-dividers} with $\ell=k$ implies that $j$th occurrence of a letter $x$ in $\bf u$ lies in the \mbox{$(k-1)$}-block $\mathbf u_i$ if and only if $j$th occurrence of a letter $x$ in $\bf v$ lies in the \mbox{$(k-1)$}-block $\mathbf v_i$ for any $x$ and any $j=1,2$. We are going to check that $\simple(\mathbf u_i)=\simple(\mathbf v_i)$ and $\mul(\mathbf u_i)=\mul(\mathbf v_i)$. Let $x\in\con(\mathbf u_i)$. Lemma~\ref{identities in K}(ii) allows us to assume that $\occ_x(\mathbf u)\le 2$. There are three possibilities. First, if the first and the second occurrences of $x$ in \textbf u lie in $\mathbf u_i$ then the first and the second occurrences of $x$ in \textbf v lie in $\mathbf v_i$, whence $x\in\mul(\mathbf u_i)$ and $x\in\mul(\mathbf v_i)$. Second, if the first occurrences of $x$ in \textbf u lies in $\mathbf u_i$ but the second one does not lie in $\mathbf u_i$ then the first occurrences of $x$ in \textbf v lies in $\mathbf v_i$ but the second one does not lie in $\mathbf v_i$, whence $x\in\simple(\mathbf u_i)$ and $x\in\simple(\mathbf v_i)$. Finally, third, if first occurrence of $x$ in \textbf u is located to the left of $\mathbf u_i$, while the second one lies in $\mathbf u_i$ then first occurrence of $x$ in \textbf v is located to the left of $\mathbf v_i$, while the second one lies in $\mathbf v_i$. In this case we can apply the identity~\eqref{xyx=xyxx}. This allows us to suppose that $x\in\mul(\mathbf u_i)$ and $x\in\mul(\mathbf v_i)$. Thus, $\simple(\mathbf u_i)=\simple(\mathbf v_i)$ and $\mul(\mathbf u_i)=\mul(\mathbf v_i)$. This implies that the \mbox{$(k-1)$}-blocks $\mathbf u_i$ and $\mathbf v_i$ have the same ``canonical form''. Repeating literally arguments given above, we obtain the variety $\mathbf F_k$ satisfies the identities $\mathbf v\approx\mathbf u^\sharp\approx\mathbf u$.
 
\medskip
 
(ii) \emph{Necessity} follows from the already proved assertion~(i) of this proposition and the evident inclusion $\mathbf F_k\subseteq\mathbf K$, while \emph{sufficiency} is proved in the same way as in the assertion~(i).
\end{proof}
 
Now we are well prepared to quickly complete the proof of the claim~3) of Proposition~\ref{L(K)}. Let $\mathbf{E\subset X\subset K}$. We have to verify that $\mathbf X\in[\mathbf F_k,\mathbf F_{k+1}]$ for some $k$. Suppose that ${\bf F}_1\nsubseteq {\bf X}$. Then there is an identity ${\bf u}\approx {\bf v}$ that holds in ${\bf X}$ but does not hold in ${\bf F}_1$. Propositions~\ref{word problem E} and~\ref{word problem F_k,K}(i) and the inclusion ${\bf E} \subseteq {\bf X}$ imply that the claims~\eqref{sim(u)=sim(v) & mul(u)=mul(v)} and~\eqref{eq D_1 subseteq X} hold, while the claim~\eqref{eq the same l-dividers} with $\ell=1$ is false. Let~\eqref{t_0u_0t_1u_1 ... t_mu_m} be the 0-decomposition of $\bf u$. Then Lemma~\ref{k-equivalent} applies with the conclusion that the 0-decomposition of $\bf v$ has the form~\eqref{t_0v_0t_1v_1 ... t_mv_m}. Since ${\bf u}\approx {\bf v}$ violates the claim~\eqref{eq the same l-dividers} with $\ell=1$ but satisfies~\eqref{eq D_1 subseteq X}, there is a letter $x$ such that $h_2^0({\bf u},x)\ne h_2^0({\bf v},x)$. Put $t_i=h_2^0({\bf u},x)$ and $t_j=h_2^0({\bf v},x)$. We may assume without loss of generality that $j<i$. Since the claim~\eqref{eq D_1 subseteq X} is true, we have $h_1^0({\bf u},x)=h_1^0({\bf v},x)=t_q$ for some $q$. Clearly, $q\le j$. Thus, the identity $\mathbf{u\approx v}$ has the form
$$
\mathbf u_1t_q\mathbf u_2x\mathbf u_3t_i\mathbf u_4x\mathbf u_5\approx\mathbf v_1t_q\mathbf v_2x\mathbf v_3x\mathbf v_4t_i\mathbf v_5
$$
for some possibly empty words $\mathbf u_s$ and $\mathbf v_s$ with $s=1,2,\dots,5$. Substitute~1 for all letters occurring in the identity ${\bf u}\approx {\bf v}$ except $x$ and $t_i$. Then we obtain an identity of the form $xt_ix^p\approx x^qt_ix^r$ where $p\ge 1$, $q\ge 2$ and $r\ge 0$. Now the identity~\eqref{xyxzx=xyxz} applies and we conclude that ${\bf X}$ satisfies $xt_ix\approx x^2t_i$. This fact together with the inclusion $\mathbf{X\subseteq K}$ implies that ${\bf X} \subseteq {\bf E}$, contradicting the choice of \textbf X. Thus, ${\bf F}_1 \subseteq {\bf X}$. If ${\bf X}$ contains an infinite number of varieties of the form ${\bf F}_k$ then Proposition~\ref{word problem F_k,K} implies that ${\bf X}={\bf K}$. Hence there is a natural number $k$ such that ${\bf F}_k \subseteq {\bf X}$ but $\mathbf F_{k+1}\nsubseteq\mathbf X$. Then Proposition~\ref{word problem F_k,K}(i) implies that the claim~\eqref{sim(u)=sim(v) & mul(u)=mul(v)} holds, while the claim~\eqref{eq the same l-dividers} with $\ell=k+1$ fails. Now we apply Lemma~\ref{V subseteq E or J_(ell-1)^(ell-1)} and conclude that ${\bf X} \subseteq {\bf J}_k^k\subset\mathbf F_{k+1}$. Thus, $\mathbf X\in[\mathbf F_k,\mathbf F_{k+1}]$. The claim~3) of Proposition~\ref{L(K)} is proved.
 
\subsection{Structure of the interval $[\mathbf F_k,\mathbf F_{k+1}]$}
\label{sufficiency: K - structure of [F_k,F_{k+1}]} 
 
Here we prove the claim~4) of Proposition~\ref{L(K)}. We divide this subsection into six subsubsections. In Subsubsections~\ref{structure of [F_k,F_{k+1}] 1 step}--\ref{structure of [F_k,F_{k+1}] 5 step} we verify that an arbitrary variety from the interval $[\mathbf F_k,\mathbf F_{k+1}]$ coincides with one of the varieties $\mathbf F_k$, $\mathbf H_k$, $\mathbf I_k$, $\mathbf J_k^1$, $\mathbf J_k^2$, \dots, $\mathbf J_k^k$, $\mathbf F_{k+1}$. In Subsubsection~\ref{structure of [F_k,F_{k+1}] 6 step} we check that all these varieties are pairwise different. These facts together with Lemma~\ref{between F_k and F_{k+1}} imply the claim~4) of Proposition~\ref{L(K)}.
 
\subsubsection{If $\mathbf F_k\subset\mathbf X\subseteq\mathbf F_{k+1}$ then $\mathbf H_k\subseteq\mathbf X$}
\label{structure of [F_k,F_{k+1}] 1 step}
 
The first step in the verification of the claim~4) of Proposition \ref{L(K)} is the following
 
\begin{lemma}
\label{F_k or over H_k}
If $\mathbf X$ is a monoid variety such that $\mathbf X \in [\mathbf F_k, \mathbf F_{k + 1}]$ then either $\mathbf{X = F}_k$ or $\mathbf{X \supseteq H}_k$.
\end{lemma}
 
To check this fact, we need
 
\begin{lemma}
\label{V=F_k}
Let ${\bf V}$ be a monoid variety with ${\bf F}_s \subseteq {\bf V} \subseteq {\bf K}$ for some $s$. If $\mathbf V$ satisfies an identity ${\bf u}\approx {\bf v}$ such that $\ell_1({\bf u},a)<\ell_1({\bf u},b)$, $\ell_1({\bf v},b)<\ell_1({\bf v},a)$ and $D({\bf u},a)=D({\bf u},b)=s$ for some $a,b\in \con({\bf u})$ then ${\bf V} = {\bf F}_s$.
\end{lemma}
 
\begin{proof}
Put $x_s= a$ and $y_s= b$. Since $\mathbf F_s\subseteq \mathbf V$, Proposition~\ref{word problem F_k,K}(i) implies that the claims~\eqref{sim(u)=sim(v) & mul(u)=mul(v)} and~\eqref{eq the same l-dividers} with $\ell=s$ are true. Suppose that
\begin{equation}
\label{order of second occurrences}
\ell_2(\mathbf u,x_s)<\ell_2(\mathbf u,y_s)\text{ and }\ell_2(\mathbf v,x_s)<\ell_2(\mathbf v,y_s).
\end{equation}
Now Lemma~\ref{form of the identity} applies with the conclusion that there are letters $x_0,x_1,\dots$, $ x_{s-1}$ such that $D(\mathbf u,x_r)=D(\mathbf v,x_r)=r$ for any $0\le r<s$ and the identity $\mathbf u \approx \mathbf v$ has the form~\eqref{form of u=v} for some possibly empty words $\mathbf u_0, \mathbf u_1,\dots,\mathbf u_{2s+1}$ and $\mathbf v_0, \mathbf v_1,\dots,\mathbf v_{2s+1}$.
 
One can verify that the first occurrences of $x_s$ and $y_s$ in \textbf u lie in the same \mbox{$(s-1)$}-block. Put $t_1=h_1^{s-1}(\mathbf u, x_s)$ and $t_2=h_1^{s-1}(\mathbf u, y_s)$. Arguing by contradiction, suppose that $t_1\ne t_2$. Since $\ell_1({\bf u},x_s)<\ell_1({\bf u},y_s)$, we have $\ell_1({\bf u},t_1)<\ell_1({\bf u},t_2)$. Lemma~\ref{k-equivalent} with $k=s-1$ implies that $\ell_1({\bf v},t_1)<\ell_1({\bf v},t_2)$. In view of the claim~\eqref{eq the same l-dividers} with $\ell=s$, $t_1=h_1^{s-1}(\mathbf v, x_s)$ and $t_2=h_1^{s-1}(\mathbf v, y_s)$. But this contradicts the fact that $\ell_1({\bf v},y_s)<\ell_1({\bf v},x_s)$. So, the first occurrences $x_s$ and $y_s$ in \textbf u lie in the same \mbox{$(s-1)$}-block. In particular, first occurrence of $y_s$ in \textbf u precedes first occurrence of $x_{s-1}$ in \textbf u because $\ell_1(\mathbf u,x_s)<\ell_1(\mathbf u,x_{s-1})$ and $x_{s-1}$ is an \mbox{$(s-1)$}-divider. This implies that $\mathbf u_{2s}=\mathbf u_{2s}'y_s\mathbf u_{2s}''$ for some possibly empty words $\mathbf u_{2s}'$ and $\mathbf u_{2s}''$. Since first occurrence $y_s$ in \textbf v precedes first occurrence of $x_s$ in \textbf v, we have $\mathbf v_{2s+1}=\mathbf v_{2s+1}'y_s\mathbf v_{2s+1}''$ for some possibly empty words $\mathbf v_{2s+1}'$ and $\mathbf v_{2s+1}''$.
 
Further, since $\ell_1(\mathbf u,y_s)<\ell_1(\mathbf u,x_{s-2})$, we apply Lemma~\ref{if first then second} with $\bf w = u$, $z=y_s$, $t=x_{s-2}$ and $r=s$ and obtain $\ell_2(\mathbf u,y_s)<\ell_1(\mathbf u,x_{s-2})$. This implies that $\mathbf u_{2s-2}=\mathbf u_{2s-2}'y_s\mathbf u_{2s-2}''$ for some possibly empty words $\mathbf u_{2s-2}'$ and $\mathbf u_{2s-2}''$. Analogously, we can verify that $\mathbf v_{2s-2}=\mathbf v_{2s-2}'y_s\mathbf v_{2s-2}''$ for some possibly empty words $\mathbf v_{2s-2}'$ and $\mathbf v_{2s-2}''$.
 
In view of the above, we have the identity $\mathbf{u\approx v}$ has the form
\begin{align*}
&\mathbf u_{2s+1} x_s \mathbf u_{2s}'\stackrel{(1)}{y_s}\mathbf u_{2s}' x_{s-1} \mathbf u_{2s-1} x_s \mathbf u_{2s-2}'\stackrel{(2)}{y_s}\mathbf u_{2s-2}' x_{s-2} \mathbf u_{2s-3} x_{s-1} \cdots\\
&\cdot\,\mathbf u_4 x_1 \mathbf u_3 x_2 \mathbf u_2 x_0 \mathbf u_1 x_1 \mathbf u_0\\[-3pt]
\approx{}&\mathbf v_{2s+1}' \stackrel{(1)}{y_s} \mathbf v_{2s+1}'' x_s \mathbf v_{2s} x_{s-1} \mathbf v_{2s-1} x_s \mathbf v_{2s-2}' \stackrel{(2)}{y_s}\mathbf v_{2s-2}'' x_{s-2} \mathbf v_{2s-3} x_{s-1} \cdots\\
&\cdot\,\mathbf v_4 x_1 \mathbf v_3 x_2 \mathbf v_2 x_0 \mathbf v_1 x_1 \mathbf v_0.
\end{align*}
Lemma~\ref{identities in K}(ii) allows us to assume that the letters $x_r$ with $1\le r\le s$ and $y_s$ occur twice in each of the words \textbf u and \textbf v. Now substitute~1 for all letters occurring in this identity except $x_0,x_1,\dots,x_s$ and $y_s$. We get the identity
$$
x_sy_sx_{s-1}x_sy_sx_{s-2}x_{s-1}\cdots x_1x_2x_0x_1\approx y_sx_sx_{s-1}x_sy_sx_{s-2}x_{s-1}\cdots x_1x_2x_0x_1,
$$
i.e., the identity $\alpha_s$.
 
Suppose now that~\eqref{order of second occurrences} is false. If $\ell_2(\mathbf u,x_s)<\ell_2(\mathbf u,y_s)$ but $\ell_2(\mathbf v,y_s)<\ell_2(\mathbf v,x_s)$ then the same considerations as above show that \textbf V satisfies the identity
$$
x_sy_sx_{s-1}x_sy_sx_{s-2}x_{s-1}\cdots x_1x_2x_0x_1\approx{}\stackrel{(1)}{y_s}\stackrel{(1)}{x_s}x_{s-1}\stackrel{(2)}{y_s}\stackrel{(2)}{x_s}x_{s-2}x_{s-1}\cdots x_1x_2x_0x_1.
$$
According to Lemma~\ref{identities in K}(i), the variety \textbf V satisfies the identity $\sigma_2$. This allows us to transpose the second occurrences of the letters $x_s$ and $y_s$ in the right-hand side of the last identity. As a result, we get $\alpha_s$ as well.
 
Finally, if $\ell_2(\mathbf u,y_s)<\ell_2(\mathbf u,x_s)$ then we can repeat the above arguments but apply Lemmas~\ref{if first then second} and~\ref{form of the identity} for the letter $y_s$ rather than $x_s$. As a result, we obtain an identity of the form
$$
x_sy_sx_{s-1}y_sx_sx_{s-2}x_{s-1}\cdots x_1x_2x_0x_1\approx y_sx_sx_{s-1}\mathbf ax_{s-2}x_{s-1}\cdots x_1x_2x_0x_1
$$
where
$$
\mathbf a=
\begin{cases}
x_sy_s&\text{whenever }\ell_2(\mathbf v,x_s)<\ell_2(\mathbf v,y_s),\\
y_sx_s&\text{otherwise}.
\end{cases}
$$
In the case when $\mathbf a=x_sy_s$ this identity coincides with $\alpha_s$, while otherwise we apply $\sigma_2$ once again and obtain $\alpha_s$ too. Thus, \textbf V satisfies $\alpha_s$ in any case, whence ${\bf V} \subseteq {\bf F}_s$.
\end{proof}
 
\begin{proposition}
\label{word problem H_k}
A non-trivial identity ${\bf u}\approx {\bf v}$ holds in the variety ${\bf H}_k$ if and only if the claims~\eqref{sim(u)=sim(v) & mul(u)=mul(v)},~\eqref{eq the same l-dividers} and
\begin{equation}
\label{h_1^l(u,x)= h_1^l(v,x)}
\text{if either }D({\bf u},x)\le \ell\text{ or }D({\bf v},x)\le \ell\text{ then }h_1^\ell({\bf u},x)= h_1^\ell({\bf v},x)
\end{equation}
with $\ell=k$ are true.
\end{proposition}
 
\begin{proof}
\emph{Necessity}. Suppose that a non-trivial identity $\mathbf u\approx \mathbf v$ holds in $\mathbf H_k$. Proposition~\ref{word problem F_k,K}(i) and the inclusion ${\bf F}_k\subseteq{\bf H}_k$ imply that the claims~\eqref{sim(u)=sim(v) & mul(u)=mul(v)} and~\eqref{eq the same l-dividers} with $\ell=k$ are true. As in the proof of necessity in Proposition~\ref{word problem F_k,K}(i), we can assume that $\mathbf u=\mathbf p\xi(\mathbf a)\mathbf q$ and $\mathbf v=\mathbf p\xi(\mathbf b)\mathbf q$ for some possibly empty words $\mathbf p$ and $\mathbf q$, an endomorphism $\xi$ of $F^1$ and an identity $\mathbf a\approx \mathbf b\in\{\Phi,\beta_k\}$. 
 
If $\mathbf a\approx\mathbf b\in\Phi$ then the claim~\eqref{eq the same l-dividers} is true for any $\ell$ by Proposition~\ref{word problem F_k,K}(ii). Evidently, this implies the required conclusion. Suppose now that $\mathbf a\approx \mathbf b$ coincides with $\beta_k$. Then 
\begin{align*}
\xi(\mathbf a)={}&\mathbf a_{k+1}\mathbf a_k\mathbf a_{k+1}\mathbf a_{k-1}\mathbf a_k\mathbf a_{k-2}\mathbf a_{k-1}\cdots \mathbf a_1\mathbf a_2\mathbf a_0\mathbf a_1,\\
\xi(\mathbf b)={}&\mathbf a_k\mathbf a_{k+1}^2\mathbf a_{k-1}\mathbf a_k\mathbf a_{k-2}\mathbf a_{k-1}\cdots \mathbf a_1\mathbf a_2\mathbf a_0\mathbf a_1
\end{align*}
for some words $\mathbf a_0,\mathbf a_1,\dots,\mathbf a_k$ and $\mathbf a_{k+1}$, whence 
\begin{align*}
\mathbf u={}&\mathbf p\mathbf a_{k+1}\mathbf a_k\mathbf a_{k+1}\mathbf a_{k-1}\mathbf a_k\mathbf a_{k-2}\mathbf a_{k-1}\cdots \mathbf a_1\mathbf a_2\mathbf a_0\mathbf a_1\mathbf q,\\
\mathbf v={}&\mathbf p\mathbf a_k\mathbf a_{k+1}^2\mathbf a_{k-1}\mathbf a_k\mathbf a_{k-2}\mathbf a_{k-1}\cdots \mathbf a_1\mathbf a_2\mathbf a_0\mathbf a_1\mathbf q.
\end{align*}
By Lemma~\ref{depth and index}, $D(\mathbf a,x),D(\mathbf a,x_k)>k-1$. Then Lemma~\ref{does not contain dividers} implies that the subword $\mathbf a_{k+1}\mathbf a_k\mathbf a_{k+1}$ of $\bf u$ located between $\mathbf p$ and $\mathbf a_{k-1}$ is contained in some \mbox{$(k-1)$}-block. Besides that, in view of Lemma~\ref{does not contain dividers}, both occurrences of the word $\mathbf a_{k+1}$ in $\bf u$ do not contain any $k$-dividers of $\bf u$ because $D(\mathbf u,x)>k$ by Lemma~\ref{depth and index}. This means that the words $\bf u$ and $\bf v$ are $k$-equivalent. Now Lemma~\ref{k-equivalent} applies with the conclusion that the claim~\eqref{h_1^l(u,x)= h_1^l(v,x)} with $\ell=k$ is true. 
 
\smallskip
 
\emph{Sufficiency}. The outline of our arguments here is the same as in the proof of sufficiency in Proposition~\ref{word problem F_k,K}(i). But the canonical form of a \mbox{$(k-1)$}-block of \textbf u looks more complicated here than in that proposition.
 
Suppose that the claims~\eqref{sim(u)=sim(v) & mul(u)=mul(v)},~\eqref{eq the same l-dividers} and~\eqref{h_1^l(u,x)= h_1^l(v,x)} with $\ell=k$ are true. Let~\eqref{t_0u_0t_1u_1 ... t_mu_m} be the \mbox{$(k-1)$}-decomposition of \textbf u. Let us fix an index $i\in\{0,1,\dots,m\}$. Let
\begin{equation}
\label{presentation for t_iu_i}
t_i\mathbf u_i=s_0\mathbf a_0s_1\mathbf a_1\cdots s_n\mathbf a_n
\end{equation}
be the presentation of the word $t_i\mathbf u_i$ as the product of alternating $k$-dividers $s_0,s_1,\dots, s_n$ and $k$-blocks $\mathbf a_0, \mathbf a_1,\dots,\mathbf a_n$. Put $\mathbf u_i^\ast=\mathbf a_0\mathbf a_1\cdots\mathbf a_n$. Let $\con(\mathbf u_i^\ast)=\{x_1,x_2,\dots,x_p\}$ and
$$
\overline{\mathbf u_i}\,=x_1^2x_2^2\cdots x_p^2s_1 s_2\cdots s_n.
$$
As we will see below, $\overline{\mathbf u_i}$ is nothing but the mentioned above ``canonical form'' of the \mbox{$(k-1)$}-block $\mathbf u_i$.
 
Clearly, $\mathbf u=\mathbf w_1\mathbf u_i\mathbf w_2$ for some possibly empty words $\mathbf w_1$ and $\mathbf w_2$. Suppose that $x\in\con(\mathbf u_i^\ast)$ but $x\notin\con(\mathbf w_1)$. If $x$ is simple in $\mathbf u_i$ then $x$ is a $k$-divider of \textbf u but this is not the case. Therefore, $x$ is multiple in $\mathbf u_i$. Since $x\notin\con(\mathbf w_1)$, this means that the first and the second occurrences of $x$ in \textbf u lie in the same \mbox{$(k-1)$}-block of \textbf u, whence $D(\mathbf u,x)>k$. Further, Lemma~\ref{k-divider and depth} implies that $D(\mathbf u,s_j)=k$ for all $j=1,\dots,n$. We see that if $a\in\con(\mathbf u_i^\ast)$ and $b\in\{s_1,s_2,\dots,s_n\}$ then either $a\in\con(\mathbf w_1)$ or $D(\mathbf u,a)\ne D(\mathbf u,b)$. Now the assertions~(ii) and~(iii) of Lemma~\ref{u'abu''=u'bau''} apply with the conclusion that the identities 
$$
\mathbf u=\mathbf w_1\mathbf u_i\mathbf w_2\approx\mathbf w_1\mathbf u_i^\ast s_1s_2\cdots s_n\mathbf w_2
$$
hold in $\mathbf H_k$. As we have seen above, if $x\in\con(\mathbf u_i^\ast)\setminus\con(\mathbf w_1)$ then $\occ_x(\mathbf u_i^\ast)\ge 2$. Further, if $x\in\con(\mathbf w_1)\,\cap\,\con(\mathbf u_i^\ast)$ then we can apply the identity~\eqref{xyx=xyxx} and obtain $\occ_x(\mathbf u_i^\ast)\ge 2$ too. Now Lemma~\ref{identities in K}(ii) applies with the conclusion that $\occ_x(\mathbf u_i^\ast)=2$ for any $x\in\con(\mathbf u_i^\ast)$. Then Lemma~\ref{identities in K}(iii) applies and we conclude that the identities
$$
\mathbf u\approx\mathbf w_1\mathbf u_i^\ast s_1s_2\cdots s_n\mathbf w_2\approx\mathbf w_1\,\overline{\mathbf u_i}\,\mathbf w_2
$$
hold in $\mathbf H_k$. 
 
So, as in the proof of Proposition~\ref{word problem F_k,K}(i), using identities which hold in the variety $\mathbf H_k$, we can replace the \mbox{$(k-1)$}-blocks $\mathbf u_i$ of \textbf u successively, one after another, by the ``canonical form'' $\overline{\mathbf u_i}$ for $i=m,m-1,\dots,0$. Then the variety $\mathbf H_k$ satisfies the identities~\eqref{u = canonical form}. Put $\mathbf u^\sharp=t_0\,\overline{\mathbf u_0}\,t_1\,\overline{\mathbf u_1}\,\cdots t_m\,\overline{\mathbf u_m}$\,.
 
One can return to the word \textbf v. By Lemma~\ref{k-equivalent}, the \mbox{$(k-1)$}-decomposition of $\bf v$ has the form~\eqref{t_0v_0t_1v_1 ... t_mv_m}. By~\eqref{h_1^l(u,x)= h_1^l(v,x)} and Lemma~\ref{k-equivalent}, the words $\bf u$ and $\bf v$ are $k$-equivalent. This means that the word $t_i\mathbf v_i $ is the product of alternating $k$-dividers $s_0,s_1,\dots,s_n$ and $k$-blocks $\mathbf b_0,\mathbf b_1,\dots,\mathbf b_n$, i.e.,
\begin{equation}
\label{presentation for t_iv_i}
t_i\mathbf v_i=s_0\mathbf b_0s_1\mathbf b_1\cdots s_n\mathbf b_n.
\end{equation}
The claim~\eqref{eq the same l-dividers} with $\ell=k$ implies that $j$th occurrence of a letter $x$ in $\bf u$ lies in the \mbox{$(k-1)$}-block $\mathbf u_i$ if and only if $j$th occurrence of a letter $x$ in $\bf v$ lies in the \mbox{$(k-1)$}-block $\mathbf v_i$ for any $x$ and any $j=1,2$. Also, Lemma~\ref{identities in K}(ii) allows us to assume that if the first and the second occurrences of the letter $x$ in $\bf u$ do not lie in the \mbox{$(k-1)$}-block $\mathbf u_i$ then this letter does not occur in $\mathbf u_i$. Then $\con(\mathbf u_i^\ast)=\con(\mathbf b_0\mathbf b_1\cdots\mathbf b_n)$. This implies that the \mbox{$(k-1)$}-blocks $\mathbf u_i$ and $\mathbf v_i$ have the same ``canonical form''. Repeating literally arguments given above, we obtain the variety $\mathbf H_k$ satisfies the identities $\mathbf v\approx\mathbf u^\sharp\approx\mathbf u$.
\end{proof}
 
Now we are well prepared to quickly complete the proof of Lemma~\ref{F_k or over H_k}. Let $\mathbf F_k\subset\mathbf{X\subseteq F}_{k+1}$. We have to verify that $\mathbf{X \supseteq H}_k$. Arguing by contradiction, suppose that ${\bf H}_k\nsubseteq{\bf X}$. Then there exists an identity ${\bf u}\approx {\bf v}$ that holds in ${\bf X}$ but does not hold in ${\bf H}_k$. Propositions~\ref{word problem F_k,K}(i) and~\ref{word problem H_k} and the inclusion ${\bf F}_k \subset {\bf X}$ together imply that the claims~\eqref{sim(u)=sim(v) & mul(u)=mul(v)} and~\eqref{eq the same l-dividers} are true, while the claim~\eqref{h_1^l(u,x)= h_1^l(v,x)} with $\ell=k$ is false. According to Lemma~\ref{the same l-dividers}, the words $\mathbf u$ and $\mathbf v$ have the same set of $k$-dividers but $\mathbf u$ and $\mathbf v$ are not $k$-equivalent by Lemma~\ref{k-equivalent}. Then there are $k$-dividers $a$, $b$ of the words $\bf u$, $\bf v$ such that $\ell_1(\mathbf u, a)< \ell_1(\mathbf u, b)$, while $\ell_1(\mathbf v, b)< \ell_1(\mathbf v, a)$. In view of Lemma~\ref{k-divider and depth}, $D(\mathbf u, a),D(\mathbf u, b)\le k$. Suppose that $D(\mathbf u, a)=r<k$. According to Lemma~\ref{h_i^k(u,x)=h_i^k(v,x) to h_i^s(u,x)=h_i^s(v,x)}, the claim~\eqref{eq the same l-dividers} with $\ell=r$ is true. Then Lemma~\ref{D(u,x)=k iff D(v,x)=k} implies that $D(\mathbf v, a)=r$. Also the words $\bf u$ and $\bf v$ are $r$-equivalent by Lemma~\ref{k-equivalent}. Put $c=h_1^r(\mathbf u, b)$. Since $a$ is an $r$-divider of $\bf u$ by Lemma~\ref{k-divider and depth}, first occurrence of $a$ in \textbf u precedes first occurrence of $c$ in \textbf u. On the other hand, the claim~\eqref{eq the same l-dividers} with $\ell=r$ implies that $c=h_1^r(\mathbf v, b)$, whence $\ell_1(\mathbf v, c)< \ell_1(\mathbf v, a)$. This contradicts the fact that the words $\bf u$ and $\bf v$ are $r$-equivalent. So, $D(\mathbf u, a)=k$. Analogously, $D(\mathbf u, b)=k$. Now Lemma~\ref{V=F_k} with $s=k$ applies and we conclude that ${\bf X} \subseteq {\bf F}_k$, a contradiction. Lemma~\ref{F_k or over H_k} is proved.\qed
 
\subsubsection{If $\mathbf H_k\subset\mathbf X\subseteq\mathbf F_{k+1}$ then $\mathbf I_k\subseteq\mathbf X$}
\label{structure of [F_k,F_{k+1}] 2 step}
 
The second step in the verification of the claim~4) of Proposition~\ref{L(K)} is the following
 
\begin{lemma}
\label{H_k or over I_k}
If $\mathbf X$ is a monoid variety such that $\mathbf X \in [\mathbf H_k, \mathbf F_{k + 1}]$ then either $\mathbf{X = H}_k$ or $\mathbf{X \supseteq I}_k$.
\end{lemma}
 
To check this fact, we need the following auxiliary result.

\begin{proposition}
\label{word problem I_k}
A non-trivial identity ${\bf u}\approx {\bf v}$ holds in the variety ${\bf I}_k$ if and only if the claims \eqref{sim(u)=sim(v) & mul(u)=mul(v)}, \eqref{eq the same l-dividers} and
\begin{equation}
\label{identities in I_k}
h_1^\ell({\bf u},x)= h_1^\ell({\bf v},x)\text{ for all }x\in \con({\bf u})
\end{equation}
with $\ell=k$ are true.
\end{proposition}
 
\begin{proof}
\emph{Necessity}. Suppose that a non-trivial identity $\mathbf u\approx \mathbf v$ holds in $\mathbf I_k$. Proposition~\ref{word problem H_k} and the inclusion ${\bf H}_k\subseteq{\bf I}_k$ imply that the claims~\eqref{sim(u)=sim(v) & mul(u)=mul(v)} and~\eqref{eq the same l-dividers} with $\ell=k$ are true. As in the proof of Proposition~\ref{word problem F_k,K}(i), we can assume that $\mathbf u=\mathbf p\xi(\mathbf a)\mathbf q$ and $\mathbf v=\mathbf p\xi(\mathbf b)\mathbf q$ for some possibly empty words $\mathbf p$ and $\mathbf q$, an endomorphism $\xi$ of $F^1$ and an identity $\mathbf a\approx \mathbf b\in\{\Phi,\gamma_k\}$. 

If $\mathbf a\approx\mathbf b\in\Phi$ then the claim~\eqref{eq the same l-dividers} is true for any $\ell$ by Proposition~\ref{word problem F_k,K}(ii). Evidently, this implies the required conclusion. Suppose now that $\mathbf a\approx \mathbf b$ coincides with $\gamma_k$. Then 
\begin{align*}
\xi(\mathbf a)={}&\mathbf b_1\mathbf b_0\mathbf b_1\mathbf a_k\mathbf a_{k-1}\mathbf a_k\mathbf a_{k-2}\mathbf a_{k-1}\cdots \mathbf a_1\mathbf a_2\mathbf a_0\mathbf a_1,\\
\xi(\mathbf b)={}&\mathbf b_1\mathbf b_0\mathbf a_k\mathbf b_1\mathbf a_{k-1}\mathbf a_k\mathbf a_{k-2}\mathbf a_{k-1}\cdots \mathbf a_1\mathbf a_2\mathbf a_0\mathbf a_1
\end{align*}
for some words $\mathbf a_0,\mathbf a_1,\dots,\mathbf a_k$ and $\mathbf b_0,\mathbf b_1$, whence 
\begin{align*}
\mathbf u={}&\mathbf p\mathbf b_1\mathbf b_0\mathbf b_1\mathbf a_k\mathbf a_{k-1}\mathbf a_k\mathbf a_{k-2}\mathbf a_{k-1}\cdots \mathbf a_1\mathbf a_2\mathbf a_0\mathbf a_1\mathbf q,\\
\mathbf v={}&\mathbf p\mathbf b_1\mathbf b_0\mathbf a_k\mathbf b_1\mathbf a_{k-1}\mathbf a_k\mathbf a_{k-2}\mathbf a_{k-1}\cdots \mathbf a_1\mathbf a_2\mathbf a_0\mathbf a_1\mathbf q.
\end{align*}
By Lemma~\ref{depth and index}, $D(\mathbf a,x_k)=k$. Then Lemma~\ref{does not contain dividers} implies that the subword $\mathbf a_k$ of $\bf u$ located between $\mathbf b_1$ and $\mathbf a_{k-1}$ does not contain any \mbox{$(k-1)$}-divider. Also, obviously, the subword $\mathbf b_1$ of $\bf u$ located between $\mathbf b_0$ and $\mathbf a_k$ does not contain any $s$-divider for all $s$. Therefore, the subword $\mathbf b_1\mathbf a_k$ of $\bf u$ located between $\mathbf b_0$ and $\mathbf a_{k-1}$ lies in some \mbox{$(k-1)$}-block. It is evident that the subword $\mathbf b_1$ of $\bf u$ located between $\mathbf b_0$ and $\mathbf a_k$ does not contain first occurrence of any letter in $\bf u$. This implies that the claim~\eqref{identities in I_k} with $\ell=k$ is true.
 
\smallskip
 
\emph{Sufficiency}. As in the proof of Proposition~\ref{word problem H_k}, the outline of our arguments here is similar to one from the proof of sufficiency in Proposition~\ref{word problem F_k,K}(i). But the canonical form of the block here is even more complicated than in Proposition~\ref{word problem H_k}.
 
Suppose that the claims~\eqref{sim(u)=sim(v) & mul(u)=mul(v)},~\eqref{eq the same l-dividers} and~\eqref{identities in I_k} with $\ell=k$ are true. As in the proof of sufficiency in Proposition~\ref{word problem H_k}, we suppose that~\eqref{t_0u_0t_1u_1 ... t_mu_m} is the \mbox{$(k-1)$}-decomposition of $\bf u$, while~\eqref{presentation for t_iu_i} is the representation of $t_i\mathbf u_i $ as the product of alternating $k$-dividers $s_0,s_1,\dots, s_n$ and $k$-blocks $\mathbf a_0, \mathbf a_1,\dots,\mathbf a_n$.
 
For any $j=0,1,\dots,n$, we put
$$
X_j=\{x\in\con(\mathbf a_j)\mid\text{ first occurrence of }x\text{ in }\mathbf u\text{ lies in }\mathbf a_j\}.
$$
Note that the set $X_j$ may be defined by another (of course, equivalent) way. Namely, it is clear that a letter $x$ lies in $X_j$ if and only if it occurs in the $k$-block $\mathbf a_j$ and the \mbox{$(1,k)$}-restrictor of $x$ in \textbf u coincides with the $k$-divider of \textbf u that immediately precedes $\mathbf a_j$. In other words,
$$
X_j=\bigl\{x\in \con(\mathbf a_j)\mid s_j=h_1^k(\mathbf u,x)\bigr\}.
$$
Put $X=X_0\cup X_1\cup\cdots \cup X_n$, $\mathbf a_j'=(\mathbf a_j)_X$ and $\mathbf u_i^\ast=\mathbf a_0'\mathbf a_1'\cdots\mathbf a_n'$. Let $X_j=\{x_{j1},x_{j2},\dots$, $x_{jq_j}\}$, $\con(\mathbf u_i^\ast)=\{c_1,c_2,\dots,c_p\}$ and 
$$
\overline{\mathbf u_i}\!=\!(c_1c_2\cdots c_p)\cdot(x_{01}^2\cdots x_{0q_0}^2)\cdot (s_1x_{11}^2\cdots x_{1q_1}^2) \cdot (s_2x_{21}^2\cdots x_{2q_2}^2) \cdots (s_nx_{n1}^2\cdots x_{nq_n}^2).
$$
As we will see below, $\overline{\mathbf u_i}$ is nothing but the ``canonical form'' of the \mbox{$(k-1)$}-block $\mathbf u_i$ mentioned above.
 
Clearly, $\mathbf u=\mathbf w_1\mathbf u_i\mathbf w_2$ for some possibly empty words $\mathbf w_1$ and $\mathbf w_2$. Let $x\in X_j$. If $x$ is simple in $\mathbf u_i$ then either $x$ coincides with one of the $k$-dividers $s_1,s_2,\dots,s_n$ or $x\in\con(\mathbf w_1)$. But both these variants contradict the choice of $x$. Therefore, $x$ is multiple in $\mathbf u_i$. In view of the identy~\eqref{xyxzx=xyxz}, we can assume that $\occ_x(\mathbf u_i)=2$. Thus, $\mathbf u=\mathbf ax\mathbf b x\mathbf c$ for possibly empty words \textbf a, \textbf b and \textbf c such that $x\mathbf bx$ is a subword of $\mathbf u_i$. One can verify that the variety $\mathbf I_k$ satisfies the identity $\mathbf u\approx\mathbf ax^2\mathbf{bc}$. If $\mathbf b=\lambda$ then this claim is evident. Let now $\mathbf b\ne\lambda$. Then we apply Lemma~\ref{u'abu''=u'bau''}(ii) and successively transpose second occurrence of $x$ in \textbf u with all the letters of the word $\mathbf b$ from right to left. Thus, we conclude that $\mathbf I_k$ satisfies the identity $\mathbf u\approx\mathbf ax^2\mathbf{bc}$. We can assume without loss of generality that $\ell_1(\mathbf u,x_{j0})<\ell_1(\mathbf u,x_{j1})<\cdots<\ell_1(\mathbf u,x_{jq_j})$. Therefore, $\bf I_k$ satisfies the identity
\begin{equation}
\label{u = long word 12}
\mathbf u \approx \mathbf w_1\cdot(x_{01}^2\cdots x_{0q_0}^2\mathbf a_0')\cdot (s_1x_{11}^2\cdots x_{1q_1}^2\mathbf a_1') \cdots (s_nx_{n1}^2\cdots x_{nq_n}^2\mathbf a_n')\cdot\mathbf w_2.
\end{equation}
The definition of the set $X$ and words of the form $\mathbf a_j'$ imply that $x\in \con(\mathbf w_1)$ for any $x\in \con(\mathbf u_i^\ast)$. Now we can apply Lemma~\ref{u'abu''=u'bau''}(ii) and obtain that the identity
$$
\mathbf u \approx \mathbf w_1\cdot\mathbf u_i^\ast\cdot(x_{01}^2\cdots x_{0q_0}^2)\cdot (s_1x_{11}^2\cdots x_{1q_1}^2) \cdots (s_nx_{n1}^2\cdots x_{nq_n}^2)\cdot\mathbf w_2
$$
holds in $\mathbf I_k$. As we have seen above, $\con(\mathbf u_i^\ast)\subseteq \con(\mathbf w_1)$. Then we can apply the identity~\eqref{xyxzx=xyxz} and obtain the word $\mathbf u_i^\ast$ is linear. Then Lemma~\ref{identities in K}(i) applies and we conclude that $\mathbf I_k$ satisfies the identities 
\begin{align*}
\mathbf u \approx{}&\mathbf w_1\cdot(c_1c_2\cdots c_p)\cdot(x_{01}^2\cdots x_{0q_0}^2)\cdot (s_1x_{11}^2\cdots x_{1q_1}^2) \cdots (s_nx_{n1}^2\cdots x_{nq_n}^2)\cdot\mathbf w_2\\
={}&\mathbf w_1\overline{\mathbf u_i}\,\mathbf w_2.
\end{align*}
 
So, as in the proof of Proposition~\ref{word problem F_k,K}(i), using identities which hold in the variety $\mathbf I_k$, we can replace the \mbox{$(k-1)$}-blocks $\mathbf u_i$ of \textbf u successively, one after another, by the ``canonical form'' $\overline{\mathbf u_i}$ for $i=m,m-1,\dots,0$. Then the variety $\mathbf I_k$ satisfies the identities~\eqref{u = canonical form}. Put $\mathbf u^\sharp=t_0\,\overline{\mathbf u_0}\,t_1\,\overline{\mathbf u_1}\,\cdots t_m\,\overline{\mathbf u_m}$\,.
 
One can return to the word \textbf v. By Lemma~\ref{k-equivalent}, the \mbox{$(k-1)$}-decomposition of $\bf v$ has the form~\eqref{t_0v_0t_1v_1 ... t_mv_m}. Furthermore, the claim~\eqref{identities in I_k} with $\ell=k$ and Lemma~\ref{k-equivalent} imply that~\eqref{presentation for t_iv_i} is a representation of $t_i\mathbf v_i$ as the product of alternating $k$-dividers $s_0,s_1,\dots,s_n$ and $k$-blocks $\mathbf b_0, \mathbf b_1,\dots,\mathbf b_n$. The claim~\eqref{identities in I_k} implies that
$$
X_j=\bigl\{x\in \con(\mathbf b_j) \mid s_j=h_1^k(\mathbf v,x)\bigr\}
$$
for all $j=0,1,\dots,r_i$. Put $\mathbf b_j'=(\mathbf b_j)_X$. The claim~\eqref{eq the same l-dividers} with $\ell=k$ implies that $j$th occurrence of a letter $x$ in $\bf u$ lies in the \mbox{$(k-1)$}-block $\mathbf u_i$ if and only if $j$th occurrence of a letter $x$ in $\bf v$ lies in the \mbox{$(k-1)$}-block $\mathbf v_i$ for any $x$ and any $j=1,2$. Also, Lemma~\ref{identities in K}(ii) allows us to assume that if the first and the second occurrences of the letter $x$ in $\bf u$ do not lie in the \mbox{$(k-1)$}-block $\mathbf u_i$ then this letter does not occur in $\mathbf u_i$. Then $\con(\mathbf u_i^\ast)=\con(\mathbf b_0'\mathbf b_1'\cdots\mathbf b_n')$. This implies that the \mbox{$(k-1)$}-blocks $\mathbf u_i$ and $\mathbf v_i$ have the same ``canonical form''. Repeating literally arguments given above, we obtain the variety $\mathbf I_k$ satisfies the identities $\mathbf v\approx\mathbf u^\sharp\approx\mathbf u$.
\end{proof}
 
Now we are well prepared to quickly complete the proof of Lemma~\ref{H_k or over I_k}. Let $\mathbf H_k\subset\mathbf{X\subseteq F}_{k+1}$. We have to verify that $\mathbf{X \supseteq I}_k$. Arguing by contradiction, suppose that ${\bf I}_k \nsubseteq {\bf X}$. Then there exists an identity ${\bf u}\approx {\bf v}$ that holds in ${\bf X}$ but does not hold in ${\bf I}_k$. Then Propositions~\ref{word problem H_k} and~\ref{word problem I_k} and the inclusion ${\bf H}_k \subset {\bf X}$ together imply that the claims~\eqref{sim(u)=sim(v) & mul(u)=mul(v)},~\eqref{eq the same l-dividers} and~\eqref{h_1^l(u,x)= h_1^l(v,x)} with $\ell=k$ are true, while the claim~\eqref{identities in I_k} with $\ell=k$ is false. Let~\eqref{t_0u_0t_1u_1 ... t_mu_m} be the $k$-decomposition of the word \textbf u. The claim~\eqref{h_1^l(u,x)= h_1^l(v,x)} and Lemma~\ref{k-equivalent} imply that the $k$-decomposition of $\bf v$ has the form~\eqref{t_0v_0t_1v_1 ... t_mv_m}. Since the claim~\eqref{identities in I_k} is false, there is a letter $x$ such that $h_1^k(\mathbf u,x)\ne h_1^k(\mathbf v,x)$. Put $t_i=h_1^k(\mathbf u,x)$ and $t_j=h_1^k(\mathbf v,x)$. Then $i\ne j$. We can assume without loss of generality that $i<j$. Then $\ell_1(\mathbf u,x)<\ell_1(\mathbf u,t_j)$, while $\ell_1(\mathbf v,t_j)<\ell_1(\mathbf u,x)$. Lemma~\ref{k-divider and depth} implies that $D(\mathbf u,t_j)\le k$. Put $D(\mathbf u,t_j)=r$. If $r=0$ then $t_j$ is a $0$-divider of $\bf u$. The claim~\eqref{sim(u)=sim(v) & mul(u)=mul(v)} implies that $t_j$ is a $0$-divider of $\bf v$ too. Then $t_j=h_1^0(\mathbf v,x)$ but $t_j\ne h_1^0(\mathbf u,x)$. In view of Lemma~\ref{h_i^k(u,x)=h_i^k(v,x) to h_i^s(u,x)=h_i^s(v,x)}, the claim~\eqref{eq the same l-dividers} with $\ell=p$ is true for any $1\le p\le k$, a contradiction. Thus, $r\ge 1$. Now Lemma~\ref{V subseteq H_s or J_s^s}(i) with $s=r$ and $x_s=t_j$ applies and we conclude that $\mathbf X \subseteq {\bf H}_r \subseteq {\bf H}_k$, a contradiction. Lemma~\ref{H_k or over I_k} is proved.\qed
 
\subsubsection{If $\mathbf I_k\subset\mathbf X\subseteq\mathbf F_{k+1}$ then $\mathbf J_k^1\subseteq\mathbf X$}
\label{structure of [F_k,F_{k+1}] 3 step}
 
The third step in the verification of the claim~4) of Proposition~\ref{L(K)} is the following
 
\begin{lemma}
\label{I_k or over J_k^1}
If $\mathbf X$ is a monoid variety such that $\mathbf X \in [\mathbf I_k, \mathbf F_{k + 1}]$ then either $\mathbf{X = I}_k$ or $\mathbf{X \supseteq J}_k^1$.
\end{lemma}
 
To check this fact, we need the following
 
\begin{lemma}
\label{V subseteq I_k or J_k^r}
Let ${\bf V}$ be a monoid variety with ${\bf V} \subseteq {\bf K}$ and $\ell$ a natural number. Suppose that ${\bf V}$ satisfies an identity ${\bf u}\approx {\bf v}$.
\begin{itemize}
\item[\textup{(i)}]If the claims~\eqref{sim(u)=sim(v) & mul(u)=mul(v)},~\eqref{eq the same l-dividers} and~\eqref{identities in I_k} are true but the claim
\begin{equation}
\label{eq V subseteq J_k^r}
\text{if }x\in \con({\bf u})\text{ and }D({\bf u},x)\le m\text{ then }h_2^\ell({\bf u},x)= h_2^\ell({\bf v},x)
\end{equation}
with $m=1$ is false then ${\bf V} \subseteq {\bf I}_\ell$.
\item[\textup{(ii)}]If the claims~\eqref{sim(u)=sim(v) & mul(u)=mul(v)},~\eqref{eq the same l-dividers},~\eqref{identities in I_k} and~\eqref{eq V subseteq J_k^r} with $m=r$ are true but the claim \eqref{eq V subseteq J_k^r} with $m=r+1$ is false then ${\bf V} \subseteq {\bf J}_\ell^r$.
\end{itemize}
\end{lemma}
 
\begin{proof}
Proofs of the assertions~(i) and~(ii) have the same initial part. Suppose that the variety \textbf V satisfies the hypothesis of one of these two claims. Then the claims~\eqref{sim(u)=sim(v) & mul(u)=mul(v)},~\eqref{eq the same l-dividers} and~\eqref{identities in I_k} are true. Let $m$ be the least natural number such that the claim~\eqref{eq V subseteq J_k^r} is false. Then there is a letter $y_m$ such that $D({\bf u},y_m)=m$ and $h_2^\ell({\bf u},y_m)\ne h_2^\ell({\bf v},y_m)$. Put $x_\ell= h_2^\ell({\bf u},y_m)$ and $z_\ell= h_2^\ell({\bf v},y_m)$. In view of Lemma~\ref{k-divider and depth}, $D({\bf u},x_\ell)$, $D({\bf u},z_\ell)\le \ell$. Note that either $D({\bf u},x_\ell)=\ell$ or $D({\bf u},z_\ell)=\ell$. Indeed, if $D({\bf u},x_\ell)$, $D({\bf u},z_\ell)<\ell$ then $D({\bf v},x_\ell),D({\bf v},z_\ell)<\ell$ by Lemma~\ref{D(u,x)=k iff D(v,x)=k}. Then $x_\ell$ and $z_\ell$ are \mbox{$(\ell-1)$}-dividers of $\bf u$ and $\bf v$, whence $x_\ell= h_2^{\ell-1}({\bf u},y_m)$ and $z_\ell= h_2^{\ell-1}({\bf v},y_m)$. But this contradicts the claim~\eqref{eq the same l-dividers}. Suppose without loss of generality that $D({\bf u},x_\ell)=\ell$. By symmetry, we may assume that first occurrence of $z_\ell$ in the word $\bf u$ precedes first occurrence of $x_\ell$ in this word. Since the claim~\eqref{identities in I_k} is true, $\ell_1(\mathbf v,z_\ell)<\ell_1(\mathbf v,x_\ell)$ by Lemma~\ref{k-equivalent}. This implies that $\ell_2(\mathbf v,y_m)<\ell_1(\mathbf v,x_\ell)$.
 
Now we apply Lemma~\ref{form of the identity} with $x_s=x_\ell$ and $s=\ell$ and conclude that there are letters $x_0,x_1,\dots, x_{\ell-1}$ such that, for any $p=0,1,\dots,\ell-1$ and $q=0,1,\dots,\ell-2$, $D(\mathbf u,x_p)=D(\mathbf v,x_p)=p$ and the inequalities
$$
\ell_1(\mathbf w,x_{p+1})<\ell_1(\mathbf w,x_p)<\ell_2(\mathbf w,x_{p+1})\text{ and }\ell_2(\mathbf w,x_{q+2})<\ell_1(\mathbf w,x_q)
$$
hold where \textbf w is any of the words \textbf u or \textbf v.
 
Put $y_{m-1}=h_2^{m-1}({\bf u},y_m)$. According to Lemma~\ref{h_2^{k-1}}, $D(\mathbf u, y_{m-1})=m-1$ and $\ell_1(\mathbf u,y_m)<\ell_1(\mathbf u,y_{m-1})$. Besides that, the claim~\eqref{eq the same l-dividers} and Lemma~\ref{h_i^k(u,x)=h_i^k(v,x) to h_i^s(u,x)=h_i^s(v,x)} imply that $h_2^{m-1}({\bf v},y_m)=h_2^{m-1}({\bf u},y_m)=y_{m-1}$. Now we apply Lemma~\ref{h_2^{k-1}} again and obtain $D(\mathbf v, y_{m-1})=m-1$ and $\ell_1(\mathbf v,y_m)<\ell_1(\mathbf v,y_{m-1})$. In view of Lemma~\ref{k-divider and depth}, the letter $x_{\ell-1}$ is an $\ell$-divider of $\bf u$, whence $\ell_2(\mathbf u,y_m)<\ell_1(\mathbf u,x_{\ell-1})$ because $x_\ell= h_2^\ell({\bf u},y_m)$ and $\ell_1(\mathbf u,x_\ell)<\ell_1(\mathbf u,x_{\ell-1})$.
 
Lemma~\ref{identities in K}(ii) allows us to assume that the letters $y_m$ and $x_p$ with $1\le p\le\ell$ occur twice in each of the words \textbf u and \textbf v. Further considerations are divided into two cases corresponding to the statements~(i) and~(ii).
 
\smallskip
 
\emph{Case} 1: $m=1$. In view of the above, we have the identity $\mathbf{u\approx v}$ has the form
\begin{align*}
& \mathbf u_{2\ell+4} y_1 \mathbf u_{2\ell+3} y_0 \mathbf u_{2\ell+2} x_\ell \mathbf u_{2\ell+1} y_1\mathbf u_{2\ell} x_{\ell-1} \mathbf u_{2\ell-1} x_\ell \mathbf u_{2\ell-2} x_{\ell-2} \mathbf u_{2\ell-3} x_{\ell-1} \cdots\\
&\cdot\,\mathbf u_4 x_1 \mathbf u_3 x_2 \mathbf u_2 x_0 \mathbf u_1 x_1 \mathbf u_0\\[-3pt]
\approx{}&\mathbf v_{2\ell+4} y_1 \mathbf v_{2\ell+3} y_0 \mathbf v_{2\ell+2} y_1 \mathbf v_{2\ell+1}x_\ell \mathbf v_{2\ell} x_{\ell-1}\mathbf v_{2\ell-1} x_\ell \mathbf v_{2\ell-2} x_{\ell-2} \mathbf v_{2\ell-3} x_{\ell-1}\cdots\\
&\cdot\,\mathbf v_4 x_1 \mathbf v_3 x_2 \mathbf v_2 x_0 \mathbf v_1 x_1 \mathbf v_0
\end{align*}
for some possibly empty words $\mathbf u_0, \mathbf u_1,\dots,\mathbf u_{2\ell+4}$ and $\mathbf v_0, \mathbf v_1,\dots,\mathbf v_{2\ell+4}$ such that $x_s,y_0$, $y_1\notin\con(\mathbf u_i\mathbf v_i)$ for $0\le s\le \ell$ and $0\le i\le 2\ell+4$. Now substitute~1 for all letters occurring in this identity except $x_0,x_1,\dots,x_\ell,y_0$ and $y_1$. We get the identity
$$
y_1y_0x_\ell y_1x_{\ell-1}x_\ell x_{\ell-2} x_{\ell-1} \cdots x_1 x_2 x_0 x_1\approx y_1y_0 y_1x_\ell x_{\ell-1}x_\ell x_{\ell-2} x_{\ell-1} \cdots x_1 x_2 x_0 x_1,
$$
i.e., the identity $\gamma_\ell$. The claim~(i) is proved.
 
\smallskip
 
\emph{Case} 2: $m>1$. Now we will prove that $\ell_2(\mathbf v, x_m)<\ell_2(\mathbf v, y_{m-1})$ and $\ell_2(\mathbf u, x_m)<\ell_2(\mathbf u, y_{m-1})$. Put $y_{m-2}=h_2^{m-2}(\mathbf v, y_{m-1})$. Since $D(\mathbf v, y_{m-1})=m-1$, Lemma~\ref{h_2^{k-1}} implies that $D(\mathbf v, y_{m-2})=m-2$ and $\ell_1({\bf v},y_{m-1})< \ell_1({\bf v},y_{m-2})$. Recall that $\ell_1(\mathbf v,y_m)<\ell_1({\bf v},y_{m-1})$, whence $\ell_1(\mathbf v,y_m)<\ell_1({\bf v},y_{m-2})$. Since $D(\mathbf v,y_m)=m$, we can apply Lemma~\ref{if first then second} and conclude that $\ell_2({\bf v},y_m)< \ell_1({\bf v},y_{m-2})$. First occurrence of $x_\ell$ in \textbf v precedes second occurrence of $y_m$, whence $\ell_1({\bf v},x_\ell)< \ell_1({\bf v},y_{m-2})$. Then Lemma~\ref{if first then second} implies that $\ell_2({\bf v},x_\ell)< \ell_1({\bf v},y_{m-2})$. This implies that $\ell_1(\mathbf v,x_{\ell-1})<\ell_2({\bf v},x_\ell)< \ell_1({\bf v},y_{m-2})$. If $\ell-1\ge m$ then Lemma~\ref{if first then second} applies with the conclusion that $\ell_2({\bf v},x_{\ell-1})< \ell_1({\bf v},y_{m-2})$. Continuing this process, we eventually obtain $\ell_2(\mathbf v, x_m)<\ell_1(\mathbf v, y_{m-2})$. In particular, $\ell_1(\mathbf v, x_m)<\ell_1(\mathbf v, y_{m-2})$. In view of Lemma~\ref{k-divider and depth}, the letters $x_m$ and $y_{m-2}$ are $\ell$-dividers of $\bf v$. Now Lemma~\ref{k-equivalent} applies with the conclusion that $\ell_1(\mathbf u, x_m)<\ell_1(\mathbf u, y_{m-2})$. Then Lemma~\ref{if first then second} shows that $\ell_2(\mathbf u, x_m)<\ell_1(\mathbf u, y_{m-2})$. The choice of $y_{m-2}$ implies that first occurrence of $y_{m-2}$ in \textbf v precedes second occurrence of $y_{m-1}$. Therefore, $\ell_2(\mathbf v, x_m)<\ell_2(\mathbf v, y_{m-1})$. In view of the claim~\eqref{eq the same l-dividers} and Lemma~\ref{h_i^k(u,x)=h_i^k(v,x) to h_i^s(u,x)=h_i^s(v,x)}, $h_2^{m-2}(\mathbf u, y_{m-1})=h_2^{m-2}(\mathbf v, y_{m-1})=y_{m-2}$, whence $\ell_2(\mathbf u, x_m)<\ell_2(\mathbf u, y_{m-1})$.
 
Let now $m>2$. Note that
$$
\ell_1(\mathbf u,y_{m-1})<\ell_2(\mathbf u,y_m)<\ell_1(\mathbf u,x_\ell)<\ell_1(\mathbf u,x_{\ell-1})<\cdots<\ell_1(\mathbf u,x_{m-3}).
$$
If $\ell_1({\bf u},x_{m-3})<\ell_2({\bf u},y_{m-1})$ then the letter $x_{m-3}$ lies between the first and the second occurrences of $y_{m-1}$ in \textbf u. Since $x_{m-3}$ is an \mbox{$(m-3)$}-divider of \textbf u, we obtain a contradiction with the equality $D(\mathbf u, y_{m-1})=m-1$. Thus, $\ell_2(\mathbf u, y_{m-1})<\ell_1(\mathbf u, x_{m-3})$ whenever $m>2$.
 
Further, there are three possibilities for second occurrence of the letter $y_{m-1}$ in \textbf u:
\begin{align}
\label{2nd occurrence y_{m-1} in u 1}&\ell_2(\mathbf u,y_{m-1})<\ell_1(\mathbf u, x_{m-2});\\
\label{2nd occurrence y_{m-1} in u 2}&\ell_1(\mathbf u, x_{m-2})<\ell_2(\mathbf u, y_{m-1})<\ell_2(\mathbf u, x_{m-1});\\
\label{2nd occurrence y_{m-1} in u 3}&\ell_2(\mathbf u, x_{m-1})<\ell_2(\mathbf u,y_{m-1}).
\end{align}
 
Suppose that the claim~\eqref{2nd occurrence y_{m-1} in u 1} is true. Then $\ell_1(\mathbf u,y_{m-2})<\ell_1(\mathbf u, x_{m-2})$. In view of Lemma~\ref{k-equivalent}, $\ell_1(\mathbf v,y_{m-2})<\ell_1(\mathbf v, x_{m-2})$. Since $y_{m-2}=h_2^{m-2}(\mathbf v, y_{m-1})$ by Lemma~\ref{h_i^k(u,x)=h_i^k(v,x) to h_i^s(u,x)=h_i^s(v,x)}, we have $\ell_2(\mathbf v,y_{m-1})<\ell_1(\mathbf v, x_{m-2})$. Now if $m>2$ then we apply Lemma~\ref{form of the identity} with $x_s=y_{m-2}$ and $s=m-2$ and conclude that there are letters $y_0,y_1,\dots, y_{m-3}$ such that $D(\mathbf u,y_p)=D(\mathbf v, y_p)=p$ and, for any $p=0,1,\dots,m-2$, the inequalities
$$
\ell_1(\mathbf w,y_{p+1})<\ell_1(\mathbf w,y_p)<\ell_2(\mathbf w,y_{p+1})\text{ and }\ell_2(\mathbf w,y_{p+2})<\ell_1(\mathbf w,y_p)
$$
hold where \textbf w is any of the words \textbf u or \textbf v. Lemma~\ref{identities in K}(ii) allows us to assume that the letters $y_p$ with $1\le p\le m$ occur twice in each of the words \textbf u and \textbf v. In view of the above, we have the identity $\mathbf{u\approx v}$ has the form
\begin{align*}
&\mathbf u_{2\ell+5} y_m \mathbf u_{2\ell+4} y_{m-1} \mathbf u_{2\ell+3} x_\ell \mathbf u_{2\ell+2} y_m\mathbf u_{2\ell+1} x_{\ell-1} \mathbf u_{2\ell} x_\ell \mathbf u_{2\ell-1} x_{\ell-2} \mathbf u_{2\ell-2} x_{\ell-1} \cdots\\
&\cdot\,\mathbf u_{2m+1}y_{m-2}\mathbf u_{2m}y_{m-1}\mathbf u_{2m-1}x_{m-1}\mathbf u_{2m-2}y_{m-3}\mathbf u_{2m-2}y_{m-2}\cdots\\
&\cdot\,\mathbf u_4 y_1 \mathbf u_3 y_2 \mathbf u_2 y_0 \mathbf u_1 y_1 \mathbf u_0\\[-3pt]
\approx{}&\mathbf v_{2\ell+5} y_m \mathbf v_{2\ell+4} y_{m-1} \mathbf v_{2\ell+3}y_m \mathbf v_{2\ell+2} x_\ell\mathbf v_{2\ell+1} x_{\ell-1}\mathbf v_{2\ell} x_\ell \mathbf v_{2\ell-1} x_{\ell-2}\mathbf v_{2\ell-2} x_{\ell-1} \cdots\\
&\cdot\,\mathbf v_{2m+1}y_{m-2}\mathbf v_{2m}y_{m-1}\mathbf v_{2m-1}x_{m-1}\mathbf v_{2m-2}y_{m-3}\mathbf v_{2m-2}y_{m-2}\cdots\\
&\cdot\,\mathbf v_4 y_1 \mathbf v_3 y_2\mathbf v_2 y_0 \mathbf v_1 y_1 \mathbf v_0
\end{align*}
for some possibly empty words $\mathbf u_0, \mathbf u_1,\dots,\mathbf u_{2\ell+5}$ and $\mathbf v_0, \mathbf v_1,\dots,\mathbf v_{2\ell+5}$ such that $x_s,y_t\notin\con(\mathbf u_i\mathbf v_i)$ for $m-1\le s\le \ell$, $0\le t\le m$ and $0\le i\le 2\ell+5$. Now substitute~1 for all letters occurring in this identity except $y_0,y_1,\dots,y_m$ and $x_{m-1},x_m,\dots,x_\ell$. We get the identity
\begin{align*}
&y_my_{m-1}x_\ell y_mx_{\ell-1}x_\ell x_{\ell-2} x_{\ell-1} \cdots y_{m-2}y_{m-1}x_{m-1}y_{m-3}y_{m-2}\cdots y_1 y_2 y_0 y_1\\
\approx{}&y_my_{m-1}y_mx_\ell x_{\ell-1}x_\ell x_{\ell-2} x_{\ell-1} \cdots y_{m-2}y_{m-1}x_{m-1}y_{m-3}y_{m-2}\cdots y_1 y_2 y_0 y_1.
\end{align*}
Now we rename in this identity $y_i$ in $x_i$ for $i=0,1,\dots,m-2$ and obtain the identity
\begin{equation}
\label{almost delta_k^{m-1}}
\begin{array}{rl}
&y_m\stackrel{(1)}{y_{m-1}}x_\ell y_mx_{\ell-1}x_\ell x_{\ell-2} x_{\ell-1} \cdots x_{m-2}\stackrel{(2)}{y_{m-1}}\stackrel{(2)}{x_{m-1}}x_{m-3}x_{m-2}\cdots\\
&\cdot\,x_1 x_2 x_0 x_1\\
\approx{}&y_m\stackrel{(1)}{y_{m-1}}y_mx_\ell x_{\ell-1}x_\ell x_{\ell-2} x_{\ell-1} \cdots x_{m-2}\stackrel{(2)}{y_{m-1}}\stackrel{(2)}{x_{m-1}}x_{m-3}x_{m-2}\cdots\\
&\cdot\,x_1 x_2 x_0 x_1.
\end{array}
\end{equation}
In view of Lemma~\ref{identities in K}(i), we may use the identity $\sigma_2$. This identity allows us to swap the second occurrences of the letters $x_{m-1}$ and $y_{m-1}$ in both the sides of the identity~\eqref{almost delta_k^{m-1}}. As a result, we get $\delta_\ell^{m-1}$.
 
Suppose now that the claim~\eqref{2nd occurrence y_{m-1} in u 2} is true. If $\ell_2(\mathbf v,y_{m-1})<\ell_1(\mathbf v, x_{m-2})$ then $\ell_1(\mathbf v,y_{m-2})<\ell_1(\mathbf v, x_{m-2})$. In view of Lemma~\ref{k-equivalent}, $\ell_1(\mathbf u,y_{m-2})<\ell_1(\mathbf u, x_{m-2})$. Since $y_{m-2}=h_2^{m-2}(\mathbf v, y_{m-1})$ by Lemma~\ref{h_i^k(u,x)=h_i^k(v,x) to h_i^s(u,x)=h_i^s(v,x)}, we have $\ell_2(\mathbf u,y_{m-1})<\ell_1(\mathbf u, x_{m-2})$. This contradicts the claim~\eqref{2nd occurrence y_{m-1} in u 2}. Thus, $\ell_1(\mathbf v,x_{m-2})<\ell_2(\mathbf v, y_{m-1})$.{\sloppy

}
 
Suppose that $\ell_2(\mathbf v,y_{m-1})<\ell_2(\mathbf v, x_{m-1})$. Then the identity $\mathbf{u\approx v}$ has the form
\begin{align*}
&\mathbf u_{2\ell+5} y_m \mathbf u_{2\ell+4} y_{m-1} \mathbf u_{2\ell+3} x_\ell \mathbf u_{2\ell+2} y_m\mathbf u_{2\ell+1} x_{\ell-1} \mathbf u_{2\ell} x_\ell \mathbf u_{2\ell-1} x_{\ell-2} \mathbf u_{2\ell-2} x_{\ell-1} \cdots\\
&\cdot\,\mathbf u_{2m+1}x_{m-2}\mathbf u_{2m}y_{m-1}\mathbf u_{2m-1}x_{m-1}\mathbf u_{2m-2}x_{m-3}\mathbf u_{2m-2}x_{m-2}\cdots\\
&\cdot\,\mathbf u_4 x_1 \mathbf u_3 x_2 \mathbf u_2 x_0 \mathbf u_1 x_1 \mathbf u_0\\
\approx{}&\mathbf v_{2\ell+5} y_m \mathbf v_{2\ell+4} y_{m-1} \mathbf v_{2\ell+3}y_m \mathbf v_{2\ell+2} x_\ell\mathbf v_{2\ell+1} x_{\ell-1} \mathbf v_{2\ell} x_\ell \mathbf v_{2\ell-1} x_{\ell-2}\mathbf v_{2\ell-2} x_{\ell-1} \cdots\\
&\cdot\,\mathbf v_{2m+1}x_{m-2}\mathbf v_{2m}y_{m-1}\mathbf v_{2m-1}x_{m-1}\mathbf v_{2m-2}x_{m-3}\mathbf v_{2m-2}x_{m-2}\cdots\\
&\cdot\,\mathbf v_4 x_1 \mathbf v_3 x_2\mathbf v_2 x_0 \mathbf v_1 x_1 \mathbf v_0
\end{align*}
for some possibly empty words $\mathbf u_0, \mathbf u_1,\dots,\mathbf u_{2\ell+5}$ and $\mathbf v_0, \mathbf v_1,\dots,\mathbf v_{2\ell+5}$ such that $x_s,y_{m-1}$, $y_m\notin\con(\mathbf u_i\mathbf v_i)$ for $0\le s\le \ell$ and $0\le i\le 2\ell+5$. Now substitute~1 for all letters occurring in this identity except $y_{m-1},y_m,x_0,x_1,\dots,x_\ell$. We get the identity~\eqref{almost delta_k^{m-1}}. As above, combining this identity with $\sigma_2$, we get $\delta_\ell^{m-1}$.
 
If $\ell_2(\mathbf v, x_{m-1})<\ell_2(\mathbf v,y_{m-1})$ then the same arguments as above show that the identity
\begin{align*}
&y_m\stackrel{(1)}{y_{m-1}}x_\ell y_mx_{\ell-1}x_\ell x_{\ell-2} x_{\ell-1} \cdots x_{m-2}\stackrel{(2)}{y_{m-1}}\stackrel{(2)}{x_{m-1}}x_{m-3}x_{m-2}\cdots x_1 x_2 x_0 x_1\\
\approx{}&y_my_{m-1}y_mx_\ell x_{\ell-1}x_\ell x_{\ell-2} x_{\ell-1} \cdots x_{m-2}x_{m-1}y_{m-1}x_{m-3}x_{m-2}\cdots x_1 x_2 x_0 x_1
\end{align*}
holds in \textbf V. Now we apply the identity $\sigma_2$ to the left-hand side of the last identity and get $\delta_\ell^{m-1}$.
 
Finally, suppose that the claim~\eqref{2nd occurrence y_{m-1} in u 3} is true. Suppose that $\ell_2(\mathbf v,y_{m-1})<\ell_2(\mathbf v, x_{m-1})$. Then the identity $\mathbf{u\approx v}$ has the form
\begin{align*}
&\mathbf u_{2\ell+5} y_m \mathbf u_{2\ell+4} y_{m-1} \mathbf u_{2\ell+3} x_\ell \mathbf u_{2\ell+2} y_m\mathbf u_{2\ell+1} x_{\ell-1} \mathbf u_{2\ell} x_\ell \mathbf u_{2\ell-1} x_{\ell-2} \mathbf u_{2\ell-2} x_{\ell-1} \cdots\\
&\cdot\,\mathbf u_{2m+1}x_{m-2}\mathbf u_{2m}x_{m-1}\mathbf u_{2m-1}y_{m-1}\mathbf u_{2m-2}x_{m-3}\mathbf u_{2m-2}x_{m-2}\cdots\\
&\cdot\,\mathbf u_4 x_1 \mathbf u_3 x_2 \mathbf u_2 x_0 \mathbf u_1 x_1 \mathbf u_0\\
\approx{}&\mathbf v_{2\ell+5} y_m \mathbf v_{2\ell+4} y_{m-1} \mathbf v_{2\ell+3}y_m \mathbf v_{2\ell+2} x_\ell\mathbf v_{2\ell+1} x_{\ell-1} \mathbf v_{2\ell} x_\ell \mathbf v_{2\ell-1} x_{\ell-2} \mathbf v_{2\ell-2} x_{\ell-1} \cdots\\
&\cdot\,\mathbf v_{2m+1}x_{m-2}\mathbf v_{2m}y_{m-1}\mathbf v_{2m-1}x_{m-1}\mathbf v_{2m-2}x_{m-3}\mathbf v_{2m-2}x_{m-2}\cdots\\
&\cdot\,\mathbf v_4 x_1 \mathbf v_3 x_2\mathbf v_2 x_0 \mathbf v_1 x_1 \mathbf v_0
\end{align*}
for some possibly empty words $\mathbf u_0, \mathbf u_1,\dots,\mathbf u_{2\ell+5}$ and $\mathbf v_0, \mathbf v_1,\dots,\mathbf v_{2\ell+5}$ such that $x_s,y_{m-1}$, $y_m\notin\con(\mathbf u_i\mathbf v_i)$ for $0\le s\le \ell$ and $0\le i\le 2\ell+5$. Now substitute~1 for all letters occurring in this identity except $y_{m-1},y_m,x_0,x_1,\dots,x_\ell$. We get the identity
\begin{align*}
&y_my_{m-1}x_\ell y_mx_{\ell-1}x_\ell x_{\ell-2} x_{\ell-1} \cdots x_{m-2}x_{m-1}y_{m-1}x_{m-3}x_{m-2}\cdots x_1 x_2 x_0 x_1\\[-3pt]
\approx{}&y_m\stackrel{(1)}{y_{m-1}}y_mx_\ell x_{\ell-1}x_\ell x_{\ell-2} x_{\ell-1} \cdots x_{m-2}\stackrel{(2)}{y_{m-1}}\stackrel{(2)}{x_{m-1}}x_{m-3}x_{m-2}\cdots x_1 x_2 x_0 x_1.
\end{align*}
Applying once again $\sigma_2$ to the right-hand side of this identity, we get $\delta_\ell^{m-1}$.
 
If $\ell_2(\mathbf v, x_{m-1})<\ell_2(\mathbf v,y_{m-1})$ then the same arguments as above show that the identity $\delta_\ell^{m-1}$ holds in \textbf V.
\end{proof}
 
\begin{proposition}
\label{word problem J_k^r}
A non-trivial identity ${\bf u}\approx {\bf v}$ holds in the variety ${\bf J}_k^r$ if and only if the claims~\eqref{sim(u)=sim(v) & mul(u)=mul(v)},~\eqref{eq the same l-dividers},~\eqref{identities in I_k} and~\eqref{eq V subseteq J_k^r} with $\ell=k$ and $m=r$ are true.
\end{proposition}
 
\begin{proof}
\emph{Necessity}. Suppose that a non-trivial identity ${\bf u}\approx {\bf v}$ holds in ${\bf J}_k^r$. The claims~\eqref{sim(u)=sim(v) & mul(u)=mul(v)},~\eqref{eq the same l-dividers} and~\eqref{identities in I_k} with $\ell=k$ follow from Proposition~\ref{word problem I_k} and the inclusion ${\bf I}_k\subseteq{\bf J}_k^r$. It remains to verify that the claim~\eqref{eq V subseteq J_k^r} with $\ell=k$ and $m=r$ is true. As in the proof of Proposition~\ref{word problem F_k,K}(i), we can assume that $\mathbf u=\mathbf p\xi(\mathbf a)\mathbf q$ and $\mathbf v=\mathbf p\xi(\mathbf b)\mathbf q$ for some possibly empty words $\mathbf p$ and $\mathbf q$, an endomorphism $\xi$ of $F^1$ and an identity $\mathbf a\approx \mathbf b\in\{\Phi,\delta_k^r\}$.

If $\mathbf a\approx\mathbf b\in\Phi$ then the claim~\eqref{eq the same l-dividers} is true for any $\ell$ by Proposition~\ref{word problem F_k,K}(ii). Evidently, this implies the required conclusion. Suppose now that $\mathbf a\approx \mathbf b$ coincides with $\delta_k^r$. Then 
\begin{align*}
\xi(\mathbf a)={}&\mathbf b_{r+1}\mathbf b_r\mathbf b_{r+1}\mathbf a_k\mathbf a_{k-1}\mathbf a_k\mathbf a_{k-2}\mathbf a_{k-1}\cdots \mathbf a_{r-1}\mathbf a_r\mathbf b_r\mathbf a_{r-2}\mathbf a_{r-1}\cdots \mathbf a_1\mathbf a_2\mathbf a_0\mathbf a_1,\\
\xi(\mathbf b)={}&\mathbf b_{r+1}\mathbf b_r\mathbf a_k\mathbf b_{r+1}\mathbf a_{k-1}\mathbf a_k\mathbf a_{k-2}\mathbf a_{k-1}\cdots \mathbf a_{r-1}\mathbf a_r\mathbf b_r\mathbf a_{r-2}\mathbf a_{r-1}\cdots \mathbf a_1\mathbf a_2\mathbf a_0\mathbf a_1
\end{align*}
for some words $\mathbf a_0,\mathbf a_1,\dots,\mathbf a_k$ and $\mathbf b_r,\mathbf b_{r+1}$, whence 
\begin{align*}
\mathbf u={}&\mathbf p\mathbf b_{r+1}\mathbf b_r\mathbf b_{r+1}\mathbf a_k\mathbf a_{k-1}\mathbf a_k\mathbf a_{k-2}\mathbf a_{k-1}\cdots \mathbf a_{r-1}\mathbf a_r\mathbf b_r\mathbf a_{r-2}\mathbf a_{r-1}\cdots \mathbf a_1\mathbf a_2\mathbf a_0\mathbf a_1\mathbf q,\\
\mathbf v={}&\mathbf p\mathbf b_{r+1}\mathbf b_r\mathbf a_k\mathbf b_{r+1}\mathbf a_{k-1}\mathbf a_k\mathbf a_{k-2}\mathbf a_{k-1}\cdots \mathbf a_{r-1}\mathbf a_r\mathbf b_r\mathbf a_{r-2}\mathbf a_{r-1}\cdots \mathbf a_1\mathbf a_2\mathbf a_0\mathbf a_1\mathbf q.
\end{align*}
By Lemma~\ref{depth and index}, $D(\mathbf a,x_k)=k$. Then Lemma~\ref{does not contain dividers} implies that the subword $\mathbf a_k$ of $\bf u$ located between $\mathbf b_{r+1}$ and $\mathbf a_{k-1}$ does not contain any \mbox{$(k-1)$}-divider. Also, obviously, the subword $\mathbf b_{r+1}$ of $\bf u$ located between $\mathbf b_r$ and $\mathbf a_k$ does not contain any $s$-divider for all $s$. Therefore, the subword $\mathbf b_{r+1}\mathbf a_k$ of $\bf u$ located between $\mathbf b_r$ and $\mathbf a_{k-1}$ lies in some \mbox{$(k-1)$}-block. Now we apply Lemma~\ref{does not contain dividers} again and obtain the subword $\mathbf b_{r+1}$ located between $\bf p$ and $\mathbf b_r$ does not contain any $s$-divider for all $s\le r$. Hence if second occurrence in $\bf u$ of some letter lies in the subword $\mathbf b_{r+1}$ located between $\mathbf b_r$ and $\mathbf a_k$ then the depth of this letter is more than $r$. This implies that the claim~\eqref{eq V subseteq J_k^r} with $\ell=k$ and $m=r$ is true.
 
\smallskip
 
\emph{Sufficiency}. As in the proofs of Propositions~\ref{word problem H_k} and~\ref{word problem I_k}, the outline of our arguments here is similar to one from the proof of sufficiency in Proposition~\ref{word problem F_k,K}(i). But the canonical form of the block here is even more complicated than in Proposition~\ref{word problem I_k}.
 
Suppose that the claims~\eqref{sim(u)=sim(v) & mul(u)=mul(v)},~\eqref{eq the same l-dividers},~\eqref{identities in I_k} and~\eqref{eq V subseteq J_k^r} with $\ell=k$ and $m=r$ are true. As in the proof of sufficiency in Proposition~\ref{word problem H_k}, we suppose that~\eqref{t_0u_0t_1u_1 ... t_mu_m} is the \mbox{$(k-1)$}-decomposition of $\bf u$, while~\eqref{presentation for t_iu_i} is the representation of $t_i\mathbf u_i $ as the product of alternating $k$-dividers $s_0,s_1,\dots, s_n$ and $k$-blocks $\mathbf a_0, \mathbf a_1,\dots,\mathbf a_n$. 
 
Clearly, $\mathbf u=\mathbf w_1\mathbf u_i\mathbf w_2$ for some possibly empty words $\mathbf w_1$ and $\mathbf w_2$. For $j=0,1,\dots,n$, we put
$$
X_j=\{x\in\con(\mathbf a_j)\mid\text{ first occurrence of }x\text{ in }\mathbf u\text{ lies in }\mathbf a_j\}.
$$
Let $X_j=\{x_{j1},x_{j2},\dots,x_{jq_j}\}$, $X=X_0\cup X_1\cup\cdots\cup X_n$, $\mathbf a_j'=(\mathbf a_j)_X$. As in the proof of sufficiency in Proposition~\ref{word problem I_k}, we can verify that
$$
X_j=\bigl\{x\in \con(\mathbf a_j)\mid s_j=h_1^k(\mathbf u,x)\bigr\}.
$$
For any $j=0,1,\dots,n$, we put
$$
Z_j=\{z\in \con(\mathbf a_j')\mid D(\mathbf u,z)\le r\},
$$
$Z=Z_0\cup Z_1\cup\cdots\cup Z_n$, $\mathbf a_j''=(\mathbf a_j')_Z$ and $\mathbf u_i^\ast=\mathbf a_0''\mathbf a_1''\cdots\mathbf a_n''$. Let $Z_j=\{z_{j1},z_{j2},\dots$, $z_{jh_j}\}$, $\con(\mathbf u_i^\ast)=\{c_1,c_2,\dots,c_p\}$ and 
\begin{align*}
\overline{\mathbf u_i}\,={}&(c_1c_2\cdots c_p)\cdot(x_{01}^2\cdots x_{0q_0}^2z_{01}\cdots z_{0h_0})\cdot (s_1x_{11}^2\cdots x_{1q_1}^2z_{11}\cdots z_{1h_1})\cdots\\
&\cdot\,(s_nx_{n1}^2\cdots x_{nq_n}^2z_{n1}\cdots z_{nh_n}).
\end{align*}
As we will see below, $\overline{\mathbf u_i}$ is nothing but the mentioned above ``canonical form'' of the \mbox{$(k-1)$}-block $\mathbf u_i$.
 
As in the proof of sufficiency in Proposition~\ref{word problem I_k}, we can verify that $\mathbf J_k^r$ satisfies the identity~\eqref{u = long word 12}. The definition of the set $X$ and words of the form $\mathbf a_j'$ imply that $z\in \con(\mathbf w_1)$ for any $z\in \con(\mathbf a_0'\mathbf a_1'\cdots\mathbf a_n')$. This implies that if $z\in Z_j$ then we can assume that $\occ_z(\mathbf u_i)=1$ because $\mathbf J_k^r$ satisfies the identity~\eqref{xyxzx=xyxz} by Lemma~\ref{identities in K}(ii). Then we can assume without loss of generality that $\ell_1(\mathbf u,z_{j1})<\ell_1(\mathbf u,z_{j2})<\cdots<\ell_1(\mathbf u,z_{jh_j})$. Since $z\in\con(\mathbf w_1)$ and $D(\mathbf u, z)>r$ for any $z\in\con(\mathbf a_j'')$, we apply Lemma~\ref{u'abu''=u'bau''}(i) with $m=r$ and obtain the identity
\begin{align*}
\mathbf u \approx{}&\mathbf w_1\cdot \mathbf u_i^\ast\cdot (x_{01}^2\cdots x_{0q_0}^2z_{01}\cdots z_{0h_0})\cdot (s_1x_{11}^2\cdots x_{1q_1}^2z_{11}\cdots z_{1h_1})\cdots\\
&\cdot\,(s_nx_{n1}^2\cdots x_{nq_n}^2z_{n1}\cdots z_{nh_n})\cdot\mathbf w_2
\end{align*}
holds in $\mathbf J_k^r$. As we have seen above, $\con(\mathbf u_i^\ast)\subseteq \con(\mathbf w_1)$. Then we can apply the identity~\eqref{xyxzx=xyxz} and obtain the word $\mathbf u_i^\ast$ is linear. Then Lemma~\ref{identities in K}(i) applies and we conclude that $\mathbf J_k^r$ satisfies the identities 
\begin{align*}
\mathbf u \approx{}&\mathbf w_1\cdot(c_1c_2\cdots c_p)\cdot(x_{01}^2\cdots x_{0q_0}^2z_{01}\cdots z_{0h_0})\cdot (s_1x_{11}^2\cdots x_{1q_1}^2z_{11}\cdots z_{1h_1})\cdots\\
&\cdot\,(s_nx_{n1}^2\cdots x_{nq_n}^2z_{n1}\cdots z_{nh_n})\cdot\mathbf w_2\\
={}&\mathbf w_1\overline{\mathbf u_i}\,\mathbf w_2.
\end{align*}
 
So, as in the proof of Proposition~\ref{word problem F_k,K}(i), using identities which hold in the variety $\mathbf J_k^r$, we can replace the \mbox{$(k-1)$}-blocks $\mathbf u_i$ of \textbf u successively, one after another, by the ``canonical form'' $\overline{\mathbf u_i}$ for $i=m,m-1,\dots,0$. Then the variety $\mathbf J_k^r$ satisfies the identities~\eqref{u = canonical form}. Put $\mathbf u^\sharp=t_0\,\overline{\mathbf u_0}\,t_1\,\overline{\mathbf u_1}\,\cdots t_m\,\overline{\mathbf u_m}$\,.
 
One can return to the word \textbf v. By Lemma~\ref{k-equivalent}, the \mbox{$(k-1)$}-decomposition of $\bf v$ has the form~\eqref{t_0v_0t_1v_1 ... t_mv_m}. Furthermore, the claim~\eqref{identities in I_k} with $\ell=k$ and Lemma~\ref{k-equivalent} imply that~\eqref{presentation for t_iv_i} is a representation of $t_i\mathbf v_i$ as the product of alternating $k$-dividers $s_0,s_1,\dots,s_n$ and $k$-blocks $\mathbf b_0, \mathbf b_1,\dots,\mathbf b_n$. The claim~\eqref{identities in I_k} with $\ell=k$ implies that
$$
X_j=\bigl\{x\in \con(\mathbf b_j) \mid s_j=h_1^k(\mathbf v,x)\bigr\}
$$
for all $j=0,1,\dots,n$. Put $\mathbf b_j'=(\mathbf b_j)_X$. In view of the claim~\eqref{eq V subseteq J_k^r} with $\ell=k$ and $m=r$, we have
$$
Z_j=\{z\in \con(\mathbf b_j')\mid D(\mathbf v,z)\le r\}
$$
for all $j=0,1,\dots,n$. Put $\mathbf b_j''=(\mathbf b_j')_Z$. The claim~\eqref{eq the same l-dividers} with $\ell=k$ implies that $j$th occurrence of a letter $x$ in $\bf u$ lies in the \mbox{$(k-1)$}-block $\mathbf u_i$ if and only if $j$th occurrence of a letter $x$ in $\bf v$ lies in the \mbox{$(k-1)$}-block $\mathbf v_i$ for any $x$ and any $j=1,2$. Also, Lemma~\ref{identities in K}(ii) allows us to assume that if the first and the second occurrences of the letter $x$ in $\bf u$ do not lie in the \mbox{$(k-1)$}-block $\mathbf u_i$ then this letter does not occur in $\mathbf u_i$. Then $\con(\mathbf u_i^\ast)=\con(\mathbf b_0''\mathbf b_1''\cdots\mathbf b_n'')$. This implies that the \mbox{$(k-1)$}-blocks $\mathbf u_i$ and $\mathbf v_i$ have the same ``canonical form''. Repeating literally arguments given above, we obtain the variety $\mathbf J_k^r$ satisfies the identities $\mathbf v\approx\mathbf u^\sharp\approx\mathbf u$.
\end{proof}
 
Now we are well prepared to quickly complete the proof of Lemma~\ref{I_k or over J_k^1}. Let $\mathbf I_k\subset\mathbf{X\subseteq F}_{k+1}$. We have to verify that $\mathbf{X \supseteq J}_k^1$. Arguing by contradiction, suppose that ${\bf J}_k^1 \nsubseteq {\bf X}$. Then there exists an identity ${\bf u}\approx {\bf v}$ that holds in ${\bf X}$ but does not hold in ${\bf J}_k^1$. Then Propositions~\ref{word problem I_k} and~\ref{word problem J_k^r} and the inclusion ${\bf I}_k \subset {\bf X}$ together imply that the claims~\eqref{sim(u)=sim(v) & mul(u)=mul(v)},~\eqref{eq the same l-dividers} and~\eqref{identities in I_k} are true, while the claim~\eqref{eq V subseteq J_k^r} with $m=1$ is false. Then Lemma~\ref{V subseteq I_k or J_k^r}(i) implies that ${\bf X} \subseteq {\bf I}_k$, a contradiction. Lemma~\ref{I_k or over J_k^1} is proved.\qed
 
\subsubsection{If $\mathbf J_k^m\subset\mathbf X\subseteq\mathbf F_{k+1}$ with $1\le m<k$ then $\mathbf J_k^{m+1}\subseteq\mathbf X$}
\label{structure of [F_k,F_{k+1}] 4 step}
 
The fourth step in the verification of the claim~4) of Proposition~\ref{L(K)} is the following
 
\begin{lemma}
\label{J_k^m or over J_k^{m+1}}
If $\mathbf X$ is a monoid variety such that $\mathbf X \in [\mathbf J_k^m, \mathbf F_{k + 1}]$ for some $1\le m<k$ then either $\mathbf{X = J}_k^m$ or $\mathbf{X \supseteq J}_k^{m+1}$.
\end{lemma}
 
\begin{proof}
Let $1\le m<k$, $\mathbf J_k^m\subset\mathbf{X\subseteq F}_{k+1}$ and ${\bf J}_k^{m+1} \nsubseteq {\bf X}$. Then there exists an identity ${\bf u}\approx {\bf v}$ that holds in ${\bf X}$ but does not hold in ${\bf J}_k^{m+1}$. Then Proposition~\ref{word problem J_k^r} and the inclusion ${\bf J}_k^m \subset {\bf X}$ imply that the claims~\eqref{sim(u)=sim(v) & mul(u)=mul(v)},~\eqref{eq the same l-dividers}, \eqref{identities in I_k} and~\eqref{eq V subseteq J_k^r} with $\ell=k$ are true, while the claim
$$
\text{if }x\in \con({\bf u})\text{ and }D({\bf u},x)\le m+1\text{ then }h_2^k({\bf u},x)= h_2^k({\bf v},x)
$$
is false. Then Lemma~\ref{V subseteq I_k or J_k^r}(ii) implies that ${\bf X} \subseteq {\bf J}_k^m$, a contradiction. We see that either $\mathbf{X=J}_k^m$ or $\mathbf J_k^{m+1}\subseteq\mathbf X$.
\end{proof}
 
\subsubsection{The interval $[\mathbf J_k^k, \mathbf F_{k+1}]$ consists of $\mathbf J_k^k$ and $\mathbf F_{k+1}$ only}
\label{structure of [F_k,F_{k+1}] 5 step}
 
The fifth step in the verification of the claim~4) of Proposition~\ref{L(K)} is the following
 
\begin{lemma}
\label{J_k^k or F_{k+1}}
If $\mathbf X$ is a monoid variety such that $\mathbf X \in [\mathbf J_k^k, \mathbf F_{k + 1}]$ then either $\mathbf{X = J}_k^k$ or $\mathbf{X = F}_{k+1}$.
\end{lemma}
 
\begin{proof}
Suppose that $\mathbf J_k^k\subset\mathbf{X\subset F}_{k+1}$. Since ${\bf F}_{k+1} \nsubseteq {\bf X}$, there exists an identity ${\bf u}\approx {\bf v}$ that holds in ${\bf X}$ but does not hold in ${\bf F}_{k+1}$. Propositions~\ref{word problem F_k,K}(i) and~\ref{word problem J_k^r} and the inclusion ${\bf J}_k^k \subset {\bf X}$ together imply that the claims~\eqref{sim(u)=sim(v) & mul(u)=mul(v)},~\eqref{eq the same l-dividers}, \eqref{identities in I_k} and the claim~\eqref{eq V subseteq J_k^r} with $\ell=m = k$ are true, while $h_2^k({\bf u},x)\ne h_2^k({\bf v},x)$ for some letter $x\in \con({\bf u})$ such that $D({\bf u},x)>k$. Then we apply Lemma~\ref{V subseteq E or J_(ell-1)^(ell-1)} for the variety ${\bf F}_{k+1}$ and obtain $\mathbf{X\subseteq J}_k^k$, a contradiction.
\end{proof}
 
\subsubsection{All inclusions are strict}
\label{structure of [F_k,F_{k+1}] 6 step}
 
Here we are going to verify the inclusions~\eqref{[F_k,F_{k+1}]}. To achieve this goal, we use Lemma~\ref{depth and index} and Table~\ref{k-decomposition of eight words} without explicit references. We note that non-strict inclusions~\eqref{non-strict inclusions} are true by Lemma~\ref{between F_k and F_{k+1}}. If $\mathbf u\approx\mathbf v$ is the identity $\alpha_k$ then $D(\mathbf u,x_k)=k$ but $h_1^k(\mathbf u,x_k)=\lambda$ and $h_1^k(\mathbf v,x_k)=y_k$. Then Proposition~\ref{word problem H_k} implies that $\mathbf F_k\subset \mathbf H_k$. Suppose that the identity $\mathbf u\approx\mathbf v$ coincides with the identity $\beta_k$. Then $h_1^k(\mathbf u,x)=\lambda$, while $h_1^k(\mathbf v,x)=x_k$. We apply Proposition~\ref{word problem I_k} and obtain $\mathbf H_k\subset \mathbf I_k$. Let now $\mathbf u\approx\mathbf v$ be equal $\gamma_k$. In this case $D(\mathbf u,y_1)=1$ but $h_2^k(\mathbf u,y_1)=y_0$ and $h_2^k(\mathbf v,y_1)=x_k$. In view of Proposition~\ref{word problem J_k^r}, $\mathbf I_k\subset \mathbf J_k^1$. Suppose now that $\mathbf u\approx\mathbf v$ coincides with the identity $\delta_k^m$ for some $1\le m<k$. Then $D(\mathbf u,y_{m+1})=m+1$ but $h_2^k(\mathbf u,y_{m+1})=y_m$ and $h_2^k(\mathbf v,y_{m+1})=x_k$. Now we apply Proposition~\ref{word problem J_k^r} again and obtain $\mathbf J_k^m\subset \mathbf J_k^{m+1}$. Finally, suppose that $\mathbf u\approx\mathbf v$ is the identity $\delta_k^k$. Since $h_2^k(\mathbf u,y_{k+1})=y_k$ and $h_2^k(\mathbf v,y_{k+1})=x_k$, Proposition~\ref{word problem F_k,K}(i) implies that $\mathbf J_k^k\subset \mathbf F_{k+1}$.
 
Thus, we have proved the inclusions~\eqref{[F_k,F_{k+1}]}. Therefore, the varieties $\mathbf F_k$, $\mathbf H_k$, $\mathbf I_k$, $\mathbf J_k^1$, $\mathbf J_k^2$, \dots, $\mathbf J_k^k$ and $\mathbf F_{k+1}$ are pairwise different. This fact and Lemmas~\ref{F_k or over H_k},~\ref{H_k or over I_k},~\ref{I_k or over J_k^1},~\ref{J_k^m or over J_k^{m+1}} (with $m=1,2,\dots,k-1$) and~\ref{J_k^k or F_{k+1}} together implies the claim~4) of Proposition~\ref{L(K)}. In view of Lemma~\ref{L(LRB+C_2)}(ii) and the results of Subsections~\ref{sufficiency: K - red to [E,K]} and~\ref{sufficiency: K - red to [F_k,F_{k+1}]}, we complete the proof of Proposition~\ref{L(K)}.\qed
 
\medskip
 
Lemmas~\ref{L(D)} and~\ref{L(BM)}(ii), Corollary~\ref{L(L)}, Propositions~\ref{L(C_n)},~\ref{L(N)} and~\ref{L(K)}, and the dual of Propositions \ref{L(N)} and~\ref{L(K)} together imply the ``if'' part of Theorem~\ref{main result}.
 
Recall that the ``only if'' part of Theorem~\ref{main result} was verified in Section~\ref{necessity}. Thus, Theorem~\ref{main result} is completely proved.\qed
 
\section{Corollaries}
\label{corollaries}
 
First of all, we indicate the exhaustive list of non-group chain varieties of monoids. Theorem~\ref{main result} together with Lemmas~\ref{L(D)} and~\ref{L(BM)}(ii), Corollary~\ref{L(L)}, Propositions~\ref{L(C_n)},~\ref{L(N)} and~\ref{L(K)}, and the dual of Propositions \ref{L(N)} and~\ref{L(K)} implies the following
 
\begin{corollary}
\label{list}
The varieties ${\bf C}_n$, ${\bf D}_k$, ${\bf D}$, ${\bf E}$, $\overleftarrow{{\bf E}}$, ${\bf F}_k$, $\overleftarrow{{\bf F}_k}$, ${\bf H}_k$, $\overleftarrow{{\bf H}_k}$, ${\bf I}_k$, $\overleftarrow{{\bf I}_k}$, ${\bf J}_k^m$, $\overleftarrow{{\bf J}_k^m}$, ${\bf K}$, $\overleftarrow{{\bf K}}$, ${\bf L}$, ${\bf LRB}$, ${\bf M}$, $\overleftarrow{{\bf M}}$, ${\bf N}$, $\overleftarrow{{\bf N}}$, ${\bf RRB}$, ${\bf SL}$ where $n\ge 2$, $k\in\mathbb N$ and $1\le m\le k$ and only they are non-group chain varieties of monoids.\qed
\end{corollary}
 
The set of all non-group chain varieties of monoids ordered by inclusion together with the variety \textbf T is shown in Fig.~\ref{all chain mon}. It is interesting to compare this figure with the diagram of the partially ordered set of all non-group chain varieties of semigroups (as we have already mentioned in Section~\ref{introduction}, such varieties were completely determined in~\cite{Sukhanov-82}). This diagram is shown in Fig.~\ref{all chain sem} where $\mathbf{LZ}=\var\{xy\approx x\}$, $\mathbf{RZ}=\var\{xy\approx y\}$, $\mathbf{ZM}=\var\{xy\approx 0\}$, $\mathbf N_k=\var\{x^2\approx x_1x_2\cdots x_k\approx 0,\,xy\approx yx\}$ for all $k\ge3$, $\mathbf N_\omega=\var\{x^2\approx 0,\,xy\approx yx\}$, $\mathbf N_3^2=\var\{x^2\approx xyz\approx 0\}$ and $\mathbf N_3^c=\var\{xyz\approx 0,\,xy\approx yx\}$ (here $\var\Sigma$ means the semigroup variety given by $\Sigma$; as is usually done when considering semigroup varieties, we write the symbolic identity $\mathbf w\approx 0$ as a short form of the identity system $\mathbf wx\approx x\mathbf w\approx\mathbf w$ where $x\notin \con(\mathbf w)$).
 
\begin{figure}[htb]
\unitlength=1mm
\linethickness{0.4pt}
\begin{center}
\begin{picture}(81,161)
\put(5,45){\circle*{1.33}}
\put(5,55){\circle*{1.33}}
\put(5,65){\circle*{1.33}}
\put(5,75){\circle*{1.33}}
\put(5,85){\circle*{1.33}}
\put(5,95){\circle*{1.33}}
\put(5,105){\circle*{1.33}}
\put(5,115){\circle*{1.33}}
\put(5,125){\circle*{1.33}}
\put(5,135){\circle*{1.33}}
\put(5,145){\circle*{1.33}}
\put(5,155){\circle*{1.33}}
\put(15,25){\circle*{1.33}}
\put(15,55){\circle*{1.33}}
\put(25,55){\circle*{1.33}}
\put(25,65){\circle*{1.33}}
\put(35,35){\circle*{1.33}}
\put(35,45){\circle*{1.33}}
\put(35,55){\circle*{1.33}}
\put(35,65){\circle*{1.33}}
\put(35,75){\circle*{1.33}}
\put(35,85){\circle*{1.33}}
\put(35,95){\circle*{1.33}}
\put(35,105){\circle*{1.33}}
\put(35,115){\circle*{1.33}}
\put(35,125){\circle*{1.33}}
\put(35,135){\circle*{1.33}}
\put(35,145){\circle*{1.33}}
\put(35,155){\circle*{1.33}}
\put(45,55){\circle*{1.33}}
\put(45,65){\circle*{1.33}}
\put(55,25){\circle*{1.33}}
\put(65,45){\circle*{1.33}}
\put(65,55){\circle*{1.33}}
\put(65,65){\circle*{1.33}}
\put(65,75){\circle*{1.33}}
\put(65,85){\circle*{1.33}}
\put(65,95){\circle*{1.33}}
\put(65,105){\circle*{1.33}}
\put(65,115){\circle*{1.33}}
\put(65,125){\circle*{1.33}}
\put(65,135){\circle*{1.33}}
\put(65,145){\circle*{1.33}}
\put(65,155){\circle*{1.33}}
\put(75,5){\circle*{1.33}}
\put(75,15){\circle*{1.33}}
\put(75,25){\circle*{1.33}}
\put(75,35){\circle*{1.33}}
\put(75,45){\circle*{1.33}}
\put(75,55){\circle*{1.33}}
\put(75,65){\circle*{1.33}}
\put(75,75){\circle*{1.33}}
\put(75,85){\circle*{1.33}}
\put(75,95){\circle*{1.33}}
\put(75,105){\circle*{1.33}}
\put(75,115){\circle*{1.33}}
\put(75,125){\circle*{1.33}}
\put(75,135){\circle*{1.33}}
\put(75,145){\circle*{1.33}}
\put(5,45){\line(0,1){102}}
\put(5,45){\line(3,-1){30}}
\put(15,25){\line(6,-1){60}}
\put(15,55){\line(2,-1){20}}
\put(25,55){\line(0,1){10}}
\put(25,55){\line(1,-1){10}}
\put(35,35){\line(0,1){112}}
\put(35,35){\line(4,-1){40}}
\put(35,35){\line(3,1){30}}
\put(35,45){\line(1,1){10}}
\put(45,55){\line(0,1){10}}
\put(55,25){\line(2,-1){20}}
\put(65,45){\line(0,1){102}}
\put(75,5){\line(0,1){142}}
\put(5,152){\makebox(0,0)[cc]{$\vdots$}}
\put(35,152){\makebox(0,0)[cc]{$\vdots$}}
\put(65,152){\makebox(0,0)[cc]{$\vdots$}}
\put(75,152){\makebox(0,0)[cc]{$\vdots$}}
\put(76,25){\makebox(0,0)[lc]{${\bf C}_2$}}
\put(76,35){\makebox(0,0)[lc]{${\bf C}_3$}}
\put(76,45){\makebox(0,0)[lc]{${\bf C}_4$}}
\put(76,55){\makebox(0,0)[lc]{${\bf C}_5$}}
\put(76,65){\makebox(0,0)[lc]{${\bf C}_6$}}
\put(76,75){\makebox(0,0)[lc]{${\bf C}_7$}}
\put(76,85){\makebox(0,0)[lc]{${\bf C}_8$}}
\put(76,95){\makebox(0,0)[lc]{${\bf C}_9$}}
\put(76,105){\makebox(0,0)[lc]{${\bf C}_{10}$}}
\put(76,115){\makebox(0,0)[lc]{${\bf C}_{11}$}}
\put(76,125){\makebox(0,0)[lc]{${\bf C}_{12}$}}
\put(76,135){\makebox(0,0)[lc]{${\bf C}_{13}$}}
\put(76,145){\makebox(0,0)[lc]{${\bf C}_{14}$}}
\put(35,158){\makebox(0,0)[cc]{${\bf D}$}}
\put(35,32){\makebox(0,0)[cc]{${\bf D}_1$}}
\put(36,44){\makebox(0,0)[lc]{${\bf D}_2$}}
\put(36,55){\makebox(0,0)[lc]{${\bf D}_3$}}
\put(36,65){\makebox(0,0)[lc]{${\bf D}_4$}}
\put(36,75){\makebox(0,0)[lc]{${\bf D}_5$}}
\put(36,85){\makebox(0,0)[lc]{${\bf D}_6$}}
\put(36,95){\makebox(0,0)[lc]{${\bf D}_7$}}
\put(36,105){\makebox(0,0)[lc]{${\bf D}_8$}}
\put(36,115){\makebox(0,0)[lc]{${\bf D}_9$}}
\put(36,125){\makebox(0,0)[lc]{${\bf D}_{10}$}}
\put(36,135){\makebox(0,0)[lc]{${\bf D}_{11}$}}
\put(36,145){\makebox(0,0)[lc]{${\bf D}_{12}$}}
\put(5,42){\makebox(0,0)[cc]{${\bf E}$}}
\put(65,41.5){\makebox(0,0)[cc]{$\overleftarrow{{\bf E}}$}}
\put(4,55){\makebox(0,0)[rc]{${\bf F}_1$}}
\put(64,55){\makebox(0,0)[rc]{$\overleftarrow{{\bf F}_1}$}}
\put(4,95){\makebox(0,0)[rc]{${\bf F}_2$}}
\put(64,95){\makebox(0,0)[rc]{$\overleftarrow{{\bf F}_2}$}}
\put(4,145){\makebox(0,0)[rc]{${\bf F}_3$}}
\put(64,145){\makebox(0,0)[rc]{$\overleftarrow{{\bf F}_3}$}}
\put(4,65){\makebox(0,0)[rc]{${\bf H}_1$}}
\put(64,65){\makebox(0,0)[rc]{$\overleftarrow{{\bf H}_1}$}}
\put(4,105){\makebox(0,0)[rc]{${\bf H}_2$}}
\put(64,105){\makebox(0,0)[rc]{$\overleftarrow{{\bf H}_2}$}}
\put(4,75){\makebox(0,0)[rc]{${\bf I}_1$}}
\put(64,75){\makebox(0,0)[rc]{$\overleftarrow{{\bf I}_1}$}}
\put(4,115){\makebox(0,0)[rc]{${\bf I}_2$}}
\put(64,115){\makebox(0,0)[rc]{$\overleftarrow{{\bf I}_2}$}}
\put(4,85){\makebox(0,0)[rc]{${\bf J}_1^1$}}
\put(64,85){\makebox(0,0)[rc]{$\overleftarrow{{\bf J}_1^1}$}}
\put(4,125){\makebox(0,0)[rc]{${\bf J}_2^1$}}
\put(64,125){\makebox(0,0)[rc]{$\overleftarrow{{\bf J}_2^1}$}}
\put(4,135){\makebox(0,0)[rc]{${\bf J}_2^2$}}
\put(64,135){\makebox(0,0)[rc]{$\overleftarrow{{\bf J}_2^2}$}}
\put(5,158){\makebox(0,0)[cc]{${\bf K}$}}
\put(65,158.5){\makebox(0,0)[cc]{$\overleftarrow{{\bf K}}$}}
\put(15,58){\makebox(0,0)[cc]{${\bf L}$}}
\put(14,25){\makebox(0,0)[rc]{${\bf LRB}$}}
\put(22,55){\makebox(0,0)[cc]{${\bf M}$}}
\put(48,55.5){\makebox(0,0)[cc]{$\overleftarrow{{\bf M}}$}}
\put(25,68){\makebox(0,0)[cc]{${\bf N}$}}
\put(45,68.5){\makebox(0,0)[cc]{$\overleftarrow{{\bf N}}$}}
\put(54,25){\makebox(0,0)[rc]{${\bf RRB}$}}
\put(76,15){\makebox(0,0)[lc]{${\bf SL}$}}
\put(75,2){\makebox(0,0)[cc]{$\bf T$}}
\end{picture}
\end{center}
\caption{All non-group chain varieties of monoids}
\label{all chain mon}
\end{figure}
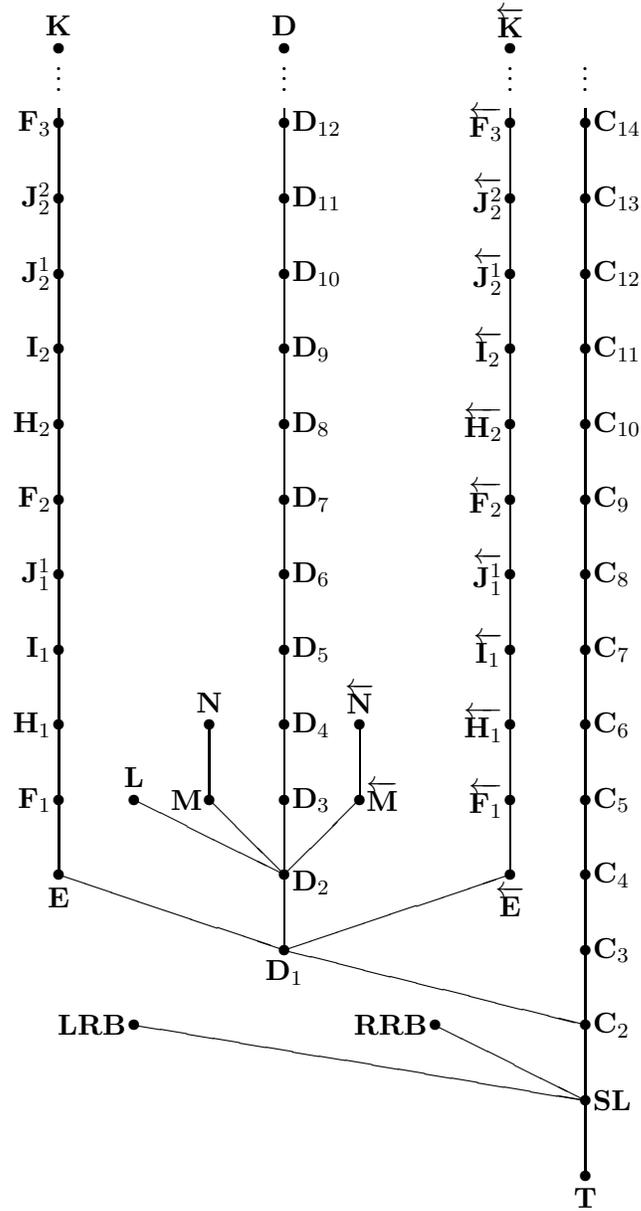
 
\begin{figure}[htb]
\unitlength=1mm
\linethickness{0.4pt}
\begin{center}
\begin{picture}(46,60)
\put(3,15){\circle*{1.33}}
\put(13,15){\circle*{1.33}}
\put(18,5){\circle*{1.33}}
\put(23,15){\circle*{1.33}}
\put(23,35){\circle*{1.33}}
\put(33,15){\circle*{1.33}}
\put(33,25){\circle*{1.33}}
\put(33,35){\circle*{1.33}}
\put(33,45){\circle*{1.33}}
\put(33,55){\circle*{1.33}}
\put(43,35){\circle*{1.33}}
\put(3,15){\line(3,-2){15}}
\put(13,15){\line(1,-2){5}}
\put(18,5){\line(1,2){5}}
\put(18,5){\line(3,2){15}}
\put(23,35){\line(1,-1){10}}
\put(33,15){\line(0,1){32}}
\put(33,25){\line(1,1){10}}
\put(33,52){\makebox(0,0)[cc]{$\vdots$}}
\put(3,18){\makebox(0,0)[cc]{\textbf{LZ}}}
\put(34,24){\makebox(0,0)[lc]{$\mathbf N_3$}}
\put(23,38){\makebox(0,0)[cc]{$\mathbf N_3^2$}}
\put(43,38){\makebox(0,0)[cc]{$\mathbf N_3^c$}}
\put(34,35){\makebox(0,0)[lc]{$\mathbf N_4$}}
\put(34,45){\makebox(0,0)[lc]{$\mathbf N_5$}}
\put(33,58){\makebox(0,0)[cc]{$\mathbf N_\omega$}}
\put(13,18){\makebox(0,0)[cc]{\textbf{RZ}}}
\put(23,18){\makebox(0,0)[cc]{\textbf{SL}}}
\put(18,2){\makebox(0,0)[cc]{\textbf T}}
\put(34,15){\makebox(0,0)[lc]{\textbf{ZM}}}
\end{picture}
\end{center}
\caption{All non-group chain varieties of semigroups}
\label{all chain sem}
\end{figure}
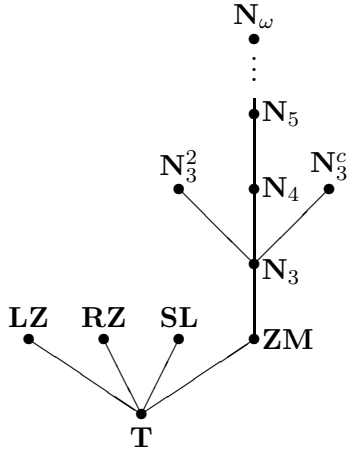
 
We see that, out of the group case, there are~1 countably infinite series and~6 ``sporadic'' chain semigroup varieties, but~10 countably infinite series and~12 ``sporadic'' chain monoid varieties. Namely, we have the countably infinite series $\mathbf N_k$ (including \textbf{ZM} as $\mathbf N_2$) and sporadic varieties \textbf{LZ}, \textbf{RZ}, \textbf{SL}, $\mathbf N_3^2$, $\mathbf N_3^c$, $\mathbf N_\omega$ in the semigroup case, and countably infinite series ${\bf C}_n$ (including \textbf{SL} as $\mathbf C_1$), ${\bf D}_k$, ${\bf F}_k$, $\overleftarrow{{\bf F}_k}$, ${\bf H}_k$, $\overleftarrow{{\bf H}_k}$, ${\bf I}_k$, $\overleftarrow{{\bf I}_k}$, ${\bf J}_k^m$, $\overleftarrow{{\bf J}_k^m}$ and sporadic varieties ${\bf D}$, ${\bf E}$, $\overleftarrow{{\bf E}}$, ${\bf K}$, $\overleftarrow{{\bf K}}$, ${\bf L}$, ${\bf LRB}$, ${\bf M}$, $\overleftarrow{{\bf M}}$, ${\bf N}$, $\overleftarrow{{\bf N}}$, ${\bf RRB}$ in the monoid case. One can say that the number of non-group chain varieties in the case of monoids is much larger (in some informal sense) than in the case of semigroups. Consequently, the partially ordered set of non-group chain varieties in the former case is much more complicated than in the latter one.
 
As we have already mentioned in Section~\ref{introduction}, a non-group chain variety of semigroups is contained in a maximal chain variety, while this is not the case for monoid varieties. The following two corollaries indicate cases when the analog of the semigroup statement is true. Fig.~\ref{all chain mon} shows that the following is true.
 
\begin{corollary}
\label{max chain iff}
A non-group chain variety of monoids $\mathbf V$ is contained in some maximal chain variety if and only if $\mathbf C_3\nsubseteq\mathbf V$.\qed
\end{corollary}
 
Theorem~\ref{main result} shows that commutative non-group chain varieties of monoids are exhausted by the varieties \textbf{SL} and $\mathbf C_n$ with $n\ge 2$. This claim and Fig.~\ref{all chain mon} imply the following
 
\begin{corollary}
\label{max chain non-commut}
A non-commutative non-group chain variety of monoids is contained in some maximal chain variety.\qed
\end{corollary}
 
In the following corollary we mention the variety \textbf O introduced in Subsection~\ref{necessity: over L}.
 
\begin{corollary}
\label{does not contain just non-chain}
Let $\mathbf X$ be a monoid variety with $\mathbf{L\subset X\subseteq O}$. Then ${\bf X}$ is not a chain variety and does not contain a just non-chain subvariety.
\end{corollary}
 
\begin{proof}
Theorem~\ref{main result} immediately implies that there are no chain monoid varieties that properly contain \textbf L, whence ${\bf X}$ is not a chain variety. It remains to check that \textbf X does not contain a just non-chain subvariety. Arguing by contradiction, suppose that ${\bf X}$ contains such a subvariety \textbf Y. In view of Theorem~\ref{main result}, any chain subvariety of the variety ${\bf O}$ is contained in ${\bf L}$. In particular, ${\bf O}$ (and therefore, \textbf Y) does not contain incomparable chain subvarieties. On the other hand, being a non-chain variety, \textbf Y contains at least two incomparable subvarieties. These two varieties are proper subvarieties of \textbf Y, whence they are chain varieties. We have a contradiction.
\end{proof}
 
The following question seems to be very interesting.
 
\begin{question}
\label{is O largest?}
Is it true that a non-chain non-group monoid variety \textbf X with $\mathbf{X\nsubseteq O}$ contains a just non-chain subvariety?
\end{question}
 
Recall that a variety of universal algebras is called \emph{locally finite} if all its finitely generated members are finite. A variety is called \emph{finitely generated} if it is generated by a finite algebra. Clearly, if a variety is contained in some finitely generated variety then it is locally finite.
 
\begin{corollary}
\label{locally finite}
An arbitrary non-group chain monoid variety is contained in some finitely generated variety; in particular, it is locally finite.
\end{corollary}
 
\begin{proof}
Clearly, it suffices to verify that each of the varieties listed in Theorem~\ref{main result} is contained in a finitely generated variety. It is well known that a proper variety of band monoids is finitely generated~\cite{Gerhard-72}. In particular, the varieties \textbf{LRB} and \textbf{RRB} have this property. It is evident that the monoid $S(\mathbf w)$ is finite for any word \textbf w. Then Lemmas~\ref{C_{n+1}=var S(x^n)} and~\ref{L = var S(xzxyty)} provide the required conclusion for the varieties $\mathbf C_n$ and \textbf L respectively. The fact that the variety $\overleftarrow{\mathbf N}$ is finitely generated follows from Example~1 in Erratum to~\cite{Jackson-05}. By symmetry, it remains to consider the varieties $\mathbf D$ and \textbf K.
 
The variety $\mathbf D$ is not finitely generated by~\cite[Theorem~2]{Lee-13}, but it is shown in~\cite[Example~5.3]{Lee-14a} that $\mathbf D$ is a subvariety of the variety generated by the well-known 6-element Brandt monoid $B_2^1=B_2\cup\{1\}$ where
$$
B_2=\langle a,b\mid a^2=b^2=0,\,aba=a,\,bab=b\rangle=\{a,b,ab,ba,0\}.
$$
 
Finally, it is easy to see that if a monoid $M$ belongs to \textbf K and consists of $k$ elements then $M$ satisfies the identity $\alpha_k$. Therefore, any finitely generated subvariety of \textbf K is contained in $\mathbf F_k$ for some $k$. In particular, the variety \textbf K is not finitely generated. But Lemma~\ref{identities in K} implies that $\mathbf K\subseteq\var\{xyxzx\approx xyxz,\,\sigma_2\}$. To complete our considerations, it remains to note that the variety $\var\{xyxzx\approx xyxz,\,\sigma_2\}$ is generated by the 5-element monoid
$$
\langle a,b\mid a^2=ab=a,\,b^2a=b^2\rangle\cup\{1\}=\{a,b,ba,b^2,1\}.
$$
This fact is proved in~\cite[Corollary~6.6]{Lee-Li-11}.
\end{proof}
 
Analog of Corollary~\ref{locally finite} for arbitrary chain varieties of monoids (including group ones) does not hold. Indeed, as we have already mentioned in Section~\ref{introduction}, it is verified in~\cite{Kozhevnikov-12} that there are uncountably many non-locally finite chain varieties of groups. But explicit examples of such varieties have not yet been specified anywhere.

\subsection*{Acknowledgments}
\label{acknowledg}

The authors express their deep appreciation to Dr. Edmond W.H.~Lee for his careful reading of the manuscript and numerous suggestions for its modification, which led to a radical processing and a significant improvement of the initial version of the manuscript. The authors are also pleased to express their gratitude to Dr. Olga Sapir, who also read the manuscript and gave a number of valuable advices for its improvement.


\begin{thebibliography}{99}
\label{bibl}
\bibitem{Almeida-94}
J. Almeida, \emph{Finite Semigroups and Universal Algebra}, World Scientific, Singapore, 1994.
\bibitem{Artamonov-78}
V.A. Artamonov, \emph{Chain varieties of groups}, Trudy Seminara Imeni I.G.Petrovskogo (Proc. of the Seminar Named After I.G.Petrovsky), \textbf 3 (1978), 3--8 [Russian].
\bibitem{Burris-Nelson-71a}
S. Burris and E. Nelson, \emph{Embedding the dual of $\Pi_m$ in the lattice of equational classes of commutative semigroups}, Proc. Amer. Math. Soc., \textbf{30} (1971), 37--39.
\bibitem{Burris-Nelson-71b}
S. Burris and E. Nelson, \emph{Embedding the dual of $\Pi_\infty$ in the lattice of equational classes of semigroups}, Algebra Universalis, \textbf 1 (1971), 248--254.
\bibitem{Gerhard-72}
J.A. Gerhard, \emph{Some subdirectly irreducible idempotent semigroups}, Semigroup Forum, \textbf 5 (1972), 362--369.
\bibitem{Gusev-18}
S.V. Gusev, \emph{On the lattice of overcommutative varieties of monoids}, Izv. Vyssh. Uchebn. Zaved. Matem., No.5 (2018), 28--32 [Russian; Engl. translation is available at: https://arxiv.org/abs/1702.08749].
\bibitem{Head-68}
T.J. Head, \emph{The varieties of commutative monoids}, Nieuw Arch. Wiskunde. III~Ser., \textbf{16} (1968), 203--206.
\bibitem{Jackson-05}
M. Jackson, \emph{Finiteness properties of varieties and the restriction to finite algebras}, Semigroup Forum, \textbf{70} (2005), 154--187; \emph{Erratum to ``Finiteness properties of varieties and the restriction to finite algebras''}, Semigroup Forum, \textbf{96} (2018), 197--198.
\bibitem{Jackson-Lee-17+}
M. Jackson and E. W.H. Lee, \emph{Monoid varieties with extreme properties}, Trans. Amer. Math. Soc., doi: https://doi.org/10.1090/tran/7091.
\bibitem{Jackson-Sapir-00}
M. Jackson and O. Sapir, \emph{Finitely based, finite sets of words}, Int. J. Algebra and Comput., \textbf{10} (2000), 683--708.
\bibitem{Kozhevnikov-12}
P.A. Kozhevnikov, \emph{On nonfinitely based varieties of groups of large prime exponent}, Commun. Algebra, \textbf{40} (2012), 2628--2644.
\bibitem{Lee-12a}
E. W.H. Lee, \emph{Varieties generated by $2$-testable monoids}, Studia Sci. Math. Hungar., \textbf{49} (2012), 366--389.
\bibitem{Lee-12b}
E. W.H. Lee, \emph{Maximal Specht varieties of monoids}, Moscow Math. J., \textbf{12} (2012), 787--802.
\bibitem{Lee-13}
E. W.H. Lee, \emph{Almost Cross varieties of aperiodic monoids with central idempotents}, Beitr\"age zur Algebra und Geometrie, \textbf{54} (2013), 121--129.
\bibitem{Lee-14a}
E. W.H. Lee, \emph{Inherently non-finitely generated varieties of aperiodic monoids with central idempotents}, Zapiski Nauchnykh Seminarov POMI (Notes of Scientific Seminars of the St.~Petersburg Branch of the Math. Institute of the Russ. Acad. of Sci.), \textbf{423} (2014), 166--182.
\bibitem{Lee-14b}
E. W.H. Lee, \emph{On certain Cross varieties of aperiodic monoids with central idempotents}, Results Math., \textbf{66} (2014), 491--510.
\bibitem{Lee-Li-11}
E. W.H. Lee and J.R. Li, \emph{Minimal non-finitely based monoids}, Dissert. Math., \textbf{475} (2011), 65~p.
\bibitem{Perkins-69}
P. Perkins, \emph{Bases for equational theories of semigroups}, J. Algebra, \textbf{11} (1969), 298--314.
\bibitem{Pollak-81}
Gy. Poll\'ak, \emph{Some lattices of varieties containing elements without cover}, Quad. Ric. Sci., \textbf{109} (1981), 91--96.
\bibitem{Sapir-15}
O. Sapir, \emph{Non-finitely based monoids}, Semigroup Forum, \textbf{90} (2015), 557--586.
\bibitem{Shevrin-Vernikov-Volkov-09}
L.N. Shevrin, B.M. Vernikov and M.V. Volkov, \emph{Lattices of semigroup varieties}, Izv. Vyssh. Uchebn. Zaved. Matem., No.3 (2009), 3--36 [Russian; Engl. translation: Russian Math. (Iz. VUZ), \textbf{53}, No.3 (2009), 1--28].
\bibitem{Sukhanov-82}
E.V. Sukhanov, \emph{Almost linear semigroup varieties}, Matem. Zametki, \textbf{32} (1982), 469--476 [Russian; Engl. translation: Math. Notes, \textbf{32} (1982), 714--717].
\bibitem{Vernikov-15}
B.M. Vernikov, \emph{Special elements in lattices of semigroup varieties}, Acta Sci. Math. (Szeged), \textbf{81} (2015), 79--109.
\bibitem{Wismath-86}
S.L. Wismath, \emph{The lattice of varieties and pseudovarieties of band monoids}, Semigroup Forum, \textbf{33} (1986), 187--198.
\end{thebibliography}
\end{document}